\documentclass[11pt]{amsart}
\usepackage{amssymb}
\usepackage{graphicx}                          
\usepackage{hyperref}

\newtheorem{lemma}{Lemma}[section]
\newtheorem{theorem}{Theorem}[section]

\newtheorem{construction}{Construction}[section]

\newenvironment{adfenumerate}{
\begin{enumerate}
\setlength{\itemsep}{0.5mm}
\setlength{\parskip}{0mm}
\setlength{\parsep}{0mm}
}{
\end{enumerate}
}

\newcommand{\adfmod}[1]{~(\mathrm{mod}~#1)}
\newcommand{\adfPENT}{\mathop{\mathrm{PENT}} }
\newcommand{\adfhat}[1]{\hat{#1}}

\newcommand{\adfhide}[1]{}

\begin{document}
\title{Generalized pentagonal geometries - II}
\author{A. D. Forbes}
\address{LSBU Business School,
London South Bank University,
103 Borough Road,
London SE1 0AA, UK.}
\email{anthony.d.forbes@gmail.com}
\author{C. G. Rutherford}
\address{LSBU Business School,
London South Bank University,
103 Borough Road,
London SE1 0AA, UK.}
\email{c.g.rutherford@lsbu.ac.uk}
\date{\today, version 6.18}

\maketitle
\begin{abstract}
A generalized pentagonal geometry PENT($k$,$r$,$w$) is a partial linear space, where
every line is incident with $k$ points,
every point is incident with $r$ lines, and
for each point, $x$, the set of points not collinear with $x$ forms the point set of
a Steiner system $S(2,k,w)$ whose blocks are lines of the geometry.
If $w = k$, the structure is called a pentagonal geometry and denoted by PENT($k$,$r$).
The deficiency graph of a PENT($k$,$r$,$w$) has as its vertices the points of the geometry, and
there is an edge between $x$ and $y$ precisely when $x$ and $y$ are not collinear.

Our primary objective is to investigate generalized pentagonal geometries PENT($k$,$r$,$w$) where the deficiency graph has girth 4.
We describe some construction methods, including a procedure that preserves deficiency graph connectedness, and
we prove a number of theorems regarding the existence spectra for $k = 3$ and various values of $w$.
In addition, we present some new PENT(4,$r$) (including PENT(4,25)) and PENT(5,$r$) with connected deficiency graphs.
Consequently, we prove that there exist pentagonal geometries PENT($k$,$r$) with deficiency graphs of girth at least 5
for $r \ge 13$, $r$ congruent to 1 modulo 4 when $k = 4$, and
for $r \ge 200000$, $r$ congruent to 0 or 1 modulo 5 when $k = 5$.
We conclude with a discussion of appropriately defined identifying codes for pentagonal geometries.
\end{abstract}


\section{Introduction}\label{sec:Introduction}
A {\em  partial linear space} is an incidence structure consisting of
a set of points $P$,
a set of lines $L$, and
an incidence relation $\mathcal{I} \subseteq P \times L$ such that
if $(p,\ell)$, $(q,\ell)$, $(p,m)$, $(q,m) \in \mathcal{I}$, then $p = q$ or $\ell = m$.
If $(p, \ell) \in \mathcal{I}$,
we say that $p$ is incident with $\ell$, or $\ell$ is incident with $p$.
However, it is often convenient to think of lines as sets of points in which case we
would use the language of set theory to describe the incidence or otherwise of $p$ and $\ell$.
A partial linear space is {\em uniform} if every line is incident with the same number of points,
and {\em regular} if every point is incident with the same number of lines.

A {\em Steiner system} $S(2,k,w)$ is an ordered pair $(W, \mathcal{B})$ such that
\begin{adfenumerate}
\item[(i)]$W$ is a set of {points} with $|W| = w$,
\item[(ii)]$\mathcal{B}$ is a set of $k$-subsets of $W$, called {blocks}, and
\item[(iii)]each pair $\{x,y\} \subseteq W$ is a subset of precisely one block of $\mathcal{B}$.
\end{adfenumerate}
We refer to Steiner systems generally as designs, and we recognize three special cases.
A {\em Steiner triple system} of order $w$, STS$(w)$, is an $S(2,3,w)$,
an {\em affine plane} of order $n$ is an $S(2,n,n^2)$, and
a {\em projective plane} of order $n$ is an $S(2, n+1, n^2 + n + 1)$.

A {\em generalized pentagonal geometry} is a uniform, regular partial linear space such that
for each point $x$, there is a Steiner system $(W_x, \mathcal{B}_x)$ where
$W_x$ is the set of points which are not collinear with $x$ and
the blocks of $\mathcal{B}_x$ are lines of the geometry.
By uniformity and regularity the number of points that are collinear with $x$ does not depend on $x$.
Therefore $|W_x|$ must also be constant.
If $k$ is the number of points incident with a given line,
$r$ is the number of lines incident with a given point
and $w$ is the point set size of the Steiner systems, then $k$, $r$ and $w$ are constants,
and we denote a generalized pentagonal geometry with these parameters by $\adfPENT(k,r,w)$.
If $w = k$, in which case the Steiner system $S(2,k,k)$ is just a single line,
the structure is called a {\em pentagonal geometry} of order $(k,r)$, \cite{BallBambergDevillersStokes2013},
and we use the notation $\adfPENT(k,r)$, \cite{GriggsStokes2016}.

To avoid triviality issues we assume $k \ge 2$ and $r \ge 1$.
The Steiner system $S(2,k,w)$ consisting of the points that are not collinear with point $x$ is
called the {\em opposite design} to $x$ and is denoted by $x^\mathrm{opp}$.
The notation $x^\mathrm{opp}$ may refer to the $S(2,k,w)$ itself, or its point set, or its block set,
whichever is appropriate from the context.
We sometimes use the term {\em opposite line} to describe any line belonging to $x^\mathrm{opp}$.
The lines of the geometry are also called blocks and we refer to the parameter $k$ of a $\adfPENT(k,r,w)$ as the block size.

The concept of a {pentagonal geometry} was introduced in \cite{BallBambergDevillersStokes2013} to provide a generalization
of the pentagon analogous to the generalization of the polygon as described in \cite{Tits1959} and \cite{FeitHigman1964}.
The geometry described by Ball, Bamberg, Devillers \& Stokes in \cite{BallBambergDevillersStokes2013}
is based on the observation that for each vertex $x$ of a pentagon, the two vertices that are not collinear with $x$ form a line.
According to their definition the pentagon is a pentagonal geometry of order $(2,2)$, or $\adfPENT(2,2)$.

Our generalization is based on the observation that a line containing $k$ points
is the single block of the Steiner system $S(2,k,k)$.
It seems natural, therefore, to extend the definition of $x^\mathrm{opp}$ to a Steiner system $S(2,k,w)$
since these designs have the same relevant property as a single line: any two points are collinear.
According to our definition the pentagon is a $\adfPENT(2,2,2)$.

The {\em deficiency graph} of a generalized pentagonal geometry $\adfPENT(k,r,w)$
has as its vertices the points of the geometry, and
there is an edge $x \sim y$ precisely when $x$ and $y$ are not collinear.
It is clear that the graph must be $w$-regular and triangle-free.

We adopt the notation {\em $(w,g)$-graph} for a graph that is $w$-regular and has girth $g$,
and {\em $(w,g^+)$-graph} for a graph that is $w$-regular and has girth at least $g$.
Thus we can say that a $\adfPENT(k,r,w)$ has a deficiency $(w,4^+)$-graph.
Also we recall some standard notation.
If $G$ is a graph, $V(G)$ is its vertex set, $E(G)$ is its edge set and we write $x \sim y$ if $\{x,y\} \in E(G)$.
If $x, y \in V(G)$, then
$d_G(x,y)$, or just $d(x,y)$ if the graph is clear from the context, is the distance between $x$ and $y$,
and $N_G(x)$ or $N(x)$ is the set of neighbours of $x$, $\{z: z \sim x\}$.

If $w \ge k \ge 2$ and there exists an $S(2,k,w)$, then there exists a $\adfPENT(k, (w - 1)/(k - 1), w)$.
In this generalized pentagonal geometry, which we regard as degenerate,
the lines are the blocks of two Steiner systems $S(2,k,w)$ with disjoint point sets, and
the deficiency graph is a complete bipartite graph $K_{w,w}$.
For any point $x$ in either one of the Steiner systems, $x^\mathrm{opp}$ is the other system.
A degenerate geometry can also occur in the line set of a larger generalized pentagonal geometry, and
we refer to the substructure as an {\em opposite design pair}, or
an {\em opposite line pair} when $w = k$, \cite{BallBambergDevillersStokes2013}.

For the general theory of pentagonal geometries as well as further background material, we refer the reader to \cite{BallBambergDevillersStokes2013}.
The subject was further developed in \cite{GriggsStokes2016}, \cite{Forbes2020P} and \cite{ForbesGriggsStokes2020}.
Generalized pentagonal geometries were introduced by the authors in \cite{ForbesRutherford2022}.

A $\adfPENT(2,r,w)$ is essentially just the complement of a $(w,4^+)$-graph with $r + w + 1$ vertices,
as explained in \cite{BallBambergDevillersStokes2013} for $w = k$ and in \cite{ForbesRutherford2022} for $w > k$.
So we have nothing further to say about this case, and
for the remainder of the paper we assume that $w \ge k \ge 3$.

\newcommand{\adfbullet}{\large$\star$}%
We identify six types of generalized pentagonal geometries, A, B, \dots, F,
according to the following scheme.
\begin{adfenumerate}
\item[\adfbullet]{\bf Type A}~
The deficiency graph has girth at least 5 and is connected.
One might consider these to be the nicest of the generalized pentagonal geometries.
Although \cite{Forbes2020P} gives a construction for infinitely many $\adfPENT(3,r)$,
no corresponding simple construction is known for $\adfPENT(k,r,w)$ of type A with $w > k \ge 3$ or $w = k \ge 4$.
Some small examples with $(k,w) \in \{(4,4), (4,13), (5,5)\}$ appear in \cite{Forbes2020P} and \cite{ForbesRutherford2022}.
Deficiency graph girth at least 5 implies that for any distinct pair of points $x$ and $y$,
$x^\mathrm{opp}$ and $y^\mathrm{opp}$ have no common block (Lemma~\ref{lem:pairwise disjoint opposite designs}).
Moreover, we have a lower bound for $r$:
$$r = \dfrac{w(w - 1)}{k - 1} + e,~~~ e \ge 0,$$
where $e = 0$ is attained only when the number of points, $v = (k - 1)r + w + 1$, is equal to $w^2 + 1$, the Moore bound for $(w,5)$-graphs.
There are $v r/k$ lines (Lemma~\ref{lem:v-b}), which are split between those that belong to opposite designs,
${v w (w - 1)}/(k(k - 1)),$
and those that do not, $e v/k$ (Lemma~\ref{lem:r >= w(w-1)/(k-1)}).
\item[\adfbullet]{\bf Type B}~
The deficiency graph has girth at least 5 but is not connected.
We distinguish these geometries from type A because our primary recursive construction, i.e.\ the
method which involves filling in the groups of group divisible designs, Theorem~\ref{thm:GDD-basic},
does not preserve deficiency graph connectedness.
\item[\adfbullet]{\bf Type C}~
The deficiency graph is connected and has girth 4, but the geometry does not contain an opposite design pair.
This type does not exist when $w = k$, \cite{BallBambergDevillersStokes2013}.
In contrast to types A and B, $x^\mathrm{opp}$ and $y^\mathrm{opp}$ can have common blocks for distinct points $x$ and $y$.
\item[\adfbullet]{\bf Type D}~
As type C except that the deficiency graph is not connected.
Again, opposite designs for distinct points can have common blocks, and this type does not exist when $w = k$.
As with types A and B, the distinction is made between types C and D because the geometries
created by Theorem~\ref{thm:GDD-basic} do not have connected deficiency graphs.
\item[\adfbullet]{\bf Type E}~
Any $\adfPENT(k,r,w)$ other than type F that contains an opposite design pair.
The deficiency graph has girth 4, is not connected, and has at least one component $K_{w,w}$.
\item[\adfbullet]{\bf Type F}~
$\adfPENT(k, (w-1)/(k-1), w)$. This is the degenerate
opposite design pair consisting of two point-wise disjoint Steiner systems $S(2,k,w)$.
The deficiency graph is $K_{w,w}$.
\end{adfenumerate}

The existence spectrum for generalized pentagonal geometries has been settled completely
for $\adfPENT(3,r)$ type E, \cite{GriggsStokes2016},             
                and type B, \cite{ForbesGriggsStokes2020}.       
It has been settled up to a small number of possible exceptions
for $\adfPENT(4,r)$ type E, \cite{ForbesGriggsStokes2020},       
for $\adfPENT(4,r)$ type B, \cite{Forbes2020P},                  
for $\adfPENT(3,r,w)$ type B, $w \in $ \{7, 9, 13, 15, 19, 21, 25, 27, 31, 33\}, \cite{ForbesRutherford2022}, and  
for $\adfPENT(4,r,13)$ types B and E, \cite{ForbesRutherford2022}.                                                 

As far as we are aware, the known $\adfPENT(k,r,w)$ with $k \ge 3$ and {\em connected} deficiency graph of girth at least 5 (i.e.\ type A)
amount to the following:
\begin{enumerate}
\item[(i)]a small number of $\adfPENT(3,r)$ constructed by hand, including
  the Desargues configuration, $\adfPENT(3,3)$, \cite{BallBambergDevillersStokes2013};
\item[(ii)]$\adfPENT(3,r)$ for all $r \equiv 3 \adfmod{6}$, $r \ge 33$, \cite{Forbes2020P};
\item[(iii)]$\adfPENT(3,r,w)$, $w \in \{3, 7, 9, 13, 15, 19, 21, 25, 27, 31, 33\}$
that can be created by hill climbing for $r$ not too large, \cite{ForbesGriggsStokes2020}, \cite{ForbesRutherford2022};
\item[(iv)] $\adfPENT(4,r)$ for $r \in $
\{13, 17, 20, 21, 24,
29, 33, 37, 40, 45, 49, 52, 53, 60, 61, 65, 69, 77, 80, 81, 85, 93, 97, 100, 101, 108, 109, 117, 120, 125, 133, 140, 141, 149, 157, 160, 165, 173, 180\}, \cite{Forbes2020P};
\item[(v)]$\adfPENT(4,r,13)$ for $r \in $ \{112, 116, 120, 124, 128, 132, 136, 140, 144, 148, 152, 156, 160, 164\}, \cite{ForbesRutherford2022};
\item[(vi)]$\adfPENT(5,r)$ for $r \in $ \{20, 25, 30, 35, 40\}, \cite{Forbes2020P};
\item[(vii)]the $\adfPENT(6,7)$ and $\adfPENT(7,7)$ based on the Hoffman--Singleton graph, \cite{BallBambergDevillersStokes2013}.
\end{enumerate}

For most of this paper, we consider geometries of types C and D.
In the next section we recall the basic properties of generalized pentagonal geometries, and
in Section~\ref{sec:Constructions} we describe some general constructions.
Section~\ref{sec:Block size 3} addresses the existence spectrum of type D geometries $\adfPENT(3, r, w)$.
For a few values of $w$, we are able to provide simple proofs that the necessary conditions, $r \equiv 0$ or $w + 1 \adfmod{3}$,
are sufficient whenever $r$ is not small.
These proofs rely on a supply of suitable directly constructed geometries, which we have
for $w \in$ \{3, 7, 9, 13, 15, 19, 21, 25, 27, 33, 39\}.

In Section~\ref{sec:Block sizes 4, 5, 6 and 7} we consider the existence of
$\adfPENT(k, r, w)$, $4 \le k \le 7$, with deficiency graph of girth 4.
Also we report some new type-A pentagonal geometries that we found during our researches:
$\adfPENT(5,r)$, $r \in$ \{21, 26, 31, 36, 41, 45\}.
As a consequence, we are able to determine the spectrum for type-B $\adfPENT(5,r)$
up to a few possible exceptions (Theorem~\ref{thm:PENT-5-r-constructed}).
In Section~\ref{sec:Identifying codes} we define identifying codes for pentagonal geometries.
For $k = 3$ and 4, we show that, apart from a few possible exceptions, there exist
pentagonal geometries $\adfPENT(k,r)$ which have identifying codes that are as small as possible.
Section~\ref{sec:Identifying codes} also includes some new $\adfPENT(4,r)$ with connected deficiency graphs.


\section{Generalized pentagonal geometries}\label{sec:Generalized pentagonal geometries}
For any $k \ge 3$, there exist Steiner systems $S(2, k, 0)$, $S(2, k, 1)$ and $S(2, k, k)$,
with block set sizes 0, 0 and 1, respectively.
If a Steiner system $S(2,k,w)$ exists, then
the number of blocks is $w(w - 1)/(k(k - 1))$ and
the number of blocks that contain a given point is a constant, $(w - 1)/(k - 1)$.
%
The conditions $w \ge 1$, $w(w - 1) \equiv 0 \adfmod{k(k - 1)}$ and $w - 1 \equiv 0 \adfmod{k - 1}$ collectively
are sufficient for the existence of a Steiner system $S(2,k,w)$
when $k = 3$, \cite{Kirkman1847}, and when $k \in \{4,5\}$, \cite{Hanani1961}.
An affine plane and a projective plane of order $n$ exist whenever $n$ is a prime power.

Our first lemma of Section~\ref{sec:Generalized pentagonal geometries} concerns the size of a $\adfPENT(k,r,w)$, which is completely determined by the three parameters $k$, $r$ and $w$.
%
\begin{lemma} \label{lem:v-b}
A generalized pentagonal geometry $\adfPENT(k, r, w)$ has $r(k-1) + w + 1$ points and $(r(k - 1) + w + 1)r/k$ lines.

A $\adfPENT(k, r, w)$ exists only if $r(r - w - 1) \equiv 0 \adfmod{k}$.
\end{lemma}
\begin{proof}
See \cite[Lemma 2.1]{ForbesRutherford2022}.
\end{proof}
Given $k$ and $w$, we say that $r$ is {\em admissible} if $r(w + 1 - r) \equiv 0 \adfmod{k}$.
%
\begin{lemma} \label{lem:deficiency-graph}
Let $D$ be the deficiency graph of a generalized pentagonal geometry $\adfPENT(k,r,w)$.
Then $D$ is $w$-regular and has girth at least $4$.

For distinct points $x$ and $y$, let $U_{x,y}$ denote the set of points common to both $x^\mathrm{opp}$ and $y^\mathrm{opp}$,
and let $u_{x,y} = |U_{x,y}|$.
Then $U_{x,y}$ is the point set of a Steiner system $S(2,k,u_{x,y})$ that exists as a
subdesign of both $x^\mathrm{opp}$ and $y^\mathrm{opp}$.
\end{lemma}
\begin{proof}
See \cite[Lemma 2.3]{ForbesRutherford2022}.
\end{proof}
The possible values of $u_{x,y}$ in Lemma~\ref{lem:deficiency-graph} always include 0, 1, $k$ and $w$.
Conversely, recalling that $w > k$ implies $w \ge k^2 - k + 1$ by Fisher's inequality for block designs,
$u_{x,y}$ can never take values in the intervals $[2, k-1]$ and $[k + 1, k^2 - k$].
%
\begin{lemma} \label{lem:pairwise disjoint opposite designs}
For a generalized pentagonal geometry $\adfPENT(k, r, w)$, the following properties are equivalent.
\begin{enumerate}
\item[(i)] For any two distinct points $x$ and $y$, $x^\mathrm{opp}$ and $y^\mathrm{opp}$ have no common block.
\item[(ii)] For any two distinct points $x$ and $y$, the point sets of $x^\mathrm{opp}$ and $y^\mathrm{opp}$ have at most one common point.
\item[(iii)] The deficiency graph has girth at least $5$.
\end{enumerate}
\end{lemma}
\begin{proof}
See \cite[Lemma 2.4]{ForbesRutherford2022}.
\end{proof}
%
\begin{lemma} \label{lem:r >= w(w-1)/(k-1)}
For a $\adfPENT(k, r, w)$ with a deficiency graph of girth at least $5$,
we have
$$r = w(w - 1)/(k - 1) + e,~ \text{~where~} e \ge 0.$$
Furthermore, the geometry contains $\dfrac{v w (w - 1)}{k(k - 1)}$ lines that belong to opposite designs and
$e v/k$ that do not, where $v = (k - 1)r + w + 1$.
\end{lemma}
\begin{proof}
See \cite[Lemma 2.5]{ForbesRutherford2022} for the first part.
By Lemma~\ref{lem:v-b}, the number of points in the geometry is $v$ and the number of lines is $v r/k$.
The number of blocks in a Steiner system $S(2,k,w)$ is $w(w - 1)/(k(k - 1))$.
Hence, by Lemma~\ref{lem:pairwise disjoint opposite designs}, the number of lines that belong to opposite designs is $v w(w - 1)/(k(k - 1))$.
Subtracting this from the number of lines gives $e v/k$.
\end{proof}
The {\em distance at least 3} graph of a generalized pentagonal geometry has the points of the geometry as its vertices
and there is an edge $x \sim y$ precisely when the distance between $x$ and $y$ in the geometry's deficiency graph is at least 3.
%
\begin{lemma} \label{lem:distance-at-least-3-neighbours}
Let $E$ be the distance at least 3 graph of a generalized pentagonal geometry $\adfPENT(k, r, w)$ with $w \ge k \ge 3$.
Then for every vertex $x$ of $E$,
the subgraph of $E$ induced by neighbours of $x$ can be partitioned into complete graphs $K_{k-1}$.

Furthermore, the degree of $x$ in $E$ is least $r(k - 1) - w(w - 1)$
with equality if the deficiency graph of the $\adfPENT(k, r, w)$ has girth at least $5$.
\end{lemma}
\begin{proof}
Let $D$ be the deficiency graph of the $\adfPENT(k, r, w)$.

Consider distinct points $x$ and $y$.
If $x$ and $y$ are not collinear, then $x \sim y$ in $D$ and therefore $x \not\sim y$ in $E$.
Now assume $x$ and $y$ are collinear.
If $x$ and $y$ belong to $z^\mathrm{opp}$ for some point $z$, then there is a path $x \sim z \sim y$ in $D$
and hence $x \not\sim y$ in $E$.
Conversely, $x \not\sim y$ in $E$ implies there is a path $x \sim z \sim y$ in $D$;
then $x$ and $y$, being neighbours of $z$ in $D$, must both belong to $z^\mathrm{opp}$.
The only other possible case is when $x$ and $y$ belong to a non-opposite line,
i.e.\ a line that is not in $z^\mathrm{opp}$ for any point $z$.
In this case there cannot be a 2-edge path from $x$ to $y$, and hence $x \sim y$ in $E$.

Therefore each non-opposite line containing the point $x$ yields a subgraph $K_k$ in $E$ containing $x$.
Moreover, these subgraphs are pointwise mutually disjoint except for a common vertex, $x$.
If there are $s$ non-opposite lines containing $x$,
then the subgraphs collectively form a windmill graph $\mathrm{Wd}(k,s$) with centre $x$.
Since all vertices of $N_E(x)$ occur in the windmill,
the subgraph induced by $N_E(x)$ can be partitioned into complete graphs $K_{k-1}$.

If $D$ has girth at least 5, then by Lemma~\ref{lem:r >= w(w-1)/(k-1)}
the number of non-opposite lines containing $x$ is $s = r - w(w - 1)/(k - 1)$.
Therefore $x$ has degree $(k-1)s$ in $E$.
However, if $D$ has girth 4, then possible duplication amongst the lines containing $x$ that belong to opposite designs
implies that $(k-1)s$ will be a lower bound for the degree of $x$ in $E$ that is not necessarily attained.
\end{proof}
If $D$ is a $(w,5^+)$-graph with $v$ vertices and $E$ is its distance at least 3 graph, then $E$ is $(v - w^2 - 1)$-regular and
we have a decomposition of the edges of $K_v$ into sets of sizes $v w/2$, $v w (w - 1)/2$, $v(v - w^2 - 1)/2$
corresponding to $d_D(x,y) = 1$, $d_D(x,y) = 2$, $d_D(x,y) \ge 3$, respectively.
This is not necessarily the case when $D$ has girth 4.
However, $E$ is $(h(v - w^2 - 1))$-regular when $D$ is derived from a $(w,5^+)$-graph with $v$ vertices by
inflating each point by a factor of $h > 1$ and replacing each edge with a $K_{h,h}$.

%

\section{Constructions}\label{sec:Constructions}
For the purpose of this paper, a {\em group divisible design}, $k$-GDD, of type $g_1^{u_1} g_2^{u_2} \dots g_n^{u_n}$ is
an ordered triple ($V,\mathcal{G},\mathcal{B}$) such that:
\begin{adfenumerate}
\item[(i)]{$V$
is a base set of cardinality $u_1 g_1 + u_2 g_2 + \dots + u_n g_n$;}
\item[(ii)]{$\mathcal{G}$
is a partition of $V$ into $u_i$ subsets of cardinality $g_i$, $i = 1, 2, \dots, n$, called \textit{groups};}
\item[(iii)]{$\mathcal{B}$
is a non-empty collection of $k$-subsets of $V,$ called \textit{blocks}; and}
\item[(iv)]{each pair of elements from distinct groups occurs in precisely one block but no pair of
elements from the same group occurs in any block.}
\end{adfenumerate}

In the first theorem of this section, we recall
the main recursive construction for creating large geometries from small ones,
\cite{GriggsStokes2016}, \cite{Forbes2020P}, \cite{ForbesGriggsStokes2020}, \cite{ForbesRutherford2022}.
It is a straightforward adaptation of Wilson's Fundamental Construction, \cite{WilsonRM1972}, \cite[Theorem IV.2.5]{GreigMullen2007}.
%
\begin{theorem}
\label{thm:GDD-basic}
Let $k \ge 3$ and $w \ge k$ be integers.
For $i = 1$, $2$, \dots, $n$, let $r_i$ be a positive integer, let $v_i = (k-1) r_i + w + 1$,
and suppose there exists a generalized pentagonal geometry $\adfPENT(k, r_i, w)$.
Suppose also that there exists a $k$-$\mathrm{GDD}$ of type $v_1^{u_1} v_2^{u_2} \dots v_n^{u_n}$.
Let $N = u_1 + u_2 + \dots + u_n$ and
$R = u_1 r_1 + u_2 r_2 + \dots + u_n r_n$.
Then there exists a generalized pentagonal geometry
$\adfPENT(k, R + (N - 1)(w + 1)/(k - 1), w)$.

Furthermore, the deficiency graph girth of the $\adfPENT(k, R + (N - 1)(w + 1)/(k - 1), w)$
is the minimum of the deficiency graph girths of the $\adfPENT(k, r_i, w)$, $i = 1$, $2$, \dots, $n$.
\end{theorem}
\begin{proof}
See \cite[Theorem 2.5]{ForbesRutherford2022}.
It is clear that the deficiency graph of the $\adfPENT(k, R + (N - 1)(w + 1)/(k - 1), w)$ is the vertex-disjoint union of
the deficiency graphs of the $\adfPENT(k, r_i, w)$, $i = 1$, $2$, \dots, $n$, and
therefore it inherits the smallest girth of its components.
\end{proof}
For convenience, we provide a summary of results concerning the existence of group divisible designs of the kind that we
require for the application of Theorem~\ref{thm:GDD-basic} and other constructions.
%
\begin{lemma}
\label{lem:k-GDD-existence}
Suppose $g$, $u$ and $m$ are positive integers.
\begin{enumerate}
\item[(i)]There exists a $3$-$\mathrm{GDD}$ of type $g^u$ if
$g$ is even, $u \ge 3$ and $u(u-1)g \equiv 0 \adfmod{3}$.

\item[(ii)]There exists a $3$-$\mathrm{GDD}$ of type $g^u m^1$ if
$g$ is even, $m$ is even, $u \ge 3$, $m \le g(u-1)$ and $g^2 u (u-1) + 2 g u m \equiv 0 \adfmod{3}$.

\item[(iii)]There exists a $4$-$\mathrm{GDD}$ of type $g^u$ if
$u \ge 4$, $g(u - 1) \equiv 0 \adfmod{3}$ and $g^2u(u - 1) \equiv 0 \adfmod{12}$,
except for $(g,u) \in \{(2,4), (6,4)\}$.

\item[(iv)]Suppose $g \equiv 5 \adfmod{6}$ and let $\mu = 3$ if $m$ is odd, $\mu = 9$ if $m$ is even.
Then there exists a $4$-$\mathrm{GDD}$ of type $g^u m^1$, if
$m \equiv 2 \adfmod{3}$, $m \le g(u - 1)/2$,
$u \ge 12$ and $u \equiv 0 \text{ or } \mu \adfmod{12}$, with possible exceptions
$u \in \{12, 24, 15, 27, 39, 51, 21, 33\}$ when $0 < m < g$.

\item[(v)]Suppose $g \in \{80, 128\}$,
$u \equiv 0 \adfmod{3}$, $u \ge 6$,
$m \equiv 2 \adfmod{3}$, $m \le g(u - 1)/2$.
Then there exists a $4$-$\mathrm{GDD}$ of type $g^u m^1$,
with possible exceptions $80^9 m^1$, $m \in \{299, 311, 317\}$.

\item[(vi)]There exists a $5$-$\mathrm{GDD}$ of type $g^5$ if $g \not\in \{2, 3, 6, 10\}$.

\item[(vii)]There exists a $5$-$\mathrm{GDD}$ of type $g^u$ if
$g \equiv 0 \adfmod{10}$, $u \equiv 1 \adfmod{2}$ and $u \ge 5$, with possible exceptions
\begin{align*}
10^5, 10^7, 10^{27}, 10^{39}, 10^{47}, 50^{27}.
\end{align*}

\item[(viii)]There exists a $5$-$\mathrm{GDD}$ of type $g^u$ if
$g \equiv 0 \adfmod{2}$, $g \ge 86$, $u \equiv 1 \adfmod{10}$ and $u \ge 21$.

\item[(ix)]There exist $5$-$\mathrm{GDD}$s of types $10^{10} 18^1$ and $10^{10} 30^1$.

\item[(x)]There exists a $6$-$\mathrm{GDD}$ of type $g^6$ if $g \ge 23$.

\item[(xi)]There exists a $7$-$\mathrm{GDD}$ of type $7^7$.

\item[(xii)]There exists an $11$-$\mathrm{GDD}$ of type $g^{11}$ if $g \equiv 1 \adfmod{2}$ and $g \ge 1937$.
\end{enumerate}
\end{lemma}
\begin{proof}
For (i) and (ii), see \cite{ColbournHoffmanRees1992} or \cite[Theorems IV.4.1 and IV.4.2]{Ge2007};
For (iii), see \cite{BrouwerSchrijverHanani1977}.

For (iv), see \cite[Theorem 1.2]{Forbes2019}.

For (v), see \cite[Theorem 1.2]{Forbes2019B}.

For (vi), (x), (xi) and (xii), the $\ell$-GDD of type $g^{\ell}$ is constructed from
$\ell - 2$ mutually orthogonal Latin squares of side $g$,
the existence of which are given by \cite[Table III.3.88]{AbelColbournDinitz2007}.

For (vii) and (viii), see \cite[Theorem 2.25]{WeiGe2014} and \cite{Forbes2022G}.

For (ix), type $10^{10} 18^1$ is given by \cite[Lemma 4.1]{Forbes2020P}, and
type $10^{10} 30^1$ is constructed from a resolvable 4-GDD of type $10^{10}$ as explained in the proof of \cite[Lemma 4.7]{Forbes2020P}.
\end{proof}

Next, we have a useful application of Theorem~\ref{thm:GDD-basic} for geometries with block size 3.
%
\begin{theorem}
\label{thm:PENT-3-r-w-constructed-large}
Let $w \ge 3$ be an integer.
For $i \in \{0, 1, 2\}$, let $v_i = 2r_i + w + 1$ and
suppose there exists a generalized pentagonal geometry $\adfPENT(3, r_i, w)$
with a deficiency graph of girth $g_i$.
Suppose also that
$r_0 \equiv 0 \adfmod{3}$,
$r_1, r_2 \not\equiv 0 \adfmod{3}$ and
$\gcd(v_1, v_2) = 6$.
Let
$$t_\mathrm{min} = 1 + \max\left(2, \left\lceil\dfrac{v_0}{v_1}\right\rceil\right),$$
$$u_\mathrm{min}(t) = 1 + \max\left(2,
                                    \left\lceil\dfrac{v_1 t + v_1}{v_2}\right\rceil,
                                    \left\lceil\dfrac{v_1 t + v_0}{v_2}\right\rceil\right),$$
$$x_\mathrm{min} = u_\mathrm{min}\left(t_\mathrm{min} + \dfrac{v_2}{6}\right)$$
and
$$r_\mathrm{min} = \dfrac{v_2 \, x_\mathrm{min}}{2} + \dfrac{v_1 \, t_\mathrm{min}}{2}
                 + \dfrac{v_1 v_2}{6} + \max(r_0, r_1).$$
Then for all $r \ge r_\mathrm{min}$, $r \equiv 0$ or $w + 1 \adfmod{3}$,
there exists a generalized pentagonal geometry $\adfPENT(3, r, w)$ with a deficiency graph of girth
$\min(g_0, g_1, g_2)$ if $r \equiv 0 \adfmod{3}$, $\min(g_1, g_2)$ otherwise.
\end{theorem}
\begin{proof}
Let $r_3 \in \{r_0, r_1\}$ and $v_3 = 2r_3 + w + 1$.
We have $w \equiv 1$ or $3 \adfmod{6}$ since we are assuming there exists a Steiner triple system of order $w$.
Also $r_1 \equiv r_2 \equiv w + 1 \adfmod{3}$ by Lemma~\ref{lem:v-b}, and
$v_1 \equiv v_2 \equiv 0 \adfmod{6}$.

Apply Theorem~\ref{thm:GDD-basic} twice, first with
a $\adfPENT(3, r_3, w)$,
$t$ copies of a $\adfPENT(3, r_1, w)$
and a 3-GDD of type $v_1^t v_3^1$ to obtain a
$$\adfPENT\left(3, \dfrac{v_1 t}{2} + r_3, w\right)~~
\text{~for~} t \ge t_\mathrm{min},$$
and then with
$u$ copies of a $\adfPENT(3, r_2, w)$
and a 3-GDD of type $v_2^u (v_1 t + v_3)^1$ to obtain a
\begin{align}\label{eqn:PENT-3-r-w-constructed}
& \adfPENT\left(3, \dfrac{v_2 u}{2} + \dfrac{v_1 t}{2} + r_3, w\right)~~ \text{~for~} u \ge u_\mathrm{min}(t),~~ t \ge t_\mathrm{min},
\end{align}
with deficiency graph girth as stated.
The lower bounds $t_\mathrm{min}$ and $u_\mathrm{min}(t)$ ensure the existence of the two 3-GDDs by Lemma~\ref{lem:k-GDD-existence}.

In (\ref{eqn:PENT-3-r-w-constructed}) the greatest common divisor of the coefficients of $u$ and $t$ is 3 and
the two options for $r_3$ belong to the two admissible residue classes $0$ and $w + 1$ modulo 3.
To show that all $r \ge r_\mathrm{min}$, $r \equiv r_3 \adfmod{3}$ are covered we argue as follows.

Write $p_1 = v_1/6$, $p_2 = v_2/6$, and note that $\gcd(p_1, p_2) = 1$.
By the Chinese Remainder Theorem, for any positive integer $n$, there exist integers $t$, $x$ and $d$ such that
$$n = p_2 x + p_1 t + d p_1 p_2,$$
where
$$t_\mathrm{min} \le t < t_\mathrm{min} + p_2 \text{~~and~~}
  x_\mathrm{min} \le x < x_\mathrm{min} + p_1.$$
Moreover, $d$ will be non-negative for all
\begin{equation} \label{eqn:n-ge-p2xmin-plus-p1tmin-plus-2p1p2}
n \ge p_2 \, x_{\mathrm{min}} + p_1 \, t_{\mathrm{min}} + 2p_1 p_2.
\end{equation}
Multiplying the expression on the right of (\ref{eqn:n-ge-p2xmin-plus-p1tmin-plus-2p1p2}) by 3,
adding $r_3$, and comparing with \ref{eqn:PENT-3-r-w-constructed},
we see that there exists a $\adfPENT(3, r, w)$ for all
$$r \ge \dfrac{v_2 \, x_{\mathrm{min}}}{2}
        + \dfrac{v_1 \, t_{\mathrm{min}}}{2} + \dfrac{v_1 v_2}{6} + r_3,~~r \equiv r_3 \adfmod{3}.$$
The theorem follows by combining the two options for $r_3$.
\end{proof}

In the next construction we are in some sense multiplying a geometry by $h$.
The girth of the constructed deficiency graph is reduced to 4 but,
unlike in Theorem~\ref{thm:GDD-basic}, connectedness is preserved.

%
\begin{theorem}
\label{thm:Product-construction}
Let $r$, $w$ and $h$ be positive integers with $w \ge k \ge 3$ and $h \ge 3$.
Suppose there exists a $k$-$\mathrm{GDD}$ of type $h^k$,
a Steiner system $S(2,k,h)$ and a $\adfPENT(k,r,w)$ with a connected deficiency graph.
Then there exists a $\adfPENT(k, hr + (h-1)/(k-1), hw)$ with a connected deficiency graph of girth $4$.
\end{theorem}
\begin{proof}
The geometry $\adfPENT(k,r,w)$ is essentially defined by its line set, which we denote by $\mathcal{P}$.
Let $U$ be its point set and let $u = |U| = (k-1)r + w + 1$.
Let $D$ be the deficiency graph of $\mathcal{P}$; then $D$ is connected, is $w$-regular and has girth at least 4.

Let $H = \{1, 2, \dots, h\}$.
Let $V$ = $U \times H$, $v = |V|$, and denote the point $(a,i)$, $a \in U$, $i \in H$ by $a_i$.

Replace each line $A \in \mathcal{P}$ with the $h^2$ lines of a $k$-GDD of type $h^k$ with groups
\begin{equation*} 
\{\{a_1, a_2, \dots, a_h\}:  a \in A\}.
\end{equation*}
For each point $a \in U$ add the lines of a Steiner system $S(2,k,h)$ with point set $\{a_1, a_2, \dots, a_h\}$.
Let $\mathcal{Q}$ denote the set of $k$-tuples thus obtained.
We claim that $\mathcal{Q}$ is the line set of a $\adfPENT(k, h r + (h -1)/(k - 1), hw)$.

Observe that $v = h u$ and that
$$|\mathcal{Q}| = h^2|\mathcal{P}| + \dfrac{u h(h - 1)}{k(k - 1)}
                = \dfrac{h^2 u r}{k} + \dfrac{u h(h - 1)}{k(k - 1)} = \dfrac{v}{k}\left(hr + \dfrac{h - 1}{k - 1}\right),$$
the correct numbers of points and lines for a $\adfPENT(k, hr + (h-1)/(k-1), hw)$.
Moreover, each line of $\mathcal{Q}$ contains $k$ points, and
each point $x$ of $V$ is incident with $rh$ blocks of $k$-GDDs and $(h-1)/(k-1)$ blocks of a Steiner system $S(2,k,h)$,
$rh + (h-1)/(k-1)$ lines altogether.

Consider two distinct points, $a_i, b_j \in V$.
If $a = b$, or if $a \neq b$ and $a$ is collinear with $b$ in $\mathcal{P}$, then
$a_i$ and $b_j$ are collinear in $\mathcal{Q}$, either in a $k$-GDD when $i \neq j$ or in an $S(2,k,h)$ when $i = j$.
In the remaining case $b$ is in the set of points that are not collinear with $a$ in $\mathcal{P}$.
The set of $w$ points of $a^\mathrm{opp}$ corresponds to a set $A(a)$ of $h w$ points in $V$, which are precisely
the points of $V$ that are not collinear with $a_i$.
The $w(w - 1)/(k(k - 1)$ lines of $a^\mathrm{opp}$ correspond to $h^2 w(w - 1)/(k(k - 1))$ blocks of $k$-GDDs of type $h^k$.
Furthermore, the $w$ points of $a^\mathrm{opp}$ correspond to a further $w h(h - 1)/(k(k - 1))$ blocks
of $w$ Steiner systems $S(2,k,h)$ that fill in the groups of the $k$-GDDs.
On collecting them together we have a set $\mathcal{A}(a)$ of
$$\dfrac{h^2 w(w - 1)}{k(k - 1)} + \dfrac{w h(h - 1)}{k(k - 1)} = \dfrac{hw(hw - 1)}{k(k - 1)}$$
lines.
Clearly, any pair of distinct elements of $A(a)$ is a subset of some line in $\mathcal{A}(a)$.
Hence $\mathcal{A}(a)$ is the block set of a Steiner system $S(2,k,hw)$ that forms $a_i^\mathrm{opp}$
in a $\adfPENT(k, hr + (h-1)/(k-1), hw)$.

The deficiency graph is obtained from $D$ be replacing each edge $a \sim b$ with a complete bipartite graph $K_{h,h}$
to yield a connected $(hw)$-regular graph of girth 4.
\end{proof}
As a special case of Theorem~\ref{thm:Product-construction} we describe a simple tripling construction,
which is self-contained---the only input is the $\adfPENT(3,r,w)$ itself.
%
\begin{theorem}
\label{thm:Tripling-construction}
Let $r$ and $w$ be positive integers with $w \ge 3$ and suppose there exists a $\adfPENT(3,r,w)$ with a connected deficiency graph.
Then there exists a $\adfPENT(3,3r+1,3w)$ with a connected deficiency graph of girth $4$.
\end{theorem}
\begin{proof}
Let $\mathcal{P}$ be the line set of the $\adfPENT(3,r,w)$,
let $U$ be its point set and let $D$ be its deficiency graph.
Let $V$ = $U \times \{1,2,3\}$ and write $a_i$ for $(a,i)$.
Let $\Pi$ be the set of permutations of $(1,2,3)$ and let
\begin{align*}
\mathcal{Q} = &\{\{a_1, a_2, a_3\}: a \in U\} \\
              & \cup \{\{a_i, b_i, c_i\}: \{a,b,c\} \in \mathcal{P}, i \in \{1, 2, 3\}\} \\
              & \cup \{\{a_h, b_i, c_j\}: \{a,b,c\} \in \mathcal{P}, (h,i,j) \in \Pi\}.
\end{align*}
In the first set of this expression we are replacing each point with a Steiner system $S(2,3,3)$, and
in the second and third sets
we are replacing each line with the nine blocks of a 3-GDD of type $3^3$.
Hence, by Theorem~\ref{thm:Product-construction}, $\mathcal{Q}$ is the line set of a $\adfPENT(3, 3r+1,3w)$ with point set $V$, and
its deficiency graph is obtained from $D$ by replacing each edge with a $K_{3,3}$.
\end{proof}

The next construction describes a way to build a $\adfPENT(k,r,w)$ of type C directly from a graph with girth at least 5.
%
\begin{construction}\label{con:PENT-k-r-w-constructed}
{\rm

Let $k$, $h$, $w$ and $v$ be integers with $w \ge h \ge k \ge 3$, $v \ge 10$ and $w, v \equiv 0 \adfmod{h}$.

Let $C$ be a connected $(w/h, 5^+)$-graph with $v/h$ vertices, and
let $V(C) = \{0, 1, \dots, v/h - 1\}$.
For $p \in V(C)$, let
$$H(p) = \{hp, hp + 1, \dots, hp + h - 1\}.$$

Construct a graph, $D$, with vertex set $\{0, 1, \dots, v-1\}$ by replacing
each vertex $p$ of $C$ with the vertices of $H(p)$ and
each edge $p \sim q$ of $C$ with the edges of a complete bipartite graph $K_{h,h}$ with parts $H(p)$ and $H(q)$.
For $x \in V(D)$, denote by $\adfhat{x}$ the vertex of $C$ from which $x$ arises by the construction;
i.e.\ $\adfhat{x} = \lfloor x/h \rfloor$.
Then $D$ has $v$ vertices, is $w$-regular and has girth 4.
For $x, y \in V(D)$, we have
$$|N_D(x) \cap N_D(y)| ~=~
\left\{\begin{array}{ll} w & \text{~if~} \adfhat{x} = \adfhat{y},\\
                         0 & \text{~if~} \adfhat{x} \sim \adfhat{y} \text{~in~} C,\\
                         h & \text{~if~} d_C(\adfhat{x}, \adfhat{y}) = 2, \\
                         0 & \text{~if~} d_C(\adfhat{x}, \adfhat{y}) \ge 3.
\end{array}\right.$$

Given a Steiner system $S(2,k,h)$ that has point set $\{0, 1, \dots, h - 1\}$,
let $\mathcal{S}(p)$ be the block set of the same Steiner system with the points relabelled by the mapping $x \mapsto x + h p$.

Given a $k$-GDD of type $h^{w/h}$ that has groups
$$\{0, 1, \dots, h - 1\}, \{h, h + 1, \dots, 2h - 1\}, \dots, \{w - h, w - h + 1, \dots, w\},$$
let $\mathcal{G}(p)$ be the block set of the same $k$-GDD with the points relabelled such that the groups are $H(q)$, $q \in N_C(p)$.

Let $r = (v - w - 1)/(k - 1)$.
We attempt to construct a $\adfPENT(k,r,w)$ with deficiency graph $D$.

For each $p \in V(C)$, let $W(p) = \bigcup_{q \in N(p)} H(q)$,
Then $\mathcal{B}(p) = \mathcal{G}(p) \cup \mathcal{S}(p)$ is the block set of a
Steiner system $S(2,k,w)$ with point set $W(p)$.

For $x \in V(D)$, we designate as $x^\mathrm{opp}$
the Steiner system $S(2,k,w)$ with point set $W(\adfhat{x})$ and block set $\mathcal{B}(\adfhat{x})$.
Observe that $N_D(x) = W(\adfhat{x})$.

Let
$$b_\mathrm{opp} = \left|\bigcup_{p \in V(C)} \mathcal{B}(p)\right|,$$
the number of lines that occur in opposite designs.
Since $C$ has girth at least 5 there is no overlapping of the group divisible designs;
hence there are
$$\dfrac{v}{h} \; h^2 \; \dfrac{w}{h}\left(\dfrac{w}{h} - 1\right) \dfrac{1}{k(k - 1)}$$
distinct lines that occur in $k$-GDDs.
To this we add the number of lines that occur in the $v/h$ Steiner systems $S(2,k,h)$ to obtain
$$b_\mathrm{opp} = v w \left(\dfrac{w}{h} - 1\right) \dfrac{1}{k(k - 1)} + \dfrac{v}{h} \; \dfrac{h(h - 1)}{k(k - 1)}
 = \dfrac{v(w(w - h) + h(h - 1))}{h k (k - 1)}.$$

If $b_\mathrm{opp} = v r/k$, the construction is complete.
On the other hand, when $b_\mathrm{opp} < v r/k$ we need to provide the remaining lines by assembling into $k$-tuples
those pairs $\{x, y\} \subseteq V(D)$ that are neither subsets of opposite designs nor edges of $D$.
If $k = 3$, we can attempt to provide these $v r/3 - b_\mathrm{opp}$ lines by hill climbing,
as described in \cite{Stinson1985} or \cite[Section 2.7.2]{ColbourRosa1999}.
However, the process is doomed if $D$ is not the deficiency graph of a generalized pentagonal geometry.
Nevertheless, experiments suggest that hill climbing is likely to succeed whenever $v$ is not too small and not too large.
If $k \ge 4$, the required non-opposite lines must be assembled by some other method.
}\end{construction}

There does not seem to be any obvious (to the authors) general method for finding $(w,4)$-graphs
that might be suitable as deficiency graphs for geometries of type C.
Lemma~\ref{lem:deficiency-graph} places a severe restriction on values of $|N(x) \cap N(y)| > 1$
and yet this inequality must be satisfied at least once.
The main feature of Construction~\ref{con:PENT-k-r-w-constructed} is that we start with a $(w/h,5^+)$-graph.
Then the multiplication by $h$ process automatically creates a deficiency graph with acceptable neighbourhood intersection sizes.

%

\section{Block size 3}\label{sec:Block size 3}

In a generalized pentagonal geometry $\adfPENT(3,r,w)$ the number of points is $2r + w + 1$
and the admissibility condition is $r \equiv 0$ or $w + 1 \adfmod{3}$.
Also $w \equiv 1$ or $3 \adfmod{6}$, a sufficient condition for the existence of a Steiner triple system of order $w$.

%
\begin{lemma}\label{lem:PENT-3-r-3-direct}
For $r = 3$ and for $r \equiv 0$ or $1 \adfmod{3}$, $9 \le r \le 33$,
there exists a pentagonal geometry $\adfPENT(3,r)$ where the deficiency graph is a
generalized Petersen graph with girth at least $5$.
\end{lemma}
\begin{proof}
These can be created as follows.
For some suitable parameter $s$,
let $D$ be the generalized Petersen graph $\mathrm{GP}(r + 2, s)$ with vertex set $V = \{0, 1, \dots, 2r+3\}$
and edges generated from $\{0, 1\}$, $\{0, 2\}$ and $\{1, 2s+1\}$ by the mapping $x \mapsto x + 2 \adfmod{2r + 4}$.
The point set of the geometry is $V$.
Let $\mathcal{L} = \{N(x): x \in V\}$.
This is $\{x^\mathrm{opp}: x \in V\}$, the set of opposite lines.
If $|\mathcal{L}| = r(2r + 4)/3$, the construction is complete---as happens when $r = 3$ and $D$ is the Petersen graph, $\mathrm{GP}(5,2)$.
When $|\mathcal{L}| < r(2r + 4)/3$ we determine the remaining lines of the geometry by hill climbing.
Observe that $D$ is connected and that the even-numbered vertices form an $(r+2)$-cycle.

For $r = 3$, see \cite{BallBambergDevillersStokes2013};
for $r = 33$, see \cite{Forbes2020P}.
For the remaining $r$, it is not difficult to construct the geometries by the method outlined above.
Therefore we save space by providing details in the appendix, which is present only in this paper's ArXiv preprint, available at
\begin{center}
\url{https://arxiv.org/abs/2111.13599}.
\end{center}
We used $\mathrm{GP}(r + 2, 4)$ for $r \in \{9, 13, 18\}$, $\mathrm{GP}(r + 2, 5)$ for the others.
As an example, we give details for the $\adfPENT(3,12)$ presented in the standard format that we employ throughout the paper.

For a generalized pentagonal geometry $\adfPENT(k,r,w)$, we specify $d$ followed by $dr$ numbers.
Let $v = (k-1)r + w + 1$.
Partition the $dr$ numbers into $dr/k$ $k$-tuples and
then develop these $k$-tuples into the $vr/k$ lines of the geometry by the mapping $x \mapsto x + d \adfmod{v}$.
The point set of the geometry is $\{0, 1, \dots, v - 1\}$.

%
{\boldmath $\adfPENT(3,12)$}, $d = 4$:
{1, 2, 26;  0, 11, 19;  0, 3, 4;  2, 13, 21;  0, 5, 10;  0, 6, 18;  0, 7, 23;  0, 8, 22;  0, 9, 12;  0, 13, 17;  0, 15, 21;  1, 3, 17;  1, 7, 14;  1, 22, 27;  2, 10, 27;  2, 11, 15}.
The deficiency graph is connected and has girth 6.
%
\end{proof}

%
\begin{lemma}\label{lem:PENT-3-r-9-direct}
There exists a generalized pentagonal geometry $\adfPENT(3,r,9)$ with a connected deficiency graph of girth $4$ for
$$r \in \{10\} \cup \{25, 28, 31, \dots, 100\}.$$
\end{lemma}
\begin{proof}
For those $r \equiv 1$ or $4 \adfmod{9}$, use Theorem~\ref{thm:Tripling-construction}
to construct the geometry from a $\adfPENT(3,(r - 1)/3)$ from Lemma~\ref{lem:PENT-3-r-3-direct}.

For $r = 25$, use Construction~\ref{con:PENT-k-r-w-constructed} with the graph generated from
$$\{\{0,4\}, \{1,5\}, \{2,6\}, \{0,3\}, \{1,3\}, \{2,3\}\}$$
by $x \mapsto x + 4 \adfmod{20}$.
This is our smallest example of a Type-C $\adfPENT(3,r,9)$ that is not obtained by tripling a $\adfPENT(3,s)$.


{\boldmath $\adfPENT(3,25,9)$}, $d = 12$:
{0, 1, 2;  0, 3, 7;  0, 4, 6;  0, 5, 8;  1, 3, 6;  1, 4, 8;  1, 5, 7;  2, 3, 8;  2, 4, 7;  2, 5, 6;  3, 4, 5;  6, 7, 8;  9, 10, 11;  9, 12, 49;  9, 13, 48;  9, 14, 50;  9, 15, 52;  9, 16, 51;  9, 17, 53;  9, 18, 55;  9, 19, 54;  9, 20, 56;  10, 12, 48;  10, 13, 50;  10, 14, 49;  10, 15, 51;  10, 16, 53;  10, 17, 52;  10, 18, 54;  10, 19, 56;  10, 20, 55;  11, 12, 50;  11, 13, 49;  11, 14, 48;  11, 15, 53;  11, 16, 52;  11, 17, 51;  11, 18, 56;  11, 19, 55;  11, 20, 54;  0, 15, 44;  0, 16, 46;  0, 17, 43;  0, 18, 45;  0, 19, 41;  0, 20, 47;  0, 27, 42;  0, 28, 54;  0, 29, 56;  0, 30, 40;  0, 31, 52;  0, 32, 39;  0, 33, 53;  0, 34, 55;  0, 35, 51;  1, 15, 32;  1, 16, 43;  1, 17, 46;  1, 18, 40;  1, 19, 47;  1, 20, 41;  1, 27, 54;  1, 28, 56;  1, 29, 42;  1, 30, 51;  1, 31, 39;  1, 33, 55;  1, 34, 45;  1, 35, 53;  1, 44, 52;  2, 15, 47;  2, 16, 45;  2, 17, 42;  2, 18, 46;  2, 19, 39;  2, 20, 40;  2, 27, 55;  2, 28, 43;  2, 29, 44;  2, 30, 53;  2, 31, 41;  2, 32, 51;  2, 33, 52;  2, 34, 56;  2, 35, 54;  3, 19, 45;  3, 33, 46;  3, 34, 54;  4, 18, 47;  4, 20, 46;  4, 35, 55;  5, 19, 46;  5, 33, 54;  5, 35, 56;  8, 33, 45;  9, 23, 47;  9, 33, 59;  9, 34, 46;  10, 23, 34;  10, 35, 47}.
The deficiency graph is connected and has girth 4.
%

For the others, use Construction~\ref{con:PENT-k-r-w-constructed} with a generalized Petersen graph $\mathrm{GP}((r + 5)/3, 2)$.
These geometries can be rather large and so details are given in the appendix.
\end{proof}
The results of Lemma~\ref{lem:PENT-3-r-9-direct} suggest that given sufficient resources one can
create a $\adfPENT(3,r,9)$ for any $r \ge 40$, $r \equiv 1 \adfmod{3}$
using a generalized Petersen graph with $(2r + 10)/3$ vertices.
Unfortunately, the problem with hill climbing is that there seems to be no way other than by trial to prove that it will work.
There is nothing special about the maximum value, $r = 100$; it is merely where we got fed up with verifying the existence of the geometries.
If we allow deficiency graphs with more than one component, i.e.\ type D, we can do better.
First, we require a type-C $\adfPENT(3,r,9)$ with $r \equiv 0 \adfmod{3}$.
%
\begin{lemma}\label{lem:PENT-3-72-9-direct}
There exists a generalized pentagonal geometry $\adfPENT(3,72,9)$ with a connected deficiency graph of girth $4$.
\end{lemma}

\begin{proof}


{\boldmath $\adfPENT(3,72,9)$}, $d = 2$:
{0, 26, 74;  0, 56, 79;  0, 77, 122;  1, 27, 107;  1, 33, 38;  1, 99, 116;  26, 48, 56;  26, 77, 128;  26, 79, 122;  27, 33, 99;  27, 38, 129;  27, 81, 116;  33, 116, 129;  38, 81, 99;  38, 107, 116;  48, 77, 79;  56, 74, 77;  56, 122, 128;  74, 79, 128;  99, 107, 129;  0, 2, 89;  0, 4, 49;  0, 7, 127;  0, 9, 16;  0, 10, 131;  0, 11, 123;  0, 12, 85;  0, 14, 115;  0, 15, 110;  0, 17, 46;  0, 19, 113;  0, 20, 62;  0, 24, 118;  0, 25, 133;  0, 28, 95;  0, 34, 84;  0, 35, 114;  0, 37, 47;  0, 41, 57;  0, 55, 93;  0, 63, 64;  0, 68, 151;  0, 97, 141;  0, 135, 139;  1, 13, 127;  1, 15, 119;  1, 21, 91;  1, 25, 87}.
The deficiency graph is connected and has girth 4.
%
\end{proof}
%
\begin{theorem}\label{thm:PENT-3-r-9-constructed}
For all admissible $r \ge 1413$, there exists a $\adfPENT(3,r,9)$ of type D
where each component of the deficiency graph has girth $4$.
\end{theorem}
\begin{proof}
Use Theorem~\ref{thm:PENT-3-r-w-constructed-large} with the $\adfPENT(3,72,9)$ from Lemma~\ref{lem:PENT-3-72-9-direct}
as well as the $\adfPENT(3,25,9)$ and $\adfPENT(3,28,9)$ from Lemma~\ref{lem:PENT-3-r-9-direct}.
The parameters of Theorem~\ref{thm:PENT-3-r-w-constructed-large} are $(r_0, r_1, r_2) = (72,25,28)$.
\end{proof}

Next, we deal with $w = 7$.
%
\begin{lemma}
\label{lem:PENT-3-r-7-direct}
There exists a $\adfPENT(3,r,7)$ with a connected deficiency graph of girth $4$ for $r \in \{47, 51, 53\}$.
\end{lemma}
\begin{proof}
Each of these is created from a connected $(7,4)$-graph with $2r + 8$ vertices
and suitable neighbourhood intersection sizes. Details are provided in the appendix.
\end{proof}

%
\begin{theorem}
\label{thm:PENT-3-r-7-constructed}
For all admissible $r \ge 3396$, there exists a $\adfPENT(3,r,7)$ of type D
where each component of the deficiency graph has girth $4$.
\end{theorem}
\begin{proof}
This follows from Lemma~\ref{lem:PENT-3-r-7-direct} and
Theorem~\ref{thm:PENT-3-r-w-constructed-large} with  $(r_0, r_1, r_2) = (51,47,53)$.
\end{proof}

We can prove similar theorems for some larger values of $w$.
%
\begin{lemma}
\label{lem:PENT-3-r-w-direct}
There exists a $\adfPENT(3,r,w)$ with a connected deficiency graph of girth $4$ for
\begin{align*}
(r,w) \in \{  &(246,  13), (248,  13), (254,  13),
               (333,  15), (55,   15), (61,   15), \\
              &(561,  19), (563,  19), (569,  19),
               (24,   21), (64,   21), (73,   21), \\
              &(1452, 25), (1454, 25), (1484, 25),
                           (31,   27), (217,  27), \\
              &
                           (310,  33), (316,  33),
               (45,   39), (136,  39), (505,  39)  \}.
\end{align*}
There exists a $\adfPENT(3,r,w)$ with a connected deficiency graph for
\begin{align*}
(r,w) \in \{  (1389, 27), (2550, 33) \}.
\end{align*}
\end{lemma}
\begin{proof}
The $\adfPENT(3,55,15)$ and the $\adfPENT(3,61,15)$ were created from $(5,5^+)$-graphs using Construction~\ref{con:PENT-k-r-w-constructed};
details are in the appendix.

For the $\adfPENT(3,24,21)$ and $\adfPENT(3,73,21)$, use Theorem~\ref{thm:Product-construction} with
the $\adfPENT(3,3)$ and the $\adfPENT(3,10)$ of Lemma~\ref{lem:PENT-3-r-3-direct}.
For the $\adfPENT(3,64,21)$, see Theorem~\ref{thm:PENT-k-r-w-from-Hoffman-Singleton}.

For the $\adfPENT(3,31,27)$ and $\adfPENT(3,217,27)$, use the tripling construction, Theorem~\ref{thm:Tripling-construction}, with
the $\adfPENT(3,10,9)$ of Lemma~\ref{lem:PENT-3-r-9-direct} and
the $\adfPENT(3,72,9)$ of Lemma~\ref{lem:PENT-3-72-9-direct}.

For the $\adfPENT(3,1389,27)$ and $\adfPENT(3,2550,33)$, see \cite[Lemma 3.4]{ForbesRutherford2022}.

The $\adfPENT(3,310,33)$ and $\adfPENT(3,316,33)$ were created from $(11,5^+)$-graphs using Construction~\ref{con:PENT-k-r-w-constructed};
details are in the appendix.

For the $\adfPENT(3,45,39)$ and $\adfPENT(3,136,39)$, use Theorem~\ref{thm:Product-construction} with $h = 13$ and
the $\adfPENT(3,s)$, $s \in \{3, 10\}$ of Lemma~\ref{lem:PENT-3-r-3-direct}.
For the $\adfPENT(3,505,39)$, use Theorem~\ref{thm:Tripling-construction} with
the $\adfPENT(3,168,13)$ of \cite[Lemma 3.4]{ForbesRutherford2022}.

The remaining geometries were created from suitable $(w,4)$-graphs, as described in Lemma~\ref{lem:PENT-3-r-3-direct};
details are in the appendix.
\end{proof}
%
\begin{theorem}
\label{thm:PENT-3-r-13-15-19-21-25-27-33-39}
For each $w \in \{13, 15, 19, 21, 25, 27, 33, 39\}$,
for all admissible $r \ge r_{w}$, where
\begin{center}
\begin{tabular}{llll}
$r_{13} = 68873$, &
$r_{15} = 6018$, &
$r_{19} = 336944$, &
$r_{21} = 7009$, \\
$r_{25} = 2209439$, &
$r_{27} = 16503$, &
$r_{33} = 120378$, &
$r_{39} = 84079$,
\end{tabular}
\end{center}
there exists a $\adfPENT(3,r,w)$ of type D.
\end{theorem}
\begin{proof}
The theorem follows from Lemma~\ref{lem:PENT-3-r-w-direct} and Theorem~\ref{thm:PENT-3-r-w-constructed-large} with parameters
$(w, r_0, r_1, r_2)$ =
(13, 246, 248, 254),
(15, 333, 55, 61),
(19, 561, 563, 569),
(21, 24, 64, 73),
(25, 1452, 1454, 1484),
(27, 1389, 31, 217),
(33, 2550, 310, 316) and
(39, 45, 136, 505).
Since $t, u > 0$ in (\ref{eqn:PENT-3-r-w-constructed}) there is always a component with girth 4
in the deficiency graph of the constructed geometry.
Recall that $r$ is admissible if $r \equiv 0$ or $w + 1 \adfmod{3}$.
\end{proof}

%

\section{Block sizes 4, 5, 6 and 7}\label{sec:Block sizes 4, 5, 6 and 7}
For $k \in \{4, 5\}$, the number of points in a generalized pentagonal geometry $\adfPENT(k,r,w)$ is $(k - 1)r + w + 1$,
the parameter $w$ must satisfy $w \equiv 1$ or $k \adfmod{k(k - 1)}$, and the admissibility condition is
$r \equiv 0$ or $w + 1 \adfmod{k}$.

Here is a type-A geometry that arrived too late to be included in \cite{ForbesRutherford2022}.
%
\begin{lemma}
\label{lem:PENT-4-r-13-direct}
There exists a $\adfPENT(4,168,13)$ with a connected deficiency graph of girth $5$.
\end{lemma}
\begin{proof}

{\boldmath $\adfPENT(4,168,13)$}, $d = 2$:
{11, 17, 42, 455;  11, 87, 179, 471;  11, 135, 343, 476;  11, 155, 265, 517;  17, 87, 343, 517;  17, 135, 265, 471;  17, 155, 179, 476;  42, 87, 265, 476;  42, 135, 179, 517;  42, 155, 343, 471;  87, 135, 155, 455;  179, 265, 343, 455;  455, 471, 476, 517;  2, 48, 64, 421;  2, 99, 340, 432;  2, 176, 384, 502;  2, 254, 364, 508;  48, 99, 384, 508;  48, 176, 364, 432;  48, 254, 340, 502;  64, 99, 364, 502;  64, 176, 340, 508;  64, 254, 384, 432;  99, 176, 254, 421;  340, 364, 384, 421;  421, 432, 502, 508;  261, 432, 363, 134;  327, 462, 113, 185;  74, 509, 23, 434;  236, 251, 60, 409;  213, 437, 476, 334;  420, 162, 211, 267;  312, 18, 413, 116;  509, 156, 299, 260;  494, 193, 185, 102;  276, 250, 154, 209;  490, 461, 200, 361;  458, 351, 321, 508;  512, 274, 399, 373;  358, 212, 250, 195;  0, 491, 200, 166;  451, 446, 23, 224;  504, 398, 431, 29;  179, 65, 125, 4;  323, 66, 89, 296;  267, 80, 319, 47;  238, 219, 3, 179;  263, 423, 297, 394;  213, 203, 496, 493;  172, 115, 112, 355;  218, 437, 98, 102;  412, 321, 215, 213;  126, 226, 337, 315;  305, 151, 342, 34;  466, 475, 444, 89;  121, 444, 205, 408;  394, 404, 210, 149;  134, 162, 421, 69;  455, 408, 462, 83;  0, 7, 232, 240;  0, 2, 21, 71;  0, 1, 314, 366;  0, 14, 77, 289;  0, 13, 140, 409;  0, 25, 214, 359;  0, 72, 423, 505;  0, 119, 313, 465;  0, 65, 159, 284;  0, 73, 216, 246;  0, 131, 154, 367;  0, 79, 323, 337;  0, 401, 439, 443;  0, 56, 123, 404;  0, 40, 247, 417;  0, 66, 151, 160;  0, 181, 403, 469;  0, 82, 281, 428;  0, 102, 250, 407;  0, 133, 282, 447;  0, 215, 251, 483;  0, 32, 203, 306;  0, 147, 267, 295;  0, 132, 333, 473;  0, 303, 399, 503}.
The deficiency graph is connected and has girth 5.
%
\end{proof}

As we indicated in Section~\ref{sec:Introduction}, small type-A pentagonal geometries with block size 5 are not common.
The only ones we are aware of are the five $\adfPENT(5,r)$ that occur in \cite[Lemma 4.4]{Forbes2020P},
where in each case $r \equiv 0 \adfmod{5}$.
Here, we expand the collection by presenting one more as well as five further examples with $r$ belonging to the other admissible residue class, $r \equiv 1 \adfmod{5}$.
%
\begin{lemma}
\label{lem:PENT-5-r-5-direct}
There exists a pentagonal geometry $\adfPENT(5,r)$ with a connected deficiency graph of girth at least $5$ for
$$r \in \{20, 21, 25, 26, 30, 31, 35, 36, 40, 41, 45\}.$$
\end{lemma}
\begin{proof}

{\boldmath $\adfPENT(5,21)$}, $d = 2$:
{1, 2, 53, 75, 88;  0, 11, 16, 38, 81;  0, 27, 31, 33, 57;  0, 25, 39, 46, 87;  0, 13, 24, 47, 59;  0, 10, 29, 58, 70;  0, 5, 28, 37, 45;  0, 6, 14, 21, 40;  0, 18, 36, 54, 72;  1, 19, 37, 55, 73}.
The last 2 blocks represent short orbits.
The deficiency graph is connected and has girth 5.
Identifying code: $\{x: x \equiv 0 \text{~or~} 2 \adfmod{6}, 0 \le x < 90\}$,
$\pi_0 = 0$, $\pi_1 = 30$, $\pi_2 = 60$.

{\boldmath $\adfPENT(5,26)$}, $d = 2$:
{15, 43, 47, 52, 58;  41, 64, 68, 71, 96;  0, 1, 60, 100, 102;  0, 21, 35, 41, 86;  0, 5, 54, 77, 90;  0, 12, 31, 81, 84;  0, 16, 33, 34, 64;  0, 29, 37, 53, 71;  0, 39, 91, 93, 103;  0, 13, 49, 75, 96;  0, 22, 44, 66, 88;  1, 23, 45, 67, 89}.
The last 2 blocks represent short orbits.
The deficiency graph is connected and has girth 5.

{\boldmath $\adfPENT(5,30)$}, $d = 2$:
{1, 2, 87, 113, 124;  0, 11, 14, 40, 117;  90, 63, 58, 22, 115;  62, 99, 8, 72, 50;  0, 6, 19, 24, 80;  0, 16, 33, 50, 98;  0, 15, 53, 60, 107;  0, 9, 30, 69, 73;  0, 7, 38, 83, 101;  0, 20, 55, 71, 79;  0, 8, 29, 31, 75;  1, 7, 29, 85, 97}.
The deficiency graph is connected and has girth 5.
Identifying code: $\{x: x \equiv 0 \text{~or~} 2 \adfmod{6}, 0 \le x < 126\}$,
$\pi_0 = 0$, $\pi_1 = 42$, $\pi_2 = 84$.

{\boldmath $\adfPENT(5,31)$}, $d = 2$:
{54, 65, 76, 81, 101;  30, 47, 50, 66, 85;  104, 108, 57, 45, 98;  108, 121, 15, 123, 21;  0, 1, 61, 74, 98;  1, 9, 41, 75, 89;  0, 29, 30, 64, 115;  0, 12, 40, 99, 109;  0, 8, 31, 70, 88;  0, 3, 7, 75, 92;  0, 9, 14, 63, 86;  0, 21, 39, 46, 48;  0, 26, 52, 78, 104;  1, 27, 53, 79, 105}.
The last 2 blocks represent short orbits.
The deficiency graph is connected and has girth 5.

{\boldmath $\adfPENT(5,36)$}, $d = 2$:
{1, 2, 11, 105, 148;  0, 25, 46, 127, 140;  15, 32, 64, 127, 106;  135, 14, 103, 52, 138;  24, 146, 95, 8, 15;  0, 5, 102, 110, 116;  0, 20, 61, 72, 126;  0, 22, 49, 80, 145;  0, 31, 43, 47, 84;  0, 18, 68, 91, 135;  0, 15, 17, 88, 125;  0, 33, 55, 69, 131;  1, 7, 27, 35, 85;  0, 29, 36, 93, 111;  0, 30, 60, 90, 120;  1, 31, 61, 91, 121}.
The last 2 blocks represent short orbits.
The deficiency graph is connected and has girth 5.
Identifying code: $\{x: x \equiv 0 \text{~or~} 2 \adfmod{6}, 0 \le x < 150\}$,
$\pi_0 = 0$, $\pi_1 = 50$, $\pi_2 = 100$.

{\boldmath $\adfPENT(5,41)$}, $d = 2$:
{26, 79, 144, 149, 157;  14, 21, 22, 92, 151;  62, 131, 143, 47, 134;  16, 167, 39, 57, 124;  35, 78, 60, 151, 155;  133, 1, 53, 62, 58;  0, 1, 15, 47, 141;  0, 37, 61, 143, 159;  0, 29, 51, 74, 163;  0, 39, 65, 67, 116;  0, 14, 56, 101, 111;  0, 2, 21, 27, 135;  0, 3, 58, 107, 138;  0, 6, 30, 46, 63;  0, 11, 28, 64, 150;  0, 10, 22, 60, 104;  0, 34, 68, 102, 136;  1, 35, 69, 103, 137}.
The last 2 blocks represent short orbits.
The deficiency graph is connected and has girth 5.

{\boldmath $\adfPENT(5,45)$}, $d = 2$:
{54, 57, 101, 119, 132;  68, 86, 91, 97, 130;  7, 64, 18, 85, 3;  79, 177, 80, 143, 150;  111, 184, 34, 17, 143;  31, 159, 179, 182, 184;  47, 77, 98, 172, 89;  0, 1, 49, 71, 85;  0, 43, 93, 95, 149;  0, 9, 37, 83, 123;  0, 41, 100, 151, 159;  0, 4, 17, 130, 137;  0, 52, 131, 141, 157;  0, 30, 75, 99, 118;  0, 10, 92, 117, 164;  0, 6, 34, 48, 87;  0, 15, 80, 146, 170;  0, 8, 20, 58, 84}.
The deficiency graph is connected and has girth 5.
Identifying code: $\{x: x \equiv 0 \text{~or~} 1 \adfmod{6}, 0 \le x < 186\}$,
$\pi_0 = 0$, $\pi_1 = 62$, $\pi_2 = 124$.

See \cite[Lemma 4.4]{Forbes2020P} for the others.
\end{proof}

Next, we have a small improvement on \cite[Theorem 12]{GriggsStokes2016}.
%
\begin{theorem}
There exist pentagonal geometries $\adfPENT(5,r)$ for all positive integers $r \equiv 1 \adfmod{5}$
except $r = 6$
and except possibly $r \in$ \{$11$, $16$, $66$, $96$, $116$\}. 
\end{theorem}
\begin{proof}
For $r \not\in \{36, 56, 81, 86\}$, see \cite[Theorem 12]{GriggsStokes2016}.
For $r \in \{36, 56, 81, 86\}$,
proceed as in the proof of \cite[Theorem 12]{GriggsStokes2016} using
5-GDDs of types $10^{u}$, $u \in \{15, 23, 33, 35\}$ from Lemma~\ref{lem:k-GDD-existence}.
\end{proof}
The construction in \cite[Theorem 12]{GriggsStokes2016} uses the degenerate $\adfPENT(5,1)$ and therefore it produces type-E geometries.
However, with Lemma~\ref{lem:PENT-5-r-5-direct} we  have enough material for a proof of the following.
\begin{theorem}
\label{thm:PENT-5-r-constructed}
There exist pentagonal geometries $\adfPENT(5,r)$ with deficiency graphs of girth at least $5$ for
all $r \ge 200000$, $r \equiv 0$ or $1 \adfmod{5}$.
\end{theorem}
\begin{proof}
Let $\rho = 0$ or 1 and let $r \ge 200000$ be an integer that is congruent to $\rho$ modulo 5.
We show how to assemble a type-B $\adfPENT(5,r)$ using type-A pentagonal geometries
$\adfPENT(5, s)$, $s \in \{20, 21, 26\}$ from Lemma~\ref{lem:PENT-5-r-5-direct}.
We refer to Lemma~\ref{lem:k-GDD-existence} for the existence of the various group divisible designs employed in the construction.

Let $v = 4r + 6$ and note that
\begin{equation}\label{eqn:PENT-5-r-v}
v \equiv 6 + 4 \rho \adfmod{20}.
\end{equation}
Let $h= 86 + 4 \rho = 86$ or 90. Let $q_\mathrm{min} = \lceil v/129 \rceil$, $q_\mathrm{max} = \lfloor v/111 \rfloor$, and
choose $q$ such that
\begin{adfenumerate}
\item[] $q$ is an odd integer in the range $q_\mathrm{min} \le q \le q_\mathrm{max}$,
\item[] $q \equiv 0 \adfmod{11}$ and
\item[] $v - 100q \equiv 0 \adfmod{h}$.
\end{adfenumerate}
This choice is possible by (\ref{eqn:PENT-5-r-v}) and the Chinese Remainder Theorem since
$\gcd(11,h) = 1$ and $q_\mathrm{max} - q_\mathrm{min}$ is sufficiently large.
Let
\begin{equation}\label{eqn:PENT-5-r-m-D}
m = v - 100q, \text{~~} D = \{10, 18, 30\},
\end{equation}
and recall that there exists a 5-GDD of type $10^{10} d^1$ for each $d \in D$.
Since $q \ge 1937$ there exists an 11-GDD of type $q^{11}$.
Also, from (\ref{eqn:PENT-5-r-v}) and (\ref{eqn:PENT-5-r-m-D}) we have
$m \equiv 2 \adfmod{4}$ and
\begin{equation}\label{eqn:PENT-5-r-M}
10q + 16 < 11q \le m \le 29q < 30q - 32.
\end{equation}
Hence $m$ can be represented as a sum of $q$ elements of $D$,
\begin{equation} \label{eqn:PENT-5-r-m}
m = d_1 + d_2 + \dots + d_q, ~~ d_1, d_2, \dots, d_q \in D.
\end{equation}

Inflate each point in ten groups of the 11-GDD by a factor of 10.
Inflate the $q$ points of the remaining group by factors $d_1$, $d_2$, \dots, $d_q$ as determined in (\ref{eqn:PENT-5-r-m}).
Overlay the blocks of the inflated 11-GDD with 5-GDDs of types $10^{10} d^1$, $d \in D$, as appropriate.
The result is a 5-GDD of type $(10q)^{10} m^1$.

Take this 5-GDD and overlay each of the ten groups that have $10q$ points with 5-GDDs of type $110^{q/11}$,
which exist because $q/11 \equiv 1 \adfmod{2}$ and $q/11 \ge 5$.
Overlay the group that has $m$ points with a 5-GDD of type $h^{b}$, where $b = m/h$.
For the existence of this 5-GDD, we have $h \equiv 0 \adfmod{2}$, $b \equiv 1 \adfmod{2}$, $b \ge 21$, and
if $h = 86$ then (\ref{eqn:PENT-5-r-v}) and (\ref{eqn:PENT-5-r-m-D}) together imply $b \equiv 1 \adfmod{10}$.

Thus we have created a 5-GDD of type $110^{10q/11} h^{m/h}$.
Now apply Theorem~\ref{thm:GDD-basic}, overlaying each group of size 86, 90 or 110 with
a type-A $\adfPENT(5, 20)$, $\adfPENT(5, 21)$ or $\adfPENT(5, 26)$, as appropriate.
The final result is a type-B pentagonal geometry with $110 \cdot 10q/11 + m = v$ points.
This is a $\adfPENT(5,r)$ with a deficiency graph of girth at least $5$.
\end{proof}

There is no type-A or type-B $\adfPENT(5,r)$ for $r \in \{1,5,6\}$.
With more work the lower bound for $r$ in Theorem~\ref{thm:PENT-5-r-constructed} can be improved.
If we make use of all known type-A $\adfPENT(5,s)$ as well as
all $q$ for which there exist 9 MOLS of side $q$, and
if in the construction, suitably modified to allow for even $q$,
we consider all $m$ that are sums of $q$ elements from $\{10, 18, 30\}$,
not just those in the range (\ref{eqn:PENT-5-r-M}),
we can reduce the lower bound from 200000 to 18431 and the number of possible exceptions to 1735.
However, the proof is rather messy and will not be presented here. 

Using the product construction, Theorem~\ref{thm:Product-construction},
with our current stock of type-A geometries with block sizes 4 and 5, namely
the $\adfPENT(4,r)$ from \cite{Forbes2020P},
the $\adfPENT(4,r,13)$ from \cite{ForbesRutherford2022},
the $\adfPENT(4,168,13)$ from Lemma~\ref{lem:PENT-4-r-13-direct}, and
the $\adfPENT(5,r)$ from Lemma~\ref{lem:PENT-5-r-5-direct},
we can generate infinitely many examples of type-C generalized pentagonal geometries.
If $h \equiv 1$ or $k \adfmod{k(k - 1)}$, we are guaranteed the existence of
the required Steiner system $S(2,k,h)$ by \cite{Hanani1961} and
the $k$-GDD of type $h^k$ by Lemma~\ref{lem:k-GDD-existence}.
For example, from the $\adfPENT(4,13)$ we obtain
$\adfPENT(4, 13h + (h - 1)/3, 4h)$, $h = 4$, 13, 16, 25, \dots.

As a further illustration of Theorem~\ref{thm:Product-construction},
assume $h \ge 31$, $h \equiv 1$ or $6 \adfmod{15}$
and that a Steiner system $S(2,6,h)$ exists, which will certainly be the case if
$h \notin $
\{36, 46,
51, 61, 81, 166, 226, 231, 256, 261, 286, 316, 321,
346, 351, 376, 406, 411, 436, 441, 471, 501, 561, 591,
616, 646, 651, 676, 771, 796, 801\} (see \cite{AdamsBryantBuchanan} and the references therein).
Since $h \ge 31$ there exists a 6-GDD of type $h^6$ by Lemma~\ref{lem:k-GDD-existence}.
Furthermore, there exists a $\adfPENT(6,7)$, \cite{BallBambergDevillersStokes2013}.
Hence, by Theorem~\ref{thm:Product-construction}, there exists a $\adfPENT(6,7h + (h - 1)/5, 6h)$.
The smallest of these, corresponding to $h = 31$, where the Steiner system is a projective plane of order 5,
is a $\adfPENT(6, 223, 186)$ with a connected deficiency graph of girth 4.

Geometries $\adfPENT(7,7)$ and $\adfPENT(3,21,7)$ were constructed from the Hoffman--Singleton graph in \cite{BallBambergDevillersStokes2013} and \cite{ForbesRutherford2022}, respectively.
Here are three more based on the same graph, including one with block size 7.
%
\begin{theorem}\label{thm:PENT-k-r-w-from-Hoffman-Singleton}
There exist generalized pentagonal geometries
$\adfPENT(3, 64, 21)$, $\adfPENT(4,57,28)$ and $\adfPENT(7,50,49)$,
where each has a connected deficiency graph of girth $4$.
\end{theorem}
\begin{proof}
Apply Construction~\ref{con:PENT-k-r-w-constructed} with the 50-vertex Hoffman--Singleton $(7,5)$-graph, \cite{HoffmanSingleton1960}, and appropriate parameters.
In each case,
$h$ is the multiplier,
$v$ is the number of points,
$b$ is the number of lines,
the Steiner system $S(2,k,h)$ is just a single block, and
Lemma~\ref{lem:k-GDD-existence} asserts the existence of the required $k$-GDD of type $h^7$.
There are no non-opposite lines. Details are in the appendix.


\vskip 1mm
{\boldmath\bf $\adfPENT(3, 64, 21)$}: $h = k = 3$, $v = 150$, $b = 3200$.

{\boldmath\bf $\adfPENT(4, 57, 28)$}: $h = k = 4$, $v = 200$, $b = 2850$.

{\boldmath\bf $\adfPENT(7, 50, 49)$}: $h = k = 7$, $v = 350$, $b = 2500$.
\end{proof}

%

\section{Identifying codes for pentagonal geometries}\label{sec:Identifying codes}

Let $G$ be a graph and let $S$ be a subset of its vertex set.
We say that $S$ is an {\em identifying code} for $G$ if
the intersection of neighbourhoods with $S$, $N(x) \cap S$, $x \in V(G)$, are distinct.
A {\em minimum identifying code} for $G$ is an identifying code for $G$ that is as small as possible.
An identifying code for a linear space is an identifying code for its deficiency graph.

This is not the only way to define identifying codes.
We could have used closed neighbourhoods, $(N(x) \cup \{x\}) \cap S$,
or more generally $B_d(x) \cap S$, where $B_d(x)$ is the set of vertices at distance $d$ or less from $x$.
For directed graphs there is also the choice of in- and out-neighbourhoods, ordinary or closed,
\cite{KarpovskyChakrabarty1998}, \cite{LaifenfeldTrachtenberg2008}.
Furthermore, in the extension to linear spaces there are two other graphs that one might consider,
the collinearity graph, i.e.\ the complement of the deficiency graph, or
the incidence graph, where the vertices are the points and lines of the space, and
point $p$ is adjacent to line $\ell$ if $p$ is incident with $\ell$,
\cite{Araujo-PardoBalbuenaMontejanoValenzuela2011}.
However the definitions we have adopted seem to be appropriate for pentagonal geometries since
in this case the neighbourhood $N(x)$ of vertex $x$ in the deficiency graph is a line of the geometry, $x^\mathrm{opp}$.

An identifying code $S$ for $G$ creates a $1:1$ correspondence between $\{N(x) \cap S: x \in V(G)\}$ and $V(G)$.
Clearly, the existence of an identifying code requires the neighbourhoods to be distinct.
Since a complete multipartite graph with at least one part of size greater than 1 does not have distinct neighbourhoods,
a pentagonal geometry can have an identifying code only if the deficiency graph has girth at least 5.
However, it is possible for a generalized pentagonal geometry of type~C
to have an identifying code even though its deficiency graph has girth 4.

Let $S$ be an identifying code for a graph, $G$.
We define $\pi_i(S)$ to be the number of vertices $x$ of $G$ such that $|N(x) \cap S| = i$.
In particular, $\pi_0(S) = 0$ if the empty set is not a member of $\{N(x) \cap S: x \in V(G)\}$,
$\pi_0(S) = 1$ otherwise.
We write just $\pi_i$ instead of $\pi_i(S)$ if the identifying code is clear from the context.

%
\begin{lemma}
\label{lem:id-code-components}
Let $G$ be a graph with components $G_1$, $G_2$, \dots, $G_m$ having identifying codes $S_1$, $S_2$, \dots, $S_m$, respectively.
Suppose $\pi_0(S_i) = 0$ for $i = 1, 2, \dots, m$ with at most one exception.
Then $S = S_1 \cup S_2 \cup \dots \cup S_m$ is an identifying code of size $|S_1| + |S_2| + \dots + |S_m|$ for $G$.
Moreover, if $S_i$ is a minimum identifying code for $G_i$, $i = 1, 2, \dots, m$,
then $S$ is a minimum identifying code for $G$.
\end{lemma}
\begin{proof}
The first part follows since the vertex sets of the components are mutually disjoint and
the empty set occurs at most once amongst all of the neighbourhood intersections
$\{N(x) \cap S_i: x \in V(G_i)\}$, $i = 1, 2, \dots, m$.
For the second part, suppose $T$ is a smaller identifying code for $G$.
Then there exists $i$ such that $|V(G_i) \cap T| < |S_i|$, a contradiction.
\end{proof}

The following theorem and its proof are adapted from \cite{BallBambergDevillersStokes2013}.

%
\begin{theorem} \label{thm:id-code-lower-bound}
Suppose $G$ is a $w$-regular graph with $v$ vertices and $S$ is an identifying code for $G$.
Then
\begin{equation}\label{eqn:id-code-|S|-lower-bound}
|S| \ge \left\lceil \dfrac{2(v - 1)}{w + 1} \right\rceil.
\end{equation}
Moreover, if $|S| = 2(v - 1)/(w + 1)$, then
\begin{equation} \label{eqn:id-code-pi0-pi1-pi2}
\left\{\begin{array}{l}
\pi_0(S) = 1, ~~~~
\pi_1(S) = |S|, \\
\pi_2(S) = \dfrac{(v - 1)(w - 1)}{w + 1}, ~~~~
\pi_3(S) = \pi_4(S) = \dots = 0.
\end{array}\right.
\end{equation}
\end{theorem}
\begin{proof}
Recall that $\pi_i = \pi_i(S)$ is the number of vertices $x$ of $G$ such that $|N(x) \cap S| = i$.
We have $\pi_0 \le 1$, $\pi_1 \le |S|$,
\begin{equation}\label{eqn:id-code-sum-pi-i}
\sum_{i = 0}^w \pi_i = v \text{\;\;\;\; and \;\;\;\;} \sum_{i = 0}^w i \pi_i = |S| w.
\end{equation}
For the latter equality, observe that each $x \in S$ contributes 1 to the sum for each neighbour of $x$.
Therefore
\begin{equation} \label{eqn:id-code-|S|w - 2v}
|S|w - 2v = \sum_{i=0}^w (i - 2)\pi_i
          = -2 \pi_0 - \pi_1 + \sum_{i = 3}^w(i - 2)\pi_i \ge -2 - |S|
\end{equation}
from which (\ref{eqn:id-code-|S|-lower-bound}) follows.
If $|S| = 2(v - 1)/(w + 1)$, then (\ref{eqn:id-code-|S|w - 2v}) becomes an equality, which implies (\ref{eqn:id-code-pi0-pi1-pi2}).
\end{proof}

A {\em best possible identifying code} for a graph or a pentagonal geometry is
an identifying code that meets the bound (\ref{eqn:id-code-|S|-lower-bound}).
%
\begin{lemma}
\label{lem:id-code-cycle}
If $n \ge 3$ and $n \not\equiv 4 \adfmod{6}$, then the $n$-cycle graph $C_n$ has a best possible identifying code.

The cycle graph $C_{6q + 4}$, $q \ge 1$, has a minimum identifying code of size $4q + 3$;
$C_4$ has no identifying code.
\end{lemma}
\begin{proof}
The graph $C_4$ does not have distinct neighbourhoods.
Otherwise, we shall determine $\mu_n$, the size of a minimum identifying code for a $C_n$, $n \neq 4$.

It is not difficult to compute $\mu_3 = 2$ and $\mu_5 = 3$, both of which meet the bound (\ref{eqn:id-code-|S|-lower-bound}).

Assume the $C_n$ has vertex set $V = \{1, 2, \dots, n\}$ and let $n = 6q + d$, $q$ a positive integer.
In Table~\ref{tab:id-code-C-n} we represent as strings of zeros and ones the characteristic functions $V \rightarrow \{0,1\}$ of
identifying codes for the $C_{6q + d}$, $d = 0, 1, \dots, 5$.
From the table we see that $\mu_{6q + 4} \le 4q + 3$ and that
$$ \mu_{6q + d} = 4q + \lceil 2(d - 1)/3\rceil = \lceil 2(n - 1)/3 \rceil,~ d \in \{0,1,2,3,5\}, $$
an equality since the right-hand side meets the bound (\ref{eqn:id-code-|S|-lower-bound}).

\begin{table}[h]
\begin{center}
\begin{tabular}{lll}
$n = 6q$     & $\chi_S = $ 111100\,111100\,\dots111100\,      & $|S| = 4q$     \\
$n = 6q + 1$ & $\chi_S = $ 111100\,111100\,\dots111100\,0     & $|S| = 4q$     \\
$n = 6q + 2$ & $\chi_S = $ 111100\,111100\,\dots111100\,01    & $|S| = 4q + 1$ \\
$n = 6q + 3$ & $\chi_S = $ 111100\,111100\,\dots111100\,011   & $|S| = 4q + 2$ \\
$n = 6q + 4$ & $\chi_S = $ 111100\,111100\,\dots111100\,0111  & $|S| = 4q + 3$ \\
$n = 6q + 5$ & $\chi_S = $ 111100\,111100\,\dots111100\,10110 & $|S| = 4q + 3$
\end{tabular}
\end{center}
\caption{Identifying codes for $n$-cycles}
\label{tab:id-code-C-n}
\end{table}

It remains to prove that there is no identifying code of length $4q + 2$ for the $C_{6q + 4}$.
Assume $q \ge 2$ and suppose $S$, with characteristic function $\chi_S: V \rightarrow \{0,1\}$, is an identifying code of length $4q + 2$ for the $C_{6q + 4}$.
Observe that $4q + 2$ is the bound (\ref{eqn:id-code-|S|-lower-bound}) without the ceiling brackets;
therefore (\ref{eqn:id-code-pi0-pi1-pi2}) holds and hence
the neighbourhood intersections $N(x) \cap S$ must contain the empty set as well as $\{x\}$ for each $x \in S$.
There are two cases to consider.

Case 1.
Assume the empty set arises from $\chi_S$ as {\bf 000}.
That is, three consecutive vertices of the cycle are not in $S$.
A straightforward argument,
where it might help to observe that $S$ cannot contain more than four consecutive vertices of the cycle,
shows that $\chi_S$ in the vicinity of {\bf 000} must look like this:
$$ \dots 0 0 1 1 1 1 0 0 1 1 1 1 {\bf 0 0 0} 1 1 1 1 0 0 1 1 1 1 0 0 \dots.$$
By adding 6-blocks 001111 on the left and 6-blocks 111100 on the right,
we see that where two branches of this pattern meet we have
$$ \dots 1 1 1 1 0 0 1 1 1 1 0 0x0 0 1 1 1 1 0 0 1 1 1 1 \dots. $$
But it is impossible to replace $x$ with two 1s.

Case 2.
Assume the empty set arises as {\bf 010}.
Then, without loss of generality, $\chi_S$ in the vicinity of {\bf 010} looks like this:
$$ \dots 0 1 1 1 1 0 0 1 1 1 1 0 {\bf 0 1 0} 1 1 0 1 1 0 \dots.$$
Now we add 011110 on the left and 110110 on the right until the two branches meet:
$$ \dots  0 1 1 1 1 0 x 0 1 1 0 1 1 \dots.$$
The neighbourhood intersection corresponding to $x$ is $\{\}$, a contradiction.

Finally, a simplified version of the previous argument confirms that $C_{10}$ cannot have an identifying code of size 6.
\end{proof}

%
\begin{theorem}
\label{thm:id-code-PENT-2-r}
A pentagonal geometry $\adfPENT(2,r)$ with a connected deficiency graph has a best possible identifying code if
$r \not\equiv 1 \adfmod{6}$.
\end{theorem}
\begin{proof}
The deficiency graph is an $(r + 3)$-cycle and therefore the theorem follows from Lemma~\ref{lem:id-code-cycle}.
\end{proof}

%
\begin{lemma}
\label{lem:id-code-GP-graph}
Let $q$ and $t$ be integers such that $q \ge 5$ and $1 \le t < q/2$.
Then the generalized Petersen graph $\mathrm{GP}(q,t)$ has a best possible identifying code with $\pi_0 = 0$.
\end{lemma}
\begin{proof}
We may assume that the vertex set of the $\mathrm{GP}(q,t)$ is $Z_q \times Z_2$ and that its edges
consists of
$$\{(i,0), (i+1,0)\},\; \{(i,1), (i+t,1)\},\; \{(i,0), (i,1)\}, \;\;\; i \in Z_q.$$
Let $S = \{(i,0): i \in Z_q\}$ and note that $|S|$ meets the bound (\ref{eqn:id-code-|S|-lower-bound}).
It is clear that $\{N((i,1)) \cap S: i \in Z_q\}$ = $S$,
that $\{N((i,0)) \cap S: i \in Z_q\}$ consists of distinct pairs of elements of $S$, and
that none of these neighbourhood intersections is empty.
\end{proof}

For a 3-regular graph, the number of vertices $v$ is even and
therefore inequality (\ref{eqn:id-code-|S|-lower-bound}) simplifies to $|S| \ge v/2$.
Hence, by Lemma~\ref{lem:id-code-components}, the size of a best possible identifying code for a non-connected 3-regular graph
is the sum of the sizes of the best possible identifying codes for the components,
provided that all with at most one exception have $\pi_0 = 0$.
This is also the case for any $w$-regular graph, $w > 3$, where $2v$ is divisible by $w + 1$, for then (\ref{eqn:id-code-|S|-lower-bound})
simplifies to the integer $2v/(w+1)$.

Next, we prove two theorems concerning the existence of pentagonal geometries $\adfPENT(3,r)$ with best possible identifying codes.
The number of points is $v = 2r + 4$ and the lower bound of Theorem~\ref{thm:id-code-lower-bound} is $v/2 = r + 2$.

%
\begin{lemma}\label{lem:id-code-PENT-3-r-direct}
For each
$$r \in \{3, 9, 10, 12, 15, 16, 19, 21, 22, 24, 25, 27\},$$
there exists a pentagonal geometry $\adfPENT(3,r)$ with a connected deficiency graph and a best possible identifying code $S$
that satisfies $\pi_0(S) = 0$.
\end{lemma}
\begin{proof}
This follows from Lemmas~\ref{lem:PENT-3-r-3-direct} and \ref{lem:id-code-GP-graph}.
\end{proof}

%
\begin{theorem} \label{thm:id-code-PENT-3-r-constructed}
There exists a $\adfPENT(3,r)$ with a best possible identifying code
for all positive integers $r \equiv 0 \text{ or } 1 \adfmod{3}$ except $r \in \{1, 4, 6, 7\}$.
\end{theorem}
\begin{proof}
We refer to Theorem 2.2 of \cite{ForbesGriggsStokes2020} in which the existence spectrum for
$\adfPENT(3,r)$ with deficiency graph girth at least 5 is determined.
The proof of \cite[Theorem 2.2]{ForbesGriggsStokes2020} uses $\adfPENT(3,q)$ for
$$q \in Q = \{3, 9, 10, 12, 15, 16, 19, 21, 22, 24, 25, 27\},$$
to construct pentagonal geometries $\adfPENT(3,r)$ for the range of $r$ as stated in our theorem.
The construction, the same as that of Theorem~\ref{thm:GDD-basic}, creates a $\adfPENT(3,r)$
by filling in the groups of a 3-GDD with $\adfPENT(3,q_i)$, $q_i \in Q$, $i = 1, 2, \dots, m$, say,
which by Lemma~\ref{lem:id-code-PENT-3-r-direct} have best possible identifying codes with $\pi_0 = 0$.
Hence we may apply Lemma~\ref{lem:id-code-components}.
The resulting deficiency graph has a best possible identifying code of size $r + 2 = \sum_{i=1}^m (q_i + 2)$.
\end{proof}

The $\adfPENT(3,r)$ constructed in the proof of Theorem~\ref{thm:id-code-PENT-3-r-constructed} do not have connected deficiency graphs.
We address this partially in the next theorem.

%
\begin{theorem} \label{thm:id-code-PENT-3-r-connected}
For all positive integers $r \equiv 3 \adfmod{6}$,
there exists a $\adfPENT(3,r)$ with a connected deficiency graph and a best possible identifying code.
\end{theorem}
\begin{proof}
See Lemma~\ref{lem:id-code-PENT-3-r-direct} for $r \in \{3, 9, 15, 21, 27\}$.

In the construction for $m \ge 5$ of a $\adfPENT(3, 6m+3)$ described in the proof of \cite[Theorem 2.1]{Forbes2020P}
the deficiency graph is a generalized Petersen graph, $\mathrm{GP}(6m+5,2)$.
Therefore, by Lemma~\ref{lem:id-code-GP-graph}, the geometry has a best possible identifying code.
\end{proof}

We now consider pentagonal geometries $\adfPENT(4,r)$.
The number of points is $3r + 5$ and a best possible identifying code has size $\lceil 2(3r + 4)/5\rceil$.
When $r \equiv 0 \adfmod{5}$, one way of constructing a suitable deficiency graph is given by the next lemma.

%
\begin{lemma}
\label{lem:id-code-5-cycle-GP-graph}
Let $q$ be a positive integer and let $t_0, t_1, \dots, t_4$ be integers such that
$1 \le t_0, t_1, \dots, t_4 < q/2$.
Let $G$ be the graph defined by $V(G) = Z_q \times Z_5$ and, writing $x_i$ for $(x,i)$,
$$E(G) = \bigcup \{\{\{x_i, (x + t_i)_i\}, \{x_i, x_{i+1}\}\}: x \in Z_q, i \in Z_5\}.$$
Then G is $4$-regular and has a best possible identifying code given by
$$S = \{x_0: x \in Z_q\} \cup \{x_2: x \in Z_q\}.$$
Moreover, $\pi_0(S) = 0$.
\end{lemma}
\begin{proof}
The restrictions on $t_0$, $t_1$, \dots, $t_4$  ensure that $G$ is 4-regular.
As $x$ runs through the elements of $Z_q$, we have
\begin{align*}
N(x_0) \cap S &= \{(x - t_0)_0, (x + t_0)_0\}, \\
N(x_2) \cap S &= \{(x - t_2)_2, (x + t_2)_2\}, \\
N(x_1) \cap S &= \{x_0, x_2\}, ~~
N(x_3) \cap S  = \{x_2\}, ~~
N(x_4) \cap S  = \{x_0\},
\end{align*}
which are clearly distinct and do not include the empty set.
\end{proof}

%
\begin{lemma} \label{lem:id-code-PENT-4-r-direct}
For each
\begin{equation} \label{eqn:id-code-PENT-4-r-direct}
\begin{aligned}r \in \{& 13, 17, 20, 21, 24, 25, 40, 45, 60, 65, 80, 85, \\
        & 100, 105, 120, 125, 140, 145, 160, 165, 180, 185\},
\end{aligned}
\end{equation}
there exists a pentagonal geometry $\adfPENT(4,r)$ with a connected deficiency graph of girth at least $5$ and a best possible identifying code.
Additionally, the identifying code of the $\adfPENT(4,r)$ has $\pi_0 = 0$ when $r = 13$ and when $r \equiv 0 \adfmod{5}$.
\end{lemma}
\begin{proof}
We provide details in the standard format.
The $\adfPENT(4,r)$ for $r \equiv 0 \adfmod{20}$ were created from deficiency graphs constructed by Lemma~\ref{lem:id-code-5-cycle-GP-graph}
with parameters $t_0 = t_2 = 1$, $t_1 = t_3 = 2$, $t_4 = 3$.
For $r \equiv 5 \adfmod{20}$ we used the graph with edges generated from
$$\{\{0, 1\}, \{0, 2\}, \{1, 5\}, \{0, 11\}\}$$
by $x \mapsto x + 2 \adfmod{3r+5}$.
This graph has an identifying code generated from
$\{0, 2, 8, 9, 11, 13, 14, 15\}$ by $x \mapsto x + 20 \adfmod{3r+5}$.

%

{\boldmath $\adfPENT(4,13)$}, $d = 4$:
{3, 38, 41, 42;  4, 27, 31, 34;  4, 8, 13, 23;  0, 17, 21, 26;  0, 6, 7, 16;  1, 14, 22, 38;  1, 18, 30, 43;  0, 8, 18, 37;  0, 2, 31, 32;  0, 11, 22, 39;  0, 1, 20, 33;  1, 3, 9, 29;  1, 15, 23, 35}.
The deficiency graph is connected and has girth 6.
Identifying code: \{4, 5, 9, 12, 14, 16, 18, 20, 22, 24, 26, 28, 30, 34, 38, 40, 41, 42\},
$\pi_0 = 0$, $\pi_1 = 18$, $\pi_2 = 24$, $\pi_3 = 2$.

{\boldmath $\adfPENT(4,17)$}, $d = 8$:
{4, 19, 34, 43;  21, 25, 33, 38;  4, 13, 24, 54;  15, 16, 22, 40;  0, 2, 46, 53;  12, 15, 41, 50;  10, 20, 25, 43;  15, 51, 53, 55;  9, 24, 37, 15;  47, 19, 11, 14;  0, 1, 3, 48;  0, 5, 26, 40;  0, 17, 20, 51;  0, 27, 33, 54;  1, 26, 27, 43;  1, 17, 36, 55;  1, 12, 13, 44;  1, 11, 37, 46;  0, 15, 25, 29;  0, 7, 49, 50;  1, 30, 50, 54;  1, 18, 23, 34;  0, 18, 37, 52;  2, 5, 31, 34;  0, 10, 22, 35;  2, 10, 30, 52;  2, 39, 45, 51;  3, 13, 14, 52;  4, 6, 14, 37;  3, 21, 37, 45;  3, 4, 12, 55;  0, 14, 23, 44;  0, 12, 28, 39;  6, 7, 31, 46}.
The deficiency graph is connected and has girth 5.
Identifying code: \{1, 2, 7, 10, 15, 16, 17, 21, 27, 29, 30, 32, 33, 36, 39, 43, 44, 48, 49, 50, 52, 55\},
$\pi_0 = 1$, $\pi_1 = 22$, $\pi_2 = 33$.

{\boldmath $\adfPENT(4,20)$}, $d = 5$:
{1, 4, 5, 60;  0, 2, 11, 56;  1, 3, 7, 62;  2, 4, 13, 58;  0, 3, 19, 54;  55, 27, 35, 62;  45, 31, 39, 2;  38, 56, 22, 8;  55, 52, 40, 13;  0, 8, 16, 40;  0, 14, 30, 58;  0, 18, 31, 43;  0, 29, 36, 52;  0, 13, 44, 47;  0, 17, 32, 39;  0, 21, 46, 53;  0, 24, 34, 42;  0, 26, 48, 63;  1, 24, 36, 63;  1, 6, 27, 47;  2, 23, 29, 42;  1, 19, 52, 64;  1, 14, 39, 43;  1, 16, 44, 49;  2, 19, 38, 43}.
The deficiency graph is connected and has girth 5.
Identifying code: $\{x: 0 \le x < 65,\; x \equiv 0 \text{ or } 2 \adfmod{5}\}$,
$\pi_0 = 0$, $\pi_1 = 26$, $\pi_2 = 39$.

{\boldmath $\adfPENT(4,21)$}, $d = 2$:
{18, 29, 39, 50;  30, 31, 39, 40;  33, 30, 8, 15;  43, 31, 11, 57;  0, 5, 33, 49;  0, 13, 19, 26;  0, 6, 37, 41;  0, 20, 63, 65;  0, 4, 12, 27;  0, 2, 16, 40;  0, 17, 34, 51}.
The last block represents a short orbit.
The deficiency graph is connected and has girth 5.
Identifying code: \{0, 2, 4, 6, 8, 10, 12, 16, 18, 22, 24, 28, 30, 34, 36, 40, 42, 46, 48, 52, 54, 56, 58, 60, 62, 64, 66\},
$\pi_0 = 1$, $\pi_1 = 26$, $\pi_2 = 41$.

{\boldmath $\adfPENT(4,24)$}, $d = 7$:
{4, 10, 55, 59;  23, 27, 30, 46;  23, 50, 57, 58;  19, 21, 34, 70;  0, 36, 39, 46;  40, 41, 47, 66;  28, 47, 52, 57;  47, 25, 55, 71;  54, 34, 51, 2;  44, 17, 8, 28;  76, 62, 39, 33;  46, 55, 34, 44;  46, 13, 62, 70;  0, 43, 54, 58;  48, 71, 50, 75;  75, 43, 29, 26;  58, 20, 18, 59;  0, 1, 6, 48;  0, 9, 27, 34;  1, 3, 17, 20;  0, 8, 14, 76;  1, 13, 18, 31;  1, 24, 34, 45;  2, 31, 55, 68;  1, 19, 41, 72;  0, 15, 50, 74;  1, 29, 66, 67;  1, 40, 53, 74;  0, 2, 22, 35;  1, 52, 59, 65;  0, 25, 37, 73;  2, 10, 45, 53;  0, 11, 31, 40;  0, 7, 30, 52;  0, 17, 47, 67;  0, 5, 38, 61;  0, 3, 18, 65;  0, 12, 51, 56;  2, 11, 12, 74;  2, 47, 59, 61;  0, 32, 60, 68;  2, 16, 33, 60}.
The deficiency graph is connected and has girth 6.
Identifying code: \{0, 2, 3, 4, 7, 9, 10, 14, 15, 17, 21, 24, 27, 30, 31, 35, 38, 39, 40, 44, 45, 49, 51, 52, 56, 58, 64, 66, 70, 72, 73\},
$\pi_0 = 1$, $\pi_1 = 30$, $\pi_2 = 44$, $\pi_3 = 2$.

{\boldmath $\adfPENT(4,25)$}, $d = 2$:
{1, 2, 11, 78;  0, 5, 70, 77;  56, 21, 9, 43;  1, 12, 42, 37;  0, 6, 35, 63;  0, 37, 51, 53;  0, 17, 47, 73;  0, 12, 43, 61;  0, 21, 27, 59;  0, 19, 28, 44;  0, 18, 41, 56;  0, 8, 22, 54;  0, 20, 40, 60;  1, 21, 41, 61}.
The last 2 blocks represent short orbits.
The deficiency graph is connected and has girth 5.
Identifying code: $\{x + 20i: x \in \{0,2,8,9,11,13,14,15\}, 0 \le i < 4\}$,
$\pi_0 = 0$, $\pi_1 = 32$, $\pi_2 = 48$.

{\boldmath $\adfPENT(4,40)$}, $d = 5$:
{1, 4, 5, 120;  0, 2, 11, 116;  1, 3, 7, 122;  2, 4, 13, 118;  0, 3, 19, 114;  65, 89, 62, 9;  104, 97, 58, 65;  64, 30, 3, 72;  71, 13, 32, 49;  97, 80, 31, 47;  62, 32, 83, 111;  6, 116, 74, 39;  67, 103, 41, 91;  71, 50, 38, 44;  101, 5, 27, 58;  81, 112, 52, 65;  122, 97, 109, 35;  94, 117, 72, 38;  115, 54, 100, 80;  0, 94, 27, 123;  0, 7, 65, 77;  0, 13, 57, 97;  0, 14, 82, 102;  0, 37, 52, 71;  0, 8, 72, 107;  1, 6, 14, 62;  0, 36, 54, 117;  1, 18, 39, 42;  1, 22, 36, 64;  1, 37, 53, 119;  2, 34, 68, 94;  2, 49, 89, 99;  2, 54, 73, 78;  1, 54, 58, 72;  0, 26, 49, 104;  1, 89, 94, 114;  0, 44, 46, 80;  1, 28, 66, 109;  1, 8, 59, 108;  0, 25, 55, 84;  1, 43, 74, 113;  0, 38, 68, 109;  0, 31, 61, 101;  0, 48, 51, 88;  0, 18, 63, 111;  0, 33, 41, 86;  0, 56, 81, 103;  0, 40, 83, 106;  0, 28, 78, 93;  0, 23, 58, 75}.
The deficiency graph is connected and has girth 5.
Identifying code: $\{x: 0 \le x < 125,\; x \equiv 0 \text{ or } 2 \adfmod{5}\}$,
$\pi_0 = 0$, $\pi_1 = 50$, $\pi_2 = 75$.

{\boldmath $\adfPENT(4,45)$}, $d = 2$:
{1, 2, 11, 138;  0, 5, 130, 137;  55, 25, 72, 19;  113, 99, 35, 118;  136, 70, 14, 64;  11, 134, 98, 91;  126, 59, 57, 9;  25, 76, 52, 119;  22, 121, 130, 44;  88, 24, 68, 115;  0, 12, 26, 124;  0, 8, 29, 110;  0, 25, 51, 82;  0, 41, 69, 80;  0, 34, 95, 119;  0, 39, 52, 100;  0, 19, 94, 125;  0, 37, 62, 111;  0, 63, 81, 103;  0, 33, 45, 117;  1, 17, 55, 99;  0, 55, 75, 107;  0, 35, 70, 105}.
The last block represents a short orbit.
The deficiency graph is connected and has girth 5.
Identifying code: $\{x + 20i: x \in \{0,2,8,9,11,13,14,15\}, 0 \le i < 7\}$,
$\pi_0 = 0$, $\pi_1 = 56$, $\pi_2 = 84$.

See the appendix for the others.
\end{proof}

As far as the authors are aware, the
$\adfPENT(4,25)$,
$\adfPENT(4,105)$ and
$\adfPENT(4,145)$
of Lemma~\ref{lem:id-code-PENT-4-r-direct} are new---in the sense that
type-A pentagonal geometries with the same parameters have not been previously published.
For the other $r$ in (\ref{eqn:id-code-PENT-4-r-direct}), examples of these $\adfPENT(4,r)$
have appeared in {\cite{Forbes2020P} and \cite{ForbesGriggsStokes2020} but without identifying codes.

Lemma~\ref{lem:id-code-PENT-4-r-direct} enables us to determine up to a small number of possible exceptions
the $r$ for which there exists a $\adfPENT(4,r)$ with a best possible identifying code.
%
\begin{theorem} \label{thm:id-code-PENT-4-r-constructed}

For each integer $r \ge 5316$, $r \equiv 0 \text{ or } 1 \adfmod{4}$,
there exists a pentagonal geometry $\adfPENT(4, r)$ with a deficiency of girth at least $5$
and a best possible identifying code.
\end{theorem}
\begin{proof}
Let
$$Q = \{13, 17, 20, 21, 24\}.$$
We invoke Theorem~\ref{thm:GDD-basic} using as ingredients the $\adfPENT(4,q)$,
$q \in Q \cup \{25\}$, of Lemma~\ref{lem:id-code-PENT-4-r-direct}, each of which
has a best possible identifying code, which we denote by $S_q$.
See Lemma~\ref{lem:k-GDD-existence} for the existence of the relevant 4-GDDs.

For $q \in Q$,
$\beta \in \{3, 4, 5, 6\}$ and
$$\gamma = 0 \text{ or }
\gamma = \left\{\begin{array}{ll} 1 & \text{ if $q$ is even} , \\
                                  3 & \text{ if $q$ is odd},
                \end{array} \right.$$
take one copy of the $\adfPENT(4,q)$,
$12 \beta + 3 \gamma$ copies of the $\adfPENT(4,20)$, and
a 4-GDD of type $65^{12 \beta + 3\gamma} (3q + 5)^1$
to obtain a
$$\adfPENT(4, 260\beta + 65 \gamma + q),~ 3 \le \beta \le 6.$$
For $\alpha \ge 45$, take this $\adfPENT(4, 260\beta + 65 \gamma + q)$,
$3 \alpha$ copies of the $\adfPENT(4,25)$, and
a 4-GDD of type $80^{3 \alpha} (780 \beta + 195 \gamma + 3q + 5)^1$
to obtain a $\adfPENT(4, r)$, where
$$r = 80 \alpha + 260\beta + 65 \gamma + q,~ \alpha \ge 45,~ 3 \le \beta \le 6.$$

The deficiency graph $D_r$ of this $\adfPENT(r)$ has $3 \alpha$ components of 80 vertices,
$12 \beta + 3 \gamma$ components of 65 vertices and one component of $3q + 5$ vertices.
Each component has a best possible identifying code and $\pi_0(S_{20}) = \pi_0(S_{25}) = 0$.
Therefore, by Lemma~\ref{lem:id-code-components},
$S_r$, the union of the $3\alpha + 12\beta + 3 \gamma + 1$ identifying codes is a minimum identifying code for $D_r$.
Also, noting that $|S_{20}| = 26$ and $|S_{25}| = 32$, we have
\begin{align*}
|S_r| &= (3 \alpha) \cdot 32 + (12 \beta + 3\gamma) \cdot 26 + \left\lceil \dfrac{2(3q + 4)}{5} \right\rceil \\
    &= \left\lceil \dfrac{2 \big((3 \alpha) \cdot 80 + (12 \beta + 3\gamma) \cdot 65 + 3q + 4\big)}{5} \right\rceil,
\end{align*}
which meets the bound (\ref{eqn:id-code-|S|-lower-bound}) of Theorem~\ref{thm:id-code-lower-bound}.
Hence $S_r$ is a best possible identifying code for the $\adfPENT(4,r)$.

\begin{table}[h]
\begin{center}
\begin{tabular}{ccc|ccc}
$r \adfmod{20}$ & $\gamma$ & $q$ & $r \adfmod{20}$ & $\gamma$ & $q$ \\
\hline
 0              & 0        & 20  & 1               & 0        & 21 \\
 4              & 0        & 24  & 5               & 1        & 20 \\
 8              & 3        & 13  & 9               & 1        & 24 \\
 12             & 3        & 17  & 13              & 0        & 13 \\
 16             & 3        & 21  & 17              & 0        & 17
\end{tabular}
\end{center}
\caption{Theorem \ref{thm:id-code-PENT-4-r-constructed} construction}
\label{tab:id-code-PENT-4-r-constructed}
\end{table}

By the Chinese Remainder Theorem, for any integer $n$, there exist integers $x$, $y$ and $d$ such that
$$n = 4 x + 13 y + 52 d,\;\; 45 \le x \le 57,\;\; 3 \le y \le 6.$$
It is easily verified that $d$ is non-negative for all $n \ge 255$.
Hence $80 \alpha + 260 \beta$ covers all multiples of 20 that are at least $20 \cdot 255 = 5100$.
Therefore, observing that $65 \gamma + q \le 216$, we conclude that
a $\adfPENT(4,r)$ with a best possible identifying code exists whenever $r \ge 5316$ and $r \equiv 5 \gamma + q \adfmod{20}$.
We choose $\gamma \in \{0,1,3\}$ and $q \in Q$ as specified in Table~\ref{tab:id-code-PENT-4-r-constructed}
to show that all $r \equiv 0 \text{ or } 1 \adfmod{4}$ are represented.
\end{proof}

Using the same technique as for the $\adfPENT(4,25)$,
we have obtained several new type-A $\adfPENT(4,r)$;
details are provided by the next lemma.
We use Lemmas~\ref{lem:id-code-PENT-4-r-direct} and \ref{lem:PENT-4-r-direct} to improve \cite[Theorem 3.2]{Forbes2020P}
and incidentally settle the existence of opposite-line-pair-free $\adfPENT(4,r)$ for odd $r \neq 9$.
%
\begin{lemma} \label{lem:PENT-4-r-direct}
For each
$$\begin{aligned}r \in \{& 41, 57, 73, 89, 113, 121, 129, 137, 153, 161, 169, \\
        & 177, 181, 189, 193, 197\},
\end{aligned}$$
there exists a pentagonal geometry $\adfPENT(4,r)$ with a connected deficiency graph of girth at least $5$.
\end{lemma}
\begin{proof}
%
%

{\boldmath $\adfPENT(4,41)$}, $d = 2$:
{3, 12, 57, 116;  15, 72, 115, 126;  38, 48, 94, 54;  49, 45, 98, 5;  35, 105, 107, 87;  108, 14, 56, 27;  108, 30, 5, 28;  88, 43, 118, 1;  0, 1, 7, 68;  0, 5, 20, 29;  0, 4, 77, 114;  0, 39, 58, 107;  0, 26, 63, 85;  0, 8, 44, 95;  0, 19, 22, 55;  0, 23, 89, 101;  0, 81, 97, 127;  0, 21, 115, 123;  0, 27, 62, 90;  0, 31, 111, 121;  0, 32, 64, 96;  1, 33, 65, 97}.
The last 2 blocks represent short orbits.
The deficiency graph is connected and has girth 6.

{\boldmath $\adfPENT(4,57)$}, $d = 2$:
{23, 82, 94, 107;  55, 70, 123, 154;  174, 112, 9, 74;  123, 140, 111, 82;  166, 156, 133, 124;  175, 115, 9, 158;  23, 81, 148, 137;  35, 96, 41, 62;  0, 1, 2, 21;  0, 3, 112, 172;  0, 5, 8, 148;  0, 6, 20, 110;  0, 15, 130, 160;  0, 18, 68, 103;  0, 22, 78, 102;  0, 37, 39, 124;  0, 43, 48, 93;  0, 47, 54, 119;  0, 59, 81, 106;  0, 40, 127, 163;  0, 69, 83, 135;  0, 137, 141, 167;  0, 26, 101, 139;  0, 57, 99, 131;  0, 49, 97, 125;  0, 55, 79, 95;  1, 19, 65, 99;  0, 63, 71, 157;  0, 44, 88, 132;  1, 45, 89, 133}.
The last 2 blocks represent short orbits.
The deficiency graph is connected and has girth 6.

See the appendix for the others.
\end{proof}

%
\begin{theorem} \label{thm:id-code-PENT-4-r-constructed}
There exist pentagonal geometries $\adfPENT(4,r)$ with deficiency graphs of girth at least $5$
for all positive integers $r \equiv 0 \textrm{~or~} 1 \adfmod{4}$
except for $r \in $
\{$1$, $4$, $5$\}
and except possibly for
$r \in $
\{$8$, $9$, $12$, $16$, $28$, $32$, $36$, $44$, $48$, $56$, $64$, $68$, $72$, $76$, $84$, $88$, $92$, $96$, $104$, $116$, $124$, $128$, $144$, $148$, $164$, $168$\}.
\end{theorem}
\begin{proof}
This follows from \cite[Theorem 3.2]{Forbes2020P}, Lemma~\ref{lem:id-code-PENT-4-r-direct}, Lemma~\ref{lem:PENT-4-r-direct},
and applications of Theorem~\ref{thm:GDD-basic} for $r \in \{212, 308\}$.

For $\adfPENT(4,212)$, use 6 copies of a $\adfPENT(4,25)$, a $\adfPENT(52)$, \cite[Lemma 3.2]{Forbes2020P}, and a 4-GDD of type $80^6 161^1$.

For $\adfPENT(4,308)$, use 6 copies of a $\adfPENT(4,41)$, a $\adfPENT(52)$ and a 4-GDD of type $128^6 161^1$.
\end{proof}

We finish the paper by noting that the $\adfPENT(5,r)$, $r \in \{21, 30, 36, 45\}$ of Lemma~\ref{lem:PENT-5-r-5-direct}
have best possible identifying codes with $\pi_0 = 0$.

%

\section*{ORCID}
\noindent A. D. Forbes     \url{https://orcid.org/0000-0003-3805-7056}\\
C. G. Rutherford           \url{https://orcid.org/0000-0003-1924-207X}

%


\newpage
\appendix

%
\section{$\adfPENT(k,r,w)$ details}
\newcommand{\adfAppGap}{~}%
{\scriptsize


{\boldmath $\adfPENT(3,9)$}, $d = 22$:
{1, 2, 20;  0, 9, 15;  0, 3, 4;  2, 11, 17;  2, 5, 6;  4, 13, 19;  4, 7, 8;  6, 15, 21;  6, 9, 10;  1, 8, 17;  8, 11, 12;  3, 10, 19;  10, 13, 14;  5, 12, 21;  12, 15, 16;  1, 7, 14;  14, 17, 18;  3, 9, 16;  16, 19, 20;  5, 11, 18;  0, 18, 21;  7, 13, 20;  0, 5, 10;  0, 6, 12;  0, 7, 11;  0, 8, 16;  0, 13, 17;  0, 14, 19;  1, 3, 5;  1, 4, 16;  1, 6, 13;  1, 10, 18;  1, 11, 21;  1, 12, 19;  2, 7, 19;  2, 8, 21;  2, 9, 14;  2, 10, 15;  2, 12, 18;  2, 13, 16;  3, 6, 18;  3, 7, 12;  3, 8, 20;  3, 13, 15;  3, 14, 21;  4, 9, 12;  4, 10, 20;  4, 11, 14;  4, 15, 18;  4, 17, 21;  5, 7, 16;  5, 8, 14;  5, 9, 20;  5, 15, 17;  6, 11, 16;  6, 14, 20;  6, 17, 19;  7, 9, 18;  7, 10, 17;  8, 13, 18;  8, 15, 19;  9, 11, 13;  9, 19, 21;  10, 16, 21;  11, 15, 20;  12, 17, 20}.
The deficiency graph is connected and has girth 6.

{\boldmath $\adfPENT(3,10)$}, $d = 24$:
{1, 2, 22;  0, 11, 15;  0, 3, 4;  2, 13, 17;  2, 5, 6;  4, 15, 19;  4, 7, 8;  6, 17, 21;  6, 9, 10;  8, 19, 23;  8, 11, 12;  1, 10, 21;  10, 13, 14;  3, 12, 23;  12, 15, 16;  1, 5, 14;  14, 17, 18;  3, 7, 16;  16, 19, 20;  5, 9, 18;  18, 21, 22;  7, 11, 20;  0, 20, 23;  9, 13, 22;  0, 5, 7;  0, 6, 13;  0, 8, 14;  0, 9, 12;  0, 10, 18;  0, 16, 21;  0, 17, 19;  1, 3, 19;  1, 4, 20;  1, 6, 16;  1, 7, 9;  1, 8, 13;  1, 12, 18;  1, 17, 23;  2, 7, 12;  2, 8, 20;  2, 9, 16;  2, 10, 19;  2, 11, 18;  2, 14, 21;  2, 15, 23;  3, 5, 10;  3, 6, 20;  3, 8, 18;  3, 9, 21;  3, 11, 14;  3, 15, 22;  4, 9, 11;  4, 10, 17;  4, 12, 21;  4, 13, 18;  4, 14, 23;  4, 16, 22;  5, 8, 16;  5, 11, 17;  5, 12, 22;  5, 13, 20;  5, 21, 23;  6, 11, 23;  6, 12, 19;  6, 14, 22;  6, 15, 18;  7, 10, 22;  7, 13, 15;  7, 14, 19;  7, 18, 23;  8, 15, 21;  8, 17, 22;  9, 14, 20;  9, 15, 17;  10, 15, 20;  10, 16, 23;  11, 13, 16;  11, 19, 22;  12, 17, 20;  13, 19, 21}.
The deficiency graph is connected and has girth 6.

{\boldmath $\adfPENT(3,12)$}, $d = 4$:
{1, 2, 26;  0, 11, 19;  0, 3, 4;  2, 13, 21;  0, 5, 10;  0, 6, 18;  0, 7, 23;  0, 8, 22;  0, 9, 12;  0, 13, 17;  0, 15, 21;  1, 3, 17;  1, 7, 14;  1, 22, 27;  2, 10, 27;  2, 11, 15}.
The deficiency graph is connected and has girth 6.

{\boldmath $\adfPENT(3,13)$}, $d = 6$:
{1, 2, 28;  0, 9, 23;  0, 3, 4;  2, 11, 25;  2, 5, 6;  4, 13, 27;  0, 5, 15;  0, 6, 27;  0, 7, 17;  0, 8, 18;  0, 10, 22;  0, 11, 14;  0, 13, 16;  0, 19, 25;  1, 3, 27;  1, 5, 20;  1, 10, 29;  1, 13, 26;  1, 16, 21;  2, 8, 22;  2, 9, 21;  2, 10, 23;  2, 14, 27;  3, 5, 29;  3, 16, 22;  4, 11, 29}.
The deficiency graph is connected and has girth 7.

{\boldmath $\adfPENT(3,15)$}, $d = 2$:
{1, 2, 32;  0, 11, 25;  0, 5, 23;  0, 6, 27;  0, 7, 29;  0, 8, 20;  0, 9, 24;  0, 13, 17;  0, 15, 18;  1, 3, 29}.
The deficiency graph is connected and has girth 7.

{\boldmath $\adfPENT(3,16)$}, $d = 12$:
{1, 2, 34;  0, 11, 27;  0, 3, 4;  2, 13, 29;  2, 5, 6;  4, 15, 31;  4, 7, 8;  6, 17, 33;  6, 9, 10;  8, 19, 35;  8, 11, 12;  1, 10, 21;  0, 5, 12;  0, 6, 13;  0, 7, 10;  0, 8, 15;  0, 9, 20;  0, 14, 30;  0, 16, 33;  0, 17, 22;  0, 18, 23;  0, 19, 25;  0, 21, 28;  0, 26, 31;  1, 3, 5;  1, 4, 16;  1, 6, 15;  1, 7, 9;  1, 8, 28;  1, 13, 35;  1, 14, 20;  1, 18, 32;  1, 19, 33;  1, 22, 29;  2, 9, 21;  2, 10, 19;  2, 11, 14;  2, 15, 33;  2, 16, 31;  2, 17, 23;  2, 20, 28;  2, 22, 32;  2, 27, 30;  3, 7, 15;  3, 8, 22;  3, 9, 17;  3, 10, 34;  3, 11, 16;  3, 18, 35;  3, 20, 33;  4, 9, 11;  4, 10, 23;  4, 17, 29;  4, 18, 34;  4, 22, 30;  5, 7, 19;  5, 8, 18;  5, 9, 22;  5, 20, 32;  5, 23, 35;  6, 18, 31;  6, 21, 35;  7, 11, 22;  7, 20, 35}.
The deficiency graph is connected and has girth 8.

{\boldmath $\adfPENT(3,18)$}, $d = 8$:
{1, 2, 38;  0, 9, 33;  0, 3, 4;  2, 11, 35;  2, 5, 6;  4, 13, 37;  4, 7, 8;  6, 15, 39;  0, 5, 8;  0, 6, 22;  0, 7, 18;  0, 10, 24;  0, 11, 17;  0, 12, 20;  0, 13, 19;  0, 14, 27;  0, 15, 29;  0, 21, 31;  0, 23, 25;  0, 28, 34;  0, 30, 35;  1, 3, 12;  1, 4, 34;  1, 5, 23;  1, 6, 27;  1, 7, 26;  1, 10, 18;  1, 11, 31;  1, 13, 15;  1, 14, 37;  1, 19, 29;  1, 20, 36;  1, 21, 30;  1, 22, 28;  2, 7, 19;  2, 13, 26;  2, 14, 21;  2, 15, 28;  2, 20, 37;  2, 22, 39;  2, 27, 30;  3, 5, 39;  3, 7, 28;  3, 14, 29;  3, 15, 20;  3, 21, 36;  4, 14, 22;  4, 15, 30}.
The deficiency graph is connected and has girth 5.

{\boldmath $\adfPENT(3,19)$}, $d = 6$:
{1, 2, 40;  0, 11, 33;  0, 3, 4;  2, 13, 35;  2, 5, 6;  4, 15, 37;  0, 5, 9;  0, 6, 25;  0, 7, 14;  0, 8, 32;  0, 10, 24;  0, 12, 29;  0, 13, 37;  0, 15, 22;  0, 16, 35;  0, 20, 26;  0, 21, 34;  0, 23, 31;  0, 27, 39;  1, 3, 39;  1, 4, 16;  1, 5, 38;  1, 7, 28;  1, 9, 26;  1, 13, 27;  1, 14, 41;  1, 17, 20;  1, 29, 34;  2, 9, 32;  2, 10, 17;  2, 15, 33;  2, 16, 34;  2, 22, 28;  2, 23, 39;  3, 5, 22;  3, 11, 40;  3, 17, 28;  5, 11, 23}.
The deficiency graph is connected and has girth 7.

{\boldmath $\adfPENT(3,21)$}, $d = 2$:
{1, 2, 44;  0, 11, 37;  0, 5, 27;  0, 6, 30;  0, 7, 20;  0, 8, 36;  0, 9, 32;  0, 12, 29;  0, 13, 31;  0, 15, 21;  0, 19, 35;  0, 25, 39;  0, 41, 43;  1, 5, 13}.
The deficiency graph is connected and has girth 8.

{\boldmath $\adfPENT(3,22)$}, $d = 48$:
{1, 2, 46;  0, 11, 39;  0, 3, 4;  2, 13, 41;  2, 5, 6;  4, 15, 43;  4, 7, 8;  6, 17, 45;  6, 9, 10;  8, 19, 47;  8, 11, 12;  1, 10, 21;  10, 13, 14;  3, 12, 23;  12, 15, 16;  5, 14, 25;  14, 17, 18;  7, 16, 27;  16, 19, 20;  9, 18, 29;  18, 21, 22;  11, 20, 31;  20, 23, 24;  13, 22, 33;  22, 25, 26;  15, 24, 35;  24, 27, 28;  17, 26, 37;  26, 29, 30;  19, 28, 39;  28, 31, 32;  21, 30, 41;  30, 33, 34;  23, 32, 43;  32, 35, 36;  25, 34, 45;  34, 37, 38;  27, 36, 47;  36, 39, 40;  1, 29, 38;  38, 41, 42;  3, 31, 40;  40, 43, 44;  5, 33, 42;  42, 45, 46;  7, 35, 44;  0, 44, 47;  9, 37, 46;  0, 5, 37;  0, 6, 30;  0, 7, 14;  0, 8, 31;  0, 9, 23;  0, 10, 17;  0, 12, 34;  0, 13, 26;  0, 15, 27;  0, 16, 43;  0, 18, 28;  0, 19, 25;  0, 20, 40;  0, 21, 45;  0, 22, 29;  0, 24, 36;  0, 32, 38;  0, 33, 41;  0, 35, 42;  1, 3, 18;  1, 4, 26;  1, 5, 12;  1, 6, 37;  1, 7, 9;  1, 8, 28;  1, 13, 16;  1, 14, 22;  1, 15, 47;  1, 17, 31;  1, 19, 41;  1, 20, 34;  1, 23, 44;  1, 24, 40;  1, 25, 42;  1, 27, 43;  1, 30, 35;  1, 32, 45;  1, 33, 36;  2, 7, 26;  2, 8, 33;  2, 9, 44;  2, 10, 36;  2, 11, 19;  2, 12, 40;  2, 14, 21;  2, 15, 29;  2, 16, 32;  2, 17, 47;  2, 18, 43;  2, 20, 38;  2, 22, 42;  2, 23, 34;  2, 24, 45;  2, 25, 28;  2, 27, 39;  2, 30, 37;  2, 31, 35;  3, 5, 46;  3, 6, 27;  3, 7, 30;  3, 8, 37;  3, 9, 25;  3, 10, 44;  3, 11, 28;  3, 14, 43;  3, 15, 34;  3, 16, 22;  3, 17, 35;  3, 19, 33;  3, 20, 36;  3, 21, 47;  3, 24, 38;  3, 26, 39;  3, 29, 45;  3, 32, 42;  4, 9, 42;  4, 10, 20;  4, 11, 47;  4, 12, 28;  4, 13, 31;  4, 14, 23;  4, 16, 38;  4, 17, 30;  4, 18, 39;  4, 19, 45;  4, 21, 46;  4, 22, 36;  4, 24, 37;  4, 25, 44;  4, 27, 29;  4, 32, 41;  4, 33, 35;  4, 34, 40;  5, 7, 41;  5, 8, 45;  5, 9, 17;  5, 10, 35;  5, 11, 23;  5, 13, 28;  5, 16, 36;  5, 18, 34;  5, 19, 26;  5, 20, 30;  5, 21, 27;  5, 22, 47;  5, 24, 32;  5, 29, 40;  5, 31, 38;  5, 39, 44;  6, 11, 29;  6, 12, 20;  6, 13, 18;  6, 14, 32;  6, 15, 28;  6, 16, 40;  6, 19, 31;  6, 21, 33;  6, 22, 46;  6, 23, 39;  6, 24, 42;  6, 25, 41;  6, 26, 36;  6, 34, 44;  6, 35, 43;  6, 38, 47;  7, 10, 43;  7, 11, 24;  7, 12, 21;  7, 13, 40;  7, 15, 22;  7, 18, 36;  7, 19, 46;  7, 20, 37;  7, 23, 25;  7, 28, 42;  7, 29, 31;  7, 32, 39;  7, 33, 38;  7, 34, 47;  8, 13, 36;  8, 14, 44;  8, 15, 32;  8, 16, 42;  8, 17, 34;  8, 18, 24;  8, 20, 27;  8, 21, 23;  8, 22, 38;  8, 25, 30;  8, 26, 40;  8, 29, 43;  8, 35, 39;  8, 41, 46;  9, 11, 41;  9, 12, 35;  9, 13, 20;  9, 14, 36;  9, 15, 40;  9, 16, 30;  9, 21, 43;  9, 22, 28;  9, 24, 39;  9, 26, 32;  9, 27, 34;  9, 31, 33;  9, 38, 45;  10, 15, 38;  10, 16, 47;  10, 18, 42;  10, 19, 34;  10, 22, 32;  10, 23, 30;  10, 24, 29;  10, 25, 31;  10, 26, 45;  10, 27, 40;  10, 28, 37;  10, 33, 46;  10, 39, 41;  11, 13, 27;  11, 14, 37;  11, 15, 26;  11, 16, 35;  11, 17, 36;  11, 18, 38;  11, 22, 45;  11, 25, 32;  11, 30, 44;  11, 33, 40;  11, 34, 42;  11, 43, 46;  12, 17, 24;  12, 18, 31;  12, 19, 43;  12, 22, 37;  12, 25, 46;  12, 26, 47;  12, 27, 45;  12, 29, 33;  12, 30, 38;  12, 32, 44;  12, 36, 41;  12, 39, 42;  13, 15, 30;  13, 17, 43;  13, 19, 24;  13, 21, 25;  13, 29, 35;  13, 32, 37;  13, 34, 46;  13, 38, 44;  13, 39, 45;  13, 42, 47;  14, 19, 38;  14, 20, 46;  14, 24, 30;  14, 26, 41;  14, 27, 31;  14, 28, 34;  14, 29, 42;  14, 33, 39;  14, 35, 47;  14, 40, 45;  15, 17, 44;  15, 18, 41;  15, 19, 36;  15, 20, 45;  15, 21, 42;  15, 23, 46;  15, 31, 39;  15, 33, 37;  16, 21, 39;  16, 23, 26;  16, 24, 33;  16, 25, 37;  16, 28, 46;  16, 29, 44;  16, 31, 34;  16, 41, 45;  17, 19, 32;  17, 20, 42;  17, 21, 38;  17, 22, 41;  17, 23, 29;  17, 25, 39;  17, 28, 33;  17, 40, 46;  18, 23, 35;  18, 25, 40;  18, 26, 46;  18, 27, 33;  18, 30, 45;  18, 32, 47;  18, 37, 44;  19, 21, 35;  19, 22, 44;  19, 23, 40;  19, 27, 30;  19, 37, 42;  20, 25, 33;  20, 26, 35;  20, 28, 44;  20, 29, 32;  20, 39, 47;  20, 41, 43;  21, 24, 44;  21, 26, 34;  21, 28, 36;  21, 29, 37;  21, 32, 40;  22, 27, 35;  22, 30, 40;  22, 31, 43;  22, 34, 39;  23, 27, 42;  23, 28, 38;  23, 31, 47;  23, 36, 45;  23, 37, 41;  24, 31, 46;  24, 34, 43;  24, 41, 47;  25, 27, 38;  25, 29, 47;  25, 36, 43;  26, 31, 42;  26, 33, 44;  26, 38, 43;  27, 32, 46;  27, 41, 44;  28, 35, 41;  28, 40, 47;  28, 43, 45;  29, 34, 41;  29, 36, 46;  30, 36, 42;  30, 39, 46;  30, 43, 47;  31, 36, 44;  31, 37, 45;  33, 45, 47;  35, 37, 40;  35, 38, 46;  37, 39, 43}.
The deficiency graph is connected and has girth 8.

{\boldmath $\adfPENT(3,24)$}, $d = 4$:
{1, 2, 50;  0, 11, 43;  0, 3, 4;  2, 13, 45;  0, 5, 18;  0, 6, 16;  0, 7, 10;  0, 8, 40;  0, 9, 49;  0, 13, 37;  0, 14, 30;  0, 15, 38;  0, 17, 23;  0, 19, 24;  0, 21, 46;  0, 22, 31;  0, 25, 41;  0, 26, 34;  0, 27, 45;  0, 29, 33;  0, 35, 39;  1, 3, 15;  1, 6, 39;  1, 9, 38;  1, 18, 31;  1, 19, 27;  1, 22, 34;  1, 23, 47;  1, 46, 51;  2, 19, 30;  2, 22, 47;  2, 23, 39}.
The deficiency graph is connected and has girth 8.

{\boldmath $\adfPENT(3,25)$}, $d = 6$:
{1, 2, 52;  0, 11, 45;  0, 3, 4;  2, 13, 47;  2, 5, 6;  4, 15, 49;  0, 5, 18;  0, 6, 22;  0, 7, 24;  0, 8, 39;  0, 9, 44;  0, 10, 51;  0, 12, 43;  0, 13, 28;  0, 14, 35;  0, 15, 46;  0, 17, 49;  0, 19, 47;  0, 20, 25;  0, 21, 33;  0, 23, 38;  0, 26, 34;  0, 27, 32;  0, 29, 40;  1, 3, 5;  1, 4, 33;  1, 7, 37;  1, 8, 43;  1, 9, 28;  1, 14, 40;  1, 15, 53;  1, 17, 46;  1, 22, 32;  1, 26, 38;  1, 27, 51;  1, 34, 39;  1, 41, 47;  2, 8, 40;  2, 9, 35;  2, 11, 16;  2, 15, 32;  2, 17, 20;  2, 22, 29;  2, 27, 45;  3, 9, 17;  3, 10, 40;  3, 35, 53;  4, 10, 46;  4, 17, 41;  4, 23, 35}.
The deficiency graph is connected and has girth 8.

{\boldmath $\adfPENT(3,27)$}, $d = 2$:
{1, 2, 56;  0, 11, 49;  0, 5, 39;  0, 6, 38;  0, 7, 43;  0, 8, 36;  0, 9, 14;  0, 10, 47;  0, 12, 41;  0, 13, 15;  0, 16, 34;  0, 17, 31;  0, 19, 35;  0, 21, 33;  0, 23, 27;  0, 25, 55;  0, 45, 51;  1, 9, 41}.
The deficiency graph is connected and has girth 8.

{\boldmath $\adfPENT(3,28)$}, $d = 12$:
{1, 2, 58;  0, 11, 51;  0, 3, 4;  2, 13, 53;  2, 5, 6;  4, 15, 55;  4, 7, 8;  6, 17, 57;  6, 9, 10;  8, 19, 59;  8, 11, 12;  1, 10, 21;  0, 5, 22;  0, 6, 32;  0, 7, 40;  0, 8, 18;  0, 9, 31;  0, 10, 55;  0, 12, 54;  0, 13, 49;  0, 14, 52;  0, 15, 39;  0, 16, 33;  0, 17, 47;  0, 19, 21;  0, 20, 57;  0, 23, 27;  0, 24, 53;  0, 25, 43;  0, 26, 35;  0, 28, 37;  0, 30, 46;  0, 34, 41;  0, 38, 45;  0, 44, 50;  1, 3, 9;  1, 4, 47;  1, 5, 8;  1, 6, 28;  1, 7, 18;  1, 13, 44;  1, 14, 20;  1, 15, 38;  1, 16, 40;  1, 17, 43;  1, 22, 27;  1, 23, 45;  1, 26, 34;  1, 29, 54;  1, 30, 35;  1, 31, 53;  1, 33, 57;  1, 39, 55;  1, 42, 56;  1, 46, 59;  2, 7, 19;  2, 14, 34;  2, 15, 28;  2, 16, 23;  2, 17, 31;  2, 18, 41;  2, 20, 38;  2, 21, 35;  2, 27, 57;  2, 29, 42;  2, 30, 43;  2, 32, 54;  2, 33, 45;  2, 46, 52;  2, 47, 55;  2, 51, 59;  3, 5, 34;  3, 6, 30;  3, 7, 28;  3, 8, 20;  3, 10, 46;  3, 15, 29;  3, 18, 55;  3, 21, 42;  3, 22, 44;  3, 31, 45;  3, 32, 56;  3, 35, 40;  3, 41, 57;  3, 47, 54;  4, 9, 46;  4, 10, 19;  4, 16, 53;  4, 17, 32;  4, 18, 33;  4, 20, 45;  4, 22, 56;  4, 23, 54;  4, 29, 35;  4, 30, 57;  4, 34, 44;  5, 7, 32;  5, 9, 41;  5, 10, 56;  5, 17, 59;  5, 23, 46;  5, 44, 57;  6, 18, 58;  6, 23, 34;  6, 31, 47;  7, 10, 43;  7, 11, 46;  7, 20, 35;  7, 33, 44;  8, 35, 47;  9, 11, 35;  9, 22, 34}.
The deficiency graph is connected and has girth 6.

{\boldmath $\adfPENT(3,30)$}, $d = 64$:
{1, 2, 62;  0, 11, 55;  0, 3, 4;  2, 13, 57;  2, 5, 6;  4, 15, 59;  4, 7, 8;  6, 17, 61;  6, 9, 10;  8, 19, 63;  8, 11, 12;  1, 10, 21;  10, 13, 14;  3, 12, 23;  12, 15, 16;  5, 14, 25;  14, 17, 18;  7, 16, 27;  16, 19, 20;  9, 18, 29;  18, 21, 22;  11, 20, 31;  20, 23, 24;  13, 22, 33;  22, 25, 26;  15, 24, 35;  24, 27, 28;  17, 26, 37;  26, 29, 30;  19, 28, 39;  28, 31, 32;  21, 30, 41;  30, 33, 34;  23, 32, 43;  32, 35, 36;  25, 34, 45;  34, 37, 38;  27, 36, 47;  36, 39, 40;  29, 38, 49;  38, 41, 42;  31, 40, 51;  40, 43, 44;  33, 42, 53;  42, 45, 46;  35, 44, 55;  44, 47, 48;  37, 46, 57;  46, 49, 50;  39, 48, 59;  48, 51, 52;  41, 50, 61;  50, 53, 54;  43, 52, 63;  52, 55, 56;  1, 45, 54;  54, 57, 58;  3, 47, 56;  56, 59, 60;  5, 49, 58;  58, 61, 62;  7, 51, 60;  0, 60, 63;  9, 53, 62;  0, 5, 53;  0, 6, 41;  0, 7, 59;  0, 8, 13;  0, 9, 61;  0, 10, 22;  0, 12, 29;  0, 14, 32;  0, 15, 57;  0, 16, 52;  0, 17, 42;  0, 18, 23;  0, 19, 40;  0, 20, 33;  0, 21, 39;  0, 24, 30;  0, 25, 50;  0, 26, 38;  0, 27, 58;  0, 28, 36;  0, 31, 37;  0, 34, 51;  0, 35, 46;  0, 43, 47;  0, 44, 56;  0, 45, 48;  0, 49, 54;  1, 3, 40;  1, 4, 31;  1, 5, 27;  1, 6, 48;  1, 7, 49;  1, 8, 15;  1, 9, 47;  1, 12, 32;  1, 13, 28;  1, 14, 29;  1, 16, 33;  1, 17, 39;  1, 18, 25;  1, 19, 51;  1, 20, 44;  1, 22, 57;  1, 23, 63;  1, 24, 34;  1, 26, 61;  1, 30, 56;  1, 35, 60;  1, 36, 46;  1, 37, 42;  1, 38, 50;  1, 41, 52;  1, 43, 58;  1, 53, 59;  2, 7, 48;  2, 8, 21;  2, 9, 52;  2, 10, 24;  2, 11, 34;  2, 12, 50;  2, 14, 47;  2, 15, 32;  2, 16, 25;  2, 17, 60;  2, 18, 51;  2, 19, 49;  2, 20, 41;  2, 22, 28;  2, 23, 40;  2, 26, 55;  2, 27, 46;  2, 29, 44;  2, 30, 58;  2, 31, 59;  2, 33, 61;  2, 35, 42;  2, 36, 56;  2, 37, 54;  2, 38, 53;  2, 39, 63;  2, 43, 45;  3, 5, 30;  3, 6, 22;  3, 7, 42;  3, 8, 18;  3, 9, 50;  3, 10, 25;  3, 11, 60;  3, 14, 39;  3, 15, 52;  3, 16, 32;  3, 17, 24;  3, 19, 37;  3, 20, 53;  3, 21, 29;  3, 26, 54;  3, 27, 38;  3, 28, 49;  3, 31, 63;  3, 33, 51;  3, 34, 46;  3, 35, 58;  3, 36, 59;  3, 41, 44;  3, 43, 61;  3, 45, 62;  3, 48, 55;  4, 9, 35;  4, 10, 53;  4, 11, 38;  4, 12, 51;  4, 13, 48;  4, 14, 28;  4, 16, 26;  4, 17, 49;  4, 18, 33;  4, 19, 52;  4, 20, 34;  4, 21, 36;  4, 22, 56;  4, 23, 44;  4, 24, 47;  4, 25, 62;  4, 27, 45;  4, 29, 55;  4, 30, 54;  4, 32, 41;  4, 37, 39;  4, 40, 61;  4, 42, 58;  4, 43, 50;  4, 46, 60;  4, 57, 63;  5, 7, 50;  5, 8, 35;  5, 9, 38;  5, 10, 40;  5, 11, 52;  5, 12, 17;  5, 13, 20;  5, 16, 48;  5, 18, 56;  5, 19, 41;  5, 21, 43;  5, 22, 45;  5, 23, 39;  5, 24, 31;  5, 26, 62;  5, 28, 55;  5, 29, 61;  5, 32, 44;  5, 33, 54;  5, 34, 60;  5, 36, 57;  5, 37, 51;  5, 42, 47;  5, 46, 63;  6, 11, 23;  6, 12, 19;  6, 13, 58;  6, 14, 38;  6, 15, 63;  6, 16, 28;  6, 18, 32;  6, 20, 37;  6, 21, 51;  6, 24, 43;  6, 25, 44;  6, 26, 56;  6, 27, 39;  6, 29, 54;  6, 30, 60;  6, 31, 33;  6, 34, 53;  6, 35, 57;  6, 36, 52;  6, 40, 50;  6, 42, 55;  6, 45, 47;  6, 46, 59;  6, 49, 62;  7, 9, 22;  7, 10, 31;  7, 11, 32;  7, 12, 55;  7, 13, 15;  7, 14, 23;  7, 18, 47;  7, 19, 26;  7, 20, 56;  7, 21, 62;  7, 24, 44;  7, 25, 36;  7, 28, 46;  7, 29, 33;  7, 30, 35;  7, 34, 52;  7, 37, 63;  7, 38, 54;  7, 39, 58;  7, 40, 53;  7, 41, 43;  7, 45, 57;  8, 14, 51;  8, 16, 30;  8, 17, 40;  8, 20, 45;  8, 22, 49;  8, 23, 61;  8, 24, 41;  8, 25, 54;  8, 26, 42;  8, 27, 43;  8, 28, 47;  8, 29, 48;  8, 31, 53;  8, 32, 55;  8, 33, 46;  8, 34, 59;  8, 36, 50;  8, 37, 56;  8, 38, 58;  8, 39, 57;  8, 44, 60;  8, 52, 62;  9, 11, 58;  9, 12, 54;  9, 13, 34;  9, 14, 59;  9, 15, 43;  9, 16, 56;  9, 17, 44;  9, 20, 27;  9, 21, 24;  9, 23, 46;  9, 25, 42;  9, 26, 51;  9, 28, 40;  9, 30, 55;  9, 31, 49;  9, 32, 39;  9, 33, 57;  9, 36, 60;  9, 37, 48;  9, 41, 45;  10, 15, 30;  10, 16, 42;  10, 17, 33;  10, 18, 36;  10, 19, 46;  10, 20, 50;  10, 23, 45;  10, 26, 35;  10, 27, 29;  10, 28, 52;  10, 32, 37;  10, 34, 54;  10, 38, 60;  10, 39, 43;  10, 41, 63;  10, 44, 58;  10, 47, 55;  10, 48, 56;  10, 49, 61;  10, 51, 62;  10, 57, 59;  11, 13, 24;  11, 14, 49;  11, 15, 54;  11, 16, 35;  11, 17, 59;  11, 18, 57;  11, 19, 42;  11, 22, 62;  11, 25, 37;  11, 26, 50;  11, 27, 33;  11, 28, 43;  11, 29, 41;  11, 30, 39;  11, 36, 45;  11, 40, 56;  11, 44, 51;  11, 46, 53;  11, 47, 61;  11, 48, 63;  12, 18, 44;  12, 20, 36;  12, 21, 35;  12, 22, 42;  12, 24, 58;  12, 25, 60;  12, 26, 46;  12, 27, 62;  12, 28, 48;  12, 30, 47;  12, 31, 38;  12, 33, 59;  12, 34, 39;  12, 37, 43;  12, 40, 45;  12, 41, 56;  12, 49, 63;  12, 52, 61;  12, 53, 57;  13, 16, 51;  13, 17, 50;  13, 18, 42;  13, 19, 61;  13, 21, 60;  13, 25, 31;  13, 26, 36;  13, 27, 30;  13, 29, 45;  13, 32, 56;  13, 35, 37;  13, 38, 52;  13, 39, 41;  13, 40, 46;  13, 43, 62;  13, 44, 54;  13, 47, 63;  13, 49, 53;  13, 55, 59;  14, 19, 50;  14, 20, 58;  14, 21, 53;  14, 22, 52;  14, 24, 57;  14, 26, 44;  14, 27, 34;  14, 30, 42;  14, 31, 48;  14, 33, 63;  14, 35, 40;  14, 36, 61;  14, 37, 45;  14, 41, 46;  14, 43, 56;  14, 54, 62;  14, 55, 60;  15, 17, 62;  15, 18, 61;  15, 19, 45;  15, 20, 48;  15, 21, 37;  15, 22, 38;  15, 23, 36;  15, 26, 39;  15, 27, 56;  15, 28, 53;  15, 29, 42;  15, 31, 46;  15, 33, 60;  15, 34, 41;  15, 40, 55;  15, 44, 49;  15, 47, 51;  15, 50, 58;  16, 21, 57;  16, 22, 31;  16, 23, 59;  16, 24, 29;  16, 34, 43;  16, 36, 49;  16, 37, 50;  16, 38, 62;  16, 39, 54;  16, 40, 47;  16, 41, 60;  16, 44, 61;  16, 45, 53;  16, 46, 58;  16, 55, 63;  17, 19, 35;  17, 20, 54;  17, 21, 45;  17, 22, 47;  17, 23, 58;  17, 25, 56;  17, 28, 57;  17, 29, 43;  17, 30, 38;  17, 31, 52;  17, 32, 48;  17, 34, 55;  17, 36, 63;  17, 41, 53;  17, 46, 51;  18, 24, 55;  18, 26, 53;  18, 27, 60;  18, 28, 38;  18, 30, 40;  18, 31, 39;  18, 34, 49;  18, 35, 52;  18, 37, 41;  18, 43, 48;  18, 45, 58;  18, 46, 62;  18, 50, 63;  18, 54, 59;  19, 21, 48;  19, 22, 53;  19, 23, 30;  19, 24, 36;  19, 25, 38;  19, 27, 59;  19, 31, 56;  19, 32, 54;  19, 33, 44;  19, 34, 57;  19, 43, 60;  19, 47, 58;  19, 55, 62;  20, 25, 43;  20, 26, 63;  20, 28, 51;  20, 29, 32;  20, 30, 46;  20, 35, 62;  20, 38, 47;  20, 39, 60;  20, 40, 52;  20, 42, 59;  20, 49, 55;  20, 57, 61;  21, 23, 47;  21, 25, 63;  21, 26, 59;  21, 27, 52;  21, 28, 56;  21, 32, 49;  21, 33, 50;  21, 34, 42;  21, 38, 44;  21, 40, 58;  21, 46, 55;  21, 54, 61;  22, 27, 55;  22, 29, 50;  22, 30, 51;  22, 32, 46;  22, 34, 63;  22, 35, 43;  22, 36, 58;  22, 37, 40;  22, 39, 44;  22, 41, 59;  22, 48, 61;  22, 54, 60;  23, 25, 49;  23, 26, 31;  23, 27, 51;  23, 28, 41;  23, 29, 57;  23, 34, 62;  23, 35, 54;  23, 37, 53;  23, 38, 48;  23, 42, 56;  23, 50, 55;  23, 52, 60;  24, 32, 61;  24, 33, 52;  24, 37, 60;  24, 38, 45;  24, 39, 62;  24, 40, 49;  24, 42, 63;  24, 46, 56;  24, 48, 53;  24, 50, 59;  24, 51, 54;  25, 27, 48;  25, 28, 61;  25, 29, 47;  25, 30, 53;  25, 32, 51;  25, 33, 41;  25, 39, 46;  25, 40, 59;  25, 52, 58;  25, 55, 57;  26, 32, 58;  26, 33, 40;  26, 34, 48;  26, 41, 57;  26, 43, 49;  26, 45, 52;  26, 47, 60;  27, 31, 57;  27, 32, 40;  27, 35, 49;  27, 41, 54;  27, 42, 50;  27, 44, 63;  27, 53, 61;  28, 33, 62;  28, 34, 44;  28, 35, 50;  28, 37, 58;  28, 42, 54;  28, 45, 60;  28, 59, 63;  29, 31, 35;  29, 34, 40;  29, 36, 51;  29, 37, 59;  29, 46, 52;  29, 53, 60;  29, 56, 62;  29, 58, 63;  30, 36, 43;  30, 37, 62;  30, 44, 59;  30, 45, 50;  30, 48, 57;  30, 49, 52;  30, 61, 63;  31, 34, 58;  31, 36, 42;  31, 43, 55;  31, 44, 62;  31, 45, 61;  31, 47, 54;  31, 50, 60;  32, 38, 57;  32, 42, 60;  32, 45, 63;  32, 47, 53;  32, 50, 62;  32, 52, 59;  33, 35, 47;  33, 36, 55;  33, 37, 49;  33, 38, 56;  33, 39, 45;  33, 48, 58;  34, 47, 50;  34, 56, 61;  35, 38, 63;  35, 39, 51;  35, 41, 48;  35, 53, 56;  35, 59, 61;  36, 41, 62;  36, 44, 53;  36, 48, 54;  37, 44, 52;  37, 55, 61;  38, 43, 59;  38, 46, 61;  38, 51, 55;  39, 42, 61;  39, 47, 52;  39, 50, 56;  39, 53, 55;  40, 48, 60;  40, 54, 63;  40, 57, 62;  41, 47, 49;  41, 55, 58;  42, 48, 62;  42, 49, 51;  42, 52, 57;  43, 46, 54;  43, 51, 57;  44, 50, 57;  45, 49, 56;  45, 51, 59;  47, 59, 62;  49, 57, 60;  51, 53, 58;  51, 56, 63}.
The deficiency graph is connected and has girth 8.

{\boldmath $\adfPENT(3,31)$}, $d = 6$:
{1, 2, 64;  0, 11, 57;  0, 3, 4;  2, 13, 59;  2, 5, 6;  4, 15, 61;  0, 5, 35;  0, 6, 18;  0, 7, 50;  0, 8, 41;  0, 9, 14;  0, 10, 39;  0, 13, 25;  0, 15, 30;  0, 16, 49;  0, 17, 56;  0, 19, 44;  0, 20, 32;  0, 21, 24;  0, 22, 34;  0, 23, 61;  0, 26, 33;  0, 27, 38;  0, 28, 58;  0, 29, 31;  0, 37, 43;  0, 40, 59;  0, 45, 53;  0, 46, 55;  0, 47, 52;  1, 3, 46;  1, 4, 39;  1, 5, 40;  1, 8, 14;  1, 9, 51;  1, 15, 63;  1, 16, 62;  1, 17, 25;  1, 19, 53;  1, 20, 28;  1, 22, 37;  1, 23, 38;  1, 27, 41;  1, 32, 45;  1, 33, 50;  2, 11, 40;  2, 16, 32;  2, 17, 28;  2, 20, 45;  2, 21, 47;  2, 23, 46;  2, 26, 65;  2, 33, 58;  2, 34, 39;  3, 5, 9;  3, 10, 57;  3, 16, 33;  3, 35, 53;  3, 41, 47;  4, 10, 52;  4, 11, 53;  4, 17, 29}.
The deficiency graph is connected and has girth 8.

\adfAppGap

{\boldmath $\adfPENT(3,47,7)$}, $d = 6$:
{0, 21, 36;  0, 34, 68;  0, 46, 83;  1, 35, 69;  1, 38, 57;  1, 64, 67;  2, 23, 48;  2, 38, 85;  3, 40, 66;  3, 59, 69;  4, 25, 50;  4, 40, 87;  5, 42, 61;  5, 68, 71;  21, 34, 83;  21, 46, 68;  23, 36, 38;  23, 70, 85;  25, 38, 87;  25, 40, 72;  34, 36, 46;  35, 38, 67;  35, 57, 64;  36, 48, 85;  36, 68, 83;  37, 40, 59;  37, 66, 69;  38, 40, 50;  38, 48, 70;  38, 64, 69;  39, 42, 71;  39, 61, 68;  40, 69, 71;  42, 68, 73;  50, 72, 87;  57, 67, 69;  59, 66, 71;  61, 71, 73;  0, 4, 62;  0, 6, 94;  0, 7, 58;  0, 8, 86;  0, 9, 50;  0, 11, 42;  0, 13, 27;  0, 14, 93;  0, 16, 20;  0, 17, 25;  0, 18, 72;  0, 23, 101;  0, 24, 98;  0, 28, 91;  0, 33, 41;  0, 40, 63;  0, 43, 61;  0, 44, 97;  0, 45, 59;  0, 51, 75;  0, 52, 82;  0, 53, 79;  0, 81, 85;  1, 2, 76;  1, 5, 22;  1, 7, 45;  1, 9, 25;  1, 10, 53;  1, 17, 55;  1, 27, 86;  1, 28, 73;  1, 29, 70;  1, 32, 43;  1, 44, 62;  1, 51, 94;  1, 59, 89;  1, 75, 80;  2, 8, 62;  2, 9, 47;  2, 10, 64;  2, 11, 27;  2, 15, 16;  2, 22, 35;  2, 29, 87;  2, 32, 83;  2, 52, 77;  2, 65, 93;  2, 88, 95;  3, 9, 51;  3, 21, 52;  3, 29, 33;  3, 53, 64;  3, 88, 94;  4, 22, 46;  4, 83, 101;  5, 11, 53}.
The deficiency graph is connected and has girth 4.

{\boldmath $\adfPENT(3,51,7)$}, $d = 2$:
{0, 36, 52;  0, 72, 104;  0, 76, 88;  1, 7, 39;  1, 23, 35;  1, 59, 75;  7, 35, 75;  35, 39, 59;  36, 76, 104;  52, 72, 76;  0, 2, 49;  0, 3, 57;  0, 5, 54;  0, 8, 25;  0, 9, 11;  0, 10, 97;  0, 13, 96;  0, 15, 107;  0, 18, 95;  0, 19, 60;  0, 21, 51;  0, 26, 67;  0, 29, 73;  0, 30, 101;  0, 31, 93;  0, 33, 83;  0, 37, 44;  0, 43, 48;  0, 45, 64;  0, 53, 79;  0, 55, 65;  0, 63, 109;  0, 81, 89;  0, 85, 99}.
The deficiency graph is connected and has girth 4.

{\boldmath $\adfPENT(3,53,7)$}, $d = 6$:
{11, 49, 87;  11, 56, 75;  11, 58, 95;  13, 58, 77;  13, 60, 97;  15, 60, 79;  15, 62, 99;  17, 20, 28;  17, 40, 104;  17, 66, 99;  19, 22, 30;  19, 42, 106;  19, 68, 101;  20, 40, 66;  20, 99, 104;  21, 24, 32;  21, 44, 108;  21, 70, 103;  22, 42, 68;  22, 101, 106;  24, 44, 70;  24, 103, 108;  28, 40, 99;  28, 66, 104;  30, 42, 101;  32, 44, 103;  49, 56, 95;  49, 58, 75;  51, 58, 97;  51, 60, 77;  53, 60, 99;  53, 62, 79;  56, 58, 87;  58, 60, 89;  60, 62, 91;  75, 87, 95;  77, 89, 97;  79, 91, 99;  0, 1, 34;  0, 3, 36;  0, 4, 66;  0, 5, 42;  0, 6, 57;  0, 7, 100;  0, 9, 113;  0, 10, 54;  0, 13, 55;  0, 14, 32;  0, 15, 28;  0, 16, 73;  0, 18, 53;  0, 21, 83;  0, 22, 74;  0, 23, 47;  0, 24, 104;  0, 25, 40;  0, 41, 82;  0, 43, 44;  0, 45, 98;  0, 61, 101;  0, 62, 63;  0, 85, 89;  0, 86, 110;  0, 92, 99;  0, 93, 97;  1, 3, 55;  1, 7, 93;  1, 11, 25;  1, 14, 67;  1, 15, 75;  1, 19, 110;  1, 22, 100;  1, 23, 38;  1, 26, 35;  1, 29, 80;  1, 32, 68;  1, 37, 106;  1, 45, 74;  1, 52, 53;  1, 57, 59;  1, 64, 112;  1, 71, 113;  1, 81, 105;  2, 5, 112;  2, 8, 82;  2, 15, 56;  2, 16, 57;  2, 17, 46;  2, 23, 44;  2, 27, 45;  2, 34, 50;  2, 47, 83;  2, 53, 106;  2, 59, 99;  2, 88, 94;  3, 4, 28;  3, 9, 89;  3, 17, 34;  3, 39, 95;  3, 40, 94;  3, 45, 112;  3, 47, 51;  3, 64, 82;  4, 17, 76;  4, 29, 47;  5, 11, 71}.
The deficiency graph is connected and has girth 4.

\adfAppGap

{\boldmath $\adfPENT(3,34,9)$}, $d = 6$:
{0, 1, 2;  0, 15, 71;  0, 16, 70;  0, 17, 69;  1, 15, 70;  1, 16, 69;  1, 17, 71;  2, 15, 69;  2, 16, 71;  2, 17, 70;  3, 4, 5;  3, 6, 74;  3, 7, 73;  3, 8, 72;  4, 6, 73;  4, 7, 72;  4, 8, 74;  5, 6, 72;  5, 7, 74;  5, 8, 73;  0, 18, 59;  0, 19, 48;  0, 20, 52;  0, 21, 44;  0, 22, 42;  0, 23, 51;  0, 24, 57;  0, 25, 46;  0, 26, 61;  0, 27, 56;  0, 28, 47;  0, 29, 43;  0, 31, 63;  0, 32, 50;  0, 34, 62;  0, 35, 53;  0, 37, 64;  0, 38, 65;  0, 39, 55;  0, 40, 45;  1, 19, 43;  1, 20, 49;  1, 21, 53;  1, 23, 59;  1, 26, 52;  1, 27, 58;  1, 29, 45;  1, 32, 51;  1, 34, 64;  1, 35, 56;  1, 38, 62;  1, 39, 57;  1, 40, 46;  1, 41, 47;  2, 22, 39;  2, 23, 58;  2, 27, 44;  2, 32, 65;  2, 33, 53;  2, 40, 47;  2, 41, 45;  2, 46, 64;  3, 9, 47;  3, 10, 45;  3, 11, 40;  3, 22, 51;  4, 35, 65;  4, 40, 77}.
The deficiency graph is connected and has girth 4.

{\boldmath $\adfPENT(3,43,9)$}, $d = 96$:
{0, 1, 2;  0, 9, 13;  0, 10, 12;  0, 11, 14;  0, 15, 88;  0, 16, 87;  0, 17, 89;  0, 84, 94;  0, 85, 93;  0, 86, 95;  1, 9, 12;  1, 10, 14;  1, 11, 13;  1, 15, 87;  1, 16, 89;  1, 17, 88;  1, 84, 93;  1, 85, 95;  1, 86, 94;  2, 9, 14;  2, 10, 13;  2, 11, 12;  2, 15, 89;  2, 16, 88;  2, 17, 87;  2, 84, 95;  2, 85, 94;  2, 86, 93;  3, 4, 5;  3, 6, 91;  3, 7, 90;  3, 8, 92;  3, 12, 28;  3, 13, 27;  3, 14, 29;  3, 75, 85;  3, 76, 84;  3, 77, 86;  4, 6, 90;  4, 7, 92;  4, 8, 91;  4, 12, 27;  4, 13, 29;  4, 14, 28;  4, 75, 84;  4, 76, 86;  4, 77, 85;  5, 6, 92;  5, 7, 91;  5, 8, 90;  5, 12, 29;  5, 13, 28;  5, 14, 27;  5, 75, 86;  5, 76, 85;  5, 77, 84;  6, 7, 8;  6, 15, 19;  6, 16, 18;  6, 17, 20;  6, 21, 94;  6, 22, 93;  6, 23, 95;  7, 15, 18;  7, 16, 20;  7, 17, 19;  7, 21, 93;  7, 22, 95;  7, 23, 94;  8, 15, 20;  8, 16, 19;  8, 17, 18;  8, 21, 95;  8, 22, 94;  8, 23, 93;  9, 10, 11;  9, 18, 34;  9, 19, 33;  9, 20, 35;  9, 81, 91;  9, 82, 90;  9, 83, 92;  10, 18, 33;  10, 19, 35;  10, 20, 34;  10, 81, 90;  10, 82, 92;  10, 83, 91;  11, 18, 35;  11, 19, 34;  11, 20, 33;  11, 81, 92;  11, 82, 91;  11, 83, 90;  12, 13, 14;  12, 21, 25;  12, 22, 24;  12, 23, 26;  13, 21, 24;  13, 22, 26;  13, 23, 25;  14, 21, 26;  14, 22, 25;  14, 23, 24;  15, 16, 17;  15, 24, 40;  15, 25, 39;  15, 26, 41;  16, 24, 39;  16, 25, 41;  16, 26, 40;  17, 24, 41;  17, 25, 40;  17, 26, 39;  18, 19, 20;  18, 27, 31;  18, 28, 30;  18, 29, 32;  19, 27, 30;  19, 28, 32;  19, 29, 31;  20, 27, 32;  20, 28, 31;  20, 29, 30;  21, 22, 23;  21, 30, 46;  21, 31, 45;  21, 32, 47;  22, 30, 45;  22, 31, 47;  22, 32, 46;  23, 30, 47;  23, 31, 46;  23, 32, 45;  24, 25, 26;  24, 33, 37;  24, 34, 36;  24, 35, 38;  25, 33, 36;  25, 34, 38;  25, 35, 37;  26, 33, 38;  26, 34, 37;  26, 35, 36;  27, 28, 29;  27, 36, 52;  27, 37, 51;  27, 38, 53;  28, 36, 51;  28, 37, 53;  28, 38, 52;  29, 36, 53;  29, 37, 52;  29, 38, 51;  30, 31, 32;  30, 39, 43;  30, 40, 42;  30, 41, 44;  31, 39, 42;  31, 40, 44;  31, 41, 43;  32, 39, 44;  32, 40, 43;  32, 41, 42;  33, 34, 35;  33, 42, 58;  33, 43, 57;  33, 44, 59;  34, 42, 57;  34, 43, 59;  34, 44, 58;  35, 42, 59;  35, 43, 58;  35, 44, 57;  36, 37, 38;  36, 45, 49;  36, 46, 48;  36, 47, 50;  37, 45, 48;  37, 46, 50;  37, 47, 49;  38, 45, 50;  38, 46, 49;  38, 47, 48;  39, 40, 41;  39, 48, 64;  39, 49, 63;  39, 50, 65;  40, 48, 63;  40, 49, 65;  40, 50, 64;  41, 48, 65;  41, 49, 64;  41, 50, 63;  42, 43, 44;  42, 51, 55;  42, 52, 54;  42, 53, 56;  43, 51, 54;  43, 52, 56;  43, 53, 55;  44, 51, 56;  44, 52, 55;  44, 53, 54;  45, 46, 47;  45, 54, 70;  45, 55, 69;  45, 56, 71;  46, 54, 69;  46, 55, 71;  46, 56, 70;  47, 54, 71;  47, 55, 70;  47, 56, 69;  48, 49, 50;  48, 57, 61;  48, 58, 60;  48, 59, 62;  49, 57, 60;  49, 58, 62;  49, 59, 61;  50, 57, 62;  50, 58, 61;  50, 59, 60;  51, 52, 53;  51, 60, 76;  51, 61, 75;  51, 62, 77;  52, 60, 75;  52, 61, 77;  52, 62, 76;  53, 60, 77;  53, 61, 76;  53, 62, 75;  54, 55, 56;  54, 63, 67;  54, 64, 66;  54, 65, 68;  55, 63, 66;  55, 64, 68;  55, 65, 67;  56, 63, 68;  56, 64, 67;  56, 65, 66;  57, 58, 59;  57, 66, 82;  57, 67, 81;  57, 68, 83;  58, 66, 81;  58, 67, 83;  58, 68, 82;  59, 66, 83;  59, 67, 82;  59, 68, 81;  60, 61, 62;  60, 69, 73;  60, 70, 72;  60, 71, 74;  61, 69, 72;  61, 70, 74;  61, 71, 73;  62, 69, 74;  62, 70, 73;  62, 71, 72;  63, 64, 65;  63, 72, 88;  63, 73, 87;  63, 74, 89;  64, 72, 87;  64, 73, 89;  64, 74, 88;  65, 72, 89;  65, 73, 88;  65, 74, 87;  66, 67, 68;  66, 75, 79;  66, 76, 78;  66, 77, 80;  67, 75, 78;  67, 76, 80;  67, 77, 79;  68, 75, 80;  68, 76, 79;  68, 77, 78;  69, 70, 71;  69, 78, 94;  69, 79, 93;  69, 80, 95;  70, 78, 93;  70, 79, 95;  70, 80, 94;  71, 78, 95;  71, 79, 94;  71, 80, 93;  72, 73, 74;  72, 81, 85;  72, 82, 84;  72, 83, 86;  73, 81, 84;  73, 82, 86;  73, 83, 85;  74, 81, 86;  74, 82, 85;  74, 83, 84;  75, 76, 77;  78, 79, 80;  78, 87, 91;  78, 88, 90;  78, 89, 92;  79, 87, 90;  79, 88, 92;  79, 89, 91;  80, 87, 92;  80, 88, 91;  80, 89, 90;  81, 82, 83;  84, 85, 86;  87, 88, 89;  90, 91, 92;  93, 94, 95;  0, 18, 42;  0, 19, 60;  0, 20, 57;  0, 21, 63;  0, 22, 81;  0, 23, 72;  0, 24, 74;  0, 25, 49;  0, 26, 53;  0, 27, 59;  0, 28, 66;  0, 29, 78;  0, 30, 75;  0, 31, 71;  0, 32, 65;  0, 33, 52;  0, 34, 61;  0, 35, 73;  0, 36, 56;  0, 37, 69;  0, 38, 76;  0, 39, 68;  0, 40, 82;  0, 41, 45;  0, 43, 62;  0, 44, 70;  0, 46, 64;  0, 47, 67;  0, 48, 80;  0, 50, 79;  0, 51, 58;  0, 54, 83;  0, 55, 77;  1, 18, 56;  1, 19, 62;  1, 20, 45;  1, 21, 49;  1, 22, 75;  1, 23, 82;  1, 24, 78;  1, 25, 60;  1, 26, 47;  1, 27, 54;  1, 28, 67;  1, 29, 65;  1, 30, 57;  1, 31, 72;  1, 32, 66;  1, 33, 71;  1, 34, 64;  1, 35, 48;  1, 36, 81;  1, 37, 83;  1, 38, 58;  1, 39, 70;  1, 40, 76;  1, 41, 77;  1, 42, 61;  1, 43, 74;  1, 44, 69;  1, 46, 51;  1, 50, 68;  1, 52, 59;  1, 53, 80;  1, 55, 73;  1, 63, 79;  2, 18, 43;  2, 19, 67;  2, 20, 64;  2, 21, 60;  2, 22, 54;  2, 23, 57;  2, 24, 53;  2, 25, 42;  2, 26, 72;  2, 27, 47;  2, 28, 82;  2, 29, 56;  2, 30, 52;  2, 31, 81;  2, 32, 50;  2, 33, 68;  2, 34, 74;  2, 35, 55;  2, 36, 66;  2, 37, 73;  2, 38, 61;  2, 39, 59;  2, 40, 80;  2, 41, 46;  2, 44, 79;  2, 45, 83;  2, 48, 77;  2, 49, 76;  2, 51, 70;  2, 58, 65;  2, 62, 78;  2, 63, 69;  2, 71, 75;  3, 9, 68;  3, 10, 61;  3, 11, 80;  3, 18, 66;  3, 19, 78;  3, 20, 67;  3, 21, 65;  3, 22, 48;  3, 23, 64;  3, 24, 47;  3, 25, 63;  3, 26, 44;  3, 30, 54;  3, 31, 83;  3, 32, 79;  3, 33, 72;  3, 34, 56;  3, 35, 82;  3, 36, 60;  3, 37, 70;  3, 38, 69;  3, 39, 58;  3, 40, 55;  3, 41, 59;  3, 42, 74;  3, 43, 93;  3, 45, 53;  3, 46, 62;  3, 49, 71;  3, 50, 94;  3, 51, 81;  3, 52, 95;  3, 57, 73;  4, 9, 40;  4, 10, 42;  4, 11, 60;  4, 18, 58;  4, 19, 73;  4, 20, 93;  4, 21, 50;  4, 22, 70;  4, 23, 63;  4, 24, 44;  4, 25, 56;  4, 26, 68;  4, 30, 74;  4, 31, 59;  4, 32, 69;  4, 33, 79;  4, 34, 54;  4, 35, 52;  4, 36, 65;  4, 37, 81;  4, 38, 57;  4, 39, 78;  4, 41, 71;  4, 43, 64;  4, 45, 72;  4, 46, 66;  4, 47, 53;  4, 48, 82;  4, 49, 67;  4, 51, 95;  4, 55, 94;  4, 61, 83;  4, 62, 80;  5, 9, 51;  5, 10, 43;  5, 11, 71;  5, 18, 50;  5, 19, 46;  5, 20, 81;  5, 21, 39;  5, 22, 53;  5, 23, 74;  5, 24, 61;  5, 25, 73;  5, 26, 66;  5, 30, 79;  5, 31, 65;  5, 32, 63;  5, 33, 80;  5, 34, 49;  5, 35, 72;  5, 36, 55;  5, 37, 60;  5, 38, 70;  5, 40, 68;  5, 41, 54;  5, 42, 67;  5, 44, 82;  5, 45, 64;  5, 47, 83;  5, 48, 69;  5, 52, 57;  5, 56, 78;  5, 58, 94;  5, 59, 93;  5, 62, 95;  6, 24, 51;  6, 25, 89;  6, 26, 63;  6, 27, 71;  6, 28, 68;  6, 29, 58;  6, 30, 84;  6, 31, 77;  6, 32, 49;  6, 33, 74;  6, 34, 70;  6, 35, 66;  6, 36, 61;  6, 37, 56;  6, 38, 80;  6, 39, 54;  6, 40, 73;  6, 41, 67;  6, 42, 78;  6, 43, 79;  6, 44, 65;  6, 45, 75;  6, 46, 81;  6, 47, 85;  6, 48, 83;  6, 50, 72;  6, 52, 69;  6, 53, 86;  6, 55, 76;  6, 57, 64;  6, 59, 88;  6, 60, 82;  6, 62, 87;  7, 24, 62;  7, 25, 52;  7, 26, 51;  7, 27, 64;  7, 28, 55;  7, 29, 81;  7, 30, 83;  7, 31, 86;  7, 32, 73;  7, 33, 48;  7, 34, 71;  7, 35, 60;  7, 36, 82;  7, 37, 75;  7, 38, 56;  7, 39, 85;  7, 40, 67;  7, 41, 68;  7, 42, 70;  7, 43, 61;  7, 44, 80;  7, 45, 65;  7, 46, 88;  7, 47, 87;  7, 49, 79;  7, 50, 77;  7, 53, 78;  7, 54, 84;  7, 57, 63;  7, 58, 72;  7, 59, 74;  7, 66, 89;  7, 69, 76;  8, 24, 68;  8, 25, 71;  8, 26, 61;  8, 27, 70;  8, 28, 57;  8, 29, 77;  8, 30, 58;  8, 31, 62;  8, 32, 51;  8, 33, 87;  8, 34, 80;  8, 35, 69;  8, 36, 88;  8, 37, 59;  8, 38, 78;  8, 39, 66;  8, 40, 72;  8, 41, 55;  8, 42, 63;  8, 43, 83;  8, 44, 89;  8, 45, 67;  8, 46, 60;  8, 47, 76;  8, 48, 73;  8, 49, 86;  8, 50, 82;  8, 52, 84;  8, 53, 74;  8, 54, 85;  8, 56, 75;  8, 64, 81;  8, 65, 79;  9, 15, 61;  9, 16, 52;  9, 17, 55;  9, 24, 76;  9, 25, 44;  9, 26, 57;  9, 27, 63;  9, 28, 65;  9, 29, 59;  9, 30, 88;  9, 31, 54;  9, 32, 62;  9, 36, 67;  9, 37, 72;  9, 38, 85;  9, 39, 69;  9, 41, 80;  9, 42, 71;  9, 43, 75;  9, 45, 74;  9, 46, 86;  9, 47, 60;  9, 48, 84;  9, 49, 78;  9, 50, 66;  9, 53, 87;  9, 56, 73;  9, 58, 77;  9, 64, 79;  9, 70, 89;  10, 15, 45;  10, 16, 48;  10, 17, 60;  10, 24, 69;  10, 25, 70;  10, 26, 64;  10, 27, 89;  10, 28, 72;  10, 29, 79;  10, 30, 49;  10, 31, 53;  10, 32, 71;  10, 36, 75;  10, 37, 66;  10, 38, 55;  10, 39, 86;  10, 40, 62;  10, 41, 78;  10, 44, 63;  10, 46, 65;  10, 47, 84;  10, 50, 67;  10, 51, 73;  10, 52, 88;  10, 54, 77;  10, 56, 76;  10, 57, 74;  10, 58, 80;  10, 59, 85;  10, 68, 87;  11, 15, 54;  11, 16, 55;  11, 17, 66;  11, 24, 58;  11, 25, 68;  11, 26, 67;  11, 27, 87;  11, 28, 84;  11, 29, 49;  11, 30, 59;  11, 31, 78;  11, 32, 72;  11, 36, 62;  11, 37, 79;  11, 38, 88;  11, 39, 46;  11, 40, 74;  11, 41, 85;  11, 42, 64;  11, 43, 73;  11, 44, 61;  11, 45, 63;  11, 47, 51;  11, 48, 76;  11, 50, 70;  11, 52, 86;  11, 53, 89;  11, 56, 77;  11, 57, 75;  11, 65, 69;  12, 30, 69;  12, 31, 67;  12, 32, 68;  12, 33, 70;  12, 34, 78;  12, 35, 85;  12, 36, 72;  12, 37, 82;  12, 38, 60;  12, 39, 61;  12, 40, 54;  12, 41, 92;  12, 42, 66;  12, 43, 87;  12, 44, 81;  12, 45, 62;  12, 46, 93;  12, 47, 74;  12, 48, 95;  12, 49, 88;  12, 50, 86;  12, 51, 90;  12, 52, 89;  12, 53, 83;  12, 55, 75;  12, 56, 79;  12, 57, 91;  12, 58, 76;  12, 59, 73;  12, 63, 94;  12, 64, 80;  12, 65, 84;  12, 71, 77;  13, 30, 61;  13, 31, 66;  13, 32, 59;  13, 33, 86;  13, 34, 69;  13, 35, 71;  13, 36, 84;  13, 37, 76;  13, 38, 73;  13, 39, 74;  13, 40, 94;  13, 41, 95;  13, 42, 88;  13, 43, 63;  13, 44, 87;  13, 45, 77;  13, 46, 68;  13, 47, 64;  13, 48, 93;  13, 49, 75;  13, 50, 92;  13, 51, 91;  13, 52, 72;  13, 53, 67;  13, 54, 80;  13, 55, 79;  13, 56, 83;  13, 57, 78;  13, 58, 90;  13, 60, 89;  13, 62, 81;  13, 65, 82;  13, 70, 85;  14, 30, 77;  14, 31, 75;  14, 32, 60;  14, 33, 53;  14, 34, 67;  14, 35, 91;  14, 36, 79;  14, 37, 89;  14, 38, 68;  14, 39, 73;  14, 40, 69;  14, 41, 83;  14, 42, 92;  14, 43, 82;  14, 44, 74;  14, 45, 86;  14, 46, 85;  14, 47, 62;  14, 48, 66;  14, 49, 72;  14, 50, 88;  14, 51, 78;  14, 52, 58;  14, 54, 94;  14, 55, 93;  14, 56, 84;  14, 57, 80;  14, 59, 87;  14, 61, 90;  14, 63, 81;  14, 64, 95;  14, 65, 70;  14, 71, 76;  15, 21, 75;  15, 22, 73;  15, 23, 66;  15, 30, 67;  15, 31, 48;  15, 32, 80;  15, 33, 65;  15, 34, 55;  15, 35, 64;  15, 36, 83;  15, 37, 63;  15, 38, 74;  15, 42, 69;  15, 43, 71;  15, 44, 85;  15, 46, 79;  15, 47, 86;  15, 49, 82;  15, 50, 84;  15, 51, 93;  15, 52, 78;  15, 53, 72;  15, 56, 81;  15, 57, 90;  15, 58, 95;  15, 59, 77;  15, 60, 94;  15, 62, 92;  15, 68, 91;  15, 70, 76;  16, 21, 74;  16, 22, 62;  16, 23, 79;  16, 30, 95;  16, 31, 70;  16, 32, 92;  16, 33, 83;  16, 34, 77;  16, 35, 56;  16, 36, 68;  16, 37, 65;  16, 38, 66;  16, 42, 72;  16, 43, 85;  16, 44, 94;  16, 45, 76;  16, 46, 78;  16, 47, 91;  16, 49, 90;  16, 50, 69;  16, 51, 67;  16, 53, 93;  16, 54, 73;  16, 57, 84;  16, 58, 75;  16, 59, 86;  16, 60, 81;  16, 61, 80;  16, 63, 71;  16, 64, 82;  17, 21, 69;  17, 22, 50;  17, 23, 77;  17, 30, 81;  17, 31, 57;  17, 32, 83;  17, 33, 62;  17, 34, 76;  17, 35, 90;  17, 36, 85;  17, 37, 64;  17, 38, 71;  17, 42, 80;  17, 43, 68;  17, 44, 84;  17, 45, 51;  17, 46, 61;  17, 47, 78;  17, 48, 79;  17, 49, 74;  17, 52, 82;  17, 53, 94;  17, 54, 72;  17, 56, 93;  17, 58, 63;  17, 59, 75;  17, 65, 91;  17, 67, 92;  17, 70, 86;  17, 73, 95;  18, 36, 86;  18, 37, 84;  18, 38, 83;  18, 39, 87;  18, 40, 71;  18, 41, 74;  18, 44, 67;  18, 45, 93;  18, 46, 73;  18, 47, 82;  18, 48, 68;  18, 49, 77;  18, 51, 80;  18, 52, 92;  18, 53, 81;  18, 54, 75;  18, 55, 78;  18, 57, 76;  18, 59, 65;  18, 60, 95;  18, 61, 79;  18, 62, 85;  18, 63, 91;  18, 64, 69;  18, 70, 88;  18, 72, 90;  18, 89, 94;  19, 36, 64;  19, 37, 86;  19, 38, 79;  19, 39, 89;  19, 40, 77;  19, 41, 70;  19, 42, 94;  19, 43, 90;  19, 44, 72;  19, 45, 80;  19, 47, 66;  19, 48, 75;  19, 49, 95;  19, 50, 85;  19, 51, 69;  19, 52, 71;  19, 53, 68;  19, 54, 91;  19, 55, 81;  19, 56, 87;  19, 57, 88;  19, 58, 74;  19, 59, 92;  19, 61, 84;  19, 63, 82;  19, 65, 93;  19, 76, 83;  20, 36, 77;  20, 37, 91;  20, 38, 75;  20, 39, 80;  20, 40, 88;  20, 41, 94;  20, 42, 79;  20, 43, 70;  20, 44, 76;  20, 46, 82;  20, 47, 63;  20, 48, 89;  20, 49, 69;  20, 50, 90;  20, 51, 68;  20, 52, 74;  20, 53, 71;  20, 54, 78;  20, 55, 87;  20, 56, 85;  20, 58, 73;  20, 59, 72;  20, 60, 86;  20, 61, 95;  20, 62, 84;  20, 65, 83;  20, 66, 92;  21, 27, 43;  21, 28, 73;  21, 29, 67;  21, 36, 71;  21, 37, 58;  21, 38, 77;  21, 40, 90;  21, 41, 81;  21, 42, 89;  21, 44, 78;  21, 48, 70;  21, 51, 88;  21, 52, 83;  21, 53, 59;  21, 54, 87;  21, 55, 80;  21, 56, 72;  21, 57, 79;  21, 61, 92;  21, 62, 82;  21, 64, 85;  21, 66, 84;  21, 68, 86;  21, 76, 91;  22, 27, 77;  22, 28, 90;  22, 29, 57;  22, 36, 80;  22, 37, 74;  22, 38, 92;  22, 39, 56;  22, 40, 66;  22, 41, 89;  22, 42, 76;  22, 43, 72;  22, 44, 60;  22, 49, 85;  22, 51, 87;  22, 52, 67;  22, 55, 91;  22, 58, 79;  22, 59, 78;  22, 61, 82;  22, 63, 83;  22, 64, 84;  22, 65, 71;  22, 68, 88;  22, 69, 86;  23, 27, 67;  23, 28, 69;  23, 29, 50;  23, 36, 73;  23, 37, 80;  23, 38, 84;  23, 39, 75;  23, 40, 78;  23, 41, 61;  23, 42, 65;  23, 43, 81;  23, 44, 62;  23, 48, 85;  23, 49, 92;  23, 51, 83;  23, 52, 68;  23, 53, 70;  23, 54, 76;  23, 55, 86;  23, 56, 91;  23, 58, 87;  23, 59, 89;  23, 60, 90;  23, 71, 88;  24, 42, 84;  24, 43, 86;  24, 45, 85;  24, 46, 63;  24, 48, 71;  24, 49, 73;  24, 50, 91;  24, 52, 80;  24, 54, 81;  24, 55, 89;  24, 56, 92;  24, 57, 77;  24, 59, 79;  24, 60, 93;  24, 64, 70;  24, 65, 95;  24, 66, 90;  24, 67, 87;  24, 72, 94;  24, 75, 83;  24, 82, 88;  25, 43, 78;  25, 45, 82;  25, 46, 75;  25, 47, 80;  25, 48, 86;  25, 50, 76;  25, 51, 79;  25, 53, 88;  25, 54, 74;  25, 55, 85;  25, 57, 65;  25, 58, 84;  25, 59, 90;  25, 61, 94;  25, 62, 83;  25, 64, 93;  25, 66, 95;  25, 67, 91;  25, 69, 87;  25, 72, 92;  25, 77, 81;  26, 42, 86;  26, 43, 95;  26, 45, 73;  26, 46, 80;  26, 48, 78;  26, 49, 70;  26, 50, 81;  26, 52, 87;  26, 54, 79;  26, 55, 90;  26, 56, 82;  26, 58, 88;  26, 59, 76;  26, 60, 91;  26, 62, 94;  26, 65, 85;  26, 69, 89;  26, 71, 84;  26, 74, 93;  26, 75, 92;  26, 77, 83;  27, 33, 84;  27, 34, 95;  27, 35, 68;  27, 42, 90;  27, 44, 73;  27, 45, 79;  27, 46, 76;  27, 48, 72;  27, 49, 81;  27, 50, 83;  27, 55, 82;  27, 56, 80;  27, 57, 85;  27, 58, 92;  27, 60, 78;  27, 61, 93;  27, 62, 88;  27, 65, 86;  27, 66, 94;  27, 69, 75;  27, 74, 91;  28, 33, 81;  28, 34, 79;  28, 35, 87;  28, 42, 95;  28, 43, 76;  28, 44, 83;  28, 45, 78;  28, 46, 77;  28, 47, 61;  28, 48, 92;  28, 49, 94;  28, 50, 71;  28, 54, 88;  28, 56, 74;  28, 58, 89;  28, 59, 80;  28, 60, 85;  28, 62, 91;  28, 63, 93;  28, 64, 86;  28, 70, 75;  29, 33, 82;  29, 34, 92;  29, 35, 76;  29, 42, 62;  29, 43, 91;  29, 44, 68;  29, 45, 66;  29, 46, 84;  29, 47, 72;  29, 48, 74;  29, 54, 86;  29, 55, 95;  29, 60, 88;  29, 61, 85;  29, 63, 80;  29, 64, 83;  29, 69, 90;  29, 70, 87;  29, 71, 89;  29, 73, 94;  29, 75, 93;  30, 48, 87;  30, 50, 73;  30, 51, 82;  30, 53, 66;  30, 55, 72;  30, 56, 86;  30, 60, 80;  30, 62, 89;  30, 63, 90;  30, 64, 92;  30, 65, 78;  30, 68, 94;  30, 70, 91;  30, 71, 85;  30, 76, 93;  31, 49, 89;  31, 50, 80;  31, 51, 85;  31, 52, 79;  31, 55, 84;  31, 56, 94;  31, 58, 64;  31, 60, 87;  31, 61, 91;  31, 63, 92;  31, 68, 93;  31, 69, 88;  31, 73, 90;  31, 74, 95;  31, 76, 82;  32, 48, 91;  32, 52, 81;  32, 53, 84;  32, 54, 82;  32, 55, 74;  32, 56, 89;  32, 57, 93;  32, 58, 85;  32, 61, 87;  32, 64, 78;  32, 67, 86;  32, 70, 77;  32, 75, 90;  32, 76, 94;  32, 88, 95;  33, 39, 60;  33, 40, 75;  33, 41, 73;  33, 49, 91;  33, 50, 78;  33, 51, 66;  33, 54, 89;  33, 55, 92;  33, 56, 90;  33, 61, 88;  33, 63, 85;  33, 64, 94;  33, 67, 93;  33, 69, 77;  33, 76, 95;  34, 39, 94;  34, 40, 85;  34, 41, 60;  34, 48, 81;  34, 50, 75;  34, 51, 72;  34, 52, 93;  34, 53, 82;  34, 62, 86;  34, 63, 84;  34, 65, 90;  34, 66, 87;  34, 68, 89;  34, 73, 91;  34, 83, 88;  35, 39, 95;  35, 40, 83;  35, 41, 62;  35, 49, 80;  35, 50, 74;  35, 51, 89;  35, 53, 79;  35, 54, 92;  35, 61, 86;  35, 63, 78;  35, 65, 94;  35, 67, 88;  35, 70, 84;  35, 75, 81;  35, 77, 93;  36, 54, 95;  36, 57, 87;  36, 58, 78;  36, 59, 94;  36, 63, 70;  36, 69, 91;  36, 74, 90;  36, 76, 92;  36, 89, 93;  37, 54, 90;  37, 55, 88;  37, 57, 94;  37, 61, 78;  37, 62, 93;  37, 67, 85;  37, 68, 92;  37, 71, 87;  37, 77, 95;  38, 54, 93;  38, 59, 64;  38, 62, 90;  38, 63, 86;  38, 65, 81;  38, 67, 95;  38, 72, 91;  38, 82, 89;  38, 87, 94;  39, 45, 91;  39, 47, 88;  39, 55, 83;  39, 57, 92;  39, 62, 79;  39, 67, 84;  39, 71, 90;  39, 72, 93;  39, 76, 81;  39, 77, 82;  40, 45, 84;  40, 46, 87;  40, 47, 93;  40, 56, 95;  40, 57, 89;  40, 58, 86;  40, 59, 91;  40, 60, 79;  40, 61, 81;  40, 70, 92;  41, 47, 79;  41, 56, 88;  41, 57, 72;  41, 58, 93;  41, 66, 86;  41, 69, 84;  41, 75, 91;  41, 76, 90;  41, 82, 87;  42, 60, 83;  42, 68, 85;  42, 73, 93;  42, 75, 82;  42, 77, 91;  42, 81, 87;  43, 60, 84;  43, 65, 80;  43, 66, 88;  43, 67, 89;  43, 69, 92;  43, 77, 94;  44, 64, 90;  44, 66, 91;  44, 71, 86;  44, 75, 95;  44, 77, 92;  44, 88, 93;  45, 52, 94;  45, 60, 92;  45, 61, 89;  45, 68, 90;  45, 81, 88;  45, 87, 95;  46, 52, 91;  46, 53, 90;  46, 67, 94;  46, 72, 95;  46, 74, 92;  46, 83, 89;  47, 52, 73;  47, 65, 92;  47, 68, 95;  47, 75, 94;  47, 77, 90;  47, 81, 89;  48, 67, 90;  48, 88, 94;  49, 66, 93;  49, 68, 84;  49, 83, 87;  50, 87, 93;  50, 89, 95;  51, 57, 86;  51, 59, 84;  51, 71, 92;  51, 74, 94;  52, 66, 85;  52, 70, 90;  53, 57, 95;  53, 58, 91;  53, 69, 85;  53, 73, 92;  59, 63, 95;  64, 71, 91}.
The deficiency graph is connected and has girth 4.

{\boldmath $\adfPENT(3,52,9)$}, $d = 6$:
{0, 1, 2;  0, 15, 107;  0, 16, 106;  0, 17, 105;  1, 15, 106;  1, 16, 105;  1, 17, 107;  2, 15, 105;  2, 16, 107;  2, 17, 106;  3, 4, 5;  3, 6, 110;  3, 7, 109;  3, 8, 108;  4, 6, 109;  4, 7, 108;  4, 8, 110;  5, 6, 108;  5, 7, 110;  5, 8, 109;  0, 18, 94;  0, 19, 69;  0, 20, 95;  0, 21, 60;  0, 22, 71;  0, 23, 67;  0, 24, 58;  0, 25, 46;  0, 26, 65;  0, 27, 99;  0, 28, 81;  0, 29, 84;  0, 31, 62;  0, 32, 87;  0, 33, 61;  0, 35, 91;  0, 36, 73;  0, 38, 66;  0, 39, 70;  0, 40, 47;  0, 41, 101;  0, 42, 93;  0, 43, 68;  0, 44, 97;  0, 45, 63;  0, 49, 88;  0, 50, 74;  0, 52, 79;  0, 53, 82;  0, 55, 92;  0, 56, 100;  0, 57, 80;  0, 64, 85;  0, 77, 98;  0, 83, 89;  1, 19, 57;  1, 20, 85;  1, 21, 91;  1, 23, 63;  1, 27, 77;  1, 28, 67;  1, 29, 81;  1, 33, 55;  1, 34, 69;  1, 35, 83;  1, 37, 101;  1, 41, 82;  1, 43, 89;  1, 44, 100;  1, 46, 80;  1, 52, 58;  1, 53, 95;  1, 56, 99;  1, 64, 86;  1, 68, 98;  1, 70, 75;  1, 74, 92;  2, 21, 75;  2, 22, 65;  2, 23, 59;  2, 27, 74;  2, 28, 56;  2, 29, 76;  2, 33, 40;  2, 34, 70;  2, 35, 50;  2, 38, 99;  2, 39, 83;  2, 47, 52;  2, 51, 100;  2, 53, 87;  2, 64, 81;  2, 71, 89;  3, 9, 76;  3, 11, 41;  3, 22, 40;  3, 23, 69;  3, 33, 89;  3, 35, 39;  3, 46, 101;  3, 58, 88;  4, 23, 76;  4, 35, 70;  4, 41, 58}.
The deficiency graph is connected and has girth 4.

{\boldmath $\adfPENT(3,61,9)$}, $d = 12$:
{0, 1, 2;  0, 9, 13;  0, 10, 12;  0, 11, 14;  0, 15, 124;  0, 16, 123;  0, 17, 125;  1, 9, 12;  1, 10, 14;  1, 11, 13;  1, 15, 123;  1, 16, 125;  1, 17, 124;  2, 9, 14;  2, 10, 13;  2, 11, 12;  2, 15, 125;  2, 16, 124;  2, 17, 123;  3, 4, 5;  3, 6, 127;  3, 7, 126;  3, 8, 128;  4, 6, 126;  4, 7, 128;  4, 8, 127;  5, 6, 128;  5, 7, 127;  5, 8, 126;  6, 7, 8;  6, 21, 130;  6, 22, 129;  6, 23, 131;  7, 21, 129;  7, 22, 131;  7, 23, 130;  8, 21, 131;  8, 22, 130;  8, 23, 129;  9, 10, 11;  0, 18, 86;  0, 19, 90;  0, 20, 93;  0, 21, 91;  0, 22, 83;  0, 23, 117;  0, 24, 73;  0, 25, 80;  0, 26, 60;  0, 27, 111;  0, 28, 56;  0, 29, 96;  0, 30, 110;  0, 31, 87;  0, 32, 50;  0, 33, 99;  0, 34, 61;  0, 35, 78;  0, 37, 81;  0, 38, 114;  0, 39, 75;  0, 40, 66;  0, 41, 67;  0, 42, 79;  0, 43, 106;  0, 44, 94;  0, 45, 85;  0, 46, 89;  0, 47, 107;  0, 48, 105;  0, 51, 101;  0, 52, 115;  0, 53, 74;  0, 54, 109;  0, 55, 88;  0, 58, 97;  0, 59, 76;  0, 62, 119;  0, 63, 92;  0, 64, 116;  0, 68, 95;  0, 69, 104;  0, 70, 100;  0, 71, 103;  0, 77, 112;  0, 82, 102;  0, 113, 118;  1, 18, 64;  1, 19, 103;  1, 20, 39;  1, 21, 73;  1, 22, 75;  1, 23, 101;  1, 25, 68;  1, 26, 77;  1, 27, 59;  1, 28, 33;  1, 29, 88;  1, 30, 52;  1, 31, 65;  1, 32, 107;  1, 34, 115;  1, 35, 90;  1, 37, 76;  1, 38, 55;  1, 41, 58;  1, 42, 86;  1, 43, 69;  1, 46, 105;  1, 47, 111;  1, 49, 114;  1, 50, 80;  1, 51, 89;  1, 53, 102;  1, 54, 117;  1, 57, 63;  1, 62, 98;  1, 67, 92;  1, 70, 91;  1, 71, 87;  1, 74, 119;  1, 79, 99;  1, 82, 110;  1, 83, 104;  1, 95, 113;  1, 100, 116;  1, 112, 118;  2, 18, 57;  2, 20, 105;  2, 21, 52;  2, 22, 79;  2, 23, 81;  2, 26, 92;  2, 27, 33;  2, 28, 55;  2, 29, 103;  2, 30, 88;  2, 31, 50;  2, 34, 99;  2, 35, 80;  2, 39, 76;  2, 40, 100;  2, 41, 74;  2, 42, 65;  2, 43, 95;  2, 44, 67;  2, 45, 112;  2, 46, 93;  2, 51, 111;  2, 56, 102;  2, 58, 94;  2, 63, 91;  2, 64, 104;  2, 69, 87;  2, 70, 77;  2, 71, 107;  2, 75, 118;  2, 82, 119;  2, 83, 114;  2, 89, 117;  3, 10, 68;  3, 11, 102;  3, 18, 83;  3, 19, 131;  3, 20, 89;  3, 21, 107;  3, 22, 66;  3, 23, 104;  3, 30, 103;  3, 33, 105;  3, 34, 106;  3, 42, 92;  3, 43, 67;  3, 44, 101;  3, 45, 91;  3, 47, 81;  3, 52, 95;  3, 54, 118;  3, 55, 100;  3, 56, 88;  3, 57, 65;  3, 58, 93;  3, 59, 64;  3, 76, 94;  3, 77, 90;  3, 78, 130;  3, 80, 114;  4, 11, 93;  4, 18, 58;  4, 19, 46;  4, 21, 41;  4, 23, 107;  4, 33, 70;  4, 34, 88;  4, 35, 65;  4, 40, 95;  4, 42, 80;  4, 43, 66;  4, 45, 81;  4, 53, 130;  4, 54, 116;  4, 55, 94;  4, 57, 89;  4, 68, 117;  4, 71, 115;  4, 79, 129;  4, 83, 102;  4, 92, 118;  5, 9, 89;  5, 11, 67;  5, 19, 92;  5, 20, 45;  5, 21, 70;  5, 23, 102;  5, 30, 56;  5, 32, 93;  5, 34, 66;  5, 35, 41;  5, 42, 78;  5, 43, 90;  5, 44, 83;  5, 46, 95;  5, 47, 116;  5, 55, 115;  5, 58, 77;  5, 69, 91;  5, 71, 106;  5, 81, 104;  6, 30, 117;  6, 31, 66;  6, 33, 90;  6, 34, 80;  6, 35, 82;  6, 55, 116;  6, 57, 105;  7, 35, 105;  7, 43, 81;  7, 44, 92;  7, 47, 106;  7, 56, 94;  7, 71, 104;  8, 32, 104;  8, 45, 106;  8, 70, 118}.
The deficiency graph is connected and has girth 4.

{\boldmath $\adfPENT(3,70,9)$}, $d = 6$:
{0, 1, 2;  0, 15, 143;  0, 16, 142;  0, 17, 141;  1, 15, 142;  1, 16, 141;  1, 17, 143;  2, 15, 141;  2, 16, 143;  2, 17, 142;  3, 4, 5;  3, 6, 146;  3, 7, 145;  3, 8, 144;  4, 6, 145;  4, 7, 144;  4, 8, 146;  5, 6, 144;  5, 7, 146;  5, 8, 145;  0, 18, 67;  0, 19, 115;  0, 20, 116;  0, 21, 101;  0, 22, 52;  0, 23, 69;  0, 24, 98;  0, 25, 89;  0, 26, 96;  0, 27, 46;  0, 28, 71;  0, 29, 35;  0, 30, 83;  0, 31, 84;  0, 32, 136;  0, 33, 111;  0, 34, 91;  0, 36, 109;  0, 37, 77;  0, 38, 59;  0, 39, 108;  0, 40, 125;  0, 41, 56;  0, 43, 113;  0, 44, 102;  0, 45, 121;  0, 47, 99;  0, 50, 103;  0, 51, 72;  0, 55, 75;  0, 57, 64;  0, 58, 117;  0, 60, 130;  0, 61, 112;  0, 62, 124;  0, 63, 134;  0, 65, 110;  0, 68, 131;  0, 76, 127;  0, 79, 106;  0, 82, 104;  0, 85, 119;  0, 86, 135;  0, 87, 105;  0, 88, 123;  0, 93, 137;  0, 94, 122;  0, 95, 128;  0, 100, 107;  0, 118, 133;  1, 19, 38;  1, 22, 115;  1, 23, 95;  1, 25, 110;  1, 26, 56;  1, 27, 135;  1, 29, 76;  1, 31, 64;  1, 32, 116;  1, 33, 93;  1, 39, 91;  1, 40, 85;  1, 43, 105;  1, 44, 103;  1, 45, 122;  1, 46, 112;  1, 47, 101;  1, 50, 118;  1, 51, 83;  1, 53, 87;  1, 57, 125;  1, 59, 77;  1, 62, 79;  1, 68, 111;  1, 69, 123;  1, 70, 107;  1, 74, 124;  1, 80, 128;  1, 81, 130;  1, 82, 131;  1, 88, 119;  1, 89, 117;  1, 104, 137;  1, 113, 129;  2, 20, 92;  2, 21, 105;  2, 22, 89;  2, 26, 125;  2, 27, 63;  2, 28, 76;  2, 29, 33;  2, 34, 94;  2, 38, 118;  2, 39, 87;  2, 40, 113;  2, 41, 110;  2, 46, 93;  2, 47, 77;  2, 53, 112;  2, 57, 95;  2, 58, 135;  2, 59, 99;  2, 69, 136;  2, 71, 88;  2, 100, 129;  2, 111, 131;  2, 117, 123;  3, 11, 40;  3, 33, 89;  3, 34, 88;  3, 46, 65;  3, 53, 148;  3, 58, 64;  3, 77, 112;  3, 82, 100;  3, 95, 136;  4, 40, 112;  4, 65, 101;  4, 83, 149;  5, 47, 95}.
The deficiency graph is connected and has girth 4.

{\boldmath $\adfPENT(3,79,9)$}, $d = 24$:
{0, 1, 2;  0, 9, 14;  0, 10, 13;  0, 11, 12;  0, 15, 161;  0, 16, 160;  0, 17, 159;  1, 9, 13;  1, 10, 12;  1, 11, 14;  1, 15, 160;  1, 16, 159;  1, 17, 161;  2, 9, 12;  2, 10, 14;  2, 11, 13;  2, 15, 159;  2, 16, 161;  2, 17, 160;  3, 4, 5;  3, 6, 164;  3, 7, 163;  3, 8, 162;  3, 12, 29;  3, 13, 28;  3, 14, 27;  4, 6, 163;  4, 7, 162;  4, 8, 164;  4, 12, 28;  4, 13, 27;  4, 14, 29;  5, 6, 162;  5, 7, 164;  5, 8, 163;  5, 12, 27;  5, 13, 29;  5, 14, 28;  6, 7, 8;  6, 15, 20;  6, 16, 19;  6, 17, 18;  6, 21, 167;  6, 22, 166;  6, 23, 165;  7, 15, 19;  7, 16, 18;  7, 17, 20;  7, 21, 166;  7, 22, 165;  7, 23, 167;  8, 15, 18;  8, 16, 20;  8, 17, 19;  8, 21, 165;  8, 22, 167;  8, 23, 166;  9, 10, 11;  9, 18, 35;  9, 19, 34;  9, 20, 33;  10, 18, 34;  10, 19, 33;  10, 20, 35;  11, 18, 33;  11, 19, 35;  11, 20, 34;  12, 13, 14;  12, 21, 26;  12, 22, 25;  12, 23, 24;  13, 21, 25;  13, 22, 24;  13, 23, 26;  14, 21, 24;  14, 22, 26;  14, 23, 25;  15, 16, 17;  18, 19, 20;  21, 22, 23;  0, 18, 95;  0, 19, 103;  0, 20, 67;  0, 21, 48;  0, 22, 51;  0, 23, 56;  0, 24, 112;  0, 25, 61;  0, 26, 121;  0, 27, 85;  0, 28, 149;  0, 29, 72;  0, 30, 57;  0, 31, 94;  0, 32, 81;  0, 33, 79;  0, 34, 84;  0, 35, 104;  0, 36, 97;  0, 37, 105;  0, 38, 107;  0, 39, 68;  0, 40, 134;  0, 41, 130;  0, 42, 145;  0, 43, 123;  0, 44, 73;  0, 45, 124;  0, 46, 91;  0, 47, 53;  0, 49, 154;  0, 50, 143;  0, 52, 93;  0, 54, 100;  0, 55, 137;  0, 58, 78;  0, 59, 135;  0, 60, 82;  0, 62, 98;  0, 63, 139;  0, 64, 69;  0, 65, 129;  0, 66, 128;  0, 70, 155;  0, 71, 151;  0, 74, 122;  0, 75, 106;  0, 76, 150;  0, 77, 114;  0, 80, 102;  0, 83, 146;  0, 86, 140;  0, 87, 115;  0, 89, 148;  0, 90, 110;  0, 92, 108;  0, 99, 126;  0, 101, 142;  0, 109, 127;  0, 111, 131;  0, 113, 152;  0, 116, 133;  0, 117, 138;  0, 118, 132;  0, 119, 147;  0, 136, 153;  1, 18, 73;  1, 19, 62;  1, 20, 76;  1, 21, 104;  1, 22, 52;  1, 23, 149;  1, 25, 69;  1, 26, 125;  1, 27, 151;  1, 28, 119;  1, 29, 99;  1, 30, 83;  1, 31, 70;  1, 32, 110;  1, 33, 153;  1, 34, 112;  1, 35, 131;  1, 36, 155;  1, 38, 116;  1, 39, 95;  1, 40, 141;  1, 41, 138;  1, 42, 117;  1, 43, 91;  1, 44, 124;  1, 46, 94;  1, 47, 115;  1, 49, 128;  1, 50, 103;  1, 51, 101;  1, 53, 84;  1, 54, 135;  1, 55, 89;  1, 56, 130;  1, 57, 136;  1, 58, 142;  1, 59, 147;  1, 60, 118;  1, 61, 90;  1, 63, 113;  1, 64, 71;  1, 65, 86;  1, 67, 85;  1, 68, 150;  1, 75, 134;  1, 77, 105;  1, 78, 100;  1, 79, 139;  1, 81, 102;  1, 82, 122;  1, 87, 127;  1, 88, 132;  1, 92, 126;  1, 93, 111;  1, 98, 129;  1, 107, 133;  1, 109, 146;  1, 123, 143;  1, 137, 152;  1, 148, 154;  2, 18, 68;  2, 19, 135;  2, 20, 125;  2, 21, 99;  2, 22, 102;  2, 23, 84;  2, 26, 103;  2, 27, 57;  2, 28, 62;  2, 29, 47;  2, 30, 127;  2, 31, 152;  2, 32, 86;  2, 34, 66;  2, 35, 39;  2, 36, 129;  2, 37, 90;  2, 38, 150;  2, 40, 151;  2, 41, 46;  2, 42, 81;  2, 43, 111;  2, 44, 106;  2, 45, 149;  2, 51, 113;  2, 52, 88;  2, 53, 124;  2, 54, 105;  2, 56, 131;  2, 58, 128;  2, 59, 154;  2, 60, 132;  2, 61, 140;  2, 63, 82;  2, 64, 98;  2, 65, 109;  2, 67, 155;  2, 69, 112;  2, 70, 110;  2, 71, 89;  2, 75, 153;  2, 76, 138;  2, 77, 116;  2, 78, 117;  2, 80, 142;  2, 83, 104;  2, 85, 147;  2, 87, 119;  2, 91, 126;  2, 92, 141;  2, 93, 137;  2, 94, 114;  2, 100, 118;  2, 108, 148;  2, 115, 143;  2, 123, 139;  3, 9, 85;  3, 10, 64;  3, 11, 92;  3, 18, 137;  3, 20, 119;  3, 21, 106;  3, 22, 108;  3, 31, 117;  3, 32, 71;  3, 35, 103;  3, 36, 55;  3, 37, 95;  3, 38, 123;  3, 39, 151;  3, 40, 115;  3, 41, 113;  3, 42, 100;  3, 43, 77;  3, 44, 80;  3, 45, 129;  3, 46, 125;  3, 47, 82;  3, 52, 152;  3, 54, 138;  3, 56, 88;  3, 57, 94;  3, 58, 99;  3, 59, 165;  3, 60, 112;  3, 63, 166;  3, 66, 133;  3, 67, 140;  3, 68, 141;  3, 69, 139;  3, 70, 136;  3, 76, 132;  3, 78, 167;  3, 79, 155;  3, 84, 116;  3, 87, 105;  3, 89, 114;  3, 90, 135;  3, 102, 131;  3, 104, 124;  3, 107, 126;  3, 110, 150;  3, 111, 153;  3, 118, 134;  3, 128, 154;  4, 9, 136;  4, 11, 111;  4, 18, 67;  4, 19, 106;  4, 20, 124;  4, 21, 43;  4, 23, 55;  4, 30, 167;  4, 31, 62;  4, 32, 141;  4, 33, 165;  4, 34, 87;  4, 35, 94;  4, 36, 79;  4, 37, 108;  4, 39, 109;  4, 41, 153;  4, 42, 84;  4, 44, 76;  4, 46, 107;  4, 47, 129;  4, 53, 58;  4, 54, 137;  4, 56, 112;  4, 57, 88;  4, 59, 102;  4, 61, 139;  4, 63, 154;  4, 64, 133;  4, 65, 117;  4, 69, 135;  4, 70, 119;  4, 71, 143;  4, 77, 103;  4, 80, 151;  4, 81, 130;  4, 82, 128;  4, 83, 138;  4, 85, 118;  4, 86, 134;  4, 89, 115;  4, 90, 127;  4, 91, 131;  4, 105, 166;  4, 110, 155;  5, 9, 56;  5, 11, 128;  5, 18, 116;  5, 19, 38;  5, 20, 89;  5, 21, 118;  5, 22, 41;  5, 30, 85;  5, 32, 140;  5, 34, 101;  5, 35, 109;  5, 37, 130;  5, 39, 54;  5, 40, 155;  5, 43, 126;  5, 45, 95;  5, 53, 141;  5, 55, 87;  5, 57, 150;  5, 58, 127;  5, 59, 111;  5, 60, 119;  5, 61, 82;  5, 62, 88;  5, 63, 132;  5, 64, 86;  5, 65, 112;  5, 66, 142;  5, 67, 108;  5, 70, 153;  5, 71, 114;  5, 78, 143;  5, 79, 138;  5, 80, 152;  5, 81, 154;  5, 83, 115;  5, 84, 107;  5, 90, 151;  5, 91, 166;  5, 92, 135;  5, 102, 134;  5, 103, 136;  5, 104, 131;  5, 105, 137;  5, 110, 133;  5, 113, 129;  5, 117, 165;  6, 30, 93;  6, 31, 161;  6, 32, 84;  6, 34, 130;  6, 36, 138;  6, 37, 85;  6, 39, 128;  6, 40, 110;  6, 41, 68;  6, 42, 86;  6, 43, 108;  6, 44, 151;  6, 46, 105;  6, 47, 107;  6, 54, 160;  6, 55, 106;  6, 56, 109;  6, 58, 156;  6, 60, 102;  6, 63, 129;  6, 64, 136;  6, 65, 111;  6, 66, 119;  6, 67, 88;  6, 70, 137;  6, 79, 127;  6, 80, 104;  6, 82, 113;  6, 83, 117;  6, 114, 133;  6, 115, 141;  6, 116, 142;  6, 118, 157;  6, 135, 158;  7, 31, 154;  7, 32, 139;  7, 33, 88;  7, 34, 155;  7, 35, 85;  7, 36, 131;  7, 37, 118;  7, 42, 105;  7, 43, 115;  7, 44, 136;  7, 45, 113;  7, 47, 94;  7, 56, 142;  7, 57, 111;  7, 59, 110;  7, 60, 86;  7, 61, 81;  7, 62, 103;  7, 65, 158;  7, 69, 116;  7, 71, 92;  7, 80, 153;  7, 82, 140;  7, 84, 159;  7, 87, 133;  7, 90, 117;  7, 108, 141;  7, 109, 161;  7, 112, 152;  7, 119, 160;  8, 33, 90;  8, 36, 141;  8, 37, 113;  8, 38, 159;  8, 39, 133;  8, 41, 131;  8, 42, 132;  8, 43, 154;  8, 44, 112;  8, 45, 65;  8, 46, 108;  8, 56, 157;  8, 58, 91;  8, 63, 142;  8, 66, 153;  8, 67, 138;  8, 68, 88;  8, 69, 110;  8, 71, 160;  8, 84, 111;  8, 85, 135;  8, 89, 119;  8, 92, 158;  8, 95, 134;  8, 105, 139;  8, 118, 156;  9, 15, 143;  9, 16, 69;  9, 17, 139;  9, 36, 91;  9, 37, 110;  9, 38, 71;  9, 39, 119;  9, 42, 134;  9, 44, 62;  9, 45, 106;  9, 47, 83;  9, 59, 115;  9, 60, 131;  9, 61, 133;  9, 67, 116;  9, 68, 156;  9, 81, 158;  9, 87, 117;  9, 89, 130;  9, 92, 132;  9, 107, 137;  9, 108, 142;  9, 112, 140;  10, 15, 66;  10, 16, 71;  10, 17, 65;  10, 36, 110;  10, 37, 141;  10, 38, 161;  10, 39, 70;  10, 40, 113;  10, 44, 140;  10, 45, 133;  10, 46, 61;  10, 47, 139;  10, 58, 132;  10, 59, 107;  10, 62, 119;  10, 69, 111;  10, 86, 115;  10, 90, 142;  10, 92, 134;  10, 95, 158;  10, 109, 136;  10, 112, 118;  10, 114, 135;  10, 138, 160;  11, 16, 36;  11, 17, 136;  11, 38, 69;  11, 39, 137;  11, 40, 71;  11, 42, 158;  11, 44, 109;  11, 46, 141;  11, 47, 114;  11, 65, 93;  11, 68, 112;  11, 86, 113;  11, 88, 133;  11, 90, 160;  11, 95, 140;  11, 115, 138;  11, 116, 157;  11, 132, 161;  11, 135, 142;  12, 36, 167;  12, 37, 165;  12, 40, 158;  12, 42, 133;  12, 43, 137;  12, 47, 132;  12, 61, 85;  12, 62, 113;  12, 63, 135;  12, 65, 134;  12, 66, 112;  12, 68, 89;  12, 69, 163;  12, 88, 159;  12, 91, 161;  12, 93, 110;  12, 116, 162;  12, 141, 157;  13, 38, 119;  13, 39, 88;  13, 41, 47;  13, 43, 67;  13, 44, 115;  13, 45, 68;  13, 62, 159;  13, 64, 118;  13, 69, 134;  13, 70, 113;  13, 95, 110;  13, 138, 167;  13, 139, 161;  13, 143, 160;  14, 38, 87;  14, 39, 93;  14, 42, 70;  14, 44, 94;  14, 46, 110;  14, 67, 114;  14, 69, 115;  14, 91, 136;  15, 21, 93;  15, 23, 139;  15, 42, 116;  15, 63, 136;  15, 68, 90;  15, 70, 114;  15, 89, 166;  15, 92, 167;  15, 119, 163;  15, 142, 164;  16, 42, 165;  16, 43, 93;  16, 45, 119;  16, 46, 136;  16, 66, 113;  16, 67, 166;  16, 68, 142;  16, 90, 115;  17, 21, 119;  17, 46, 67;  17, 68, 92;  17, 70, 140;  17, 71, 90;  17, 93, 162;  17, 95, 143;  18, 42, 138;  18, 44, 69;  18, 91, 118;  19, 44, 164;  19, 70, 167;  20, 47, 142;  20, 71, 117;  21, 70, 142}.
The deficiency graph is connected and has girth 4.

{\boldmath $\adfPENT(3,88,9)$}, $d = 6$:
{0, 1, 2;  0, 15, 179;  0, 16, 178;  0, 17, 177;  1, 15, 178;  1, 16, 177;  1, 17, 179;  2, 15, 177;  2, 16, 179;  2, 17, 178;  3, 4, 5;  3, 6, 182;  3, 7, 181;  3, 8, 180;  4, 6, 181;  4, 7, 180;  4, 8, 182;  5, 6, 180;  5, 7, 182;  5, 8, 181;  0, 18, 98;  0, 19, 55;  0, 20, 166;  0, 21, 94;  0, 22, 141;  0, 23, 172;  0, 24, 156;  0, 25, 129;  0, 26, 76;  0, 27, 43;  0, 28, 58;  0, 29, 127;  0, 31, 56;  0, 32, 112;  0, 33, 131;  0, 34, 148;  0, 35, 149;  0, 36, 128;  0, 37, 120;  0, 38, 84;  0, 39, 153;  0, 40, 123;  0, 41, 86;  0, 42, 160;  0, 44, 62;  0, 45, 167;  0, 46, 133;  0, 47, 89;  0, 48, 130;  0, 49, 159;  0, 50, 110;  0, 51, 115;  0, 52, 163;  0, 53, 143;  0, 57, 65;  0, 59, 164;  0, 60, 137;  0, 61, 109;  0, 63, 154;  0, 64, 107;  0, 67, 145;  0, 68, 111;  0, 69, 124;  0, 70, 87;  0, 71, 158;  0, 72, 146;  0, 73, 113;  0, 75, 96;  0, 78, 169;  0, 79, 170;  0, 81, 122;  0, 83, 139;  0, 85, 105;  0, 88, 135;  0, 93, 173;  0, 95, 121;  0, 97, 125;  0, 99, 119;  0, 100, 171;  0, 101, 151;  0, 104, 142;  0, 106, 147;  0, 116, 152;  0, 117, 134;  0, 136, 157;  0, 155, 161;  1, 19, 65;  1, 20, 163;  1, 22, 152;  1, 23, 125;  1, 27, 95;  1, 28, 61;  1, 31, 140;  1, 32, 121;  1, 33, 82;  1, 34, 116;  1, 35, 55;  1, 38, 59;  1, 39, 146;  1, 40, 143;  1, 43, 136;  1, 45, 155;  1, 46, 112;  1, 50, 104;  1, 51, 73;  1, 52, 170;  1, 53, 134;  1, 56, 153;  1, 57, 128;  1, 58, 97;  1, 62, 149;  1, 63, 103;  1, 64, 80;  1, 68, 101;  1, 69, 119;  1, 70, 160;  1, 71, 87;  1, 74, 113;  1, 75, 141;  1, 77, 172;  1, 81, 158;  1, 83, 135;  1, 86, 159;  1, 88, 130;  1, 93, 99;  1, 106, 142;  1, 107, 122;  1, 117, 173;  1, 118, 124;  1, 129, 164;  2, 21, 116;  2, 22, 149;  2, 26, 155;  2, 27, 59;  2, 28, 166;  2, 29, 50;  2, 32, 130;  2, 33, 118;  2, 34, 63;  2, 39, 86;  2, 44, 119;  2, 46, 80;  2, 47, 113;  2, 51, 160;  2, 53, 94;  2, 57, 88;  2, 64, 123;  2, 65, 136;  2, 68, 161;  2, 69, 92;  2, 71, 124;  2, 76, 137;  2, 87, 135;  2, 105, 112;  2, 125, 142;  2, 129, 159;  3, 21, 129;  3, 22, 154;  3, 39, 77;  3, 40, 149;  3, 45, 185;  3, 46, 99;  3, 47, 105;  3, 57, 161;  3, 64, 124;  3, 65, 95;  3, 82, 89;  3, 100, 155;  3, 119, 184;  4, 22, 101;  4, 23, 161;  4, 35, 71;  4, 53, 88;  4, 77, 82;  4, 89, 143;  5, 23, 83}.
The deficiency graph is connected and has girth 4.

{\boldmath $\adfPENT(3,97,9)$}, $d = 12$:
{0, 1, 2;  0, 9, 13;  0, 10, 12;  0, 11, 14;  0, 15, 196;  0, 16, 195;  0, 17, 197;  1, 9, 12;  1, 10, 14;  1, 11, 13;  1, 15, 195;  1, 16, 197;  1, 17, 196;  2, 9, 14;  2, 10, 13;  2, 11, 12;  2, 15, 197;  2, 16, 196;  2, 17, 195;  3, 4, 5;  3, 6, 199;  3, 7, 198;  3, 8, 200;  4, 6, 198;  4, 7, 200;  4, 8, 199;  5, 6, 200;  5, 7, 199;  5, 8, 198;  6, 7, 8;  6, 21, 202;  6, 22, 201;  6, 23, 203;  7, 21, 201;  7, 22, 203;  7, 23, 202;  8, 21, 203;  8, 22, 202;  8, 23, 201;  9, 10, 11;  0, 18, 77;  0, 19, 53;  0, 20, 178;  0, 21, 115;  0, 22, 80;  0, 23, 173;  0, 24, 186;  0, 25, 54;  0, 26, 100;  0, 27, 58;  0, 28, 166;  0, 29, 65;  0, 30, 124;  0, 31, 140;  0, 32, 72;  0, 33, 110;  0, 34, 116;  0, 35, 89;  0, 36, 154;  0, 37, 82;  0, 38, 169;  0, 39, 83;  0, 40, 84;  0, 41, 172;  0, 42, 176;  0, 43, 91;  0, 44, 98;  0, 45, 191;  0, 46, 97;  0, 47, 128;  0, 48, 161;  0, 49, 95;  0, 50, 73;  0, 51, 183;  0, 52, 148;  0, 55, 145;  0, 56, 101;  0, 57, 141;  0, 59, 105;  0, 60, 138;  0, 61, 131;  0, 62, 86;  0, 63, 153;  0, 64, 165;  0, 66, 119;  0, 67, 126;  0, 68, 109;  0, 69, 152;  0, 70, 151;  0, 71, 146;  0, 74, 158;  0, 75, 107;  0, 76, 187;  0, 79, 129;  0, 81, 177;  0, 85, 189;  0, 87, 134;  0, 88, 133;  0, 90, 123;  0, 92, 108;  0, 93, 139;  0, 94, 163;  0, 99, 175;  0, 102, 170;  0, 103, 122;  0, 104, 135;  0, 106, 136;  0, 111, 150;  0, 112, 167;  0, 114, 179;  0, 117, 157;  0, 121, 159;  0, 125, 171;  0, 127, 185;  0, 130, 174;  0, 137, 155;  0, 142, 190;  0, 143, 181;  0, 147, 182;  0, 149, 184;  1, 18, 186;  1, 19, 161;  1, 20, 190;  1, 21, 177;  1, 22, 116;  1, 23, 107;  1, 25, 100;  1, 26, 65;  1, 27, 175;  1, 28, 170;  1, 29, 121;  1, 31, 59;  1, 32, 176;  1, 33, 128;  1, 34, 119;  1, 35, 157;  1, 37, 122;  1, 38, 91;  1, 40, 110;  1, 41, 185;  1, 42, 114;  1, 43, 125;  1, 44, 126;  1, 45, 146;  1, 50, 136;  1, 51, 191;  1, 52, 140;  1, 53, 178;  1, 54, 131;  1, 55, 99;  1, 56, 183;  1, 57, 189;  1, 58, 134;  1, 61, 149;  1, 62, 139;  1, 63, 109;  1, 64, 112;  1, 66, 90;  1, 67, 162;  1, 68, 187;  1, 69, 101;  1, 70, 163;  1, 73, 166;  1, 75, 102;  1, 77, 111;  1, 78, 173;  1, 79, 158;  1, 80, 123;  1, 81, 174;  1, 82, 143;  1, 87, 148;  1, 88, 127;  1, 92, 153;  1, 93, 130;  1, 95, 147;  1, 98, 150;  1, 103, 129;  1, 104, 152;  1, 106, 138;  1, 117, 172;  1, 118, 155;  1, 124, 142;  1, 135, 151;  1, 137, 179;  1, 141, 184;  1, 171, 188;  2, 18, 147;  2, 19, 58;  2, 20, 191;  2, 21, 110;  2, 22, 162;  2, 23, 93;  2, 27, 92;  2, 28, 50;  2, 29, 59;  2, 30, 134;  2, 31, 146;  2, 32, 82;  2, 33, 186;  2, 34, 149;  2, 35, 139;  2, 38, 190;  2, 39, 95;  2, 40, 161;  2, 42, 179;  2, 43, 106;  2, 44, 68;  2, 45, 116;  2, 46, 167;  2, 47, 160;  2, 51, 155;  2, 52, 67;  2, 53, 165;  2, 56, 113;  2, 57, 183;  2, 63, 164;  2, 65, 94;  2, 66, 174;  2, 69, 111;  2, 70, 112;  2, 71, 173;  2, 75, 83;  2, 77, 115;  2, 78, 176;  2, 80, 153;  2, 81, 126;  2, 87, 135;  2, 89, 175;  2, 90, 172;  2, 99, 137;  2, 100, 141;  2, 101, 185;  2, 103, 125;  2, 104, 140;  2, 107, 189;  2, 114, 166;  2, 118, 151;  2, 119, 178;  2, 123, 143;  2, 124, 150;  2, 128, 177;  2, 142, 163;  2, 148, 188;  3, 9, 45;  3, 10, 140;  3, 18, 130;  3, 21, 136;  3, 22, 76;  3, 31, 77;  3, 32, 126;  3, 33, 99;  3, 39, 203;  3, 40, 177;  3, 43, 70;  3, 44, 102;  3, 46, 166;  3, 52, 83;  3, 53, 125;  3, 54, 92;  3, 55, 138;  3, 56, 188;  3, 57, 63;  3, 58, 118;  3, 65, 202;  3, 66, 94;  3, 67, 154;  3, 69, 123;  3, 71, 160;  3, 82, 117;  3, 88, 186;  3, 89, 103;  3, 90, 190;  3, 91, 116;  3, 95, 172;  3, 100, 187;  3, 101, 114;  3, 105, 127;  3, 106, 137;  3, 112, 148;  3, 113, 128;  3, 115, 150;  3, 119, 149;  3, 124, 175;  3, 129, 152;  3, 131, 178;  3, 139, 179;  3, 142, 189;  3, 162, 191;  4, 9, 139;  4, 10, 46;  4, 11, 53;  4, 18, 188;  4, 20, 179;  4, 21, 68;  4, 23, 161;  4, 31, 103;  4, 32, 94;  4, 33, 71;  4, 34, 130;  4, 41, 128;  4, 42, 162;  4, 47, 64;  4, 54, 129;  4, 56, 140;  4, 57, 113;  4, 58, 149;  4, 65, 176;  4, 66, 152;  4, 67, 201;  4, 69, 150;  4, 70, 76;  4, 78, 101;  4, 79, 143;  4, 80, 107;  4, 81, 163;  4, 82, 155;  4, 83, 127;  4, 88, 190;  4, 89, 164;  4, 90, 203;  4, 104, 189;  4, 116, 151;  4, 117, 177;  4, 118, 174;  4, 137, 186;  4, 138, 175;  4, 167, 187;  5, 9, 66;  5, 10, 126;  5, 11, 82;  5, 21, 138;  5, 22, 161;  5, 30, 93;  5, 31, 78;  5, 32, 128;  5, 33, 119;  5, 42, 90;  5, 44, 174;  5, 45, 166;  5, 46, 165;  5, 55, 191;  5, 56, 103;  5, 57, 115;  5, 58, 201;  5, 68, 154;  5, 69, 102;  5, 79, 139;  5, 81, 101;  5, 83, 176;  5, 95, 131;  5, 104, 143;  5, 105, 140;  5, 106, 203;  5, 129, 188;  5, 141, 190;  5, 162, 202;  6, 31, 155;  6, 32, 107;  6, 33, 142;  6, 45, 95;  6, 47, 66;  6, 55, 167;  6, 56, 190;  6, 67, 105;  6, 68, 139;  6, 79, 141;  6, 82, 189;  6, 91, 188;  6, 103, 164;  6, 130, 179;  6, 131, 187;  7, 31, 106;  7, 43, 95;  7, 44, 103;  7, 56, 154;  7, 58, 80;  7, 83, 188;  7, 91, 177;  7, 105, 179;  7, 128, 166;  8, 33, 95;  8, 34, 143;  8, 45, 155;  8, 59, 191;  8, 71, 105;  9, 82, 154;  9, 107, 142;  11, 59, 119}.
The deficiency graph is connected and has girth 4.

\adfAppGap

{\boldmath $\adfPENT(3,246,13)$}, $d = 2$:
{12, 114, 402;  12, 122, 356;  12, 124, 482;  12, 146, 270;  12, 158, 182;  12, 198, 294;  25, 105, 309;  25, 151, 349;  25, 213, 393;  25, 237, 325;  25, 361, 385;  25, 383, 495;  105, 151, 325;  105, 213, 349;  105, 237, 495;  105, 361, 383;  105, 385, 393;  114, 122, 294;  114, 124, 158;  114, 146, 356;  114, 182, 482;  114, 198, 270;  122, 124, 402;  122, 182, 198;  124, 146, 198;  124, 182, 356;  146, 158, 402;  151, 213, 309;  151, 237, 393;  151, 361, 495;  151, 383, 385;  158, 198, 356;  158, 294, 482;  182, 270, 402;  198, 402, 482;  213, 385, 495;  237, 309, 361;  270, 356, 482;  294, 356, 402;  309, 325, 385;  309, 349, 383;  309, 393, 495;  325, 383, 393;  349, 361, 393;  0, 1, 7;  0, 3, 208;  0, 4, 141;  0, 5, 364;  0, 6, 139;  0, 9, 373;  0, 11, 317;  0, 13, 51;  0, 14, 359;  0, 15, 118;  0, 17, 459;  0, 18, 503;  0, 19, 266;  0, 20, 339;  0, 21, 175;  0, 23, 268;  0, 26, 465;  0, 27, 369;  0, 28, 271;  0, 29, 104;  0, 30, 457;  0, 31, 340;  0, 33, 66;  0, 35, 347;  0, 37, 354;  0, 38, 221;  0, 39, 452;  0, 41, 436;  0, 42, 335;  0, 43, 304;  0, 45, 270;  0, 47, 140;  0, 48, 135;  0, 49, 56;  0, 50, 121;  0, 53, 381;  0, 55, 59;  0, 57, 400;  0, 61, 424;  0, 63, 83;  0, 64, 421;  0, 65, 169;  0, 67, 128;  0, 69, 312;  0, 73, 269;  0, 75, 241;  0, 77, 453;  0, 78, 415;  0, 79, 315;  0, 81, 356;  0, 85, 386;  0, 89, 107;  0, 90, 479;  0, 91, 192;  0, 92, 463;  0, 94, 207;  0, 95, 483;  0, 97, 196;  0, 98, 333;  0, 99, 437;  0, 100, 227;  0, 101, 417;  0, 103, 384;  0, 109, 201;  0, 114, 267;  0, 115, 275;  0, 117, 376;  0, 119, 330;  0, 123, 154;  0, 125, 365;  0, 129, 185;  0, 131, 246;  0, 144, 397;  0, 145, 353;  0, 149, 395;  0, 155, 338;  0, 157, 239;  0, 159, 216;  0, 160, 351;  0, 161, 190;  0, 164, 455;  0, 165, 379;  0, 167, 489;  0, 171, 299;  0, 173, 273;  0, 177, 219;  0, 178, 461;  0, 179, 447;  0, 181, 321;  0, 184, 435;  0, 187, 217;  0, 193, 423;  0, 195, 230;  0, 199, 254;  0, 200, 497;  0, 203, 419;  0, 209, 331;  0, 211, 387;  0, 214, 501;  0, 215, 285;  0, 223, 425;  0, 229, 327;  0, 233, 487;  0, 249, 441;  0, 255, 399;  0, 257, 311;  0, 265, 343;  0, 277, 367;  0, 279, 307;  0, 289, 409;  0, 303, 329;  0, 305, 411;  0, 313, 363;  0, 341, 493;  0, 355, 505;  0, 375, 469;  0, 377, 491;  0, 401, 467;  0, 429, 443;  0, 433, 481}.
The deficiency graph is connected and has girth 4.

{\boldmath $\adfPENT(3,248,13)$}, $d = 6$:
{7, 13, 281;  7, 69, 183;  7, 133, 419;  7, 139, 181;  7, 205, 439;  7, 407, 425;  9, 15, 283;  9, 71, 185;  9, 135, 421;  9, 141, 183;  9, 207, 441;  9, 409, 427;  11, 17, 285;  11, 73, 187;  11, 137, 423;  11, 143, 185;  11, 209, 443;  11, 411, 429;  13, 69, 439;  13, 133, 183;  13, 181, 407;  13, 205, 419;  15, 71, 441;  15, 135, 185;  15, 183, 409;  15, 207, 421;  17, 73, 443;  17, 137, 187;  17, 185, 411;  17, 209, 423;  69, 133, 181;  69, 139, 419;  69, 205, 407;  69, 281, 425;  71, 135, 183;  71, 141, 421;  71, 207, 409;  71, 283, 427;  72, 86, 378;  72, 92, 230;  72, 104, 498;  72, 306, 328;  72, 330, 504;  72, 372, 442;  73, 137, 185;  73, 143, 423;  73, 209, 411;  73, 285, 429;  74, 88, 380;  74, 94, 232;  74, 106, 500;  74, 308, 330;  74, 332, 506;  74, 374, 444;  76, 90, 382;  76, 96, 234;  76, 108, 502;  76, 310, 332;  76, 334, 508;  76, 376, 446;  86, 92, 328;  86, 104, 330;  86, 230, 442;  86, 306, 504;  86, 372, 498;  88, 94, 330;  88, 106, 332;  88, 232, 444;  88, 308, 506;  88, 374, 500;  90, 96, 332;  90, 108, 334;  90, 234, 446;  90, 310, 508;  90, 376, 502;  92, 104, 306;  92, 330, 372;  92, 442, 498;  94, 106, 308;  94, 332, 374;  94, 444, 500;  96, 108, 310;  96, 334, 376;  96, 446, 502;  104, 442, 504;  106, 444, 506;  108, 446, 508;  133, 205, 281;  133, 425, 439;  135, 207, 283;  135, 427, 441;  137, 209, 285;  137, 429, 443;  139, 183, 205;  139, 281, 439;  141, 185, 207;  141, 283, 441;  143, 187, 209;  143, 285, 443;  181, 183, 439;  181, 205, 425;  181, 281, 419;  183, 185, 441;  183, 207, 427;  183, 283, 421;  185, 187, 443;  185, 209, 429;  185, 285, 423;  230, 306, 372;  230, 330, 378;  232, 308, 374;  232, 332, 380;  234, 310, 376;  234, 334, 382;  306, 330, 498;  306, 378, 442;  308, 332, 500;  308, 380, 444;  310, 334, 502;  310, 382, 446;  328, 330, 442;  328, 372, 504;  328, 378, 498;  330, 332, 444;  330, 374, 506;  330, 380, 500;  332, 334, 446;  332, 376, 508;  332, 382, 502;  407, 419, 439;  409, 421, 441;  411, 423, 443;  0, 1, 200;  0, 3, 109;  0, 4, 157;  0, 5, 36;  0, 8, 95;  0, 9, 248;  0, 10, 351;  0, 11, 491;  0, 15, 494;  0, 16, 191;  0, 17, 305;  0, 19, 165;  0, 21, 408;  0, 23, 141;  0, 25, 253;  0, 26, 241;  0, 27, 216;  0, 28, 319;  0, 29, 215;  0, 30, 495;  0, 31, 452;  0, 33, 328;  0, 34, 119;  0, 35, 134;  0, 37, 501;  0, 38, 193;  0, 39, 456;  0, 40, 201;  0, 41, 127;  0, 43, 58;  0, 45, 330;  0, 46, 75;  0, 47, 471;  0, 49, 327;  0, 51, 222;  0, 52, 507;  0, 53, 68;  0, 55, 232;  0, 57, 263;  0, 59, 358;  0, 60, 403;  0, 61, 128;  0, 63, 190;  0, 65, 493;  0, 67, 382;  0, 71, 277;  0, 73, 89;  0, 74, 239;  0, 77, 267;  0, 79, 124;  0, 80, 441;  0, 81, 177;  0, 82, 385;  0, 83, 326;  0, 85, 428;  0, 86, 197;  0, 87, 103;  0, 88, 179;  0, 90, 469;  0, 91, 271;  0, 94, 483;  0, 96, 449;  0, 97, 458;  0, 99, 122;  0, 101, 265;  0, 105, 283;  0, 106, 129;  0, 107, 499;  0, 108, 331;  0, 111, 373;  0, 113, 171;  0, 115, 152;  0, 117, 324;  0, 118, 167;  0, 121, 381;  0, 125, 213;  0, 130, 485;  0, 131, 199;  0, 135, 237;  0, 137, 287;  0, 143, 178;  0, 145, 307;  0, 146, 415;  0, 147, 484;  0, 149, 422;  0, 150, 459;  0, 151, 467;  0, 153, 481;  0, 155, 279;  0, 156, 377;  0, 159, 282;  0, 161, 470;  0, 162, 497;  0, 163, 315;  0, 164, 349;  0, 166, 405;  0, 169, 246;  0, 173, 475;  0, 175, 184;  0, 182, 317;  0, 185, 208;  0, 187, 240;  0, 188, 249;  0, 189, 196;  0, 194, 247;  0, 195, 313;  0, 203, 443;  0, 206, 357;  0, 207, 482;  0, 209, 299;  0, 211, 257;  0, 217, 464;  0, 219, 316;  0, 227, 393;  0, 229, 392;  0, 231, 472;  0, 233, 329;  0, 235, 273;  0, 243, 489;  0, 245, 404;  0, 250, 429;  0, 251, 431;  0, 255, 259;  0, 260, 359;  0, 261, 389;  0, 262, 341;  0, 269, 502;  0, 275, 463;  0, 278, 477;  0, 285, 435;  0, 289, 314;  0, 291, 371;  0, 293, 386;  0, 295, 423;  0, 297, 506;  0, 301, 451;  0, 302, 505;  0, 304, 337;  0, 310, 411;  0, 311, 445;  0, 320, 391;  0, 322, 509;  0, 323, 332;  0, 325, 413;  0, 333, 367;  0, 344, 347;  0, 345, 427;  0, 346, 363;  0, 355, 365;  0, 361, 364;  0, 369, 437;  0, 375, 383;  0, 376, 417;  0, 380, 473;  0, 388, 442;  0, 395, 447;  0, 397, 503;  0, 399, 424;  0, 401, 461;  0, 409, 436;  0, 416, 453;  0, 421, 476;  0, 430, 487;  0, 455, 500;  1, 2, 165;  1, 5, 286;  1, 8, 321;  1, 9, 376;  1, 14, 437;  1, 20, 41;  1, 22, 272;  1, 27, 388;  1, 29, 214;  1, 31, 130;  1, 32, 58;  1, 34, 251;  1, 35, 274;  1, 37, 502;  1, 40, 386;  1, 44, 187;  1, 50, 327;  1, 52, 95;  1, 53, 244;  1, 55, 472;  1, 59, 106;  1, 61, 182;  1, 62, 394;  1, 64, 223;  1, 69, 392;  1, 70, 473;  1, 74, 417;  1, 75, 340;  1, 76, 298;  1, 80, 271;  1, 81, 302;  1, 82, 268;  1, 87, 350;  1, 88, 459;  1, 91, 284;  1, 94, 278;  1, 97, 428;  1, 98, 195;  1, 103, 482;  1, 109, 232;  1, 110, 292;  1, 112, 217;  1, 116, 464;  1, 118, 500;  1, 123, 206;  1, 125, 382;  1, 128, 136;  1, 131, 280;  1, 134, 488;  1, 135, 508;  1, 140, 207;  1, 142, 485;  1, 146, 490;  1, 148, 304;  1, 152, 185;  1, 154, 424;  1, 157, 470;  1, 158, 365;  1, 160, 176;  1, 166, 430;  1, 167, 452;  1, 170, 359;  1, 172, 247;  1, 179, 310;  1, 184, 446;  1, 188, 226;  1, 189, 254;  1, 190, 218;  1, 191, 250;  1, 196, 453;  1, 197, 476;  1, 201, 460;  1, 202, 496;  1, 212, 387;  1, 224, 404;  1, 233, 322;  1, 236, 501;  1, 238, 400;  1, 256, 345;  1, 260, 471;  1, 262, 431;  1, 263, 434;  1, 266, 423;  1, 290, 370;  1, 303, 338;  1, 311, 374;  1, 315, 364;  1, 329, 398;  1, 334, 381;  1, 346, 503;  1, 383, 506;  1, 389, 494;  1, 410, 483;  1, 412, 416;  1, 448, 484;  2, 3, 243;  2, 11, 362;  2, 17, 195;  2, 21, 351;  2, 27, 56;  2, 29, 416;  2, 32, 111;  2, 38, 321;  2, 41, 202;  2, 45, 230;  2, 47, 405;  2, 51, 224;  2, 53, 266;  2, 57, 179;  2, 59, 124;  2, 62, 231;  2, 65, 460;  2, 70, 333;  2, 76, 317;  2, 77, 261;  2, 83, 88;  2, 87, 275;  2, 92, 221;  2, 93, 293;  2, 104, 357;  2, 105, 502;  2, 107, 418;  2, 110, 395;  2, 117, 322;  2, 119, 136;  2, 123, 326;  2, 129, 316;  2, 143, 305;  2, 147, 435;  2, 148, 461;  2, 149, 251;  2, 154, 471;  2, 155, 501;  2, 173, 394;  2, 190, 465;  2, 196, 491;  2, 197, 393;  2, 208, 219;  2, 215, 323;  2, 218, 473;  2, 225, 371;  2, 242, 485;  2, 263, 466;  2, 280, 335;  2, 281, 454;  2, 297, 478;  2, 304, 455;  2, 311, 430;  2, 329, 365;  2, 347, 401;  2, 359, 369;  2, 375, 509;  2, 377, 459;  2, 381, 411;  2, 382, 479;  2, 383, 387;  2, 388, 423;  2, 399, 453;  2, 406, 437;  2, 417, 507;  2, 424, 495;  2, 431, 472;  3, 4, 453;  3, 16, 263;  3, 22, 231;  3, 29, 418;  3, 34, 448;  3, 39, 202;  3, 40, 77;  3, 41, 118;  3, 46, 136;  3, 70, 189;  3, 76, 473;  3, 82, 305;  3, 111, 286;  3, 112, 214;  3, 148, 351;  3, 154, 159;  3, 160, 219;  3, 185, 328;  3, 197, 280;  3, 220, 479;  3, 226, 406;  3, 251, 346;  3, 281, 460;  3, 292, 485;  3, 316, 383;  3, 322, 382;  3, 358, 467;  3, 388, 407;  3, 419, 430;  3, 436, 497;  4, 5, 232;  4, 29, 154;  4, 34, 107;  4, 112, 377;  4, 131, 413;  4, 149, 305;  4, 167, 461;  4, 239, 485}.
The deficiency graph is connected and has girth 4.

{\boldmath $\adfPENT(3,254,13)$}, $d = 6$:
{42, 90, 158;  42, 136, 368;  42, 140, 454;  42, 156, 278;  42, 202, 242;  42, 288, 302;  44, 92, 160;  44, 138, 370;  44, 142, 456;  44, 158, 280;  44, 204, 244;  44, 290, 304;  46, 94, 162;  46, 140, 372;  46, 144, 458;  46, 160, 282;  46, 206, 246;  46, 292, 306;  69, 155, 221;  69, 235, 281;  69, 245, 321;  69, 365, 383;  69, 367, 387;  69, 433, 481;  71, 157, 223;  71, 237, 283;  71, 247, 323;  71, 367, 385;  71, 369, 389;  71, 435, 483;  73, 159, 225;  73, 239, 285;  73, 249, 325;  73, 369, 387;  73, 371, 391;  73, 437, 485;  90, 136, 302;  90, 140, 202;  90, 156, 242;  90, 278, 454;  90, 288, 368;  92, 138, 304;  92, 142, 204;  92, 158, 244;  92, 280, 456;  92, 290, 370;  94, 140, 306;  94, 144, 206;  94, 160, 246;  94, 282, 458;  94, 292, 372;  136, 140, 158;  136, 156, 454;  136, 242, 278;  138, 142, 160;  138, 158, 456;  138, 244, 280;  140, 144, 162;  140, 156, 288;  140, 160, 458;  140, 242, 302;  140, 246, 282;  140, 278, 368;  142, 158, 290;  142, 244, 304;  142, 280, 370;  144, 160, 292;  144, 246, 306;  144, 282, 372;  155, 235, 433;  155, 245, 387;  155, 281, 365;  155, 383, 481;  156, 158, 302;  157, 237, 435;  157, 247, 389;  157, 283, 367;  157, 385, 483;  158, 160, 304;  158, 202, 454;  158, 242, 368;  158, 278, 288;  159, 239, 437;  159, 249, 391;  159, 285, 369;  159, 387, 485;  160, 162, 306;  160, 204, 456;  160, 244, 370;  160, 280, 290;  162, 206, 458;  162, 246, 372;  162, 282, 292;  202, 278, 302;  204, 280, 304;  206, 282, 306;  221, 235, 383;  221, 245, 281;  221, 321, 365;  221, 367, 481;  223, 237, 385;  223, 247, 283;  223, 323, 367;  223, 369, 483;  225, 239, 387;  225, 249, 285;  225, 325, 369;  225, 371, 485;  235, 245, 481;  235, 365, 367;  237, 247, 483;  237, 367, 369;  239, 249, 485;  239, 369, 371;  245, 365, 433;  245, 367, 383;  247, 367, 435;  247, 369, 385;  249, 369, 437;  249, 371, 387;  281, 321, 481;  281, 383, 387;  283, 323, 483;  283, 385, 389;  285, 325, 485;  285, 387, 391;  321, 383, 433;  323, 385, 435;  325, 387, 437;  365, 387, 481;  367, 389, 483;  369, 391, 485;  0, 1, 431;  0, 3, 222;  0, 5, 266;  0, 6, 359;  0, 7, 323;  0, 8, 181;  0, 9, 284;  0, 11, 99;  0, 12, 319;  0, 13, 457;  0, 15, 171;  0, 17, 189;  0, 19, 499;  0, 21, 488;  0, 23, 511;  0, 25, 369;  0, 26, 513;  0, 27, 243;  0, 28, 217;  0, 29, 153;  0, 30, 341;  0, 31, 403;  0, 32, 496;  0, 33, 337;  0, 34, 177;  0, 35, 139;  0, 37, 332;  0, 38, 409;  0, 39, 305;  0, 41, 450;  0, 42, 91;  0, 43, 335;  0, 45, 213;  0, 47, 458;  0, 51, 388;  0, 52, 61;  0, 53, 517;  0, 54, 449;  0, 55, 179;  0, 56, 519;  0, 57, 197;  0, 58, 131;  0, 59, 489;  0, 63, 285;  0, 64, 103;  0, 65, 74;  0, 67, 342;  0, 70, 399;  0, 71, 280;  0, 73, 386;  0, 75, 407;  0, 77, 82;  0, 78, 435;  0, 79, 327;  0, 81, 475;  0, 83, 354;  0, 85, 340;  0, 87, 237;  0, 88, 379;  0, 89, 430;  0, 92, 259;  0, 93, 344;  0, 95, 361;  0, 96, 471;  0, 97, 169;  0, 101, 220;  0, 104, 439;  0, 105, 345;  0, 107, 250;  0, 108, 333;  0, 109, 336;  0, 111, 302;  0, 115, 494;  0, 117, 330;  0, 118, 143;  0, 119, 477;  0, 121, 136;  0, 123, 277;  0, 124, 151;  0, 125, 275;  0, 127, 466;  0, 128, 455;  0, 129, 253;  0, 133, 292;  0, 134, 493;  0, 135, 191;  0, 137, 372;  0, 140, 257;  0, 141, 254;  0, 145, 218;  0, 147, 452;  0, 149, 283;  0, 154, 463;  0, 156, 389;  0, 157, 274;  0, 159, 264;  0, 161, 195;  0, 163, 467;  0, 164, 215;  0, 165, 339;  0, 167, 418;  0, 170, 393;  0, 172, 411;  0, 173, 453;  0, 174, 423;  0, 175, 304;  0, 178, 301;  0, 182, 495;  0, 183, 447;  0, 184, 441;  0, 185, 338;  0, 187, 194;  0, 190, 211;  0, 193, 349;  0, 199, 514;  0, 201, 216;  0, 202, 445;  0, 203, 282;  0, 205, 316;  0, 206, 440;  0, 207, 368;  0, 209, 291;  0, 214, 271;  0, 219, 413;  0, 223, 315;  0, 227, 355;  0, 229, 459;  0, 230, 241;  0, 231, 484;  0, 234, 473;  0, 238, 397;  0, 242, 437;  0, 248, 267;  0, 255, 373;  0, 256, 261;  0, 263, 381;  0, 265, 434;  0, 268, 487;  0, 269, 377;  0, 272, 469;  0, 273, 470;  0, 279, 382;  0, 289, 503;  0, 293, 509;  0, 297, 448;  0, 299, 320;  0, 308, 317;  0, 313, 401;  0, 325, 429;  0, 328, 515;  0, 329, 394;  0, 331, 485;  0, 343, 461;  0, 347, 521;  0, 350, 497;  0, 351, 425;  0, 352, 427;  0, 358, 465;  0, 363, 391;  0, 371, 464;  0, 385, 415;  0, 398, 419;  0, 404, 501;  0, 405, 483;  0, 421, 491;  0, 451, 505;  0, 479, 490;  1, 2, 341;  1, 4, 122;  1, 7, 404;  1, 9, 320;  1, 10, 298;  1, 13, 490;  1, 14, 68;  1, 20, 32;  1, 22, 193;  1, 26, 467;  1, 27, 436;  1, 28, 218;  1, 29, 92;  1, 33, 110;  1, 34, 452;  1, 38, 171;  1, 39, 346;  1, 40, 179;  1, 44, 387;  1, 46, 428;  1, 50, 244;  1, 52, 80;  1, 53, 520;  1, 56, 285;  1, 57, 178;  1, 58, 374;  1, 62, 355;  1, 64, 338;  1, 65, 254;  1, 70, 191;  1, 75, 190;  1, 76, 250;  1, 82, 434;  1, 83, 142;  1, 86, 269;  1, 88, 515;  1, 94, 460;  1, 97, 337;  1, 98, 442;  1, 100, 273;  1, 104, 449;  1, 106, 491;  1, 109, 430;  1, 116, 465;  1, 124, 135;  1, 128, 383;  1, 134, 265;  1, 137, 394;  1, 141, 328;  1, 146, 172;  1, 148, 275;  1, 154, 236;  1, 165, 292;  1, 166, 303;  1, 173, 344;  1, 175, 358;  1, 176, 482;  1, 181, 406;  1, 182, 224;  1, 183, 220;  1, 185, 440;  1, 194, 489;  1, 195, 206;  1, 196, 412;  1, 200, 223;  1, 203, 362;  1, 207, 268;  1, 208, 359;  1, 212, 321;  1, 217, 506;  1, 221, 458;  1, 235, 494;  1, 238, 239;  1, 243, 266;  1, 248, 251;  1, 255, 286;  1, 262, 353;  1, 267, 488;  1, 272, 351;  1, 274, 453;  1, 281, 508;  1, 284, 435;  1, 308, 471;  1, 309, 446;  1, 310, 388;  1, 329, 332;  1, 333, 416;  1, 339, 424;  1, 370, 441;  1, 376, 410;  1, 382, 518;  1, 386, 464;  1, 418, 472;  1, 422, 485;  1, 459, 476;  1, 470, 497;  2, 3, 339;  2, 8, 15;  2, 10, 207;  2, 17, 203;  2, 27, 299;  2, 32, 69;  2, 33, 63;  2, 34, 359;  2, 35, 485;  2, 40, 449;  2, 41, 220;  2, 45, 260;  2, 47, 356;  2, 51, 152;  2, 58, 159;  2, 59, 243;  2, 74, 453;  2, 75, 428;  2, 76, 83;  2, 77, 188;  2, 87, 460;  2, 89, 419;  2, 93, 105;  2, 94, 129;  2, 95, 182;  2, 101, 467;  2, 107, 224;  2, 110, 285;  2, 117, 401;  2, 123, 155;  2, 125, 256;  2, 130, 147;  2, 131, 405;  2, 136, 495;  2, 137, 304;  2, 141, 232;  2, 143, 425;  2, 158, 317;  2, 166, 261;  2, 167, 194;  2, 173, 322;  2, 176, 383;  2, 179, 393;  2, 183, 311;  2, 184, 233;  2, 189, 395;  2, 191, 244;  2, 195, 377;  2, 201, 268;  2, 221, 291;  2, 227, 274;  2, 239, 284;  2, 245, 509;  2, 251, 519;  2, 255, 473;  2, 267, 305;  2, 275, 400;  2, 279, 375;  2, 281, 454;  2, 286, 477;  2, 293, 429;  2, 310, 399;  2, 321, 370;  2, 340, 417;  2, 352, 471;  2, 357, 491;  2, 436, 459;  2, 455, 472;  3, 4, 291;  3, 9, 478;  3, 10, 333;  3, 11, 364;  3, 16, 112;  3, 22, 214;  3, 28, 471;  3, 29, 484;  3, 45, 274;  3, 46, 323;  3, 58, 400;  3, 75, 298;  3, 76, 184;  3, 100, 322;  3, 107, 496;  3, 111, 358;  3, 136, 345;  3, 148, 371;  3, 172, 347;  3, 173, 442;  3, 178, 461;  3, 208, 394;  3, 220, 467;  3, 233, 262;  3, 292, 497;  3, 305, 473;  3, 460, 466;  4, 16, 323;  4, 17, 23;  4, 34, 455;  4, 35, 268;  4, 41, 76;  4, 46, 89;  4, 65, 107;  4, 83, 305;  4, 101, 197;  4, 113, 376;  4, 119, 149;  4, 161, 341;  4, 167, 245;  4, 172, 503;  4, 203, 286;  4, 215, 449;  4, 221, 233;  4, 365, 419}.
The deficiency graph is connected and has girth 4.

\adfAppGap

{\boldmath $\adfPENT(3,55,15)$}, $d = 6$:
{0, 1, 2;  0, 3, 11;  0, 4, 9;  0, 5, 81;  0, 6, 29;  0, 7, 27;  0, 8, 54;  0, 10, 67;  0, 15, 62;  0, 16, 60;  0, 17, 69;  0, 21, 47;  0, 22, 45;  0, 28, 52;  0, 46, 49;  0, 48, 121;  0, 50, 122;  0, 51, 55;  0, 53, 56;  0, 61, 64;  0, 63, 70;  0, 65, 71;  0, 68, 83;  0, 74, 123;  0, 79, 100;  0, 99, 124;  0, 101, 125;  1, 3, 81;  1, 5, 67;  1, 7, 51;  1, 8, 52;  1, 9, 82;  1, 10, 68;  1, 11, 83;  1, 15, 69;  1, 16, 63;  1, 17, 64;  1, 23, 49;  1, 27, 55;  1, 28, 53;  1, 29, 56;  1, 46, 50;  1, 47, 122;  1, 62, 71;  1, 74, 100;  1, 80, 125;  2, 3, 83;  2, 4, 68;  2, 5, 9;  2, 8, 27;  2, 10, 82;  2, 15, 71;  2, 16, 65;  2, 22, 50;  2, 23, 45;  2, 29, 52;  2, 63, 69;  2, 99, 123;  3, 4, 5;  3, 52, 58;  4, 11, 82;  4, 77, 125;  0, 12, 87;  0, 13, 36;  0, 14, 104;  0, 18, 42;  0, 19, 107;  0, 20, 43;  0, 25, 119;  0, 26, 38;  0, 37, 93;  0, 44, 85;  0, 57, 94;  0, 58, 110;  0, 59, 92;  0, 76, 86;  0, 77, 118;  0, 88, 117;  0, 89, 106;  0, 91, 109;  0, 95, 115;  0, 105, 116;  1, 13, 37;  1, 14, 117;  1, 20, 44;  1, 26, 43;  1, 38, 93;  1, 58, 92;  1, 59, 118;  1, 75, 94;  1, 76, 119;  1, 77, 116;  1, 87, 105;  1, 88, 106;  2, 20, 106;  2, 44, 119;  2, 58, 93;  2, 59, 75;  2, 87, 107;  3, 33, 64;  3, 35, 69;  3, 39, 101;  3, 41, 71;  3, 45, 88;  3, 47, 89;  3, 70, 112;  4, 23, 41;  4, 34, 70;  4, 35, 95;  4, 65, 101}.
The deficiency graph is connected and has girth 4.

{\boldmath $\adfPENT(3,61,15)$}, $d = 6$:
{0, 1, 2;  0, 3, 118;  0, 4, 16;  0, 5, 15;  0, 12, 94;  0, 13, 34;  0, 14, 33;  0, 17, 44;  0, 21, 127;  0, 22, 25;  0, 23, 24;  0, 26, 83;  0, 35, 38;  0, 36, 95;  0, 37, 93;  0, 42, 119;  0, 43, 117;  0, 57, 136;  0, 58, 103;  0, 75, 112;  0, 76, 97;  0, 81, 128;  0, 98, 101;  0, 99, 113;  0, 100, 111;  0, 104, 116;  0, 115, 135;  1, 3, 15;  1, 4, 117;  1, 5, 43;  1, 13, 93;  1, 14, 37;  1, 16, 44;  1, 17, 119;  1, 23, 82;  1, 25, 83;  1, 26, 128;  1, 34, 38;  1, 35, 95;  1, 58, 135;  1, 76, 111;  1, 77, 100;  1, 98, 113;  1, 99, 112;  1, 104, 137;  2, 3, 117;  2, 4, 44;  2, 15, 119;  2, 16, 118;  2, 22, 83;  2, 23, 81;  2, 26, 82;  2, 33, 95;  2, 34, 94;  2, 57, 135;  2, 75, 111;  2, 76, 101;  2, 77, 99;  3, 4, 5;  3, 29, 41;  4, 17, 118;  4, 107, 131;  0, 6, 129;  0, 7, 46;  0, 8, 45;  0, 9, 109;  0, 10, 20;  0, 11, 60;  0, 18, 125;  0, 19, 72;  0, 30, 86;  0, 31, 41;  0, 32, 71;  0, 39, 55;  0, 40, 110;  0, 47, 87;  0, 54, 134;  0, 61, 130;  0, 62, 105;  0, 67, 74;  0, 68, 133;  0, 69, 124;  0, 70, 121;  0, 73, 79;  0, 88, 106;  0, 122, 131;  1, 9, 130;  1, 10, 105;  1, 19, 73;  1, 20, 125;  1, 31, 71;  1, 32, 61;  1, 45, 89;  1, 46, 80;  1, 47, 129;  1, 56, 86;  1, 62, 69;  1, 68, 134;  1, 87, 107;  1, 106, 122;  1, 124, 131;  2, 8, 46;  2, 10, 62;  2, 20, 87;  2, 47, 131;  2, 56, 125;  2, 89, 105;  2, 123, 129;  3, 10, 99;  3, 11, 76;  3, 21, 75;  3, 22, 77;  3, 51, 137;  3, 53, 94;  3, 70, 136;  3, 71, 88;  3, 95, 100;  4, 10, 94;  4, 23, 95;  4, 46, 89;  4, 53, 71;  5, 11, 101}.
The deficiency graph is connected and has girth 4.

{\boldmath $\adfPENT(3,333,15)$}, $d = 2$:
{0, 112, 486;  0, 174, 672;  0, 230, 472;  0, 236, 428;  0, 270, 664;  0, 332, 340;  0, 386, 680;  1, 3, 413;  1, 11, 509;  1, 19, 297;  1, 197, 211;  1, 255, 447;  1, 343, 351;  1, 453, 571;  3, 19, 453;  3, 197, 447;  3, 211, 509;  3, 255, 571;  3, 297, 351;  11, 197, 255;  11, 211, 453;  11, 297, 413;  11, 447, 571;  19, 197, 413;  19, 343, 571;  112, 174, 270;  112, 230, 664;  112, 236, 340;  112, 332, 386;  112, 428, 680;  112, 472, 672;  174, 230, 680;  174, 340, 486;  174, 386, 472;  197, 297, 571;  197, 343, 509;  197, 351, 453;  211, 255, 297;  211, 343, 413;  211, 351, 571;  230, 236, 386;  230, 270, 340;  230, 332, 486;  230, 428, 672;  236, 270, 680;  236, 486, 672;  255, 343, 453;  255, 351, 413;  270, 386, 672;  270, 472, 486;  297, 343, 447;  297, 453, 509;  332, 472, 680;  340, 386, 664;  340, 428, 472;  386, 428, 486;  413, 447, 453;  486, 664, 680;  0, 4, 527;  0, 5, 540;  0, 7, 548;  0, 9, 562;  0, 12, 577;  0, 13, 382;  0, 15, 315;  0, 17, 376;  0, 20, 369;  0, 21, 32;  0, 22, 573;  0, 23, 400;  0, 24, 427;  0, 25, 518;  0, 26, 411;  0, 27, 320;  0, 28, 245;  0, 29, 38;  0, 30, 73;  0, 31, 265;  0, 33, 237;  0, 35, 554;  0, 36, 89;  0, 37, 437;  0, 39, 420;  0, 41, 409;  0, 45, 355;  0, 47, 207;  0, 48, 277;  0, 49, 663;  0, 50, 219;  0, 51, 511;  0, 52, 381;  0, 55, 544;  0, 57, 476;  0, 59, 125;  0, 60, 579;  0, 61, 406;  0, 63, 91;  0, 64, 593;  0, 65, 592;  0, 66, 181;  0, 67, 352;  0, 68, 525;  0, 69, 107;  0, 71, 502;  0, 72, 609;  0, 74, 493;  0, 75, 584;  0, 76, 275;  0, 77, 109;  0, 78, 557;  0, 79, 214;  0, 80, 379;  0, 81, 621;  0, 82, 643;  0, 83, 601;  0, 84, 555;  0, 85, 421;  0, 87, 500;  0, 92, 667;  0, 93, 569;  0, 94, 325;  0, 95, 143;  0, 97, 258;  0, 99, 489;  0, 101, 675;  0, 103, 576;  0, 105, 177;  0, 108, 513;  0, 111, 234;  0, 113, 188;  0, 117, 318;  0, 119, 268;  0, 121, 172;  0, 123, 152;  0, 126, 669;  0, 127, 261;  0, 131, 485;  0, 133, 458;  0, 135, 303;  0, 136, 599;  0, 137, 629;  0, 139, 175;  0, 144, 603;  0, 145, 363;  0, 148, 353;  0, 149, 171;  0, 151, 289;  0, 153, 328;  0, 157, 398;  0, 159, 423;  0, 160, 417;  0, 161, 309;  0, 162, 549;  0, 165, 170;  0, 167, 433;  0, 168, 641;  0, 179, 535;  0, 183, 302;  0, 185, 267;  0, 187, 445;  0, 190, 443;  0, 191, 390;  0, 195, 226;  0, 201, 415;  0, 203, 487;  0, 204, 499;  0, 213, 393;  0, 215, 439;  0, 218, 501;  0, 221, 383;  0, 222, 515;  0, 223, 637;  0, 225, 407;  0, 227, 517;  0, 233, 597;  0, 235, 361;  0, 239, 304;  0, 241, 321;  0, 243, 623;  0, 247, 307;  0, 249, 627;  0, 259, 589;  0, 260, 595;  0, 264, 635;  0, 266, 539;  0, 271, 314;  0, 279, 655;  0, 281, 338;  0, 285, 545;  0, 287, 310;  0, 290, 681;  0, 291, 367;  0, 311, 649;  0, 317, 401;  0, 319, 356;  0, 327, 647;  0, 331, 425;  0, 333, 336;  0, 339, 467;  0, 341, 365;  0, 345, 375;  0, 347, 373;  0, 359, 531;  0, 377, 497;  0, 395, 657;  0, 399, 587;  0, 429, 503;  0, 431, 495;  0, 435, 661;  0, 449, 585;  0, 451, 455;  0, 461, 613;  0, 465, 477;  0, 469, 567;  0, 475, 619;  0, 491, 581;  0, 505, 583;  0, 541, 633;  0, 553, 605;  0, 591, 611;  0, 615, 665;  1, 107, 277}.
The deficiency graph is connected and has girth 4.

\adfAppGap

{\boldmath $\adfPENT(3,561,19)$}, $d = 2$:
{0, 32, 494;  0, 90, 1016;  0, 210, 264;  0, 236, 962;  0, 250, 752;  0, 332, 460;  0, 508, 966;  0, 514, 790;  0, 768, 832;  1, 127, 683;  1, 177, 635;  1, 181, 907;  1, 311, 1111;  1, 353, 629;  1, 375, 811;  1, 391, 893;  1, 649, 1053;  1, 879, 933;  32, 90, 460;  32, 210, 236;  32, 250, 966;  32, 264, 832;  32, 332, 962;  32, 508, 790;  32, 514, 752;  32, 768, 1016;  90, 210, 250;  90, 236, 514;  90, 264, 768;  90, 332, 494;  90, 508, 962;  90, 752, 790;  90, 832, 966;  127, 177, 649;  127, 311, 933;  127, 353, 1053;  127, 375, 907;  127, 893, 1111;  177, 181, 1053;  177, 311, 879;  177, 375, 893;  177, 391, 683;  177, 629, 907;  177, 933, 1111;  181, 311, 649;  181, 353, 893;  181, 375, 635;  181, 629, 1111;  181, 811, 879;  210, 332, 514;  210, 494, 1016;  210, 508, 832;  210, 752, 966;  210, 768, 790;  236, 250, 494;  236, 264, 790;  236, 332, 832;  236, 460, 508;  236, 752, 768;  236, 966, 1016;  250, 332, 768;  250, 790, 962;  250, 832, 1016;  264, 332, 752;  264, 460, 494;  264, 962, 966;  311, 353, 391;  311, 375, 683;  311, 629, 1053;  311, 635, 907;  311, 811, 893;  353, 375, 879;  353, 635, 1111;  353, 649, 683;  353, 907, 933;  375, 391, 933;  375, 629, 649;  375, 1053, 1111;  391, 629, 811;  391, 635, 649;  391, 879, 1053;  391, 907, 1111;  460, 752, 832;  460, 768, 966;  460, 790, 1016;  494, 514, 966;  494, 768, 962;  494, 790, 832;  508, 514, 768;  514, 832, 962;  629, 635, 933;  635, 683, 1053;  649, 811, 933;  649, 879, 1111;  683, 811, 1111;  683, 879, 907;  811, 907, 1053;  893, 933, 1053;  0, 2, 571;  0, 3, 785;  0, 5, 990;  0, 7, 652;  0, 8, 485;  0, 9, 459;  0, 10, 537;  0, 11, 438;  0, 12, 1047;  0, 13, 1037;  0, 15, 798;  0, 17, 414;  0, 18, 863;  0, 19, 1019;  0, 21, 1090;  0, 23, 246;  0, 24, 81;  0, 25, 396;  0, 27, 994;  0, 29, 1066;  0, 30, 973;  0, 31, 394;  0, 33, 89;  0, 35, 952;  0, 36, 927;  0, 37, 99;  0, 39, 252;  0, 41, 852;  0, 43, 1002;  0, 44, 407;  0, 45, 550;  0, 46, 977;  0, 47, 200;  0, 49, 393;  0, 51, 415;  0, 53, 185;  0, 55, 719;  0, 56, 995;  0, 59, 915;  0, 60, 595;  0, 61, 306;  0, 62, 303;  0, 63, 641;  0, 65, 1070;  0, 66, 475;  0, 67, 710;  0, 69, 902;  0, 70, 283;  0, 71, 855;  0, 74, 341;  0, 75, 1024;  0, 77, 192;  0, 78, 1029;  0, 79, 447;  0, 83, 972;  0, 84, 765;  0, 85, 519;  0, 86, 887;  0, 87, 913;  0, 88, 237;  0, 91, 167;  0, 92, 689;  0, 93, 992;  0, 94, 825;  0, 95, 886;  0, 97, 646;  0, 98, 803;  0, 100, 427;  0, 101, 941;  0, 102, 1057;  0, 103, 202;  0, 104, 923;  0, 106, 279;  0, 107, 188;  0, 108, 217;  0, 110, 573;  0, 111, 350;  0, 112, 1049;  0, 113, 1004;  0, 114, 385;  0, 115, 215;  0, 116, 581;  0, 117, 604;  0, 119, 229;  0, 121, 223;  0, 123, 552;  0, 124, 691;  0, 125, 676;  0, 129, 443;  0, 131, 701;  0, 132, 677;  0, 133, 744;  0, 135, 186;  0, 136, 539;  0, 139, 612;  0, 141, 607;  0, 142, 1113;  0, 143, 599;  0, 144, 549;  0, 145, 220;  0, 147, 774;  0, 151, 197;  0, 153, 967;  0, 154, 381;  0, 155, 750;  0, 156, 809;  0, 159, 557;  0, 161, 367;  0, 163, 802;  0, 164, 749;  0, 165, 365;  0, 166, 827;  0, 168, 589;  0, 169, 450;  0, 171, 179;  0, 187, 664;  0, 189, 474;  0, 191, 265;  0, 195, 1017;  0, 199, 1073;  0, 201, 737;  0, 203, 388;  0, 205, 1101;  0, 206, 697;  0, 207, 732;  0, 209, 358;  0, 211, 572;  0, 219, 1007;  0, 221, 840;  0, 222, 467;  0, 228, 901;  0, 231, 769;  0, 233, 625;  0, 234, 1099;  0, 235, 623;  0, 239, 760;  0, 247, 839;  0, 249, 262;  0, 255, 564;  0, 257, 425;  0, 259, 280;  0, 261, 305;  0, 263, 822;  0, 268, 823;  0, 269, 740;  0, 273, 529;  0, 275, 408;  0, 277, 413;  0, 281, 509;  0, 285, 782;  0, 286, 1039;  0, 287, 733;  0, 288, 947;  0, 289, 356;  0, 291, 830;  0, 293, 449;  0, 294, 837;  0, 295, 696;  0, 297, 841;  0, 299, 829;  0, 301, 1077;  0, 307, 598;  0, 313, 565;  0, 314, 755;  0, 315, 317;  0, 316, 959;  0, 319, 379;  0, 321, 348;  0, 322, 1065;  0, 323, 525;  0, 325, 647;  0, 326, 1005;  0, 328, 1031;  0, 329, 618;  0, 333, 457;  0, 334, 935;  0, 335, 831;  0, 337, 453;  0, 339, 1011;  0, 343, 536;  0, 345, 511;  0, 346, 763;  0, 347, 354;  0, 349, 1109;  0, 355, 757;  0, 357, 380;  0, 361, 439;  0, 364, 1089;  0, 366, 979;  0, 369, 963;  0, 371, 751;  0, 373, 513;  0, 377, 985;  0, 378, 859;  0, 383, 1117;  0, 387, 470;  0, 389, 528;  0, 395, 905;  0, 397, 551;  0, 399, 883;  0, 401, 695;  0, 411, 686;  0, 419, 510;  0, 423, 433;  0, 429, 1133;  0, 431, 917;  0, 434, 889;  0, 435, 507;  0, 437, 777;  0, 445, 969;  0, 451, 1041;  0, 461, 650;  0, 469, 1083;  0, 471, 563;  0, 473, 587;  0, 479, 875;  0, 483, 644;  0, 484, 1063;  0, 486, 1105;  0, 487, 1131;  0, 489, 576;  0, 493, 548;  0, 495, 639;  0, 501, 609;  0, 505, 1081;  0, 517, 991;  0, 521, 877;  0, 533, 1023;  0, 534, 1095;  0, 541, 775;  0, 553, 717;  0, 559, 657;  0, 575, 663;  0, 577, 797;  0, 605, 795;  0, 611, 945;  0, 615, 1107;  0, 627, 739;  0, 631, 1045;  0, 633, 983;  0, 651, 675;  0, 667, 1015;  0, 685, 1079;  0, 687, 873;  0, 693, 759;  0, 699, 961;  0, 707, 987;  0, 709, 1141;  0, 711, 881;  0, 721, 869;  0, 723, 911;  0, 727, 1137;  0, 729, 921;  0, 735, 1025;  0, 747, 1125;  0, 761, 1001;  0, 767, 1093;  0, 773, 1085;  0, 783, 1071;  0, 787, 925;  0, 789, 807;  0, 791, 821;  0, 793, 1139;  0, 799, 835;  0, 805, 909;  0, 813, 899;  0, 815, 885;  0, 817, 1123;  0, 843, 895;  0, 847, 1069;  0, 849, 999;  0, 871, 965;  0, 975, 1127;  0, 997, 1103;  0, 1013, 1097;  0, 1021, 1033}.
The deficiency graph is connected and has girth 4.

{\boldmath $\adfPENT(3,563,19)$}, $d = 6$:
{0, 96, 604;  0, 146, 890;  0, 170, 322;  0, 270, 290;  0, 488, 934;  0, 588, 784;  0, 688, 812;  0, 744, 1000;  0, 908, 998;  1, 147, 403;  1, 149, 825;  1, 213, 559;  1, 239, 857;  1, 257, 1001;  1, 335, 459;  1, 363, 977;  1, 543, 1051;  1, 659, 877;  2, 148, 892;  2, 172, 324;  2, 272, 292;  2, 490, 936;  2, 590, 786;  2, 690, 814;  2, 910, 1000;  3, 149, 405;  3, 151, 827;  3, 215, 561;  3, 241, 859;  3, 365, 979;  3, 661, 879;  4, 174, 326;  4, 274, 294;  4, 492, 938;  4, 592, 788;  4, 912, 1002;  5, 153, 829;  5, 217, 563;  5, 243, 861;  5, 367, 981;  5, 663, 881;  96, 146, 1000;  96, 170, 934;  96, 270, 890;  96, 290, 998;  96, 322, 688;  96, 488, 812;  96, 588, 744;  98, 148, 1002;  98, 172, 936;  98, 272, 892;  98, 292, 1000;  98, 324, 690;  98, 490, 814;  98, 590, 746;  100, 150, 1004;  100, 174, 938;  100, 274, 894;  100, 294, 1002;  100, 326, 692;  100, 492, 816;  100, 592, 748;  146, 170, 290;  146, 270, 604;  146, 322, 908;  146, 488, 998;  146, 588, 812;  146, 688, 784;  146, 744, 934;  147, 149, 543;  147, 213, 857;  147, 239, 1001;  147, 257, 877;  147, 335, 1051;  147, 363, 659;  147, 459, 977;  147, 559, 825;  148, 172, 292;  148, 324, 910;  148, 490, 1000;  148, 590, 814;  148, 746, 936;  149, 151, 545;  149, 213, 239;  149, 215, 859;  149, 241, 1003;  149, 257, 335;  149, 259, 879;  149, 337, 1053;  149, 363, 403;  149, 365, 661;  149, 459, 659;  149, 461, 979;  149, 559, 1051;  149, 561, 827;  149, 857, 1001;  149, 877, 977;  150, 174, 294;  150, 326, 912;  150, 492, 1002;  150, 592, 816;  150, 748, 938;  151, 153, 547;  151, 215, 241;  151, 217, 861;  151, 243, 1005;  151, 259, 337;  151, 261, 881;  151, 339, 1055;  151, 365, 405;  151, 367, 663;  151, 461, 661;  151, 463, 981;  151, 561, 1053;  151, 563, 829;  151, 859, 1003;  153, 217, 243;  153, 261, 339;  153, 367, 407;  153, 463, 663;  153, 563, 1055;  153, 861, 1005;  170, 270, 998;  170, 604, 812;  170, 688, 1000;  170, 744, 908;  170, 784, 890;  172, 272, 1000;  172, 606, 814;  172, 690, 1002;  172, 746, 910;  172, 786, 892;  174, 608, 816;  174, 692, 1004;  174, 748, 912;  174, 788, 894;  213, 257, 1051;  213, 335, 659;  213, 363, 825;  213, 403, 877;  213, 459, 543;  213, 977, 1001;  215, 259, 1053;  215, 337, 661;  215, 365, 827;  215, 405, 879;  215, 461, 545;  215, 979, 1003;  217, 261, 1055;  217, 339, 663;  217, 367, 829;  217, 407, 881;  217, 463, 547;  217, 981, 1005;  239, 257, 977;  239, 335, 877;  239, 403, 559;  239, 459, 825;  239, 543, 659;  241, 259, 979;  241, 405, 561;  241, 461, 827;  241, 545, 661;  243, 261, 981;  243, 407, 563;  243, 463, 829;  243, 547, 663;  257, 363, 559;  257, 459, 857;  257, 543, 825;  259, 365, 561;  259, 461, 859;  259, 545, 827;  261, 367, 563;  261, 463, 861;  261, 547, 829;  270, 322, 488;  270, 588, 688;  270, 744, 784;  270, 812, 908;  270, 934, 1000;  272, 324, 490;  272, 746, 786;  272, 936, 1002;  274, 326, 492;  274, 748, 788;  274, 938, 1004;  290, 322, 784;  290, 488, 688;  290, 588, 1000;  290, 604, 890;  290, 744, 812;  290, 908, 934;  292, 324, 786;  292, 490, 690;  292, 590, 1002;  292, 606, 892;  292, 746, 814;  292, 910, 936;  294, 326, 788;  294, 492, 692;  294, 592, 1004;  294, 608, 894;  294, 748, 816;  294, 912, 938;  322, 588, 998;  322, 604, 744;  322, 812, 934;  322, 890, 1000;  324, 590, 1000;  324, 606, 746;  324, 814, 936;  324, 892, 1002;  326, 592, 1002;  326, 608, 748;  326, 816, 938;  326, 894, 1004;  335, 363, 543;  335, 403, 857;  335, 559, 1001;  335, 825, 977;  337, 365, 545;  337, 405, 859;  337, 561, 1003;  337, 827, 979;  339, 367, 547;  339, 407, 861;  339, 563, 1005;  339, 829, 981;  363, 459, 1001;  363, 857, 877;  365, 859, 879;  367, 861, 881;  403, 459, 1051;  403, 543, 977;  403, 825, 1001;  405, 461, 1053;  405, 545, 979;  405, 827, 1003;  407, 463, 1055;  407, 547, 981;  407, 829, 1005;  459, 559, 877;  461, 561, 879;  488, 604, 784;  488, 908, 1000;  490, 606, 786;  490, 910, 1002;  492, 608, 788;  492, 912, 1004;  543, 559, 857;  543, 877, 1001;  545, 561, 859;  545, 879, 1003;  547, 563, 861;  588, 604, 908;  588, 890, 934;  590, 606, 910;  590, 892, 936;  592, 608, 912;  592, 894, 938;  604, 688, 934;  604, 998, 1000;  606, 690, 936;  606, 1000, 1002;  608, 692, 938;  608, 1002, 1004;  659, 825, 857;  659, 1001, 1051;  661, 827, 859;  661, 1003, 1053;  663, 829, 861;  663, 1005, 1055;  688, 744, 998;  688, 890, 908;  690, 746, 1000;  690, 892, 910;  692, 748, 1002;  692, 894, 912;  784, 812, 1000;  784, 934, 998;  786, 814, 1002;  786, 936, 1000;  788, 816, 1004;  788, 938, 1002;  812, 890, 998;  814, 892, 1000;  816, 894, 1002;  825, 877, 1051;  827, 879, 1053;  829, 881, 1055;  857, 977, 1051;  859, 979, 1053;  861, 981, 1055;  0, 3, 404;  0, 4, 949;  0, 5, 662;  0, 6, 345;  0, 7, 506;  0, 8, 479;  0, 9, 373;  0, 10, 509;  0, 11, 914;  0, 12, 613;  0, 13, 550;  0, 14, 861;  0, 15, 520;  0, 17, 1064;  0, 19, 525;  0, 21, 691;  0, 22, 229;  0, 23, 251;  0, 25, 823;  0, 27, 118;  0, 29, 356;  0, 30, 607;  0, 31, 204;  0, 33, 766;  0, 34, 1013;  0, 35, 1099;  0, 36, 539;  0, 37, 1085;  0, 38, 1101;  0, 39, 303;  0, 41, 53;  0, 42, 921;  0, 43, 173;  0, 45, 1055;  0, 46, 795;  0, 47, 1061;  0, 48, 205;  0, 49, 121;  0, 51, 590;  0, 54, 225;  0, 55, 169;  0, 57, 737;  0, 58, 359;  0, 59, 845;  0, 60, 573;  0, 61, 1029;  0, 62, 765;  0, 63, 486;  0, 65, 288;  0, 67, 284;  0, 69, 798;  0, 70, 593;  0, 71, 206;  0, 72, 1007;  0, 73, 232;  0, 75, 781;  0, 76, 975;  0, 77, 687;  0, 79, 770;  0, 80, 265;  0, 81, 631;  0, 82, 575;  0, 83, 649;  0, 85, 562;  0, 86, 739;  0, 87, 576;  0, 88, 133;  0, 89, 136;  0, 91, 374;  0, 93, 1142;  0, 94, 1137;  0, 95, 499;  0, 97, 497;  0, 98, 507;  0, 99, 1143;  0, 101, 741;  0, 102, 1111;  0, 103, 893;  0, 104, 771;  0, 105, 1047;  0, 107, 1125;  0, 109, 846;  0, 111, 540;  0, 112, 331;  0, 113, 460;  0, 114, 551;  0, 115, 635;  0, 117, 626;  0, 119, 425;  0, 123, 1017;  0, 125, 707;  0, 126, 883;  0, 127, 481;  0, 128, 285;  0, 129, 184;  0, 130, 791;  0, 131, 683;  0, 132, 605;  0, 134, 355;  0, 135, 496;  0, 137, 581;  0, 138, 341;  0, 139, 608;  0, 141, 1059;  0, 142, 493;  0, 143, 391;  0, 145, 154;  0, 151, 450;  0, 153, 999;  0, 155, 965;  0, 158, 645;  0, 159, 222;  0, 160, 1041;  0, 161, 165;  0, 162, 1105;  0, 163, 663;  0, 167, 616;  0, 168, 1005;  0, 172, 853;  0, 175, 411;  0, 177, 234;  0, 178, 197;  0, 179, 252;  0, 181, 215;  0, 182, 746;  0, 183, 360;  0, 185, 1010;  0, 187, 742;  0, 189, 797;  0, 191, 787;  0, 192, 1123;  0, 193, 469;  0, 195, 549;  0, 199, 1012;  0, 201, 745;  0, 207, 985;  0, 209, 910;  0, 210, 805;  0, 211, 964;  0, 217, 799;  0, 219, 400;  0, 221, 1108;  0, 223, 702;  0, 227, 917;  0, 228, 1095;  0, 230, 445;  0, 231, 1084;  0, 233, 806;  0, 235, 731;  0, 236, 243;  0, 237, 1045;  0, 240, 911;  0, 241, 344;  0, 245, 680;  0, 247, 778;  0, 248, 299;  0, 249, 283;  0, 250, 669;  0, 253, 487;  0, 255, 461;  0, 258, 1133;  0, 259, 725;  0, 260, 603;  0, 261, 957;  0, 262, 395;  0, 263, 452;  0, 264, 1043;  0, 267, 878;  0, 268, 719;  0, 269, 1066;  0, 271, 674;  0, 272, 535;  0, 273, 1039;  0, 274, 673;  0, 275, 1076;  0, 276, 627;  0, 277, 721;  0, 278, 389;  0, 279, 582;  0, 280, 501;  0, 281, 1049;  0, 287, 566;  0, 289, 1004;  0, 291, 732;  0, 293, 435;  0, 295, 1069;  0, 297, 986;  0, 301, 817;  0, 305, 862;  0, 306, 983;  0, 307, 585;  0, 309, 1023;  0, 311, 1025;  0, 313, 343;  0, 315, 715;  0, 317, 530;  0, 319, 756;  0, 321, 609;  0, 323, 885;  0, 325, 768;  0, 326, 353;  0, 327, 820;  0, 328, 375;  0, 329, 1103;  0, 332, 521;  0, 333, 536;  0, 336, 955;  0, 337, 1034;  0, 338, 927;  0, 340, 869;  0, 347, 1070;  0, 349, 809;  0, 350, 941;  0, 354, 759;  0, 357, 937;  0, 361, 633;  0, 364, 429;  0, 365, 368;  0, 367, 693;  0, 369, 884;  0, 370, 1139;  0, 371, 651;  0, 372, 1145;  0, 376, 665;  0, 377, 524;  0, 379, 863;  0, 380, 947;  0, 381, 698;  0, 383, 565;  0, 385, 545;  0, 386, 899;  0, 387, 1048;  0, 388, 1129;  0, 393, 615;  0, 397, 1037;  0, 399, 424;  0, 401, 740;  0, 406, 587;  0, 407, 646;  0, 413, 865;  0, 415, 953;  0, 419, 441;  0, 421, 761;  0, 423, 968;  0, 427, 599;  0, 428, 547;  0, 431, 1028;  0, 432, 1087;  0, 433, 1079;  0, 439, 670;  0, 440, 833;  0, 443, 1018;  0, 447, 544;  0, 448, 753;  0, 449, 993;  0, 451, 622;  0, 453, 690;  0, 455, 814;  0, 457, 556;  0, 463, 987;  0, 464, 529;  0, 465, 695;  0, 466, 653;  0, 467, 516;  0, 471, 749;  0, 472, 1019;  0, 475, 868;  0, 476, 995;  0, 477, 849;  0, 478, 989;  0, 483, 1093;  0, 484, 747;  0, 485, 981;  0, 489, 569;  0, 491, 755;  0, 495, 775;  0, 500, 967;  0, 505, 553;  0, 511, 686;  0, 515, 963;  0, 517, 992;  0, 519, 1052;  0, 523, 897;  0, 527, 681;  0, 531, 629;  0, 533, 811;  0, 537, 552;  0, 538, 959;  0, 541, 722;  0, 555, 668;  0, 557, 1141;  0, 561, 621;  0, 563, 772;  0, 567, 729;  0, 571, 713;  0, 579, 1003;  0, 580, 907;  0, 583, 1035;  0, 584, 807;  0, 589, 933;  0, 591, 851;  0, 596, 829;  0, 597, 1081;  0, 602, 819;  0, 610, 991;  0, 611, 871;  0, 617, 905;  0, 623, 891;  0, 625, 783;  0, 637, 650;  0, 639, 988;  0, 640, 997;  0, 641, 873;  0, 643, 1067;  0, 647, 818;  0, 661, 839;  0, 675, 733;  0, 679, 776;  0, 682, 1115;  0, 685, 699;  0, 689, 1088;  0, 694, 1053;  0, 697, 767;  0, 701, 1124;  0, 706, 813;  0, 711, 974;  0, 718, 1063;  0, 727, 769;  0, 735, 1121;  0, 743, 1042;  0, 751, 1136;  0, 758, 1127;  0, 760, 1135;  0, 763, 1011;  0, 777, 919;  0, 782, 859;  0, 785, 1016;  0, 789, 835;  0, 790, 1107;  0, 793, 940;  0, 796, 827;  0, 801, 961;  0, 802, 1109;  0, 803, 962;  0, 808, 901;  0, 815, 895;  0, 821, 951;  0, 831, 1132;  0, 841, 896;  0, 855, 1060;  0, 866, 889;  0, 872, 1077;  0, 874, 1119;  0, 881, 916;  0, 886, 1065;  0, 887, 898;  0, 903, 1021;  0, 913, 1100;  0, 915, 1027;  0, 925, 929;  0, 939, 945;  0, 971, 1057;  0, 979, 1117;  0, 1015, 1075;  0, 1031, 1138;  0, 1033, 1058;  0, 1071, 1113;  0, 1091, 1112;  1, 2, 376;  1, 4, 87;  1, 7, 496;  1, 8, 638;  1, 9, 22;  1, 11, 652;  1, 13, 140;  1, 16, 859;  1, 20, 895;  1, 23, 818;  1, 28, 1133;  1, 32, 733;  1, 34, 958;  1, 37, 484;  1, 38, 553;  1, 39, 746;  1, 40, 627;  1, 44, 989;  1, 46, 308;  1, 47, 1034;  1, 50, 771;  1, 52, 183;  1, 55, 640;  1, 58, 389;  1, 59, 1112;  1, 62, 993;  1, 63, 646;  1, 64, 1053;  1, 68, 974;  1, 70, 634;  1, 74, 200;  1, 76, 416;  1, 77, 508;  1, 80, 259;  1, 81, 364;  1, 82, 429;  1, 86, 1011;  1, 88, 710;  1, 89, 1090;  1, 92, 196;  1, 94, 671;  1, 95, 206;  1, 99, 1084;  1, 103, 224;  1, 105, 836;  1, 106, 223;  1, 110, 340;  1, 112, 280;  1, 113, 332;  1, 116, 873;  1, 118, 281;  1, 119, 886;  1, 124, 898;  1, 127, 1036;  1, 129, 832;  1, 130, 1019;  1, 133, 920;  1, 134, 585;  1, 135, 764;  1, 136, 622;  1, 137, 316;  1, 142, 704;  1, 152, 691;  1, 154, 577;  1, 155, 542;  1, 158, 807;  1, 163, 1004;  1, 164, 1078;  1, 166, 811;  1, 169, 412;  1, 178, 937;  1, 184, 616;  1, 185, 718;  1, 190, 1013;  1, 193, 970;  1, 194, 292;  1, 205, 730;  1, 207, 350;  1, 208, 338;  1, 212, 815;  1, 214, 229;  1, 220, 557;  1, 226, 626;  1, 230, 986;  1, 231, 712;  1, 233, 794;  1, 236, 661;  1, 238, 1139;  1, 241, 868;  1, 242, 683;  1, 248, 451;  1, 250, 405;  1, 251, 1010;  1, 254, 706;  1, 256, 896;  1, 260, 428;  1, 261, 902;  1, 262, 1043;  1, 263, 1052;  1, 265, 566;  1, 266, 620;  1, 268, 874;  1, 269, 1048;  1, 274, 841;  1, 275, 1096;  1, 278, 617;  1, 285, 434;  1, 286, 885;  1, 296, 548;  1, 298, 651;  1, 301, 656;  1, 310, 545;  1, 314, 669;  1, 320, 915;  1, 326, 892;  1, 328, 357;  1, 329, 1018;  1, 333, 952;  1, 334, 544;  1, 344, 1137;  1, 346, 911;  1, 351, 482;  1, 352, 824;  1, 358, 1071;  1, 361, 866;  1, 362, 433;  1, 365, 758;  1, 368, 387;  1, 370, 741;  1, 371, 860;  1, 374, 875;  1, 377, 454;  1, 379, 1060;  1, 380, 956;  1, 382, 1064;  1, 391, 808;  1, 392, 797;  1, 398, 686;  1, 400, 658;  1, 407, 1022;  1, 410, 1109;  1, 422, 681;  1, 424, 560;  1, 430, 572;  1, 436, 568;  1, 440, 675;  1, 442, 897;  1, 449, 1108;  1, 452, 938;  1, 458, 1120;  1, 460, 821;  1, 464, 941;  1, 465, 1088;  1, 472, 700;  1, 473, 694;  1, 479, 1042;  1, 512, 554;  1, 514, 963;  1, 518, 856;  1, 520, 541;  1, 524, 699;  1, 530, 664;  1, 531, 862;  1, 536, 1066;  1, 550, 1006;  1, 562, 719;  1, 574, 676;  1, 578, 773;  1, 584, 777;  1, 590, 838;  1, 591, 850;  1, 592, 826;  1, 596, 734;  1, 598, 872;  1, 602, 760;  1, 609, 904;  1, 610, 770;  1, 611, 992;  1, 614, 917;  1, 623, 662;  1, 644, 806;  1, 650, 767;  1, 668, 1100;  1, 670, 1024;  1, 674, 1114;  1, 687, 922;  1, 707, 1118;  1, 728, 1065;  1, 736, 1017;  1, 740, 803;  1, 742, 800;  1, 752, 819;  1, 759, 782;  1, 761, 830;  1, 776, 975;  1, 779, 1046;  1, 809, 946;  1, 844, 1143;  1, 848, 1125;  1, 854, 916;  1, 863, 994;  1, 878, 1106;  1, 879, 1016;  1, 880, 950;  1, 964, 1012;  1, 976, 1138;  1, 980, 1094;  1, 982, 988;  1, 1040, 1142;  1, 1058, 1130;  1, 1059, 1144;  1, 1072, 1085;  1, 1077, 1136;  2, 5, 944;  2, 8, 987;  2, 10, 393;  2, 11, 923;  2, 14, 1085;  2, 15, 32;  2, 16, 227;  2, 17, 291;  2, 23, 1121;  2, 27, 729;  2, 33, 41;  2, 35, 338;  2, 38, 257;  2, 39, 280;  2, 40, 153;  2, 45, 719;  2, 47, 563;  2, 50, 621;  2, 51, 957;  2, 56, 335;  2, 57, 1043;  2, 59, 130;  2, 62, 557;  2, 63, 308;  2, 71, 1115;  2, 75, 452;  2, 77, 608;  2, 81, 135;  2, 82, 731;  2, 83, 801;  2, 87, 628;  2, 88, 689;  2, 89, 352;  2, 93, 1047;  2, 105, 1083;  2, 107, 953;  2, 111, 897;  2, 117, 994;  2, 123, 592;  2, 125, 652;  2, 129, 362;  2, 131, 243;  2, 134, 951;  2, 137, 263;  2, 141, 212;  2, 143, 909;  2, 147, 711;  2, 155, 358;  2, 165, 848;  2, 167, 419;  2, 171, 872;  2, 173, 183;  2, 179, 554;  2, 184, 851;  2, 185, 807;  2, 189, 1029;  2, 194, 1103;  2, 203, 262;  2, 208, 381;  2, 213, 789;  2, 224, 1059;  2, 231, 363;  2, 233, 964;  2, 236, 683;  2, 238, 299;  2, 239, 293;  2, 249, 260;  2, 251, 704;  2, 255, 776;  2, 266, 665;  2, 267, 1025;  2, 269, 693;  2, 273, 899;  2, 274, 585;  2, 275, 675;  2, 285, 1055;  2, 286, 377;  2, 287, 682;  2, 297, 383;  2, 303, 734;  2, 309, 1114;  2, 311, 1072;  2, 315, 329;  2, 317, 670;  2, 321, 760;  2, 327, 598;  2, 328, 1049;  2, 333, 380;  2, 334, 387;  2, 346, 741;  2, 350, 1113;  2, 351, 1091;  2, 369, 965;  2, 370, 833;  2, 375, 784;  2, 382, 1143;  2, 388, 995;  2, 399, 692;  2, 406, 843;  2, 413, 724;  2, 423, 471;  2, 429, 1030;  2, 430, 743;  2, 431, 820;  2, 435, 921;  2, 437, 495;  2, 455, 1071;  2, 467, 671;  2, 477, 753;  2, 478, 483;  2, 485, 839;  2, 501, 1095;  2, 502, 1041;  2, 509, 687;  2, 513, 976;  2, 527, 867;  2, 531, 875;  2, 538, 681;  2, 539, 1054;  2, 543, 555;  2, 549, 887;  2, 567, 1107;  2, 579, 1019;  2, 581, 657;  2, 599, 1066;  2, 604, 677;  2, 610, 635;  2, 611, 1097;  2, 623, 971;  2, 629, 1138;  2, 641, 1119;  2, 645, 970;  2, 647, 1102;  2, 653, 778;  2, 663, 777;  2, 688, 785;  2, 699, 742;  2, 700, 1139;  2, 707, 969;  2, 735, 993;  2, 772, 813;  2, 779, 941;  2, 783, 1036;  2, 791, 873;  2, 797, 857;  2, 803, 1061;  2, 815, 1007;  2, 825, 861;  2, 863, 893;  2, 868, 1133;  2, 880, 975;  2, 891, 898;  2, 905, 939;  2, 963, 1126;  2, 983, 1144;  2, 1023, 1060;  2, 1031, 1067;  3, 4, 689;  3, 22, 141;  3, 33, 670;  3, 34, 340;  3, 41, 142;  3, 52, 329;  3, 64, 239;  3, 65, 316;  3, 70, 196;  3, 75, 574;  3, 76, 934;  3, 82, 346;  3, 107, 526;  3, 112, 251;  3, 118, 351;  3, 124, 718;  3, 129, 520;  3, 130, 382;  3, 136, 977;  3, 137, 1132;  3, 154, 381;  3, 161, 1012;  3, 178, 592;  3, 185, 862;  3, 190, 213;  3, 202, 455;  3, 214, 915;  3, 220, 735;  3, 226, 393;  3, 232, 693;  3, 275, 1078;  3, 280, 813;  3, 287, 406;  3, 310, 593;  3, 322, 1055;  3, 335, 772;  3, 353, 1024;  3, 358, 407;  3, 359, 742;  3, 370, 424;  3, 371, 964;  3, 376, 820;  3, 377, 460;  3, 418, 519;  3, 430, 1060;  3, 454, 658;  3, 467, 580;  3, 478, 538;  3, 479, 1090;  3, 503, 532;  3, 514, 556;  3, 568, 1138;  3, 569, 784;  3, 598, 898;  3, 628, 899;  3, 646, 682;  3, 652, 1061;  3, 665, 952;  3, 676, 779;  3, 724, 1072;  3, 748, 760;  3, 785, 940;  3, 821, 826;  3, 874, 1103;  3, 946, 1031;  3, 958, 1079;  3, 965, 1114;  4, 11, 382;  4, 34, 203;  4, 41, 760;  4, 47, 953;  4, 59, 131;  4, 71, 700;  4, 76, 395;  4, 83, 1061;  4, 113, 119;  4, 118, 659;  4, 142, 1007;  4, 196, 401;  4, 197, 647;  4, 221, 244;  4, 227, 280;  4, 245, 977;  4, 251, 821;  4, 340, 965;  4, 359, 749;  4, 364, 959;  4, 377, 917;  4, 485, 707;  4, 677, 815;  4, 743, 857;  4, 809, 1085;  4, 833, 875;  4, 923, 1133}.
The deficiency graph is connected and has girth 4.

{\boldmath $\adfPENT(3,569,19)$}, $d = 6$:
{0, 42, 390;  0, 84, 498;  0, 92, 378;  0, 400, 976;  0, 432, 964;  0, 518, 522;  0, 558, 1070;  0, 570, 946;  1, 89, 589;  1, 183, 759;  1, 195, 641;  1, 213, 661;  1, 287, 379;  1, 601, 1075;  1, 637, 727;  1, 769, 1117;  2, 44, 392;  2, 86, 500;  2, 402, 978;  2, 434, 966;  2, 520, 524;  2, 560, 1072;  2, 572, 948;  3, 91, 591;  3, 185, 761;  3, 197, 643;  3, 215, 663;  3, 289, 381;  3, 603, 1077;  3, 639, 729;  3, 771, 1119;  4, 46, 394;  4, 88, 502;  4, 404, 980;  4, 436, 968;  4, 522, 526;  4, 562, 1074;  4, 574, 950;  5, 93, 593;  5, 187, 763;  5, 199, 645;  5, 217, 665;  5, 605, 1079;  5, 641, 731;  5, 773, 1121;  42, 92, 780;  42, 378, 976;  42, 400, 522;  42, 498, 946;  42, 518, 964;  42, 558, 872;  42, 570, 1070;  44, 94, 782;  44, 380, 978;  44, 402, 524;  44, 500, 948;  44, 520, 966;  44, 560, 874;  44, 572, 1072;  46, 96, 784;  46, 382, 980;  46, 404, 526;  46, 502, 950;  46, 522, 968;  46, 562, 876;  46, 574, 1074;  84, 92, 976;  84, 378, 522;  84, 390, 558;  84, 400, 946;  84, 518, 1070;  84, 570, 872;  84, 780, 964;  86, 94, 978;  86, 380, 524;  86, 392, 560;  86, 402, 948;  86, 520, 1072;  86, 572, 874;  86, 782, 966;  88, 96, 980;  88, 382, 526;  88, 394, 562;  88, 404, 950;  88, 522, 1074;  88, 574, 876;  88, 784, 968;  89, 183, 727;  89, 195, 379;  89, 213, 287;  89, 601, 641;  89, 637, 1117;  89, 661, 1067;  89, 759, 1075;  89, 769, 781;  91, 185, 729;  91, 197, 381;  91, 215, 289;  91, 603, 643;  91, 639, 1119;  91, 761, 1077;  92, 390, 518;  92, 400, 432;  92, 498, 1070;  92, 522, 946;  92, 558, 570;  92, 872, 964;  93, 187, 731;  93, 199, 383;  93, 217, 291;  93, 605, 645;  93, 641, 1121;  93, 763, 1079;  94, 392, 520;  94, 402, 434;  94, 500, 1072;  94, 524, 948;  94, 560, 572;  94, 874, 966;  96, 394, 522;  96, 404, 436;  96, 526, 950;  183, 195, 661;  183, 213, 379;  183, 287, 1117;  183, 589, 769;  183, 601, 637;  183, 641, 781;  183, 1067, 1075;  185, 197, 663;  185, 215, 381;  185, 289, 1119;  185, 591, 771;  185, 603, 639;  185, 643, 783;  185, 1069, 1077;  187, 217, 383;  187, 291, 1121;  187, 605, 641;  187, 645, 785;  187, 1071, 1079;  195, 213, 727;  195, 287, 1067;  195, 589, 1117;  195, 637, 759;  195, 769, 1075;  197, 215, 729;  197, 591, 1119;  197, 639, 761;  197, 771, 1077;  199, 217, 731;  199, 593, 1121;  199, 641, 763;  199, 773, 1079;  213, 589, 1075;  213, 601, 769;  213, 637, 1067;  213, 641, 759;  213, 781, 1117;  215, 591, 1077;  215, 603, 771;  215, 639, 1069;  215, 643, 761;  215, 783, 1119;  217, 593, 1079;  217, 605, 773;  217, 641, 1071;  217, 645, 763;  217, 785, 1121;  287, 589, 661;  287, 601, 727;  287, 637, 769;  287, 641, 1075;  287, 759, 781;  289, 591, 663;  289, 603, 729;  289, 639, 771;  289, 643, 1077;  289, 761, 783;  291, 593, 665;  291, 605, 731;  291, 641, 773;  291, 645, 1079;  291, 763, 785;  378, 400, 780;  378, 432, 570;  378, 498, 872;  378, 518, 946;  378, 558, 964;  379, 589, 759;  379, 601, 661;  379, 637, 641;  379, 781, 1075;  379, 1067, 1117;  380, 402, 782;  380, 434, 572;  380, 500, 874;  380, 520, 948;  381, 591, 761;  381, 603, 663;  381, 639, 643;  381, 783, 1077;  381, 1069, 1119;  382, 404, 784;  382, 436, 574;  382, 502, 876;  382, 522, 950;  383, 593, 763;  383, 605, 665;  383, 641, 645;  383, 785, 1079;  383, 1071, 1121;  390, 400, 872;  390, 498, 522;  390, 946, 964;  392, 402, 874;  392, 500, 524;  392, 948, 966;  394, 404, 876;  394, 502, 526;  394, 950, 968;  400, 498, 570;  400, 518, 558;  400, 964, 1070;  402, 500, 572;  402, 520, 560;  402, 966, 1072;  404, 502, 574;  404, 522, 562;  404, 968, 1074;  432, 498, 518;  432, 522, 558;  432, 872, 946;  432, 976, 1070;  434, 500, 520;  434, 524, 560;  434, 874, 948;  434, 978, 1072;  436, 502, 522;  436, 526, 562;  436, 876, 950;  436, 980, 1074;  498, 558, 780;  498, 964, 976;  500, 560, 782;  502, 562, 784;  518, 570, 780;  518, 872, 976;  520, 572, 782;  520, 874, 978;  522, 570, 964;  522, 574, 784;  522, 780, 976;  522, 872, 1070;  522, 876, 980;  524, 572, 966;  524, 782, 978;  524, 874, 1072;  526, 574, 968;  526, 784, 980;  526, 876, 1074;  558, 946, 976;  560, 948, 978;  562, 950, 980;  589, 637, 781;  589, 641, 727;  591, 639, 783;  591, 643, 729;  593, 641, 785;  593, 645, 731;  601, 759, 1117;  603, 761, 1119;  605, 763, 1121;  637, 661, 1075;  639, 663, 1077;  641, 661, 1117;  641, 665, 1079;  641, 769, 1067;  643, 663, 1119;  643, 771, 1069;  645, 665, 1121;  645, 773, 1071;  661, 727, 781;  661, 759, 769;  663, 729, 783;  663, 761, 771;  665, 731, 785;  665, 763, 773;  727, 759, 1067;  729, 761, 1069;  731, 763, 1071;  780, 946, 1070;  782, 948, 1072;  784, 950, 1074;  0, 2, 923;  0, 3, 73;  0, 5, 858;  0, 6, 59;  0, 7, 592;  0, 9, 329;  0, 11, 375;  0, 13, 218;  0, 14, 657;  0, 15, 1142;  0, 16, 839;  0, 17, 865;  0, 19, 162;  0, 21, 318;  0, 23, 618;  0, 25, 233;  0, 26, 109;  0, 27, 103;  0, 28, 565;  0, 29, 906;  0, 31, 739;  0, 33, 115;  0, 34, 1095;  0, 35, 849;  0, 37, 1008;  0, 38, 695;  0, 39, 859;  0, 41, 43;  0, 44, 935;  0, 45, 652;  0, 46, 587;  0, 47, 697;  0, 49, 268;  0, 51, 386;  0, 55, 264;  0, 56, 877;  0, 57, 1013;  0, 58, 703;  0, 61, 723;  0, 62, 631;  0, 63, 384;  0, 64, 547;  0, 65, 452;  0, 67, 93;  0, 68, 871;  0, 69, 332;  0, 70, 413;  0, 71, 766;  0, 75, 1071;  0, 76, 1031;  0, 77, 1099;  0, 78, 411;  0, 79, 1063;  0, 80, 129;  0, 81, 450;  0, 82, 1133;  0, 83, 625;  0, 85, 496;  0, 87, 1055;  0, 91, 617;  0, 95, 101;  0, 96, 299;  0, 97, 994;  0, 99, 583;  0, 100, 447;  0, 102, 543;  0, 105, 527;  0, 107, 311;  0, 110, 303;  0, 111, 1121;  0, 112, 257;  0, 113, 530;  0, 114, 283;  0, 116, 331;  0, 117, 1130;  0, 119, 1096;  0, 121, 874;  0, 123, 604;  0, 125, 1132;  0, 127, 706;  0, 131, 189;  0, 133, 347;  0, 134, 741;  0, 135, 895;  0, 136, 275;  0, 137, 1107;  0, 139, 882;  0, 141, 1094;  0, 142, 649;  0, 143, 908;  0, 145, 385;  0, 146, 1059;  0, 147, 562;  0, 148, 753;  0, 149, 995;  0, 151, 918;  0, 152, 167;  0, 153, 221;  0, 154, 979;  0, 155, 551;  0, 156, 539;  0, 157, 1097;  0, 159, 746;  0, 160, 515;  0, 161, 1157;  0, 163, 279;  0, 164, 341;  0, 165, 312;  0, 171, 973;  0, 173, 801;  0, 174, 1089;  0, 175, 834;  0, 176, 335;  0, 177, 412;  0, 178, 651;  0, 179, 956;  0, 181, 220;  0, 185, 1046;  0, 186, 777;  0, 188, 265;  0, 190, 249;  0, 191, 1017;  0, 193, 555;  0, 197, 1137;  0, 199, 419;  0, 200, 241;  0, 201, 926;  0, 202, 566;  0, 204, 421;  0, 205, 1021;  0, 206, 327;  0, 207, 416;  0, 208, 803;  0, 209, 477;  0, 211, 397;  0, 214, 1155;  0, 215, 898;  0, 216, 1105;  0, 219, 887;  0, 223, 507;  0, 225, 467;  0, 226, 907;  0, 227, 444;  0, 228, 479;  0, 229, 360;  0, 230, 325;  0, 231, 1119;  0, 232, 1003;  0, 234, 571;  0, 235, 933;  0, 237, 633;  0, 238, 1135;  0, 239, 1153;  0, 242, 833;  0, 243, 534;  0, 244, 475;  0, 245, 879;  0, 246, 927;  0, 247, 897;  0, 248, 847;  0, 250, 705;  0, 253, 1141;  0, 255, 319;  0, 256, 459;  0, 259, 783;  0, 260, 371;  0, 261, 950;  0, 263, 309;  0, 267, 690;  0, 269, 550;  0, 270, 1091;  0, 271, 1057;  0, 272, 805;  0, 273, 809;  0, 277, 812;  0, 280, 361;  0, 284, 771;  0, 285, 835;  0, 288, 963;  0, 289, 829;  0, 291, 1076;  0, 292, 883;  0, 293, 662;  0, 295, 717;  0, 297, 648;  0, 301, 1049;  0, 307, 509;  0, 310, 521;  0, 313, 526;  0, 315, 1048;  0, 317, 663;  0, 320, 1051;  0, 321, 1019;  0, 322, 471;  0, 323, 653;  0, 326, 495;  0, 330, 925;  0, 334, 557;  0, 338, 955;  0, 339, 748;  0, 342, 701;  0, 343, 1088;  0, 344, 607;  0, 345, 493;  0, 346, 483;  0, 349, 885;  0, 351, 635;  0, 352, 917;  0, 353, 573;  0, 355, 851;  0, 356, 437;  0, 357, 561;  0, 362, 621;  0, 363, 674;  0, 364, 461;  0, 365, 1124;  0, 366, 1037;  0, 367, 1139;  0, 368, 431;  0, 369, 722;  0, 372, 1065;  0, 373, 502;  0, 377, 824;  0, 381, 666;  0, 382, 1083;  0, 387, 1022;  0, 389, 1144;  0, 392, 685;  0, 393, 503;  0, 395, 729;  0, 396, 893;  0, 398, 993;  0, 399, 627;  0, 401, 866;  0, 403, 1093;  0, 404, 827;  0, 405, 742;  0, 407, 763;  0, 408, 1123;  0, 409, 1061;  0, 410, 647;  0, 417, 1125;  0, 422, 711;  0, 423, 698;  0, 425, 689;  0, 427, 831;  0, 429, 902;  0, 433, 1101;  0, 435, 937;  0, 436, 1035;  0, 439, 944;  0, 443, 855;  0, 445, 802;  0, 449, 655;  0, 451, 1005;  0, 453, 920;  0, 455, 620;  0, 457, 776;  0, 460, 733;  0, 463, 1100;  0, 464, 755;  0, 465, 1112;  0, 469, 524;  0, 473, 575;  0, 481, 799;  0, 484, 553;  0, 485, 790;  0, 487, 629;  0, 489, 596;  0, 490, 939;  0, 491, 536;  0, 501, 953;  0, 504, 1149;  0, 505, 1042;  0, 506, 985;  0, 508, 975;  0, 511, 616;  0, 513, 593;  0, 517, 579;  0, 519, 795;  0, 523, 677;  0, 525, 650;  0, 529, 656;  0, 531, 687;  0, 533, 580;  0, 535, 1039;  0, 537, 745;  0, 538, 983;  0, 541, 814;  0, 542, 817;  0, 545, 1009;  0, 549, 1087;  0, 554, 1053;  0, 559, 971;  0, 567, 1147;  0, 569, 951;  0, 577, 899;  0, 578, 967;  0, 581, 1025;  0, 585, 857;  0, 597, 961;  0, 599, 841;  0, 603, 863;  0, 605, 643;  0, 608, 945;  0, 609, 1078;  0, 611, 832;  0, 613, 1016;  0, 615, 914;  0, 619, 765;  0, 622, 1007;  0, 623, 1131;  0, 628, 1115;  0, 632, 1109;  0, 634, 1043;  0, 639, 1023;  0, 659, 772;  0, 665, 916;  0, 667, 754;  0, 668, 891;  0, 669, 683;  0, 673, 929;  0, 679, 791;  0, 691, 970;  0, 694, 1001;  0, 699, 794;  0, 707, 980;  0, 709, 743;  0, 713, 1073;  0, 719, 1085;  0, 721, 998;  0, 725, 1082;  0, 731, 1024;  0, 736, 869;  0, 737, 938;  0, 747, 1047;  0, 749, 901;  0, 751, 932;  0, 757, 813;  0, 760, 1120;  0, 761, 1113;  0, 767, 1077;  0, 773, 826;  0, 775, 853;  0, 779, 1079;  0, 785, 910;  0, 787, 806;  0, 793, 928;  0, 796, 1127;  0, 797, 982;  0, 811, 1102;  0, 815, 1041;  0, 819, 965;  0, 820, 991;  0, 823, 836;  0, 825, 881;  0, 838, 1143;  0, 843, 1006;  0, 845, 1129;  0, 848, 919;  0, 875, 1010;  0, 878, 913;  0, 886, 1145;  0, 890, 1033;  0, 903, 1151;  0, 905, 1114;  0, 909, 1090;  0, 911, 1156;  0, 921, 1081;  0, 931, 959;  0, 940, 977;  0, 943, 1103;  0, 947, 1004;  0, 952, 1069;  0, 957, 968;  0, 958, 997;  0, 969, 1058;  0, 981, 987;  0, 989, 1012;  0, 999, 1111;  0, 1029, 1045;  1, 2, 344;  1, 3, 1154;  1, 4, 321;  1, 7, 736;  1, 8, 233;  1, 10, 383;  1, 15, 728;  1, 16, 172;  1, 17, 874;  1, 22, 807;  1, 26, 82;  1, 28, 411;  1, 32, 997;  1, 34, 277;  1, 38, 907;  1, 39, 232;  1, 44, 1133;  1, 45, 406;  1, 46, 939;  1, 47, 1012;  1, 50, 1009;  1, 52, 97;  1, 58, 777;  1, 59, 382;  1, 62, 911;  1, 64, 800;  1, 65, 1004;  1, 68, 574;  1, 69, 1052;  1, 70, 890;  1, 71, 226;  1, 74, 347;  1, 76, 503;  1, 77, 1142;  1, 80, 484;  1, 81, 376;  1, 83, 926;  1, 86, 956;  1, 94, 797;  1, 98, 716;  1, 100, 667;  1, 101, 830;  1, 103, 1005;  1, 104, 409;  1, 110, 367;  1, 111, 1028;  1, 112, 899;  1, 115, 778;  1, 116, 597;  1, 118, 124;  1, 122, 623;  1, 134, 974;  1, 135, 458;  1, 140, 1150;  1, 142, 1003;  1, 146, 1034;  1, 148, 871;  1, 149, 164;  1, 152, 694;  1, 153, 184;  1, 154, 177;  1, 158, 825;  1, 160, 424;  1, 165, 370;  1, 166, 833;  1, 170, 322;  1, 176, 567;  1, 178, 261;  1, 179, 682;  1, 188, 743;  1, 189, 610;  1, 190, 952;  1, 191, 620;  1, 196, 1108;  1, 201, 970;  1, 202, 1101;  1, 205, 422;  1, 207, 364;  1, 208, 304;  1, 212, 265;  1, 217, 532;  1, 219, 1148;  1, 224, 649;  1, 227, 860;  1, 229, 500;  1, 230, 764;  1, 231, 880;  1, 235, 818;  1, 236, 814;  1, 238, 530;  1, 239, 836;  1, 242, 614;  1, 243, 662;  1, 244, 859;  1, 247, 1066;  1, 248, 524;  1, 249, 482;  1, 250, 314;  1, 251, 604;  1, 254, 410;  1, 256, 799;  1, 260, 572;  1, 266, 369;  1, 268, 752;  1, 269, 958;  1, 273, 460;  1, 281, 596;  1, 284, 313;  1, 286, 386;  1, 293, 446;  1, 296, 909;  1, 302, 829;  1, 308, 345;  1, 310, 496;  1, 325, 1048;  1, 326, 850;  1, 327, 856;  1, 328, 453;  1, 332, 981;  1, 333, 998;  1, 335, 1040;  1, 340, 445;  1, 346, 353;  1, 350, 983;  1, 352, 635;  1, 362, 508;  1, 365, 986;  1, 368, 755;  1, 374, 393;  1, 380, 791;  1, 385, 992;  1, 387, 670;  1, 392, 581;  1, 394, 644;  1, 397, 1144;  1, 398, 790;  1, 416, 700;  1, 417, 464;  1, 418, 686;  1, 430, 921;  1, 434, 539;  1, 437, 1022;  1, 440, 556;  1, 442, 957;  1, 448, 494;  1, 452, 1006;  1, 454, 512;  1, 461, 1102;  1, 465, 578;  1, 466, 838;  1, 470, 1115;  1, 472, 488;  1, 476, 554;  1, 485, 566;  1, 490, 737;  1, 491, 674;  1, 502, 938;  1, 507, 878;  1, 520, 1103;  1, 526, 742;  1, 531, 550;  1, 535, 1132;  1, 543, 1126;  1, 548, 945;  1, 551, 758;  1, 562, 914;  1, 563, 908;  1, 593, 824;  1, 602, 932;  1, 605, 796;  1, 628, 1043;  1, 629, 950;  1, 633, 1156;  1, 634, 710;  1, 640, 879;  1, 646, 980;  1, 650, 712;  1, 656, 891;  1, 658, 767;  1, 668, 929;  1, 688, 1096;  1, 692, 868;  1, 698, 1025;  1, 704, 839;  1, 706, 927;  1, 707, 772;  1, 718, 788;  1, 722, 1017;  1, 723, 1000;  1, 740, 849;  1, 747, 832;  1, 760, 813;  1, 761, 920;  1, 766, 1131;  1, 776, 827;  1, 782, 867;  1, 784, 1125;  1, 794, 808;  1, 802, 1112;  1, 806, 940;  1, 812, 1072;  1, 815, 844;  1, 821, 968;  1, 826, 995;  1, 837, 982;  1, 842, 969;  1, 848, 892;  1, 862, 1097;  1, 887, 1058;  1, 904, 1136;  1, 910, 1023;  1, 915, 1036;  1, 922, 971;  1, 933, 934;  1, 959, 994;  1, 962, 1091;  1, 1010, 1079;  1, 1013, 1084;  1, 1018, 1113;  1, 1024, 1138;  1, 1030, 1049;  1, 1046, 1059;  1, 1060, 1094;  1, 1100, 1145;  2, 4, 423;  2, 5, 412;  2, 8, 987;  2, 11, 28;  2, 23, 365;  2, 27, 914;  2, 29, 971;  2, 35, 947;  2, 40, 407;  2, 41, 326;  2, 45, 986;  2, 57, 1000;  2, 59, 190;  2, 63, 101;  2, 69, 437;  2, 70, 347;  2, 75, 202;  2, 77, 814;  2, 81, 183;  2, 82, 855;  2, 89, 1131;  2, 93, 459;  2, 98, 1089;  2, 99, 965;  2, 104, 1101;  2, 112, 303;  2, 116, 1055;  2, 117, 441;  2, 119, 189;  2, 125, 819;  2, 135, 764;  2, 141, 759;  2, 143, 797;  2, 149, 598;  2, 152, 561;  2, 153, 692;  2, 155, 670;  2, 159, 974;  2, 164, 785;  2, 165, 371;  2, 166, 677;  2, 167, 969;  2, 173, 506;  2, 201, 922;  2, 203, 452;  2, 206, 755;  2, 208, 533;  2, 209, 970;  2, 213, 617;  2, 218, 881;  2, 220, 363;  2, 230, 477;  2, 232, 735;  2, 233, 813;  2, 236, 1151;  2, 242, 509;  2, 244, 675;  2, 245, 323;  2, 250, 1011;  2, 251, 362;  2, 254, 717;  2, 255, 429;  2, 257, 909;  2, 266, 945;  2, 267, 653;  2, 274, 1125;  2, 279, 904;  2, 281, 302;  2, 285, 1048;  2, 287, 1127;  2, 305, 454;  2, 309, 1006;  2, 311, 748;  2, 315, 446;  2, 321, 952;  2, 322, 1121;  2, 327, 388;  2, 328, 873;  2, 333, 585;  2, 334, 521;  2, 335, 1132;  2, 341, 579;  2, 346, 957;  2, 351, 368;  2, 353, 449;  2, 357, 658;  2, 358, 387;  2, 359, 767;  2, 364, 983;  2, 369, 1031;  2, 370, 549;  2, 377, 633;  2, 386, 1133;  2, 400, 833;  2, 405, 1095;  2, 410, 905;  2, 417, 1037;  2, 418, 461;  2, 419, 1023;  2, 443, 599;  2, 453, 1107;  2, 466, 1005;  2, 467, 711;  2, 473, 555;  2, 485, 725;  2, 491, 975;  2, 494, 1073;  2, 507, 650;  2, 515, 899;  2, 519, 935;  2, 532, 707;  2, 537, 940;  2, 538, 1059;  2, 543, 1018;  2, 545, 1024;  2, 567, 753;  2, 568, 843;  2, 569, 723;  2, 587, 821;  2, 605, 934;  2, 610, 1067;  2, 611, 933;  2, 622, 1061;  2, 627, 867;  2, 629, 879;  2, 634, 737;  2, 641, 958;  2, 652, 719;  2, 657, 891;  2, 664, 1041;  2, 671, 903;  2, 683, 963;  2, 689, 1157;  2, 699, 795;  2, 700, 705;  2, 701, 801;  2, 747, 1109;  2, 765, 916;  2, 777, 779;  2, 778, 869;  2, 809, 1085;  2, 831, 857;  2, 838, 1137;  2, 863, 1078;  2, 880, 1043;  2, 921, 1097;  2, 939, 946;  2, 982, 1155;  2, 1019, 1083;  3, 16, 1001;  3, 28, 363;  3, 40, 352;  3, 46, 447;  3, 47, 274;  3, 52, 219;  3, 58, 136;  3, 65, 664;  3, 70, 513;  3, 76, 394;  3, 81, 514;  3, 82, 333;  3, 106, 934;  3, 112, 401;  3, 117, 1042;  3, 118, 976;  3, 137, 214;  3, 142, 592;  3, 153, 819;  3, 155, 964;  3, 167, 748;  3, 172, 249;  3, 178, 544;  3, 202, 267;  3, 203, 568;  3, 226, 545;  3, 232, 1133;  3, 233, 1000;  3, 244, 315;  3, 250, 868;  3, 256, 291;  3, 292, 983;  3, 316, 850;  3, 321, 1150;  3, 322, 953;  3, 328, 375;  3, 329, 430;  3, 334, 599;  3, 341, 508;  3, 346, 1060;  3, 358, 611;  3, 370, 838;  3, 382, 753;  3, 395, 454;  3, 413, 808;  3, 448, 676;  3, 466, 947;  3, 490, 1025;  3, 496, 537;  3, 502, 1006;  3, 538, 1145;  3, 569, 952;  3, 574, 898;  3, 580, 754;  3, 598, 1127;  3, 604, 725;  3, 635, 904;  3, 659, 802;  3, 682, 916;  3, 880, 1030;  3, 892, 1019;  3, 1049, 1144;  3, 1054, 1085;  4, 17, 527;  4, 29, 863;  4, 59, 1060;  4, 65, 689;  4, 77, 965;  4, 83, 244;  4, 89, 317;  4, 119, 778;  4, 166, 617;  4, 203, 1073;  4, 208, 923;  4, 209, 496;  4, 221, 761;  4, 233, 899;  4, 245, 256;  4, 274, 821;  4, 275, 983;  4, 280, 575;  4, 292, 629;  4, 346, 1043;  4, 497, 749;  4, 509, 514;  4, 557, 743;  4, 653, 1025;  4, 677, 851;  4, 737, 887;  4, 815, 929}.
The deficiency graph is connected and has girth 4.

\adfAppGap

{\boldmath $\adfPENT(3,64,21)$}, $d = 30$:
{0, 1, 2;  0, 3, 99;  0, 4, 70;  0, 5, 71;  0, 21, 27;  0, 22, 146;  0, 23, 29;  0, 24, 117;  0, 25, 97;  0, 26, 98;  0, 28, 101;  0, 66, 81;  0, 67, 125;  0, 68, 83;  0, 69, 144;  0, 82, 119;  0, 96, 123;  0, 100, 145;  0, 118, 124;  1, 3, 145;  1, 4, 21;  1, 5, 22;  1, 23, 69;  1, 24, 124;  1, 25, 66;  1, 26, 67;  1, 27, 70;  1, 28, 144;  1, 29, 101;  1, 68, 96;  1, 71, 99;  1, 81, 97;  1, 82, 123;  1, 83, 119;  1, 98, 117;  1, 100, 146;  1, 118, 125;  2, 3, 28;  2, 4, 145;  2, 5, 99;  2, 21, 100;  2, 22, 101;  2, 23, 70;  2, 24, 82;  2, 25, 124;  2, 26, 117;  2, 27, 71;  2, 29, 144;  2, 66, 118;  2, 67, 119;  2, 68, 97;  2, 69, 146;  2, 81, 98;  2, 83, 123;  2, 96, 125;  3, 4, 5;  3, 21, 144;  3, 22, 100;  3, 23, 101;  3, 24, 105;  3, 25, 91;  3, 26, 92;  3, 27, 69;  3, 29, 71;  3, 51, 114;  3, 52, 109;  3, 53, 110;  3, 57, 102;  3, 58, 85;  3, 59, 86;  3, 60, 78;  3, 61, 119;  3, 62, 80;  3, 70, 146;  3, 72, 81;  3, 73, 122;  3, 74, 83;  3, 75, 87;  3, 76, 143;  3, 77, 89;  3, 79, 116;  3, 82, 104;  3, 84, 120;  3, 88, 107;  3, 90, 141;  3, 103, 121;  3, 106, 142;  3, 108, 117;  3, 115, 118;  4, 22, 27;  4, 23, 28;  4, 24, 142;  4, 25, 75;  4, 26, 76;  4, 29, 99;  4, 51, 118;  4, 52, 60;  4, 53, 61;  4, 57, 121;  4, 58, 72;  4, 59, 73;  4, 62, 108;  4, 69, 100;  4, 71, 146;  4, 74, 84;  4, 77, 90;  4, 78, 109;  4, 79, 117;  4, 80, 116;  4, 81, 85;  4, 82, 120;  4, 83, 104;  4, 86, 102;  4, 87, 91;  4, 88, 141;  4, 89, 107;  4, 92, 105;  4, 101, 144;  4, 103, 122;  4, 106, 143;  4, 110, 114;  4, 115, 119;  5, 21, 70;  5, 23, 144;  5, 24, 88;  5, 25, 142;  5, 26, 105;  5, 27, 145;  5, 28, 146;  5, 29, 100;  5, 51, 79;  5, 52, 118;  5, 53, 114;  5, 57, 82;  5, 58, 121;  5, 59, 102;  5, 60, 115;  5, 61, 116;  5, 62, 109;  5, 69, 101;  5, 72, 103;  5, 73, 104;  5, 74, 85;  5, 75, 106;  5, 76, 107;  5, 77, 91;  5, 78, 110;  5, 80, 117;  5, 81, 86;  5, 83, 120;  5, 84, 122;  5, 87, 92;  5, 89, 141;  5, 90, 143;  5, 108, 119;  6, 7, 8;  6, 9, 108;  6, 10, 76;  6, 11, 77;  6, 15, 102;  6, 16, 79;  6, 17, 80;  6, 30, 42;  6, 31, 122;  6, 32, 44;  6, 39, 54;  6, 40, 146;  6, 41, 56;  6, 43, 110;  6, 55, 104;  6, 75, 120;  6, 78, 144;  6, 103, 145;  6, 109, 121;  7, 9, 121;  7, 10, 30;  7, 11, 31;  7, 15, 145;  7, 16, 39;  7, 17, 40;  7, 32, 75;  7, 41, 78;  7, 42, 76;  7, 43, 120;  7, 44, 110;  7, 54, 79;  7, 55, 144;  7, 56, 104;  7, 77, 108;  7, 80, 102;  7, 103, 146;  7, 109, 122;  8, 9, 43;  8, 10, 121;  8, 11, 108;  8, 15, 55;  8, 16, 145;  8, 17, 102;  8, 30, 109;  8, 31, 110;  8, 32, 76;  8, 39, 103;  8, 40, 104;  8, 41, 79;  8, 42, 77;  8, 44, 120;  8, 54, 80;  8, 56, 144;  8, 75, 122;  8, 78, 146;  9, 10, 11;  9, 12, 36;  9, 13, 34;  9, 14, 35;  9, 15, 18;  9, 16, 125;  9, 17, 20;  9, 19, 38;  9, 30, 120;  9, 31, 109;  9, 32, 110;  9, 33, 123;  9, 37, 124;  9, 42, 75;  9, 44, 77;  9, 45, 138;  9, 46, 103;  9, 47, 104;  9, 51, 96;  9, 52, 143;  9, 53, 98;  9, 57, 132;  9, 58, 106;  9, 59, 107;  9, 66, 78;  9, 67, 149;  9, 68, 80;  9, 76, 122;  9, 79, 134;  9, 97, 140;  9, 102, 141;  9, 105, 147;  9, 133, 148;  9, 139, 142;  10, 12, 124;  10, 13, 15;  10, 14, 16;  10, 17, 33;  10, 18, 34;  10, 19, 123;  10, 20, 38;  10, 31, 42;  10, 32, 43;  10, 35, 36;  10, 37, 125;  10, 44, 108;  10, 45, 142;  10, 46, 51;  10, 47, 52;  10, 53, 102;  10, 57, 148;  10, 58, 66;  10, 59, 67;  10, 68, 105;  10, 75, 109;  10, 77, 122;  10, 78, 106;  10, 79, 147;  10, 80, 134;  10, 96, 103;  10, 97, 141;  10, 98, 140;  10, 104, 138;  10, 107, 132;  10, 110, 120;  10, 133, 149;  10, 139, 143;  11, 12, 19;  11, 13, 124;  11, 14, 36;  11, 15, 37;  11, 16, 38;  11, 17, 34;  11, 18, 35;  11, 20, 123;  11, 30, 76;  11, 32, 120;  11, 33, 125;  11, 42, 121;  11, 43, 122;  11, 44, 109;  11, 45, 97;  11, 46, 142;  11, 47, 138;  11, 51, 139;  11, 52, 140;  11, 53, 103;  11, 57, 79;  11, 58, 148;  11, 59, 132;  11, 66, 133;  11, 67, 134;  11, 68, 106;  11, 75, 110;  11, 78, 107;  11, 80, 147;  11, 96, 104;  11, 98, 141;  11, 102, 143;  11, 105, 149;  12, 13, 14;  12, 15, 123;  12, 16, 37;  12, 17, 38;  12, 18, 33;  12, 20, 35;  12, 34, 125;  13, 16, 18;  13, 17, 19;  13, 20, 36;  13, 33, 37;  13, 35, 125;  13, 38, 123;  14, 15, 34;  14, 17, 123;  14, 18, 124;  14, 19, 125;  14, 20, 37;  14, 33, 38;  15, 16, 17;  15, 19, 124;  15, 20, 125;  15, 33, 36;  15, 35, 38;  15, 39, 144;  15, 40, 103;  15, 41, 104;  15, 54, 78;  15, 56, 80;  15, 79, 146;  16, 19, 33;  16, 20, 34;  16, 35, 123;  16, 36, 124;  16, 40, 54;  16, 41, 55;  16, 56, 102;  16, 78, 103;  16, 80, 146;  16, 104, 144;  17, 18, 37;  17, 35, 124;  17, 36, 125;  17, 39, 79;  17, 41, 144;  17, 54, 145;  17, 55, 146;  17, 56, 103;  17, 78, 104;  18, 19, 20;  18, 36, 123;  18, 38, 125;  19, 34, 36;  19, 35, 37;  20, 33, 124;  21, 22, 23;  21, 28, 145;  21, 29, 146;  21, 69, 99;  21, 71, 101;  22, 28, 69;  22, 29, 70;  22, 71, 144;  22, 99, 145;  23, 27, 100;  23, 71, 145;  23, 99, 146;  24, 25, 26;  24, 66, 123;  24, 67, 118;  24, 68, 119;  24, 75, 141;  24, 76, 106;  24, 77, 107;  24, 81, 96;  24, 83, 98;  24, 87, 90;  24, 89, 92;  24, 91, 143;  24, 97, 125;  25, 67, 81;  25, 68, 82;  25, 76, 87;  25, 77, 88;  25, 83, 117;  25, 89, 105;  25, 90, 106;  25, 92, 143;  25, 96, 118;  25, 98, 125;  25, 107, 141;  25, 119, 123;  26, 66, 97;  26, 68, 123;  26, 75, 91;  26, 77, 141;  26, 81, 124;  26, 82, 125;  26, 83, 118;  26, 87, 142;  26, 88, 143;  26, 89, 106;  26, 90, 107;  26, 96, 119;  27, 28, 29;  27, 99, 144;  27, 101, 146;  28, 70, 99;  28, 71, 100;  29, 69, 145;  30, 43, 121;  30, 44, 122;  30, 75, 108;  30, 77, 110;  31, 43, 75;  31, 44, 76;  31, 77, 120;  31, 108, 121;  32, 42, 109;  32, 77, 121;  32, 108, 122;  34, 37, 123;  34, 38, 124;  39, 55, 145;  39, 56, 146;  39, 78, 102;  39, 80, 104;  40, 55, 78;  40, 56, 79;  40, 80, 144;  40, 102, 145;  41, 54, 103;  41, 80, 145;  41, 102, 146;  42, 108, 120;  42, 110, 122;  43, 76, 108;  43, 77, 109;  44, 75, 121;  45, 51, 141;  45, 52, 139;  45, 53, 140;  45, 96, 102;  45, 98, 104;  45, 103, 143;  46, 52, 96;  46, 53, 97;  46, 98, 138;  46, 102, 139;  46, 104, 143;  46, 140, 141;  47, 51, 103;  47, 53, 141;  47, 96, 142;  47, 97, 143;  47, 98, 139;  47, 102, 140;  51, 60, 117;  51, 61, 115;  51, 62, 116;  51, 78, 108;  51, 80, 110;  51, 97, 142;  51, 98, 143;  51, 102, 138;  51, 104, 140;  51, 109, 119;  52, 61, 78;  52, 62, 79;  52, 80, 114;  52, 97, 102;  52, 98, 103;  52, 104, 141;  52, 108, 115;  52, 110, 119;  52, 116, 117;  52, 138, 142;  53, 60, 109;  53, 62, 117;  53, 78, 118;  53, 79, 119;  53, 80, 115;  53, 96, 139;  53, 104, 142;  53, 108, 116;  53, 138, 143;  54, 102, 144;  54, 104, 146;  55, 79, 102;  55, 80, 103;  56, 78, 145;  57, 66, 147;  57, 67, 133;  57, 68, 134;  57, 72, 120;  57, 73, 103;  57, 74, 104;  57, 78, 105;  57, 80, 107;  57, 81, 84;  57, 83, 86;  57, 85, 122;  57, 106, 149;  58, 67, 78;  58, 68, 79;  58, 73, 81;  58, 74, 82;  58, 80, 132;  58, 83, 102;  58, 84, 103;  58, 86, 122;  58, 104, 120;  58, 105, 133;  58, 107, 149;  58, 134, 147;  59, 66, 106;  59, 68, 147;  59, 72, 85;  59, 74, 120;  59, 78, 148;  59, 79, 149;  59, 80, 133;  59, 81, 121;  59, 82, 122;  59, 83, 103;  59, 84, 104;  59, 105, 134;  60, 79, 118;  60, 80, 119;  60, 108, 114;  60, 110, 116;  61, 79, 108;  61, 80, 109;  61, 110, 117;  61, 114, 118;  62, 78, 115;  62, 110, 118;  62, 114, 119;  66, 79, 148;  66, 80, 149;  66, 82, 124;  66, 83, 125;  66, 96, 117;  66, 98, 119;  66, 105, 132;  66, 107, 134;  67, 79, 105;  67, 80, 106;  67, 82, 96;  67, 83, 97;  67, 98, 123;  67, 107, 147;  67, 117, 124;  67, 132, 148;  68, 78, 133;  68, 81, 118;  68, 98, 124;  68, 107, 148;  68, 117, 125;  68, 132, 149;  70, 100, 144;  70, 101, 145;  72, 82, 121;  72, 83, 122;  72, 84, 102;  72, 86, 104;  73, 82, 84;  73, 83, 85;  73, 86, 120;  73, 102, 121;  74, 81, 103;  74, 86, 121;  74, 102, 122;  75, 88, 142;  75, 89, 143;  75, 90, 105;  75, 92, 107;  76, 88, 90;  76, 89, 91;  76, 92, 141;  76, 105, 142;  76, 109, 120;  76, 110, 121;  77, 87, 106;  77, 92, 142;  77, 105, 143;  78, 114, 117;  78, 116, 119;  78, 132, 147;  78, 134, 149;  79, 103, 144;  79, 104, 145;  79, 106, 132;  79, 107, 133;  79, 109, 114;  79, 110, 115;  80, 105, 148;  80, 108, 118;  81, 102, 120;  81, 104, 122;  81, 117, 123;  81, 119, 125;  82, 85, 102;  82, 86, 103;  82, 97, 117;  82, 98, 118;  83, 84, 121;  83, 96, 124;  85, 103, 120;  85, 104, 121;  87, 105, 141;  87, 107, 143;  88, 91, 105;  88, 92, 106;  89, 90, 142;  91, 106, 141;  91, 107, 142;  96, 138, 141;  96, 140, 143;  97, 103, 138;  97, 104, 139;  97, 118, 123;  97, 119, 124;  98, 102, 142;  103, 139, 141;  103, 140, 142;  106, 133, 147;  106, 134, 148;  109, 115, 117;  109, 116, 118}.
The deficiency graph is connected and has girth 4.

\adfAppGap

{\boldmath $\adfPENT(3,1452,25)$}, $d = 2$:
{0, 454, 2442;  0, 632, 752;  0, 756, 1458;  0, 962, 2574;  0, 1048, 1958;  0, 1182, 2198;  0, 1346, 1376;  0, 1560, 2802;  0, 1624, 2898;  0, 1644, 2762;  0, 2180, 2790;  0, 2182, 2824;  1, 33, 1307;  1, 107, 129;  1, 141, 1371;  1, 169, 1473;  1, 357, 1749;  1, 489, 2175;  1, 733, 1585;  1, 749, 751;  1, 973, 2179;  1, 1287, 2299;  1, 1555, 1883;  1, 1969, 2477;  33, 107, 749;  33, 129, 1473;  33, 141, 489;  33, 169, 1585;  33, 357, 2179;  33, 733, 1371;  33, 751, 1749;  33, 973, 1287;  33, 1555, 1969;  33, 1883, 2299;  33, 2175, 2477;  107, 141, 1287;  107, 169, 733;  107, 357, 1883;  107, 489, 1749;  107, 751, 2179;  107, 973, 1307;  107, 1371, 2477;  107, 1473, 2175;  107, 1555, 1585;  107, 1969, 2299;  129, 141, 973;  129, 169, 2179;  129, 357, 1585;  129, 489, 1287;  129, 733, 1307;  129, 749, 2477;  129, 751, 1555;  129, 1371, 1969;  129, 1749, 2299;  129, 1883, 2175;  141, 169, 749;  141, 357, 1555;  141, 733, 1883;  141, 751, 1585;  141, 1307, 1969;  141, 1473, 1749;  141, 2175, 2299;  141, 2179, 2477;  169, 357, 2175;  169, 489, 2299;  169, 1287, 1749;  169, 1307, 1555;  169, 1883, 1969;  357, 489, 1371;  357, 733, 2477;  357, 749, 1473;  357, 751, 1969;  357, 973, 2299;  357, 1287, 1307;  454, 632, 1376;  454, 752, 2790;  454, 756, 2198;  454, 962, 1182;  454, 1458, 1624;  454, 1560, 2824;  454, 1644, 2574;  454, 1958, 2762;  454, 2180, 2898;  454, 2182, 2802;  489, 733, 2179;  489, 749, 1555;  489, 751, 1883;  489, 973, 1969;  489, 1307, 1585;  489, 1473, 2477;  632, 756, 1346;  632, 962, 1624;  632, 1048, 2762;  632, 1182, 2802;  632, 1458, 2180;  632, 1560, 2898;  632, 1644, 2442;  632, 1958, 2198;  632, 2182, 2790;  632, 2574, 2824;  733, 749, 2299;  733, 751, 973;  733, 1287, 2175;  733, 1473, 1555;  733, 1749, 1969;  749, 973, 1749;  749, 1287, 1883;  749, 1307, 2179;  749, 1585, 1969;  751, 1287, 2477;  751, 1307, 2175;  751, 1473, 2299;  752, 756, 962;  752, 1048, 2898;  752, 1182, 2182;  752, 1376, 2824;  752, 1458, 2762;  752, 1560, 2442;  752, 1624, 2802;  752, 1958, 2574;  752, 2180, 2198;  756, 1048, 2180;  756, 1182, 2762;  756, 1624, 1644;  756, 2442, 2802;  756, 2574, 2898;  756, 2790, 2824;  962, 1048, 1376;  962, 1346, 2180;  962, 1458, 2198;  962, 1560, 1644;  962, 1958, 2182;  962, 2442, 2898;  962, 2762, 2824;  962, 2790, 2802;  973, 1371, 1883;  973, 1473, 1585;  1048, 1182, 1624;  1048, 1458, 1644;  1048, 1560, 2790;  1048, 2182, 2574;  1048, 2198, 2802;  1048, 2442, 2824;  1182, 1346, 2790;  1182, 1376, 2574;  1182, 1458, 2898;  1182, 1560, 1958;  1182, 1644, 2824;  1182, 2180, 2442;  1287, 1371, 1555;  1287, 1473, 1969;  1307, 1371, 1473;  1307, 1749, 1883;  1307, 2299, 2477;  1346, 1458, 2182;  1346, 1560, 2198;  1346, 1624, 1958;  1346, 2442, 2762;  1346, 2574, 2802;  1346, 2824, 2898;  1371, 1585, 2299;  1371, 1749, 2179;  1376, 1458, 2790;  1376, 1560, 1624;  1376, 1644, 2898;  1376, 2182, 2442;  1376, 2198, 2762;  1458, 1560, 2574;  1458, 1958, 2442;  1458, 2802, 2824;  1473, 1883, 2179;  1555, 1749, 2477;  1555, 2179, 2299;  1585, 1749, 2175;  1624, 2180, 2574;  1624, 2182, 2762;  1624, 2198, 2824;  1624, 2442, 2790;  1644, 1958, 2802;  1644, 2180, 2182;  1644, 2198, 2790;  1958, 2180, 2824;  1958, 2790, 2898;  1969, 2175, 2179;  2182, 2198, 2898;  2198, 2442, 2574;  2574, 2762, 2790;  2762, 2802, 2898;  0, 3, 2279;  0, 5, 2788;  0, 6, 1533;  0, 7, 919;  0, 8, 1087;  0, 9, 2271;  0, 10, 2472;  0, 11, 1036;  0, 13, 2244;  0, 14, 690;  0, 15, 412;  0, 17, 77;  0, 19, 2874;  0, 21, 321;  0, 23, 1899;  0, 24, 495;  0, 25, 177;  0, 26, 2239;  0, 27, 2885;  0, 29, 597;  0, 31, 1222;  0, 35, 1695;  0, 36, 2859;  0, 37, 2649;  0, 38, 1099;  0, 39, 1878;  0, 41, 306;  0, 42, 1729;  0, 43, 1614;  0, 44, 111;  0, 45, 94;  0, 46, 2303;  0, 47, 1635;  0, 48, 861;  0, 49, 974;  0, 50, 1921;  0, 51, 143;  0, 52, 825;  0, 53, 526;  0, 54, 2621;  0, 55, 396;  0, 57, 2676;  0, 58, 1199;  0, 59, 1479;  0, 60, 2342;  0, 61, 944;  0, 63, 464;  0, 65, 1664;  0, 66, 507;  0, 68, 2765;  0, 69, 1384;  0, 70, 1797;  0, 71, 2482;  0, 72, 1387;  0, 73, 87;  0, 76, 232;  0, 78, 1774;  0, 79, 2703;  0, 80, 1547;  0, 81, 1395;  0, 83, 463;  0, 85, 1634;  0, 88, 1589;  0, 89, 1341;  0, 90, 575;  0, 91, 1431;  0, 92, 2263;  0, 93, 1041;  0, 95, 1040;  0, 97, 532;  0, 98, 2233;  0, 99, 678;  0, 100, 2757;  0, 101, 288;  0, 103, 176;  0, 104, 1071;  0, 105, 274;  0, 109, 1778;  0, 110, 695;  0, 113, 509;  0, 114, 1925;  0, 115, 2648;  0, 116, 402;  0, 117, 838;  0, 118, 1155;  0, 119, 521;  0, 121, 2211;  0, 122, 1075;  0, 123, 1820;  0, 125, 2456;  0, 126, 1817;  0, 127, 370;  0, 130, 281;  0, 131, 1570;  0, 133, 1205;  0, 135, 2029;  0, 137, 1940;  0, 138, 504;  0, 139, 2780;  0, 144, 2723;  0, 145, 1992;  0, 146, 1559;  0, 148, 2759;  0, 149, 976;  0, 152, 2907;  0, 153, 771;  0, 154, 1136;  0, 155, 1312;  0, 157, 565;  0, 158, 511;  0, 159, 275;  0, 160, 2545;  0, 161, 2921;  0, 162, 1675;  0, 163, 273;  0, 165, 1689;  0, 167, 196;  0, 170, 890;  0, 171, 2455;  0, 172, 853;  0, 173, 672;  0, 174, 617;  0, 175, 667;  0, 179, 1901;  0, 180, 2715;  0, 181, 1783;  0, 182, 1065;  0, 183, 2429;  0, 185, 2851;  0, 187, 2067;  0, 189, 2255;  0, 190, 2362;  0, 191, 680;  0, 192, 1393;  0, 193, 440;  0, 195, 2595;  0, 197, 1078;  0, 198, 848;  0, 199, 310;  0, 200, 2396;  0, 201, 2424;  0, 202, 2309;  0, 203, 2863;  0, 204, 1443;  0, 205, 2238;  0, 207, 2622;  0, 208, 2805;  0, 209, 2325;  0, 211, 725;  0, 212, 1093;  0, 213, 2467;  0, 215, 1054;  0, 217, 1272;  0, 218, 482;  0, 219, 337;  0, 221, 2605;  0, 223, 1013;  0, 225, 2194;  0, 226, 1709;  0, 227, 939;  0, 229, 1181;  0, 230, 1462;  0, 231, 1994;  0, 233, 2089;  0, 234, 573;  0, 235, 778;  0, 236, 819;  0, 237, 2302;  0, 238, 2877;  0, 239, 1673;  0, 241, 1002;  0, 242, 764;  0, 243, 1540;  0, 245, 1324;  0, 246, 677;  0, 247, 1114;  0, 249, 2521;  0, 251, 2237;  0, 252, 1587;  0, 253, 1095;  0, 255, 1830;  0, 256, 1835;  0, 257, 2138;  0, 258, 535;  0, 259, 1846;  0, 261, 2598;  0, 263, 1470;  0, 265, 411;  0, 266, 2449;  0, 267, 465;  0, 269, 2498;  0, 270, 1280;  0, 271, 2243;  0, 272, 1735;  0, 279, 1885;  0, 280, 2592;  0, 283, 1477;  0, 284, 2071;  0, 285, 2496;  0, 287, 935;  0, 290, 2187;  0, 291, 2815;  0, 293, 2160;  0, 294, 2119;  0, 295, 1553;  0, 297, 1577;  0, 299, 2100;  0, 300, 1237;  0, 301, 823;  0, 303, 2256;  0, 304, 2410;  0, 305, 1273;  0, 307, 1094;  0, 309, 1082;  0, 312, 1025;  0, 313, 2024;  0, 315, 661;  0, 316, 1887;  0, 317, 2432;  0, 318, 1278;  0, 319, 755;  0, 322, 1027;  0, 323, 1961;  0, 325, 1786;  0, 326, 1143;  0, 327, 2276;  0, 329, 2491;  0, 331, 2167;  0, 333, 516;  0, 335, 851;  0, 336, 2419;  0, 340, 1939;  0, 341, 727;  0, 342, 1085;  0, 343, 1937;  0, 344, 999;  0, 345, 1086;  0, 346, 2113;  0, 347, 1258;  0, 349, 646;  0, 350, 2236;  0, 351, 419;  0, 352, 857;  0, 354, 1151;  0, 355, 1685;  0, 358, 1495;  0, 359, 1518;  0, 361, 1131;  0, 362, 2065;  0, 363, 2231;  0, 364, 2869;  0, 365, 780;  0, 367, 1826;  0, 368, 2777;  0, 369, 1389;  0, 371, 1348;  0, 372, 1489;  0, 373, 1407;  0, 374, 1398;  0, 375, 2615;  0, 377, 1399;  0, 379, 1320;  0, 380, 2541;  0, 381, 1766;  0, 383, 1510;  0, 385, 2891;  0, 386, 1722;  0, 387, 2054;  0, 388, 2919;  0, 389, 1353;  0, 390, 2475;  0, 391, 2193;  0, 393, 2390;  0, 395, 1857;  0, 399, 839;  0, 400, 1417;  0, 401, 2813;  0, 403, 1357;  0, 404, 2663;  0, 405, 544;  0, 406, 2821;  0, 407, 1231;  0, 408, 1009;  0, 409, 450;  0, 413, 1282;  0, 415, 1475;  0, 417, 1737;  0, 418, 2725;  0, 420, 1843;  0, 421, 1733;  0, 422, 1663;  0, 423, 856;  0, 424, 1557;  0, 425, 2406;  0, 427, 446;  0, 428, 2066;  0, 429, 1402;  0, 433, 1531;  0, 435, 774;  0, 436, 1543;  0, 437, 1880;  0, 438, 1308;  0, 439, 1590;  0, 444, 2273;  0, 445, 849;  0, 447, 1505;  0, 449, 1019;  0, 451, 630;  0, 452, 1161;  0, 453, 1268;  0, 455, 1046;  0, 457, 1022;  0, 459, 1309;  0, 460, 959;  0, 461, 2389;  0, 466, 2847;  0, 467, 1197;  0, 469, 2797;  0, 470, 1655;  0, 472, 1319;  0, 473, 2018;  0, 475, 926;  0, 476, 2245;  0, 477, 2185;  0, 478, 1121;  0, 479, 1145;  0, 480, 2843;  0, 481, 802;  0, 483, 1978;  0, 487, 2691;  0, 490, 2927;  0, 491, 1572;  0, 492, 1781;  0, 493, 1511;  0, 494, 1281;  0, 497, 578;  0, 501, 969;  0, 502, 2333;  0, 503, 2382;  0, 510, 2241;  0, 513, 1602;  0, 514, 2081;  0, 517, 1030;  0, 518, 2274;  0, 523, 2853;  0, 525, 1250;  0, 527, 1534;  0, 528, 1459;  0, 529, 1535;  0, 530, 2321;  0, 531, 1226;  0, 533, 1935;  0, 537, 2701;  0, 539, 652;  0, 541, 783;  0, 542, 2707;  0, 543, 634;  0, 545, 1011;  0, 546, 1327;  0, 547, 1438;  0, 549, 785;  0, 551, 1192;  0, 552, 1847;  0, 553, 1523;  0, 555, 2207;  0, 557, 1789;  0, 559, 1042;  0, 561, 1628;  0, 562, 1251;  0, 563, 968;  0, 566, 2399;  0, 567, 1671;  0, 569, 2203;  0, 570, 1425;  0, 571, 2463;  0, 572, 1877;  0, 577, 1163;  0, 579, 1661;  0, 581, 2655;  0, 584, 1229;  0, 586, 2195;  0, 587, 1636;  0, 589, 2260;  0, 591, 1683;  0, 595, 1949;  0, 600, 2573;  0, 602, 2627;  0, 603, 2769;  0, 605, 1974;  0, 606, 2433;  0, 607, 2270;  0, 609, 1892;  0, 611, 2913;  0, 613, 1410;  0, 614, 2509;  0, 615, 1184;  0, 619, 793;  0, 621, 2151;  0, 623, 1842;  0, 625, 2739;  0, 627, 1889;  0, 629, 2705;  0, 631, 1659;  0, 633, 2099;  0, 635, 2275;  0, 636, 1581;  0, 637, 683;  0, 639, 2899;  0, 640, 1751;  0, 641, 2077;  0, 647, 1896;  0, 649, 2353;  0, 651, 1761;  0, 653, 809;  0, 657, 1603;  0, 658, 1782;  0, 659, 1630;  0, 663, 1372;  0, 665, 854;  0, 666, 1667;  0, 668, 1415;  0, 669, 2317;  0, 671, 1003;  0, 673, 1860;  0, 675, 1998;  0, 679, 916;  0, 684, 1793;  0, 685, 1362;  0, 687, 1813;  0, 688, 2327;  0, 691, 1253;  0, 693, 1943;  0, 696, 2729;  0, 697, 827;  0, 698, 2164;  0, 703, 1064;  0, 704, 2717;  0, 708, 1537;  0, 710, 2003;  0, 711, 1255;  0, 712, 2355;  0, 715, 1736;  0, 717, 1503;  0, 721, 918;  0, 723, 1221;  0, 726, 2619;  0, 729, 2375;  0, 730, 2053;  0, 731, 1493;  0, 735, 2036;  0, 737, 1284;  0, 739, 2215;  0, 741, 1542;  0, 742, 1616;  0, 745, 947;  0, 746, 1802;  0, 753, 2176;  0, 757, 1692;  0, 759, 1976;  0, 761, 2059;  0, 762, 2695;  0, 763, 1697;  0, 765, 1058;  0, 767, 1549;  0, 768, 1719;  0, 769, 1562;  0, 775, 993;  0, 777, 2523;  0, 779, 2139;  0, 782, 2909;  0, 786, 1981;  0, 789, 1434;  0, 790, 2080;  0, 791, 2735;  0, 794, 2771;  0, 795, 2088;  0, 796, 1869;  0, 799, 1432;  0, 801, 2109;  0, 803, 1057;  0, 805, 1950;  0, 807, 2369;  0, 811, 1649;  0, 812, 1753;  0, 814, 2745;  0, 816, 2427;  0, 820, 2827;  0, 821, 2875;  0, 828, 2587;  0, 831, 2675;  0, 833, 1924;  0, 835, 1867;  0, 837, 1944;  0, 840, 2811;  0, 841, 1347;  0, 843, 1391;  0, 845, 1374;  0, 846, 2131;  0, 859, 2565;  0, 860, 2861;  0, 863, 1140;  0, 867, 1679;  0, 869, 1248;  0, 871, 1757;  0, 873, 1721;  0, 875, 2141;  0, 877, 2763;  0, 878, 2407;  0, 879, 1262;  0, 885, 1596;  0, 886, 1832;  0, 887, 2261;  0, 889, 2123;  0, 891, 917;  0, 893, 2699;  0, 895, 1076;  0, 898, 1853;  0, 899, 2331;  0, 900, 2553;  0, 901, 1160;  0, 902, 2767;  0, 903, 1859;  0, 904, 1966;  0, 905, 2323;  0, 907, 1762;  0, 908, 2073;  0, 909, 1113;  0, 911, 1349;  0, 913, 1565;  0, 914, 2371;  0, 915, 1758;  0, 921, 2825;  0, 923, 2079;  0, 925, 1047;  0, 927, 1805;  0, 929, 2603;  0, 934, 2559;  0, 943, 1509;  0, 948, 2227;  0, 957, 1369;  0, 958, 2905;  0, 963, 2173;  0, 965, 1945;  0, 966, 2035;  0, 970, 2539;  0, 971, 2801;  0, 975, 2117;  0, 978, 2561;  0, 979, 1955;  0, 983, 1788;  0, 985, 2397;  0, 987, 1576;  0, 989, 1600;  0, 991, 2483;  0, 995, 1705;  0, 997, 1904;  0, 1005, 1377;  0, 1007, 2121;  0, 1008, 2803;  0, 1015, 1461;  0, 1018, 2197;  0, 1020, 2923;  0, 1021, 1508;  0, 1023, 1454;  0, 1028, 2727;  0, 1029, 1494;  0, 1031, 1383;  0, 1032, 2773;  0, 1033, 1621;  0, 1035, 1837;  0, 1039, 1256;  0, 1043, 2217;  0, 1045, 2055;  0, 1053, 1427;  0, 1055, 1595;  0, 1059, 1654;  0, 1060, 2489;  0, 1067, 2681;  0, 1072, 2677;  0, 1074, 2775;  0, 1077, 2043;  0, 1081, 1105;  0, 1089, 2145;  0, 1092, 2667;  0, 1097, 2345;  0, 1101, 1879;  0, 1103, 1224;  0, 1115, 1252;  0, 1119, 1913;  0, 1122, 2295;  0, 1123, 2471;  0, 1125, 2513;  0, 1126, 2795;  0, 1135, 2469;  0, 1139, 2041;  0, 1147, 2287;  0, 1149, 2887;  0, 1153, 2721;  0, 1157, 2027;  0, 1159, 2453;  0, 1162, 2487;  0, 1169, 2903;  0, 1171, 1734;  0, 1175, 2585;  0, 1177, 1499;  0, 1183, 1409;  0, 1187, 1453;  0, 1189, 1401;  0, 1191, 2883;  0, 1193, 2329;  0, 1203, 1351;  0, 1207, 2249;  0, 1209, 2871;  0, 1210, 2643;  0, 1211, 1212;  0, 1213, 1539;  0, 1215, 2267;  0, 1217, 1512;  0, 1223, 1298;  0, 1225, 1524;  0, 1227, 2101;  0, 1235, 2403;  0, 1243, 1963;  0, 1245, 2159;  0, 1246, 2787;  0, 1247, 1861;  0, 1249, 1983;  0, 1257, 1645;  0, 1265, 2093;  0, 1269, 1965;  0, 1270, 2571;  0, 1271, 1329;  0, 1275, 1365;  0, 1277, 1400;  0, 1283, 1588;  0, 1291, 1573;  0, 1297, 2719;  0, 1299, 2651;  0, 1303, 2425;  0, 1311, 2459;  0, 1313, 2473;  0, 1317, 1352;  0, 1321, 1755;  0, 1333, 1627;  0, 1337, 2583;  0, 1339, 2037;  0, 1345, 1951;  0, 1356, 2731;  0, 1361, 2293;  0, 1363, 1439;  0, 1367, 2451;  0, 1373, 1677;  0, 1379, 2225;  0, 1397, 1873;  0, 1403, 1745;  0, 1405, 2549;  0, 1411, 2301;  0, 1419, 2111;  0, 1421, 2057;  0, 1437, 1941;  0, 1441, 1693;  0, 1445, 2191;  0, 1447, 1927;  0, 1449, 2845;  0, 1451, 1911;  0, 1455, 1613;  0, 1465, 1801;  0, 1481, 2645;  0, 1485, 1775;  0, 1497, 2575;  0, 1515, 2555;  0, 1517, 2915;  0, 1519, 1967;  0, 1521, 2045;  0, 1525, 1975;  0, 1551, 2445;  0, 1563, 2461;  0, 1591, 2277;  0, 1593, 2893;  0, 1597, 2377;  0, 1601, 2359;  0, 1617, 2391;  0, 1619, 1987;  0, 1623, 1631;  0, 1641, 2711;  0, 1651, 1777;  0, 1657, 2021;  0, 1665, 2673;  0, 1707, 2379;  0, 1715, 1763;  0, 1717, 2395;  0, 1725, 2661;  0, 1747, 2229;  0, 1765, 2841;  0, 1799, 2153;  0, 1807, 2837;  0, 1809, 2503;  0, 1815, 2789;  0, 1819, 2493;  0, 1821, 1915;  0, 1845, 2641;  0, 1855, 1993;  0, 1891, 2829;  0, 1907, 2265;  0, 1917, 2709;  0, 1919, 2835;  0, 1929, 2617;  0, 1979, 2155;  0, 1991, 2647;  0, 1997, 2897;  0, 1999, 2343;  0, 2009, 2051;  0, 2011, 2831;  0, 2015, 2599;  0, 2017, 2435;  0, 2031, 2097;  0, 2069, 2349;  0, 2095, 2199;  0, 2105, 2607;  0, 2147, 2537;  0, 2149, 2421;  0, 2163, 2867;  0, 2177, 2423;  0, 2181, 2281;  0, 2201, 2779;  0, 2223, 2373;  0, 2247, 2613;  0, 2251, 2485;  0, 2269, 2577;  0, 2283, 2623;  0, 2291, 2601;  0, 2305, 2799;  0, 2311, 2543;  0, 2313, 2629;  0, 2315, 2393;  0, 2337, 2481;  0, 2347, 2501;  0, 2357, 2785;  0, 2405, 2411;  0, 2413, 2669;  0, 2439, 2519;  0, 2507, 2737;  0, 2511, 2873;  0, 2517, 2689;  0, 2527, 2563;  0, 2557, 2753;  0, 2581, 2679;  0, 2593, 2879;  0, 2659, 2917;  0, 2685, 2865;  0, 2781, 2833;  0, 2783, 2925;  1, 11, 2281;  1, 39, 2619;  1, 45, 573;  1, 51, 2461;  1, 55, 533;  1, 57, 1359;  1, 71, 2511;  1, 89, 2509;  1, 115, 2291;  1, 161, 445;  1, 163, 1941;  1, 183, 2561;  1, 191, 1575;  1, 193, 401;  1, 201, 2657;  1, 239, 919;  1, 289, 1759;  1, 339, 1047;  1, 433, 2459;  1, 453, 2405;  1, 459, 1089;  1, 465, 2005;  1, 535, 2189;  1, 543, 2101;  1, 635, 1659;  1, 737, 1719;  1, 861, 1769;  1, 907, 1867}.
The deficiency graph is connected and has girth 4.

{\boldmath $\adfPENT(3,1454,25)$}, $d = 6$:
{0, 168, 230;  0, 298, 2296;  0, 408, 934;  0, 468, 1546;  0, 488, 1872;  0, 800, 2180;  0, 818, 1188;  0, 870, 2496;  0, 1042, 2332;  0, 1920, 2506;  0, 1952, 2450;  1, 429, 603;  1, 439, 2065;  1, 485, 1747;  1, 639, 2705;  1, 755, 2527;  1, 983, 2001;  1, 1015, 2467;  1, 1063, 2447;  1, 1389, 2637;  1, 1893, 2135;  1, 2117, 2767;  2, 170, 232;  2, 300, 2298;  2, 410, 936;  2, 470, 1548;  2, 490, 1874;  2, 802, 2182;  2, 820, 1190;  2, 872, 2498;  2, 1044, 2334;  2, 1922, 2508;  2, 1954, 2452;  3, 431, 605;  3, 441, 2067;  3, 487, 1749;  3, 641, 2707;  3, 757, 2529;  3, 985, 2003;  3, 1017, 2469;  3, 1065, 2449;  3, 1391, 2639;  3, 1895, 2137;  3, 2119, 2769;  4, 172, 234;  4, 302, 2300;  4, 412, 938;  4, 472, 1550;  4, 492, 1876;  4, 804, 2184;  4, 822, 1192;  4, 874, 2500;  4, 1046, 2336;  4, 1924, 2510;  4, 1956, 2454;  5, 433, 607;  5, 443, 2069;  5, 489, 1751;  5, 643, 2709;  5, 759, 2531;  5, 987, 2005;  5, 1019, 2471;  5, 1067, 2451;  5, 1393, 2641;  5, 1897, 2139;  5, 2121, 2771;  168, 298, 488;  168, 408, 1952;  168, 468, 1042;  168, 800, 1626;  168, 818, 2450;  168, 870, 2180;  168, 934, 1546;  168, 1188, 2506;  168, 1872, 2064;  168, 1920, 2332;  168, 2296, 2496;  170, 300, 490;  170, 410, 1954;  170, 470, 1044;  170, 802, 1628;  170, 820, 2452;  170, 872, 2182;  170, 936, 1548;  170, 1190, 2508;  170, 1874, 2066;  170, 1922, 2334;  170, 2298, 2498;  172, 302, 492;  172, 412, 1956;  172, 472, 1046;  172, 804, 1630;  172, 822, 2454;  172, 874, 2184;  172, 938, 1550;  172, 1192, 2510;  172, 1876, 2068;  172, 1924, 2336;  172, 2300, 2500;  230, 298, 1546;  230, 408, 870;  230, 468, 2296;  230, 488, 800;  230, 818, 1042;  230, 934, 1952;  230, 1188, 2064;  230, 1626, 1872;  230, 1920, 2496;  230, 2180, 2450;  230, 2332, 2506;  232, 300, 1548;  232, 410, 872;  232, 470, 2298;  232, 490, 802;  232, 820, 1044;  232, 936, 1954;  232, 1190, 2066;  232, 1628, 1874;  232, 1922, 2498;  232, 2182, 2452;  232, 2334, 2508;  234, 302, 1550;  234, 412, 874;  234, 472, 2300;  234, 492, 804;  234, 822, 1046;  234, 938, 1956;  234, 1192, 2068;  234, 1630, 1876;  234, 1924, 2500;  234, 2184, 2454;  234, 2336, 2510;  298, 408, 818;  298, 468, 2180;  298, 800, 934;  298, 870, 1920;  298, 1042, 2450;  298, 1188, 1872;  298, 1626, 2332;  298, 1952, 2496;  298, 2064, 2506;  300, 410, 820;  300, 470, 2182;  300, 802, 936;  300, 872, 1922;  300, 1044, 2452;  300, 1190, 1874;  300, 1628, 2334;  300, 1954, 2498;  300, 2066, 2508;  302, 412, 822;  302, 472, 2184;  302, 804, 938;  302, 874, 1924;  302, 1046, 2454;  302, 1192, 1876;  302, 1630, 2336;  302, 1956, 2500;  302, 2068, 2510;  408, 468, 1920;  408, 488, 1042;  408, 800, 2506;  408, 1188, 2332;  408, 1546, 1872;  408, 1626, 2450;  408, 2064, 2296;  408, 2180, 2496;  410, 470, 1922;  410, 490, 1044;  410, 802, 2508;  410, 1190, 2334;  410, 1628, 2452;  410, 2066, 2298;  410, 2182, 2498;  412, 472, 1924;  412, 804, 2510;  412, 1192, 2336;  412, 1630, 2454;  412, 2068, 2300;  412, 2184, 2500;  429, 439, 2705;  429, 485, 2767;  429, 639, 2447;  429, 755, 1893;  429, 871, 2135;  429, 983, 1063;  429, 1015, 1747;  429, 1309, 2065;  429, 1389, 2467;  429, 2001, 2117;  429, 2527, 2637;  431, 441, 2707;  431, 487, 2769;  431, 641, 2449;  431, 873, 2137;  431, 1017, 1749;  431, 1311, 2067;  431, 1391, 2469;  431, 2529, 2639;  433, 443, 2709;  433, 489, 2771;  433, 643, 2451;  433, 875, 2139;  433, 1019, 1751;  433, 1313, 2069;  433, 1393, 2471;  433, 2531, 2641;  439, 485, 1893;  439, 603, 2117;  439, 639, 2467;  439, 755, 983;  439, 871, 2527;  439, 1015, 2637;  439, 1063, 2001;  439, 1389, 2135;  439, 2447, 2767;  441, 487, 1895;  441, 605, 2119;  441, 641, 2469;  441, 757, 985;  441, 873, 2529;  441, 1017, 2639;  441, 1065, 2003;  441, 1391, 2137;  441, 2449, 2769;  443, 489, 1897;  443, 607, 2121;  443, 643, 2471;  443, 759, 987;  443, 875, 2531;  443, 1019, 2641;  443, 1067, 2005;  443, 1393, 2139;  443, 2451, 2771;  468, 488, 2450;  468, 800, 1188;  468, 818, 1872;  468, 870, 2332;  468, 934, 1626;  468, 1952, 2064;  468, 2496, 2506;  470, 490, 2452;  470, 802, 1190;  470, 820, 1874;  470, 872, 2334;  470, 936, 1628;  470, 1954, 2066;  470, 2498, 2508;  472, 492, 2454;  472, 804, 1192;  472, 822, 1876;  472, 874, 2336;  472, 938, 1630;  472, 1956, 2068;  472, 2500, 2510;  485, 603, 2065;  485, 639, 2527;  485, 755, 1015;  485, 871, 1063;  485, 983, 1389;  485, 1309, 2117;  485, 2001, 2637;  485, 2135, 2447;  485, 2467, 2705;  487, 605, 2067;  487, 641, 2529;  487, 757, 1017;  487, 873, 1065;  487, 985, 1391;  487, 1311, 2119;  487, 2003, 2639;  487, 2137, 2449;  487, 2469, 2707;  488, 818, 2296;  488, 870, 1188;  488, 934, 1920;  488, 1546, 2506;  488, 1626, 1952;  488, 2064, 2496;  488, 2180, 2332;  489, 607, 2069;  489, 643, 2531;  489, 759, 1019;  489, 875, 1067;  489, 987, 1393;  489, 1313, 2121;  489, 2005, 2641;  489, 2139, 2451;  489, 2471, 2709;  490, 820, 2298;  490, 872, 1190;  490, 936, 1922;  490, 1548, 2508;  490, 2066, 2498;  490, 2182, 2334;  492, 822, 2300;  492, 874, 1192;  492, 938, 1924;  492, 1550, 2510;  492, 2068, 2500;  492, 2184, 2336;  603, 639, 1063;  603, 755, 2135;  603, 871, 2767;  603, 983, 1015;  603, 1309, 2467;  603, 1389, 1893;  603, 1747, 2001;  603, 2447, 2637;  603, 2527, 2705;  605, 641, 1065;  605, 757, 2137;  605, 873, 2769;  605, 985, 1017;  605, 1311, 2469;  605, 1391, 1895;  605, 1749, 2003;  605, 2449, 2639;  605, 2529, 2707;  607, 643, 1067;  607, 759, 2139;  607, 875, 2771;  607, 987, 1019;  607, 1313, 2471;  607, 1393, 1897;  607, 1751, 2005;  607, 2451, 2641;  607, 2531, 2709;  639, 871, 983;  639, 1015, 1309;  639, 1389, 2117;  639, 1747, 2767;  639, 1893, 2065;  639, 2135, 2637;  641, 873, 985;  641, 1017, 1311;  641, 1391, 2119;  641, 1749, 2769;  641, 1895, 2067;  641, 2137, 2639;  643, 875, 987;  643, 1019, 1313;  643, 1393, 2121;  643, 1751, 2771;  643, 1897, 2069;  643, 2139, 2641;  755, 871, 2117;  755, 1063, 2705;  755, 1309, 1389;  755, 1747, 2447;  755, 2065, 2767;  755, 2467, 2637;  757, 873, 2119;  757, 1065, 2707;  757, 1311, 1391;  757, 1749, 2449;  757, 2067, 2769;  757, 2469, 2639;  759, 1067, 2709;  759, 1751, 2451;  759, 2069, 2771;  759, 2471, 2641;  800, 818, 2496;  800, 870, 1042;  800, 1546, 1952;  800, 1872, 2296;  800, 1920, 2450;  800, 2064, 2332;  802, 820, 2498;  802, 872, 1044;  802, 1548, 1954;  802, 1874, 2298;  802, 1922, 2452;  802, 2066, 2334;  804, 822, 2500;  804, 874, 1046;  804, 1550, 1956;  804, 1876, 2300;  804, 1924, 2454;  804, 2068, 2336;  818, 870, 1952;  818, 934, 2506;  818, 1546, 2332;  818, 1626, 1920;  818, 2064, 2180;  820, 872, 1954;  820, 936, 2508;  820, 1548, 2334;  820, 1628, 1922;  822, 874, 1956;  822, 1550, 2336;  822, 1630, 1924;  870, 934, 1872;  870, 1546, 2296;  870, 1626, 2506;  870, 2064, 2450;  871, 1015, 2447;  871, 1389, 1747;  871, 1893, 2467;  871, 2001, 2705;  871, 2065, 2637;  872, 936, 1874;  872, 1548, 2298;  872, 1628, 2508;  872, 2066, 2452;  873, 1017, 2449;  873, 1391, 1749;  873, 1895, 2469;  873, 2003, 2707;  873, 2067, 2639;  874, 938, 1876;  874, 1550, 2300;  874, 1630, 2510;  874, 2068, 2454;  875, 1019, 2451;  875, 1393, 1751;  875, 1897, 2471;  875, 2005, 2709;  875, 2069, 2641;  934, 1042, 2064;  934, 1188, 2450;  934, 2332, 2496;  936, 1044, 2066;  936, 1190, 2452;  936, 2334, 2498;  938, 1046, 2068;  938, 1192, 2454;  938, 2336, 2500;  983, 1309, 2447;  983, 1747, 1893;  983, 2065, 2705;  983, 2117, 2637;  983, 2135, 2527;  983, 2467, 2767;  985, 1749, 1895;  985, 2067, 2707;  985, 2119, 2639;  985, 2137, 2529;  985, 2469, 2769;  987, 1751, 1897;  987, 2069, 2709;  987, 2121, 2641;  987, 2139, 2531;  987, 2471, 2771;  1015, 1063, 1893;  1015, 1389, 2065;  1015, 2001, 2135;  1015, 2117, 2527;  1015, 2705, 2767;  1017, 1065, 1895;  1017, 1391, 2067;  1017, 2003, 2137;  1017, 2119, 2529;  1017, 2707, 2769;  1019, 1067, 1897;  1019, 1393, 2069;  1019, 2005, 2139;  1019, 2121, 2531;  1019, 2709, 2771;  1042, 1188, 1546;  1042, 1626, 2296;  1042, 1872, 2496;  1042, 1920, 1952;  1042, 2180, 2506;  1044, 1190, 1548;  1044, 1628, 2298;  1044, 1874, 2498;  1044, 1922, 1954;  1046, 1192, 1550;  1046, 1630, 2300;  1046, 1876, 2500;  1046, 1924, 1956;  1063, 1309, 2637;  1063, 1389, 2527;  1063, 1747, 2117;  1063, 2065, 2467;  1063, 2135, 2767;  1065, 1311, 2639;  1065, 1749, 2119;  1065, 2067, 2469;  1065, 2137, 2769;  1067, 1313, 2641;  1067, 1751, 2121;  1067, 2069, 2471;  1067, 2139, 2771;  1188, 1920, 2296;  1188, 1952, 2180;  1190, 1922, 2298;  1190, 1954, 2182;  1192, 1924, 2300;  1192, 1956, 2184;  1309, 1893, 2001;  1309, 2135, 2705;  1309, 2527, 2767;  1311, 1895, 2003;  1311, 2137, 2707;  1311, 2529, 2769;  1313, 1897, 2005;  1313, 2139, 2709;  1313, 2531, 2771;  1389, 2001, 2767;  1389, 2447, 2705;  1391, 2003, 2769;  1391, 2449, 2707;  1393, 2005, 2771;  1393, 2451, 2709;  1546, 1626, 2180;  1546, 1920, 2064;  1546, 2450, 2496;  1548, 1922, 2066;  1548, 2452, 2498;  1550, 1924, 2068;  1550, 2454, 2500;  1747, 2065, 2527;  1747, 2135, 2467;  1747, 2637, 2705;  1749, 2067, 2529;  1749, 2137, 2469;  1749, 2639, 2707;  1751, 2069, 2531;  1751, 2139, 2471;  1751, 2641, 2709;  1872, 1920, 2180;  1872, 2332, 2450;  1874, 1922, 2182;  1874, 2334, 2452;  1876, 1924, 2184;  1876, 2336, 2454;  1893, 2117, 2705;  1893, 2637, 2767;  1895, 2119, 2707;  1895, 2639, 2769;  1897, 2121, 2709;  1897, 2641, 2771;  1952, 2296, 2332;  1954, 2298, 2334;  1956, 2300, 2336;  2001, 2065, 2447;  2001, 2467, 2527;  2003, 2067, 2449;  2003, 2469, 2529;  2005, 2069, 2451;  2005, 2471, 2531;  2065, 2117, 2135;  2067, 2119, 2137;  2069, 2121, 2139;  2117, 2447, 2467;  2119, 2449, 2469;  2121, 2451, 2471;  2296, 2450, 2506;  2298, 2452, 2508;  2300, 2454, 2510;  0, 2, 2145;  0, 3, 1881;  0, 4, 79;  0, 5, 2163;  0, 6, 2507;  0, 7, 1647;  0, 8, 781;  0, 9, 1733;  0, 11, 772;  0, 12, 1145;  0, 13, 2720;  0, 14, 2442;  0, 15, 2926;  0, 16, 1751;  0, 17, 113;  0, 19, 1775;  0, 21, 1921;  0, 22, 1287;  0, 23, 2417;  0, 24, 1727;  0, 25, 2667;  0, 26, 2785;  0, 27, 2823;  0, 28, 1577;  0, 29, 1305;  0, 30, 635;  0, 31, 1980;  0, 33, 2059;  0, 34, 1693;  0, 35, 1401;  0, 37, 141;  0, 38, 179;  0, 39, 2699;  0, 40, 1651;  0, 41, 1927;  0, 42, 2789;  0, 43, 175;  0, 44, 1151;  0, 45, 2654;  0, 47, 213;  0, 49, 2353;  0, 50, 256;  0, 51, 2827;  0, 53, 1367;  0, 54, 1738;  0, 55, 1689;  0, 57, 2843;  0, 58, 2247;  0, 59, 1840;  0, 61, 1059;  0, 63, 2203;  0, 65, 1323;  0, 66, 1525;  0, 67, 556;  0, 69, 1028;  0, 71, 2701;  0, 72, 2187;  0, 73, 2456;  0, 74, 2918;  0, 75, 2650;  0, 76, 2577;  0, 77, 1758;  0, 78, 1008;  0, 81, 1620;  0, 82, 127;  0, 83, 2511;  0, 84, 1275;  0, 85, 540;  0, 86, 2149;  0, 87, 2840;  0, 88, 1045;  0, 89, 752;  0, 90, 685;  0, 91, 2629;  0, 92, 2403;  0, 93, 1707;  0, 94, 995;  0, 95, 1189;  0, 96, 1297;  0, 97, 214;  0, 98, 1817;  0, 99, 2624;  0, 100, 2357;  0, 101, 233;  0, 102, 968;  0, 103, 2470;  0, 104, 2151;  0, 105, 1628;  0, 106, 893;  0, 107, 2068;  0, 109, 657;  0, 111, 389;  0, 114, 1945;  0, 115, 1559;  0, 117, 660;  0, 119, 1593;  0, 120, 821;  0, 121, 817;  0, 122, 131;  0, 123, 1400;  0, 124, 1209;  0, 125, 1928;  0, 126, 2599;  0, 128, 935;  0, 129, 2082;  0, 132, 1537;  0, 133, 2181;  0, 135, 1219;  0, 136, 472;  0, 137, 2445;  0, 138, 2225;  0, 139, 1640;  0, 140, 1603;  0, 142, 1578;  0, 143, 2459;  0, 145, 1813;  0, 147, 989;  0, 148, 1391;  0, 149, 2591;  0, 150, 1261;  0, 151, 2721;  0, 153, 2773;  0, 155, 334;  0, 156, 1260;  0, 157, 2795;  0, 158, 2750;  0, 159, 2406;  0, 160, 2893;  0, 161, 2092;  0, 162, 1977;  0, 163, 1898;  0, 165, 2325;  0, 166, 1911;  0, 167, 1595;  0, 169, 365;  0, 171, 2754;  0, 173, 2592;  0, 176, 691;  0, 177, 2333;  0, 181, 1602;  0, 182, 898;  0, 183, 1286;  0, 184, 2531;  0, 185, 2141;  0, 186, 2917;  0, 187, 2613;  0, 188, 2469;  0, 189, 1143;  0, 191, 1186;  0, 193, 2817;  0, 194, 1633;  0, 195, 854;  0, 196, 558;  0, 197, 420;  0, 198, 2652;  0, 199, 2267;  0, 201, 1210;  0, 202, 2465;  0, 203, 1963;  0, 204, 1623;  0, 205, 1060;  0, 206, 2729;  0, 207, 903;  0, 208, 2713;  0, 209, 2549;  0, 211, 2420;  0, 212, 1857;  0, 215, 2094;  0, 216, 1743;  0, 217, 2090;  0, 218, 811;  0, 219, 573;  0, 220, 777;  0, 221, 1245;  0, 222, 1377;  0, 223, 2906;  0, 225, 1608;  0, 226, 1665;  0, 227, 2424;  0, 229, 364;  0, 231, 1379;  0, 234, 860;  0, 235, 1320;  0, 236, 1340;  0, 237, 2186;  0, 239, 1035;  0, 241, 2131;  0, 243, 908;  0, 244, 997;  0, 245, 1417;  0, 247, 2441;  0, 248, 1885;  0, 249, 2632;  0, 250, 2494;  0, 251, 534;  0, 252, 834;  0, 253, 2733;  0, 255, 2536;  0, 257, 564;  0, 259, 2129;  0, 261, 414;  0, 262, 2103;  0, 263, 1848;  0, 264, 973;  0, 265, 1851;  0, 266, 851;  0, 267, 1469;  0, 269, 2612;  0, 271, 2660;  0, 272, 1824;  0, 273, 2239;  0, 274, 1792;  0, 275, 1203;  0, 276, 2899;  0, 277, 1944;  0, 278, 2045;  0, 279, 1223;  0, 280, 2739;  0, 281, 616;  0, 283, 2535;  0, 284, 1353;  0, 285, 2890;  0, 286, 1627;  0, 287, 956;  0, 288, 1499;  0, 289, 1548;  0, 290, 795;  0, 291, 2193;  0, 292, 1437;  0, 293, 1424;  0, 295, 833;  0, 296, 1276;  0, 297, 1570;  0, 299, 1448;  0, 301, 304;  0, 302, 788;  0, 303, 486;  0, 305, 1970;  0, 306, 2142;  0, 307, 1255;  0, 309, 1860;  0, 310, 1955;  0, 311, 1506;  0, 313, 620;  0, 314, 2727;  0, 315, 1590;  0, 317, 396;  0, 319, 2196;  0, 321, 949;  0, 322, 2482;  0, 323, 1483;  0, 324, 1909;  0, 325, 2839;  0, 327, 2573;  0, 328, 2037;  0, 329, 563;  0, 331, 1259;  0, 333, 963;  0, 335, 765;  0, 336, 2205;  0, 337, 861;  0, 338, 1429;  0, 339, 373;  0, 340, 1517;  0, 341, 1174;  0, 343, 779;  0, 345, 2160;  0, 346, 971;  0, 347, 2787;  0, 348, 927;  0, 349, 1330;  0, 352, 1403;  0, 353, 1966;  0, 354, 2123;  0, 355, 2462;  0, 356, 2377;  0, 357, 962;  0, 359, 2876;  0, 360, 2393;  0, 361, 2334;  0, 362, 925;  0, 363, 2218;  0, 366, 1321;  0, 367, 2132;  0, 368, 902;  0, 369, 1786;  0, 371, 1324;  0, 372, 1375;  0, 375, 2596;  0, 377, 604;  0, 378, 2221;  0, 379, 1169;  0, 381, 463;  0, 383, 1065;  0, 384, 1087;  0, 385, 1625;  0, 387, 2315;  0, 390, 1865;  0, 391, 2853;  0, 393, 2512;  0, 394, 1755;  0, 395, 481;  0, 397, 2014;  0, 398, 1277;  0, 399, 2642;  0, 400, 1999;  0, 401, 1205;  0, 403, 1717;  0, 404, 2517;  0, 405, 1768;  0, 407, 1777;  0, 409, 1291;  0, 411, 1822;  0, 413, 2426;  0, 415, 2529;  0, 416, 1214;  0, 417, 1165;  0, 418, 2813;  0, 419, 1500;  0, 421, 2081;  0, 422, 1883;  0, 423, 2505;  0, 425, 2016;  0, 426, 2412;  0, 427, 1187;  0, 430, 699;  0, 431, 2551;  0, 433, 2398;  0, 434, 1601;  0, 435, 2162;  0, 436, 793;  0, 437, 2365;  0, 440, 1957;  0, 441, 524;  0, 443, 1132;  0, 444, 2382;  0, 445, 629;  0, 447, 865;  0, 448, 1041;  0, 449, 2243;  0, 450, 1721;  0, 451, 1118;  0, 452, 547;  0, 453, 2503;  0, 454, 2259;  0, 455, 1449;  0, 456, 1759;  0, 457, 1148;  0, 458, 1237;  0, 459, 1805;  0, 461, 1894;  0, 464, 1673;  0, 465, 2486;  0, 467, 1714;  0, 469, 1685;  0, 470, 2133;  0, 471, 1940;  0, 473, 2860;  0, 474, 2171;  0, 475, 1197;  0, 476, 673;  0, 477, 742;  0, 478, 2080;  0, 479, 1153;  0, 482, 1453;  0, 483, 513;  0, 487, 1325;  0, 489, 2644;  0, 490, 1381;  0, 491, 2227;  0, 493, 619;  0, 494, 1040;  0, 495, 680;  0, 496, 589;  0, 497, 837;  0, 499, 1164;  0, 500, 1009;  0, 501, 2251;  0, 503, 2777;  0, 505, 2463;  0, 507, 2167;  0, 508, 1105;  0, 509, 828;  0, 511, 1430;  0, 512, 2833;  0, 514, 763;  0, 516, 1265;  0, 517, 2105;  0, 519, 1465;  0, 521, 2475;  0, 523, 1115;  0, 525, 2253;  0, 527, 1668;  0, 529, 2638;  0, 531, 2663;  0, 532, 1122;  0, 533, 2307;  0, 535, 655;  0, 536, 2083;  0, 537, 1709;  0, 538, 2732;  0, 539, 1671;  0, 541, 1710;  0, 542, 1866;  0, 543, 2291;  0, 545, 1564;  0, 546, 1658;  0, 548, 1837;  0, 549, 2797;  0, 550, 1993;  0, 551, 2272;  0, 553, 1723;  0, 555, 686;  0, 557, 1149;  0, 559, 1859;  0, 560, 2241;  0, 561, 1905;  0, 562, 1283;  0, 565, 2114;  0, 566, 2452;  0, 567, 1819;  0, 568, 2137;  0, 569, 1160;  0, 571, 2089;  0, 575, 1117;  0, 577, 2668;  0, 580, 869;  0, 581, 917;  0, 583, 2124;  0, 585, 1109;  0, 587, 1217;  0, 590, 1861;  0, 591, 2347;  0, 592, 2569;  0, 593, 698;  0, 594, 1288;  0, 597, 1125;  0, 598, 863;  0, 599, 2755;  0, 601, 1411;  0, 607, 1809;  0, 608, 2457;  0, 609, 2883;  0, 610, 1495;  0, 611, 2518;  0, 613, 1518;  0, 614, 1579;  0, 615, 1922;  0, 617, 816;  0, 618, 2896;  0, 621, 1342;  0, 622, 1413;  0, 623, 1507;  0, 625, 2061;  0, 627, 1615;  0, 628, 1077;  0, 630, 2583;  0, 631, 2443;  0, 633, 1903;  0, 637, 1051;  0, 641, 2458;  0, 642, 2240;  0, 643, 2858;  0, 644, 2471;  0, 645, 2600;  0, 646, 1599;  0, 647, 1392;  0, 648, 2062;  0, 649, 1842;  0, 651, 1902;  0, 653, 2379;  0, 654, 2717;  0, 656, 1791;  0, 659, 2859;  0, 661, 2453;  0, 662, 1039;  0, 663, 2752;  0, 664, 1005;  0, 665, 2086;  0, 666, 2439;  0, 667, 1083;  0, 669, 2909;  0, 671, 1097;  0, 672, 1708;  0, 674, 2022;  0, 675, 999;  0, 677, 895;  0, 678, 1995;  0, 679, 717;  0, 681, 1878;  0, 682, 2633;  0, 683, 1541;  0, 688, 979;  0, 689, 1026;  0, 690, 2049;  0, 693, 1959;  0, 695, 2497;  0, 696, 2559;  0, 697, 2111;  0, 705, 1612;  0, 707, 2146;  0, 708, 1539;  0, 710, 2395;  0, 711, 1471;  0, 712, 2093;  0, 713, 1350;  0, 715, 1669;  0, 716, 2281;  0, 718, 2273;  0, 719, 2557;  0, 721, 923;  0, 722, 1094;  0, 723, 1154;  0, 724, 1301;  0, 725, 2901;  0, 727, 2216;  0, 729, 2069;  0, 730, 2245;  0, 731, 2235;  0, 733, 1376;  0, 734, 1175;  0, 735, 2314;  0, 739, 877;  0, 740, 2801;  0, 741, 1895;  0, 743, 2672;  0, 745, 2191;  0, 747, 849;  0, 748, 2547;  0, 751, 2427;  0, 753, 2169;  0, 757, 919;  0, 758, 1011;  0, 759, 2537;  0, 760, 1604;  0, 761, 1735;  0, 762, 2875;  0, 767, 791;  0, 768, 1795;  0, 769, 2335;  0, 771, 1369;  0, 773, 1034;  0, 775, 1365;  0, 776, 2015;  0, 778, 2685;  0, 783, 802;  0, 784, 2159;  0, 785, 2743;  0, 787, 2217;  0, 789, 2686;  0, 790, 1567;  0, 794, 2842;  0, 796, 2619;  0, 797, 1798;  0, 798, 2897;  0, 799, 1657;  0, 801, 1981;  0, 803, 2758;  0, 804, 2255;  0, 805, 2894;  0, 807, 1146;  0, 809, 2327;  0, 813, 1662;  0, 814, 2144;  0, 815, 1761;  0, 819, 2519;  0, 820, 1427;  0, 822, 2810;  0, 823, 2222;  0, 825, 1592;  0, 827, 1652;  0, 829, 2052;  0, 835, 1699;  0, 838, 2561;  0, 839, 2215;  0, 841, 1531;  0, 842, 2821;  0, 843, 2046;  0, 844, 1677;  0, 845, 2791;  0, 847, 1220;  0, 848, 1207;  0, 850, 1806;  0, 853, 1583;  0, 855, 1783;  0, 856, 1637;  0, 857, 1941;  0, 858, 2381;  0, 859, 1973;  0, 862, 2809;  0, 864, 2311;  0, 867, 2534;  0, 872, 1821;  0, 873, 2635;  0, 875, 1845;  0, 879, 2413;  0, 881, 2175;  0, 883, 1549;  0, 884, 1173;  0, 885, 1170;  0, 886, 1327;  0, 887, 1629;  0, 889, 2313;  0, 891, 1457;  0, 894, 2665;  0, 897, 1116;  0, 899, 1510;  0, 901, 2386;  0, 905, 2173;  0, 907, 2708;  0, 909, 1142;  0, 911, 1334;  0, 913, 2812;  0, 914, 2867;  0, 915, 2830;  0, 920, 1185;  0, 921, 1841;  0, 924, 1972;  0, 928, 2433;  0, 929, 1956;  0, 931, 1252;  0, 932, 1225;  0, 933, 1073;  0, 937, 1171;  0, 939, 1226;  0, 940, 1473;  0, 941, 2182;  0, 942, 2378;  0, 943, 1504;  0, 944, 1667;  0, 945, 1346;  0, 946, 1439;  0, 947, 1229;  0, 951, 1282;  0, 953, 1989;  0, 957, 2084;  0, 959, 1719;  0, 964, 1423;  0, 965, 1487;  0, 966, 2477;  0, 967, 1294;  0, 969, 1250;  0, 970, 2219;  0, 974, 1339;  0, 975, 1649;  0, 976, 1515;  0, 977, 1560;  0, 980, 1064;  0, 987, 1871;  0, 988, 1557;  0, 991, 1811;  0, 993, 2387;  0, 994, 2589;  0, 998, 2017;  0, 1000, 2025;  0, 1001, 1524;  0, 1004, 2155;  0, 1006, 1085;  0, 1007, 1802;  0, 1012, 2201;  0, 1013, 2096;  0, 1016, 1195;  0, 1017, 1785;  0, 1019, 1960;  0, 1021, 1227;  0, 1023, 2020;  0, 1024, 1089;  0, 1025, 2200;  0, 1029, 1479;  0, 1030, 2343;  0, 1031, 1431;  0, 1033, 2630;  0, 1037, 1887;  0, 1043, 1558;  0, 1044, 2277;  0, 1047, 2835;  0, 1049, 1474;  0, 1053, 2423;  0, 1061, 1159;  0, 1066, 1192;  0, 1067, 2107;  0, 1069, 1371;  0, 1071, 2211;  0, 1075, 1571;  0, 1076, 2101;  0, 1079, 2285;  0, 1080, 2805;  0, 1081, 1648;  0, 1084, 2678;  0, 1088, 2252;  0, 1091, 1778;  0, 1093, 2587;  0, 1095, 1939;  0, 1096, 1781;  0, 1099, 2210;  0, 1101, 2156;  0, 1103, 2212;  0, 1107, 2342;  0, 1113, 1409;  0, 1114, 1281;  0, 1121, 1969;  0, 1123, 1605;  0, 1124, 1734;  0, 1127, 1789;  0, 1129, 1306;  0, 1131, 1204;  0, 1135, 1406;  0, 1136, 2643;  0, 1137, 1553;  0, 1139, 1140;  0, 1141, 2435;  0, 1147, 1216;  0, 1156, 2847;  0, 1157, 1234;  0, 1161, 1387;  0, 1163, 2249;  0, 1166, 1967;  0, 1167, 1236;  0, 1172, 1639;  0, 1177, 2097;  0, 1178, 2275;  0, 1179, 2671;  0, 1180, 2236;  0, 1181, 2147;  0, 1183, 2489;  0, 1184, 2421;  0, 1190, 2318;  0, 1193, 2048;  0, 1196, 2290;  0, 1198, 1655;  0, 1199, 1638;  0, 1202, 1589;  0, 1206, 2771;  0, 1208, 1455;  0, 1213, 2737;  0, 1215, 2396;  0, 1221, 1753;  0, 1231, 2579;  0, 1232, 2361;  0, 1235, 2007;  0, 1239, 1930;  0, 1240, 2487;  0, 1241, 1347;  0, 1243, 2698;  0, 1247, 2369;  0, 1249, 1718;  0, 1251, 1855;  0, 1257, 2887;  0, 1258, 2095;  0, 1263, 2794;  0, 1268, 1975;  0, 1269, 2158;  0, 1270, 2889;  0, 1273, 1971;  0, 1274, 2675;  0, 1279, 2270;  0, 1285, 2806;  0, 1289, 2836;  0, 1293, 2072;  0, 1295, 1875;  0, 1299, 2021;  0, 1300, 1961;  0, 1304, 1461;  0, 1307, 2320;  0, 1311, 2570;  0, 1313, 1749;  0, 1315, 2410;  0, 1319, 1947;  0, 1322, 2425;  0, 1329, 2143;  0, 1331, 2026;  0, 1333, 2476;  0, 1335, 1394;  0, 1336, 1341;  0, 1337, 2749;  0, 1345, 1846;  0, 1348, 2792;  0, 1351, 2464;  0, 1352, 1801;  0, 1355, 2765;  0, 1357, 2825;  0, 1361, 2594;  0, 1363, 1503;  0, 1364, 1568;  0, 1366, 2355;  0, 1368, 2903;  0, 1370, 2213;  0, 1373, 2194;  0, 1385, 1488;  0, 1397, 1494;  0, 1399, 2109;  0, 1407, 1829;  0, 1412, 2769;  0, 1415, 1490;  0, 1421, 1607;  0, 1425, 1619;  0, 1433, 2931;  0, 1435, 2288;  0, 1441, 1917;  0, 1442, 1543;  0, 1443, 2384;  0, 1444, 2271;  0, 1445, 1676;  0, 1460, 1493;  0, 1463, 1736;  0, 1466, 1520;  0, 1467, 1468;  0, 1477, 1519;  0, 1481, 2480;  0, 1485, 1613;  0, 1486, 2399;  0, 1489, 2354;  0, 1491, 2138;  0, 1492, 2260;  0, 1497, 2565;  0, 1501, 2345;  0, 1505, 2606;  0, 1509, 2662;  0, 1521, 1780;  0, 1522, 1617;  0, 1528, 2911;  0, 1529, 1643;  0, 1533, 1591;  0, 1534, 1611;  0, 1540, 2303;  0, 1545, 2405;  0, 1547, 2498;  0, 1552, 2656;  0, 1555, 2204;  0, 1561, 1575;  0, 1563, 1929;  0, 1569, 2738;  0, 1573, 1850;  0, 1581, 2815;  0, 1582, 2319;  0, 1587, 2174;  0, 1588, 1600;  0, 1594, 2683;  0, 1597, 2192;  0, 1609, 2929;  0, 1621, 2402;  0, 1630, 2279;  0, 1631, 2746;  0, 1634, 2483;  0, 1635, 2312;  0, 1641, 2920;  0, 1645, 2326;  0, 1646, 2871;  0, 1653, 1757;  0, 1660, 1910;  0, 1661, 2689;  0, 1663, 2199;  0, 1664, 1799;  0, 1666, 2709;  0, 1679, 2161;  0, 1681, 2075;  0, 1682, 2363;  0, 1687, 2775;  0, 1691, 2513;  0, 1694, 1807;  0, 1695, 1816;  0, 1700, 2422;  0, 1701, 2223;  0, 1702, 2301;  0, 1705, 2829;  0, 1713, 1750;  0, 1715, 2659;  0, 1720, 2803;  0, 1724, 2759;  0, 1726, 2153;  0, 1729, 1934;  0, 1730, 2431;  0, 1732, 2019;  0, 1742, 2140;  0, 1744, 2415;  0, 1745, 2793;  0, 1748, 2657;  0, 1754, 1901;  0, 1756, 2539;  0, 1760, 1931;  0, 1762, 2317;  0, 1763, 2493;  0, 1767, 2774;  0, 1774, 2609;  0, 1779, 2150;  0, 1787, 2039;  0, 1797, 2919;  0, 1803, 2691;  0, 1810, 2043;  0, 1820, 2373;  0, 1823, 2005;  0, 1825, 1991;  0, 1827, 2372;  0, 1833, 2907;  0, 1835, 2515;  0, 1838, 2912;  0, 1839, 2740;  0, 1847, 2566;  0, 1858, 1983;  0, 1867, 2811;  0, 1868, 2183;  0, 1870, 2197;  0, 1873, 1899;  0, 1874, 2849;  0, 1877, 2440;  0, 1879, 2035;  0, 1886, 1937;  0, 1889, 2545;  0, 1891, 1918;  0, 1897, 2491;  0, 1900, 2397;  0, 1904, 2725;  0, 1906, 2905;  0, 1913, 2799;  0, 1915, 2305;  0, 1919, 2757;  0, 1923, 2229;  0, 1925, 2677;  0, 1933, 2383;  0, 1935, 2645;  0, 1936, 2261;  0, 1943, 2578;  0, 1946, 2119;  0, 1949, 2073;  0, 1951, 2869;  0, 1954, 2299;  0, 1958, 2648;  0, 1964, 2295;  0, 1965, 2179;  0, 1979, 2500;  0, 1985, 2407;  0, 1987, 2553;  0, 1990, 2525;  0, 1994, 2263;  0, 1997, 2540;  0, 2002, 2041;  0, 2003, 2863;  0, 2006, 2374;  0, 2009, 2828;  0, 2011, 2321;  0, 2013, 2209;  0, 2023, 2555;  0, 2027, 2246;  0, 2031, 2125;  0, 2032, 2641;  0, 2036, 2337;  0, 2047, 2851;  0, 2050, 2695;  0, 2051, 2504;  0, 2053, 2653;  0, 2055, 2852;  0, 2057, 2323;  0, 2067, 2516;  0, 2071, 2455;  0, 2074, 2655;  0, 2077, 2741;  0, 2078, 2861;  0, 2079, 2265;  0, 2091, 2571;  0, 2120, 2647;  0, 2121, 2283;  0, 2127, 2798;  0, 2139, 2366;  0, 2157, 2693;  0, 2165, 2900;  0, 2176, 2607;  0, 2177, 2385;  0, 2185, 2722;  0, 2195, 2308;  0, 2207, 2716;  0, 2224, 2523;  0, 2231, 2368;  0, 2233, 2601;  0, 2237, 2714;  0, 2248, 2339;  0, 2254, 2723;  0, 2257, 2530;  0, 2287, 2913;  0, 2289, 2881;  0, 2293, 2639;  0, 2306, 2499;  0, 2309, 2359;  0, 2329, 2930;  0, 2330, 2543;  0, 2331, 2873;  0, 2336, 2891;  0, 2341, 2669;  0, 2349, 2509;  0, 2367, 2681;  0, 2371, 2444;  0, 2375, 2611;  0, 2389, 2485;  0, 2392, 2783;  0, 2401, 2588;  0, 2409, 2575;  0, 2419, 2884;  0, 2428, 2745;  0, 2429, 2719;  0, 2434, 2617;  0, 2437, 2877;  0, 2438, 2927;  0, 2449, 2697;  0, 2451, 2786;  0, 2461, 2687;  0, 2481, 2885;  0, 2521, 2908;  0, 2533, 2895;  0, 2541, 2581;  0, 2563, 2915;  0, 2567, 2932;  0, 2582, 2625;  0, 2585, 2728;  0, 2593, 2684;  0, 2603, 2631;  0, 2605, 2879;  0, 2620, 2921;  0, 2621, 2846;  0, 2661, 2923;  0, 2673, 2761;  0, 2679, 2726;  0, 2690, 2703;  0, 2707, 2768;  0, 2753, 2857;  0, 2763, 2848;  0, 2776, 2834;  0, 2779, 2845;  0, 2807, 2819;  0, 2841, 2925;  1, 2, 1743;  1, 3, 2524;  1, 5, 1935;  1, 7, 343;  1, 8, 1065;  1, 9, 2488;  1, 10, 2122;  1, 13, 2249;  1, 14, 2647;  1, 16, 599;  1, 17, 758;  1, 20, 1252;  1, 22, 2495;  1, 23, 2375;  1, 25, 1834;  1, 26, 640;  1, 29, 2223;  1, 31, 2704;  1, 32, 1461;  1, 34, 130;  1, 35, 773;  1, 38, 2176;  1, 40, 2500;  1, 41, 763;  1, 44, 2864;  1, 45, 1967;  1, 46, 2048;  1, 50, 2461;  1, 51, 2600;  1, 52, 257;  1, 55, 196;  1, 56, 741;  1, 58, 823;  1, 59, 101;  1, 64, 1345;  1, 68, 2367;  1, 73, 997;  1, 75, 2404;  1, 76, 1739;  1, 77, 1137;  1, 79, 619;  1, 80, 2492;  1, 82, 2330;  1, 83, 2720;  1, 85, 748;  1, 86, 1676;  1, 87, 1424;  1, 88, 2480;  1, 89, 2324;  1, 91, 358;  1, 93, 2648;  1, 94, 1064;  1, 95, 2631;  1, 98, 915;  1, 99, 1048;  1, 100, 1653;  1, 103, 957;  1, 104, 2248;  1, 106, 2750;  1, 107, 1634;  1, 110, 1945;  1, 112, 2573;  1, 115, 2649;  1, 116, 314;  1, 122, 1505;  1, 123, 1322;  1, 124, 2733;  1, 125, 364;  1, 128, 365;  1, 129, 2314;  1, 134, 2163;  1, 137, 436;  1, 140, 749;  1, 143, 1934;  1, 146, 2792;  1, 148, 526;  1, 149, 1820;  1, 151, 2751;  1, 152, 857;  1, 154, 1745;  1, 158, 629;  1, 160, 1839;  1, 161, 1544;  1, 164, 2663;  1, 166, 2505;  1, 170, 2167;  1, 172, 1376;  1, 177, 1630;  1, 181, 1648;  1, 182, 459;  1, 183, 1130;  1, 184, 1905;  1, 187, 1136;  1, 189, 1030;  1, 190, 673;  1, 194, 263;  1, 195, 1235;  1, 199, 1798;  1, 200, 790;  1, 205, 2041;  1, 208, 1909;  1, 209, 2218;  1, 212, 1017;  1, 213, 392;  1, 214, 2143;  1, 215, 1516;  1, 217, 1190;  1, 218, 291;  1, 219, 580;  1, 220, 1761;  1, 221, 1538;  1, 223, 1258;  1, 224, 1215;  1, 226, 2079;  1, 232, 2428;  1, 236, 2465;  1, 237, 379;  1, 238, 1994;  1, 242, 1193;  1, 244, 967;  1, 245, 2242;  1, 248, 2620;  1, 250, 2120;  1, 251, 706;  1, 253, 1289;  1, 254, 1287;  1, 256, 1799;  1, 260, 2777;  1, 262, 1238;  1, 265, 1696;  1, 266, 1609;  1, 267, 2024;  1, 273, 1779;  1, 277, 2534;  1, 279, 523;  1, 280, 2327;  1, 281, 2422;  1, 283, 2813;  1, 284, 754;  1, 285, 2014;  1, 286, 1167;  1, 287, 1507;  1, 290, 1847;  1, 292, 679;  1, 293, 1076;  1, 296, 2897;  1, 297, 1364;  1, 304, 2353;  1, 307, 1754;  1, 310, 2728;  1, 316, 1013;  1, 320, 836;  1, 323, 1352;  1, 325, 2564;  1, 326, 1781;  1, 332, 535;  1, 334, 1937;  1, 335, 2596;  1, 338, 848;  1, 339, 2657;  1, 340, 1113;  1, 341, 626;  1, 344, 1033;  1, 346, 1719;  1, 349, 1228;  1, 350, 2396;  1, 352, 1624;  1, 355, 1162;  1, 356, 604;  1, 357, 976;  1, 361, 1762;  1, 362, 2247;  1, 367, 1292;  1, 368, 1574;  1, 370, 617;  1, 373, 806;  1, 376, 2702;  1, 380, 2056;  1, 382, 2804;  1, 386, 2021;  1, 394, 1435;  1, 398, 487;  1, 399, 2018;  1, 400, 1375;  1, 401, 1514;  1, 404, 838;  1, 405, 1318;  1, 406, 566;  1, 410, 2174;  1, 412, 2407;  1, 416, 572;  1, 418, 1061;  1, 419, 2383;  1, 422, 2690;  1, 423, 532;  1, 424, 1855;  1, 427, 2421;  1, 428, 1177;  1, 431, 1016;  1, 435, 994;  1, 440, 1029;  1, 442, 2093;  1, 445, 1769;  1, 446, 2151;  1, 448, 947;  1, 449, 2227;  1, 452, 719;  1, 453, 465;  1, 454, 1041;  1, 455, 496;  1, 457, 1191;  1, 458, 820;  1, 460, 1825;  1, 464, 1965;  1, 471, 742;  1, 472, 676;  1, 473, 1690;  1, 476, 2554;  1, 478, 2000;  1, 479, 2374;  1, 481, 2068;  1, 482, 1384;  1, 484, 517;  1, 491, 1394;  1, 493, 2780;  1, 494, 1843;  1, 495, 686;  1, 500, 1532;  1, 501, 844;  1, 506, 2345;  1, 507, 704;  1, 508, 1211;  1, 509, 1000;  1, 511, 2318;  1, 512, 2145;  1, 513, 1244;  1, 514, 2161;  1, 515, 581;  1, 518, 2119;  1, 520, 634;  1, 530, 1491;  1, 536, 2002;  1, 542, 1346;  1, 543, 896;  1, 544, 1972;  1, 547, 1282;  1, 548, 2414;  1, 550, 2499;  1, 551, 1348;  1, 554, 1393;  1, 556, 2926;  1, 557, 2164;  1, 559, 1387;  1, 560, 1147;  1, 561, 1771;  1, 563, 1957;  1, 565, 1904;  1, 569, 898;  1, 574, 2723;  1, 578, 2198;  1, 584, 2900;  1, 586, 739;  1, 590, 874;  1, 592, 2774;  1, 598, 835;  1, 605, 1316;  1, 608, 1108;  1, 609, 878;  1, 610, 1353;  1, 611, 1335;  1, 615, 2546;  1, 616, 1749;  1, 620, 1744;  1, 621, 842;  1, 622, 2589;  1, 623, 2458;  1, 628, 2063;  1, 632, 1201;  1, 638, 2435;  1, 643, 1534;  1, 645, 2362;  1, 646, 1738;  1, 647, 890;  1, 649, 2536;  1, 652, 998;  1, 655, 2798;  1, 656, 2416;  1, 657, 904;  1, 658, 1052;  1, 661, 2402;  1, 662, 1214;  1, 663, 1058;  1, 670, 2003;  1, 674, 1887;  1, 675, 2008;  1, 680, 1280;  1, 681, 2146;  1, 683, 1562;  1, 688, 1867;  1, 689, 2582;  1, 694, 1803;  1, 695, 2788;  1, 698, 1580;  1, 700, 2097;  1, 710, 1526;  1, 712, 986;  1, 713, 1444;  1, 716, 1660;  1, 717, 2440;  1, 718, 1208;  1, 722, 1751;  1, 724, 1579;  1, 725, 2392;  1, 728, 2141;  1, 730, 1946;  1, 734, 1129;  1, 740, 2836;  1, 743, 1786;  1, 746, 1700;  1, 752, 1210;  1, 753, 1709;  1, 759, 1022;  1, 760, 1095;  1, 764, 1216;  1, 766, 1117;  1, 770, 2410;  1, 772, 873;  1, 776, 1132;  1, 777, 2560;  1, 778, 2047;  1, 784, 2309;  1, 785, 1622;  1, 788, 1143;  1, 789, 1948;  1, 794, 2219;  1, 796, 2932;  1, 797, 1640;  1, 799, 2276;  1, 802, 2350;  1, 803, 2924;  1, 812, 2133;  1, 814, 1826;  1, 824, 2470;  1, 826, 832;  1, 830, 1942;  1, 841, 2618;  1, 843, 2134;  1, 849, 2858;  1, 850, 1342;  1, 851, 1204;  1, 853, 2224;  1, 860, 1603;  1, 861, 2685;  1, 862, 2651;  1, 863, 2306;  1, 867, 963;  1, 868, 1202;  1, 884, 1931;  1, 887, 1958;  1, 899, 2606;  1, 902, 2759;  1, 903, 1156;  1, 908, 2429;  1, 910, 2005;  1, 913, 2138;  1, 916, 2445;  1, 922, 2355;  1, 928, 1595;  1, 932, 1120;  1, 933, 1040;  1, 941, 1941;  1, 943, 2201;  1, 944, 1849;  1, 946, 2715;  1, 952, 1150;  1, 962, 1521;  1, 964, 2894;  1, 965, 1858;  1, 968, 2086;  1, 970, 1912;  1, 975, 2595;  1, 980, 1831;  1, 981, 1955;  1, 995, 1906;  1, 1001, 1922;  1, 1004, 1222;  1, 1005, 1378;  1, 1009, 2612;  1, 1010, 1748;  1, 1012, 1297;  1, 1018, 1181;  1, 1024, 1286;  1, 1025, 2015;  1, 1028, 2552;  1, 1031, 2471;  1, 1034, 2859;  1, 1046, 1695;  1, 1054, 1205;  1, 1057, 2624;  1, 1060, 2668;  1, 1066, 1406;  1, 1067, 1360;  1, 1070, 2912;  1, 1072, 1347;  1, 1075, 2866;  1, 1077, 2567;  1, 1078, 2655;  1, 1082, 2354;  1, 1084, 1528;  1, 1085, 2325;  1, 1088, 1331;  1, 1090, 1919;  1, 1094, 2313;  1, 1102, 1106;  1, 1118, 2271;  1, 1119, 1172;  1, 1124, 1804;  1, 1126, 1850;  1, 1138, 1720;  1, 1141, 2848;  1, 1142, 2037;  1, 1148, 2594;  1, 1149, 1702;  1, 1154, 2225;  1, 1155, 1612;  1, 1160, 1438;  1, 1161, 2344;  1, 1166, 1227;  1, 1168, 2062;  1, 1174, 2744;  1, 1178, 1856;  1, 1180, 1868;  1, 1184, 2104;  1, 1186, 1569;  1, 1187, 1256;  1, 1192, 1246;  1, 1196, 2824;  1, 1197, 2828;  1, 1198, 1875;  1, 1207, 2930;  1, 1209, 1358;  1, 1220, 2933;  1, 1221, 2387;  1, 1225, 2614;  1, 1232, 2177;  1, 1233, 2746;  1, 1237, 2747;  1, 1240, 1261;  1, 1251, 1306;  1, 1253, 1556;  1, 1262, 2126;  1, 1264, 2074;  1, 1268, 2033;  1, 1269, 2726;  1, 1270, 1979;  1, 1271, 2432;  1, 1273, 2830;  1, 1274, 1365;  1, 1277, 2284;  1, 1283, 1328;  1, 1294, 1605;  1, 1300, 1351;  1, 1304, 2282;  1, 1310, 1647;  1, 1312, 2795;  1, 1323, 1466;  1, 1330, 2583;  1, 1337, 2906;  1, 1341, 2882;  1, 1343, 2498;  1, 1354, 1454;  1, 1367, 1870;  1, 1371, 2621;  1, 1377, 1426;  1, 1382, 2572;  1, 1383, 2168;  1, 1395, 1412;  1, 1396, 1649;  1, 1407, 1616;  1, 1408, 2818;  1, 1413, 2312;  1, 1414, 2027;  1, 1430, 1911;  1, 1436, 1892;  1, 1442, 1571;  1, 1449, 1954;  1, 1450, 1864;  1, 1460, 2516;  1, 1462, 1631;  1, 1467, 2638;  1, 1474, 1899;  1, 1475, 1484;  1, 1480, 2072;  1, 1487, 2399;  1, 1493, 2398;  1, 1498, 2102;  1, 1499, 2356;  1, 1508, 2710;  1, 1510, 1928;  1, 1520, 2770;  1, 1529, 2734;  1, 1535, 1772;  1, 1540, 1940;  1, 1553, 2579;  1, 1564, 2542;  1, 1576, 2321;  1, 1582, 2528;  1, 1583, 2193;  1, 1593, 1930;  1, 1599, 2474;  1, 1601, 2147;  1, 1606, 2140;  1, 1610, 2139;  1, 1611, 1688;  1, 1613, 1816;  1, 1628, 1703;  1, 1629, 1886;  1, 1636, 1988;  1, 1642, 2672;  1, 1652, 2369;  1, 1658, 2085;  1, 1659, 2530;  1, 1670, 1828;  1, 1672, 2230;  1, 1678, 2794;  1, 1682, 2727;  1, 1684, 2038;  1, 1685, 2150;  1, 1694, 2519;  1, 1706, 1730;  1, 1712, 2439;  1, 1714, 2537;  1, 1718, 2816;  1, 1721, 1726;  1, 1727, 2278;  1, 1731, 2411;  1, 1732, 2807;  1, 1733, 2907;  1, 1737, 1982;  1, 1750, 2206;  1, 1760, 1923;  1, 1768, 2012;  1, 1774, 2633;  1, 1780, 1901;  1, 1787, 2872;  1, 1790, 2441;  1, 1796, 2786;  1, 1811, 2205;  1, 1814, 2687;  1, 1817, 1918;  1, 1821, 2116;  1, 1822, 1984;  1, 1823, 2385;  1, 1859, 2636;  1, 1862, 2921;  1, 1869, 2462;  1, 1876, 2260;  1, 1880, 2613;  1, 1882, 2159;  1, 1898, 2026;  1, 1924, 2504;  1, 1929, 2918;  1, 1971, 1976;  1, 1973, 2194;  1, 1990, 2270;  1, 1996, 2482;  1, 2006, 2541;  1, 2020, 2878;  1, 2032, 2678;  1, 2036, 2518;  1, 2042, 2291;  1, 2054, 2132;  1, 2060, 2236;  1, 2066, 2597;  1, 2073, 2927;  1, 2078, 2423;  1, 2081, 2360;  1, 2128, 2434;  1, 2157, 2211;  1, 2188, 2843;  1, 2200, 2835;  1, 2204, 2319;  1, 2217, 2764;  1, 2222, 2920;  1, 2241, 2854;  1, 2252, 2931;  1, 2254, 2729;  1, 2264, 2787;  1, 2266, 2446;  1, 2272, 2870;  1, 2289, 2806;  1, 2294, 2714;  1, 2300, 2379;  1, 2302, 2522;  1, 2308, 2487;  1, 2315, 2566;  1, 2320, 2891;  1, 2336, 2861;  1, 2373, 2913;  1, 2378, 2444;  1, 2386, 2602;  1, 2397, 2776;  1, 2438, 2483;  1, 2456, 2692;  1, 2457, 2607;  1, 2459, 2800;  1, 2464, 2716;  1, 2477, 2876;  1, 2501, 2810;  1, 2510, 2512;  1, 2588, 2741;  1, 2626, 2919;  1, 2630, 2660;  1, 2632, 2888;  1, 2654, 2909;  1, 2656, 2852;  1, 2661, 2696;  1, 2662, 2811;  1, 2679, 2740;  1, 2680, 2812;  1, 2722, 2908;  1, 2758, 2829;  1, 2768, 2799;  2, 5, 285;  2, 8, 950;  2, 9, 665;  2, 10, 236;  2, 14, 111;  2, 16, 1814;  2, 17, 1761;  2, 21, 1462;  2, 23, 123;  2, 27, 1910;  2, 28, 287;  2, 29, 2728;  2, 39, 2129;  2, 40, 2183;  2, 41, 2495;  2, 44, 2217;  2, 46, 2747;  2, 51, 323;  2, 52, 2907;  2, 57, 2729;  2, 59, 1733;  2, 65, 737;  2, 69, 1840;  2, 74, 2314;  2, 76, 1067;  2, 83, 2309;  2, 87, 1436;  2, 88, 1871;  2, 89, 2660;  2, 94, 471;  2, 95, 2750;  2, 98, 2531;  2, 100, 2549;  2, 101, 1979;  2, 104, 957;  2, 105, 1925;  2, 106, 2849;  2, 107, 122;  2, 113, 2237;  2, 119, 1469;  2, 124, 550;  2, 125, 2381;  2, 128, 2075;  2, 129, 1983;  2, 134, 1107;  2, 135, 1198;  2, 140, 371;  2, 141, 2354;  2, 142, 1499;  2, 147, 567;  2, 152, 2158;  2, 153, 1780;  2, 161, 713;  2, 164, 2009;  2, 167, 568;  2, 171, 1334;  2, 177, 760;  2, 179, 375;  2, 182, 2181;  2, 184, 1270;  2, 185, 1727;  2, 189, 1121;  2, 191, 1654;  2, 196, 1355;  2, 197, 1863;  2, 201, 2224;  2, 203, 2404;  2, 207, 2042;  2, 209, 1388;  2, 213, 964;  2, 214, 2672;  2, 218, 812;  2, 219, 771;  2, 221, 1544;  2, 224, 1438;  2, 225, 2444;  2, 227, 2363;  2, 231, 650;  2, 237, 309;  2, 243, 2558;  2, 254, 909;  2, 261, 1402;  2, 263, 2123;  2, 268, 477;  2, 273, 2134;  2, 274, 1156;  2, 275, 783;  2, 284, 2522;  2, 293, 2885;  2, 297, 2588;  2, 298, 1439;  2, 304, 1803;  2, 305, 1940;  2, 308, 2817;  2, 311, 2681;  2, 315, 2801;  2, 316, 933;  2, 321, 497;  2, 326, 1661;  2, 329, 644;  2, 335, 2165;  2, 338, 1817;  2, 340, 1600;  2, 341, 2819;  2, 345, 945;  2, 351, 2187;  2, 353, 622;  2, 356, 2285;  2, 359, 2369;  2, 362, 1725;  2, 363, 2114;  2, 368, 1713;  2, 369, 1595;  2, 377, 2079;  2, 383, 2896;  2, 386, 1953;  2, 392, 2523;  2, 393, 2499;  2, 395, 2235;  2, 398, 1491;  2, 399, 1996;  2, 401, 2367;  2, 405, 2375;  2, 406, 2303;  2, 407, 2477;  2, 413, 2162;  2, 417, 2194;  2, 418, 2597;  2, 423, 920;  2, 424, 867;  2, 425, 2914;  2, 428, 1719;  2, 435, 646;  2, 437, 741;  2, 442, 2253;  2, 446, 1557;  2, 447, 1011;  2, 449, 1348;  2, 452, 2915;  2, 453, 755;  2, 455, 2925;  2, 459, 958;  2, 461, 1276;  2, 465, 1593;  2, 466, 813;  2, 467, 1352;  2, 476, 1636;  2, 479, 2168;  2, 482, 2276;  2, 485, 647;  2, 489, 2038;  2, 495, 1611;  2, 496, 1253;  2, 501, 1001;  2, 503, 2327;  2, 508, 1283;  2, 509, 731;  2, 513, 633;  2, 515, 2073;  2, 519, 2084;  2, 521, 1565;  2, 526, 1091;  2, 530, 1449;  2, 538, 632;  2, 539, 2027;  2, 542, 1426;  2, 543, 850;  2, 545, 926;  2, 549, 1222;  2, 551, 910;  2, 560, 1383;  2, 562, 1617;  2, 563, 1042;  2, 566, 1313;  2, 569, 2121;  2, 573, 1900;  2, 575, 2024;  2, 579, 1125;  2, 581, 2643;  2, 585, 2595;  2, 597, 1811;  2, 599, 1912;  2, 603, 2241;  2, 609, 1892;  2, 615, 2264;  2, 617, 2219;  2, 623, 765;  2, 627, 1268;  2, 628, 1991;  2, 629, 2177;  2, 635, 2399;  2, 639, 1541;  2, 656, 2697;  2, 658, 1959;  2, 659, 1856;  2, 663, 2879;  2, 664, 2127;  2, 669, 842;  2, 671, 1804;  2, 675, 1226;  2, 676, 1199;  2, 677, 893;  2, 689, 1571;  2, 695, 2006;  2, 699, 773;  2, 705, 969;  2, 710, 2417;  2, 711, 1790;  2, 712, 1971;  2, 736, 2493;  2, 743, 1563;  2, 747, 1361;  2, 753, 1870;  2, 759, 2421;  2, 761, 2261;  2, 764, 2097;  2, 777, 1174;  2, 778, 2144;  2, 779, 1168;  2, 789, 1161;  2, 790, 2781;  2, 791, 1433;  2, 795, 1376;  2, 796, 1419;  2, 797, 1329;  2, 801, 1445;  2, 821, 1041;  2, 830, 2116;  2, 831, 2423;  2, 833, 2589;  2, 836, 1913;  2, 837, 1370;  2, 839, 1389;  2, 843, 2553;  2, 844, 1527;  2, 849, 1797;  2, 856, 2855;  2, 857, 1181;  2, 860, 2373;  2, 861, 1210;  2, 862, 2451;  2, 863, 1277;  2, 868, 2447;  2, 869, 1997;  2, 874, 2333;  2, 879, 1324;  2, 885, 1928;  2, 887, 1151;  2, 891, 1919;  2, 899, 929;  2, 903, 1610;  2, 905, 1336;  2, 916, 1623;  2, 917, 1193;  2, 935, 968;  2, 939, 2632;  2, 941, 1384;  2, 959, 1322;  2, 965, 1580;  2, 970, 1379;  2, 971, 2650;  2, 976, 2789;  2, 983, 2386;  2, 989, 1202;  2, 995, 1760;  2, 999, 1961;  2, 1000, 2405;  2, 1013, 2579;  2, 1018, 2771;  2, 1019, 2788;  2, 1023, 1852;  2, 1025, 1581;  2, 1029, 1397;  2, 1030, 2854;  2, 1036, 2823;  2, 1055, 2547;  2, 1077, 2644;  2, 1078, 1793;  2, 1083, 2571;  2, 1085, 1643;  2, 1088, 2243;  2, 1089, 1779;  2, 1090, 2577;  2, 1097, 2086;  2, 1103, 2176;  2, 1112, 2795;  2, 1115, 2489;  2, 1118, 2339;  2, 1119, 1233;  2, 1127, 2500;  2, 1139, 1948;  2, 1143, 2609;  2, 1144, 1640;  2, 1149, 2193;  2, 1150, 2074;  2, 1163, 1508;  2, 1167, 2122;  2, 1173, 1521;  2, 1175, 1750;  2, 1186, 1589;  2, 1191, 1756;  2, 1197, 1676;  2, 1203, 1941;  2, 1205, 1601;  2, 1217, 1319;  2, 1228, 2189;  2, 1235, 2559;  2, 1238, 2920;  2, 1245, 1533;  2, 1251, 2897;  2, 1257, 1923;  2, 1259, 1281;  2, 1263, 1547;  2, 1275, 2445;  2, 1289, 2297;  2, 1295, 2422;  2, 1299, 2164;  2, 1306, 2933;  2, 1307, 1594;  2, 1325, 1744;  2, 1331, 1751;  2, 1341, 2753;  2, 1343, 2055;  2, 1349, 1917;  2, 1353, 1442;  2, 1354, 2777;  2, 1372, 2669;  2, 1373, 1522;  2, 1377, 1791;  2, 1378, 2913;  2, 1396, 1451;  2, 1401, 2409;  2, 1407, 2255;  2, 1408, 2147;  2, 1418, 2871;  2, 1421, 2019;  2, 1425, 1973;  2, 1432, 1655;  2, 1437, 2091;  2, 1443, 2357;  2, 1444, 2794;  2, 1450, 2039;  2, 1461, 2482;  2, 1473, 2873;  2, 1475, 2429;  2, 1485, 2279;  2, 1497, 2800;  2, 1505, 2152;  2, 1511, 2811;  2, 1517, 2307;  2, 1529, 1679;  2, 1535, 2380;  2, 1539, 2908;  2, 1545, 2337;  2, 1551, 2848;  2, 1575, 2626;  2, 1577, 2608;  2, 1583, 2458;  2, 1588, 1839;  2, 1605, 1649;  2, 1607, 1967;  2, 1612, 2843;  2, 1631, 2105;  2, 1641, 1877;  2, 1660, 2463;  2, 1666, 2561;  2, 1667, 2585;  2, 1671, 1695;  2, 1673, 2770;  2, 1689, 1972;  2, 1691, 2507;  2, 1696, 2368;  2, 1697, 2045;  2, 1702, 1815;  2, 1726, 2247;  2, 1731, 2765;  2, 1738, 2625;  2, 1745, 2813;  2, 1763, 2013;  2, 1785, 1883;  2, 1821, 2619;  2, 1822, 1875;  2, 1845, 2607;  2, 1876, 2199;  2, 1889, 2211;  2, 1899, 2807;  2, 1901, 1905;  2, 1911, 2733;  2, 1930, 2320;  2, 1935, 2865;  2, 1936, 1965;  2, 1942, 2759;  2, 1943, 2433;  2, 2007, 2488;  2, 2021, 2411;  2, 2025, 2464;  2, 2069, 2441;  2, 2085, 2308;  2, 2087, 2735;  2, 2103, 2206;  2, 2109, 2291;  2, 2135, 2254;  2, 2163, 2812;  2, 2188, 2315;  2, 2207, 2213;  2, 2218, 2685;  2, 2223, 2379;  2, 2272, 2645;  2, 2319, 2860;  2, 2325, 2661;  2, 2326, 2398;  2, 2361, 2830;  2, 2387, 2613;  2, 2397, 2615;  2, 2427, 2631;  2, 2428, 2909;  2, 2453, 2799;  2, 2475, 2741;  2, 2481, 2841;  2, 2506, 2902;  2, 2543, 2895;  2, 2567, 2752;  2, 2572, 2835;  2, 2614, 2721;  2, 2637, 2687;  2, 2722, 2825;  2, 2734, 2775;  2, 2783, 2837;  3, 5, 1234;  3, 9, 1739;  3, 10, 2861;  3, 11, 1091;  3, 16, 238;  3, 17, 982;  3, 28, 1029;  3, 29, 2008;  3, 34, 2547;  3, 41, 459;  3, 45, 772;  3, 46, 2500;  3, 69, 1822;  3, 70, 867;  3, 81, 2776;  3, 89, 561;  3, 93, 184;  3, 95, 2800;  3, 100, 2291;  3, 118, 1443;  3, 125, 2680;  3, 129, 718;  3, 130, 640;  3, 135, 2128;  3, 136, 2559;  3, 142, 345;  3, 148, 2777;  3, 154, 1714;  3, 160, 945;  3, 161, 826;  3, 166, 2272;  3, 172, 2843;  3, 178, 1607;  3, 183, 388;  3, 190, 2464;  3, 191, 850;  3, 196, 2855;  3, 201, 2661;  3, 202, 2483;  3, 209, 1959;  3, 215, 940;  3, 219, 813;  3, 220, 1395;  3, 225, 1102;  3, 237, 898;  3, 244, 2890;  3, 251, 934;  3, 255, 2002;  3, 280, 2326;  3, 285, 886;  3, 292, 2308;  3, 293, 2080;  3, 304, 993;  3, 316, 2170;  3, 322, 2392;  3, 341, 2746;  3, 358, 1198;  3, 359, 1618;  3, 365, 964;  3, 370, 443;  3, 387, 796;  3, 394, 1601;  3, 399, 2818;  3, 406, 807;  3, 418, 2295;  3, 429, 2259;  3, 430, 1745;  3, 437, 1413;  3, 447, 2055;  3, 454, 2032;  3, 455, 2926;  3, 461, 2121;  3, 466, 2075;  3, 473, 1768;  3, 479, 2242;  3, 485, 646;  3, 489, 1594;  3, 495, 1846;  3, 496, 2427;  3, 515, 2153;  3, 519, 1378;  3, 526, 1167;  3, 532, 1966;  3, 537, 1312;  3, 538, 1090;  3, 563, 2566;  3, 574, 689;  3, 580, 856;  3, 585, 1845;  3, 593, 1642;  3, 610, 2476;  3, 611, 1089;  3, 621, 1587;  3, 623, 1451;  3, 628, 2603;  3, 634, 1660;  3, 651, 2710;  3, 658, 2428;  3, 665, 1203;  3, 670, 1289;  3, 675, 2878;  3, 682, 1888;  3, 688, 2692;  3, 700, 719;  3, 701, 2782;  3, 706, 791;  3, 711, 1551;  3, 712, 1941;  3, 736, 1079;  3, 742, 755;  3, 760, 1522;  3, 766, 1605;  3, 784, 1503;  3, 790, 2025;  3, 808, 838;  3, 814, 2609;  3, 837, 1906;  3, 861, 1834;  3, 869, 1252;  3, 897, 2542;  3, 921, 1960;  3, 922, 1588;  3, 928, 2068;  3, 994, 2854;  3, 1001, 2446;  3, 1018, 1930;  3, 1019, 1870;  3, 1036, 1654;  3, 1048, 2879;  3, 1067, 2735;  3, 1078, 1409;  3, 1115, 1204;  3, 1121, 2518;  3, 1127, 1655;  3, 1139, 1486;  3, 1145, 1240;  3, 1179, 2656;  3, 1199, 1210;  3, 1222, 2573;  3, 1223, 2572;  3, 1239, 2794;  3, 1264, 1546;  3, 1270, 1359;  3, 1288, 2903;  3, 1307, 2764;  3, 1324, 1426;  3, 1325, 1569;  3, 1337, 2530;  3, 1355, 2278;  3, 1367, 2116;  3, 1377, 2902;  3, 1390, 2681;  3, 1427, 2290;  3, 1462, 1810;  3, 1492, 1907;  3, 1510, 1667;  3, 1523, 2548;  3, 1589, 2189;  3, 1606, 1972;  3, 1631, 2212;  3, 1666, 1990;  3, 1685, 1744;  3, 1721, 1876;  3, 1738, 1871;  3, 1757, 2362;  3, 1798, 2645;  3, 1816, 2176;  3, 1840, 1978;  3, 1841, 2410;  3, 1924, 2801;  3, 1942, 2081;  3, 1949, 2315;  3, 1997, 2849;  3, 2014, 2038;  3, 2020, 2651;  3, 2039, 2752;  3, 2044, 2441;  3, 2074, 2627;  3, 2086, 2897;  3, 2123, 2434;  3, 2236, 2753;  3, 2273, 2806;  3, 2333, 2458;  3, 2374, 2716;  3, 2423, 2632;  3, 2584, 2740;  3, 2722, 2921;  3, 2723, 2920;  4, 11, 1109;  4, 29, 185;  4, 35, 1769;  4, 41, 1942;  4, 46, 641;  4, 47, 563;  4, 53, 550;  4, 65, 1492;  4, 70, 341;  4, 71, 2315;  4, 82, 682;  4, 88, 1859;  4, 94, 2921;  4, 101, 838;  4, 113, 1348;  4, 124, 1205;  4, 149, 2885;  4, 154, 509;  4, 179, 952;  4, 191, 1025;  4, 197, 2134;  4, 215, 238;  4, 221, 2009;  4, 245, 1139;  4, 268, 1762;  4, 287, 532;  4, 299, 856;  4, 311, 376;  4, 317, 2344;  4, 323, 1240;  4, 353, 1894;  4, 365, 646;  4, 371, 820;  4, 389, 1984;  4, 424, 2230;  4, 437, 1091;  4, 454, 2135;  4, 455, 2621;  4, 491, 1768;  4, 515, 1570;  4, 526, 1624;  4, 544, 1673;  4, 545, 1906;  4, 605, 634;  4, 652, 1972;  4, 658, 2099;  4, 683, 989;  4, 695, 1457;  4, 700, 1469;  4, 731, 1847;  4, 755, 1126;  4, 796, 1601;  4, 803, 2393;  4, 869, 2807;  4, 887, 1714;  4, 893, 1150;  4, 923, 1013;  4, 947, 994;  4, 953, 1930;  4, 971, 2297;  4, 1049, 1642;  4, 1061, 2723;  4, 1073, 2309;  4, 1097, 1871;  4, 1115, 1403;  4, 1127, 1199;  4, 1187, 2411;  4, 1204, 2615;  4, 1223, 1349;  4, 1229, 1307;  4, 1241, 2561;  4, 1265, 1715;  4, 1270, 2543;  4, 1283, 1421;  4, 1289, 2231;  4, 1373, 1546;  4, 1451, 1571;  4, 1523, 1967;  4, 1781, 2159;  4, 1877, 2057;  4, 1991, 2501;  4, 2171, 2867;  4, 2237, 2903;  4, 2321, 2525;  5, 89, 1037;  5, 359, 893;  5, 389, 2147;  5, 491, 1847;  5, 845, 1775}.
The deficiency graph is connected and has girth 4.

{\boldmath $\adfPENT(3,1484,25)$}, $d = 6$:
{0, 226, 2968;  0, 312, 1248;  0, 470, 2194;  0, 592, 2408;  0, 800, 1270;  0, 1160, 2908;  0, 1162, 2830;  0, 1256, 1554;  0, 1312, 2730;  0, 1472, 2092;  0, 1740, 2176;  0, 2542, 2876;  1, 27, 819;  1, 87, 1835;  1, 119, 453;  1, 165, 1683;  1, 265, 1441;  1, 587, 2403;  1, 801, 2525;  1, 903, 1833;  1, 1255, 2769;  1, 1523, 1739;  1, 1725, 2195;  1, 1747, 2683;  2, 228, 2970;  2, 594, 2410;  2, 1162, 2910;  2, 1164, 2832;  2, 1258, 1556;  2, 1314, 2732;  2, 1474, 2094;  2, 1742, 2178;  2, 2544, 2878;  3, 29, 821;  3, 89, 1837;  3, 121, 455;  3, 167, 1685;  3, 267, 1443;  3, 589, 2405;  3, 905, 1835;  3, 1257, 2771;  3, 1525, 1741;  4, 230, 2972;  4, 596, 2412;  4, 1164, 2912;  4, 1166, 2834;  4, 1260, 1558;  4, 1316, 2734;  4, 1476, 2096;  4, 1744, 2180;  4, 2546, 2880;  5, 31, 823;  5, 91, 1839;  5, 123, 457;  5, 169, 1687;  5, 269, 1445;  5, 591, 2407;  5, 907, 1837;  5, 1259, 2773;  5, 1527, 1743;  27, 87, 2195;  27, 119, 1523;  27, 165, 2769;  27, 265, 1739;  27, 453, 1747;  27, 587, 1683;  27, 801, 1725;  27, 903, 2525;  27, 1255, 1833;  27, 1441, 1835;  27, 2403, 2683;  29, 89, 2197;  29, 121, 1525;  29, 167, 2771;  29, 267, 1741;  29, 455, 1749;  29, 589, 1685;  29, 803, 1727;  29, 905, 2527;  29, 1257, 1835;  29, 1443, 1837;  29, 2405, 2685;  31, 91, 2199;  31, 123, 1527;  31, 169, 2773;  31, 269, 1743;  31, 457, 1751;  31, 591, 1687;  31, 805, 1729;  31, 907, 2529;  31, 1259, 1837;  31, 1445, 1839;  31, 2407, 2687;  87, 119, 587;  87, 165, 801;  87, 265, 2683;  87, 453, 2403;  87, 819, 1725;  87, 903, 1747;  87, 1255, 1441;  87, 1523, 1683;  87, 1739, 2525;  87, 1833, 2769;  89, 121, 589;  89, 167, 803;  89, 267, 2685;  89, 455, 2405;  89, 821, 1727;  89, 905, 1749;  89, 1257, 1443;  89, 1525, 1685;  89, 1741, 2527;  91, 123, 591;  91, 169, 805;  91, 269, 2687;  91, 457, 2407;  91, 823, 1729;  91, 907, 1751;  91, 1259, 1445;  91, 1527, 1687;  91, 1743, 2529;  119, 165, 1833;  119, 265, 2525;  119, 801, 2403;  119, 819, 903;  119, 1255, 1739;  119, 1441, 2683;  119, 1683, 2195;  119, 1725, 1747;  119, 1835, 2769;  121, 167, 1835;  121, 267, 2527;  121, 803, 2405;  121, 821, 905;  121, 1257, 1741;  121, 1443, 2685;  121, 1685, 2197;  121, 1727, 1749;  121, 1837, 2771;  123, 169, 1837;  123, 269, 2529;  123, 805, 2407;  123, 823, 907;  123, 1259, 1743;  123, 1445, 2687;  123, 1687, 2199;  123, 1729, 1751;  123, 1839, 2773;  165, 265, 1835;  165, 453, 1255;  165, 587, 903;  165, 819, 2403;  165, 1441, 2525;  165, 1523, 1747;  165, 1725, 2683;  165, 1739, 2195;  167, 267, 1837;  167, 455, 1257;  167, 589, 905;  167, 821, 2405;  167, 1443, 2527;  167, 1525, 1749;  167, 1727, 2685;  167, 1741, 2197;  169, 269, 1839;  169, 457, 1259;  169, 591, 907;  169, 823, 2407;  169, 1445, 2529;  169, 1527, 1751;  169, 1729, 2687;  169, 1743, 2199;  226, 470, 1554;  226, 592, 1270;  226, 800, 1312;  226, 1160, 1248;  226, 1162, 2908;  226, 1256, 2092;  226, 1740, 2408;  226, 2194, 2876;  226, 2730, 2830;  228, 472, 1556;  228, 594, 1272;  228, 802, 1314;  228, 1162, 1250;  228, 1258, 2094;  228, 1742, 2410;  228, 2196, 2878;  228, 2732, 2832;  230, 474, 1558;  230, 596, 1274;  230, 804, 1316;  230, 1164, 1252;  230, 1260, 2096;  230, 1744, 2412;  230, 2198, 2880;  230, 2734, 2834;  265, 453, 801;  265, 587, 1725;  265, 819, 2195;  265, 903, 2769;  265, 1255, 1523;  265, 1683, 1747;  265, 1833, 2403;  267, 455, 803;  267, 589, 1727;  267, 821, 2197;  267, 905, 2771;  267, 1257, 1525;  267, 1685, 1749;  267, 1835, 2405;  269, 457, 805;  269, 591, 1729;  269, 823, 2199;  269, 907, 2773;  269, 1259, 1527;  269, 1687, 1751;  269, 1837, 2407;  312, 470, 2092;  312, 592, 2730;  312, 800, 2176;  312, 1160, 1312;  312, 1162, 1256;  312, 1270, 2830;  312, 1554, 2408;  312, 1740, 2908;  312, 2194, 2542;  312, 2876, 2968;  314, 472, 2094;  314, 594, 2732;  314, 802, 2178;  314, 1162, 1314;  314, 1164, 1258;  314, 1272, 2832;  314, 1556, 2410;  314, 1742, 2910;  314, 2196, 2544;  314, 2878, 2970;  316, 474, 2096;  316, 596, 2734;  316, 804, 2180;  316, 1164, 1316;  316, 1166, 1260;  316, 1274, 2834;  316, 1558, 2412;  316, 1744, 2912;  316, 2198, 2546;  316, 2880, 2972;  453, 587, 1739;  453, 903, 1683;  453, 1441, 2195;  453, 1523, 2525;  453, 1725, 1835;  453, 1833, 2683;  455, 589, 1741;  455, 905, 1685;  455, 1443, 2197;  455, 1525, 2527;  455, 1727, 1837;  455, 1835, 2685;  457, 591, 1743;  457, 907, 1687;  457, 1445, 2199;  457, 1527, 2529;  457, 1729, 1839;  457, 1837, 2687;  470, 592, 800;  470, 1160, 2876;  470, 1162, 1312;  470, 1248, 2968;  470, 1256, 2908;  470, 1270, 1740;  470, 1472, 2730;  470, 2176, 2830;  470, 2408, 2542;  472, 594, 802;  472, 1162, 2878;  472, 1164, 1314;  472, 1250, 2970;  472, 1258, 2910;  472, 1474, 2732;  472, 2178, 2832;  472, 2410, 2544;  474, 596, 804;  474, 1164, 2880;  474, 1166, 1316;  474, 1252, 2972;  474, 1260, 2912;  474, 1476, 2734;  474, 2180, 2834;  474, 2412, 2546;  587, 801, 1441;  587, 819, 2683;  587, 1255, 2195;  587, 1523, 1835;  587, 2525, 2769;  589, 803, 1443;  589, 821, 2685;  589, 1257, 2197;  589, 2527, 2771;  591, 805, 1445;  591, 823, 2687;  591, 1259, 2199;  591, 2529, 2773;  592, 1160, 2830;  592, 1162, 2092;  592, 1248, 2876;  592, 1256, 1740;  592, 1312, 1554;  592, 1472, 2194;  592, 2176, 2968;  594, 1162, 2832;  594, 1164, 2094;  594, 1250, 2878;  594, 1258, 1742;  594, 1314, 1556;  594, 1474, 2196;  594, 2178, 2970;  596, 1164, 2834;  596, 1166, 2096;  596, 1252, 2880;  596, 1260, 1744;  596, 1316, 1558;  596, 1476, 2198;  596, 2180, 2972;  800, 1160, 1740;  800, 1162, 1554;  800, 1248, 1256;  800, 1472, 2908;  800, 2092, 2876;  800, 2194, 2830;  800, 2408, 2730;  800, 2542, 2968;  801, 819, 1255;  801, 903, 1523;  801, 1683, 2769;  801, 1739, 1747;  801, 1833, 2195;  801, 1835, 2683;  802, 1162, 1742;  802, 1164, 1556;  802, 1250, 1258;  802, 1474, 2910;  802, 2094, 2878;  802, 2196, 2832;  802, 2410, 2732;  802, 2544, 2970;  803, 821, 1257;  803, 905, 1525;  803, 1685, 2771;  803, 1741, 1749;  803, 1835, 2197;  803, 1837, 2685;  804, 1164, 1744;  804, 1166, 1558;  804, 1252, 1260;  804, 1476, 2912;  804, 2096, 2880;  804, 2198, 2834;  804, 2412, 2734;  804, 2546, 2972;  805, 823, 1259;  805, 907, 1527;  805, 1687, 2773;  805, 1743, 1751;  805, 1837, 2199;  805, 1839, 2687;  819, 1441, 1523;  819, 1683, 1835;  819, 1739, 1833;  819, 1747, 2525;  821, 1443, 1525;  821, 1685, 1837;  821, 1741, 1835;  821, 1749, 2527;  823, 1445, 1527;  823, 1687, 1839;  823, 1743, 1837;  823, 1751, 2529;  903, 1255, 2403;  903, 1441, 1725;  903, 1739, 2683;  903, 1835, 2195;  905, 1257, 2405;  905, 1443, 1727;  905, 1741, 2685;  905, 1837, 2197;  907, 1259, 2407;  907, 1445, 1729;  907, 1743, 2687;  907, 1839, 2199;  1160, 1162, 2968;  1160, 1256, 2194;  1160, 1270, 2092;  1160, 1472, 2408;  1160, 1554, 2176;  1160, 2542, 2730;  1162, 1164, 2970;  1162, 1258, 2196;  1162, 1270, 2730;  1162, 1272, 2094;  1162, 1472, 2542;  1162, 1556, 2178;  1162, 1740, 2194;  1162, 2176, 2876;  1162, 2544, 2732;  1164, 1166, 2972;  1164, 1260, 2198;  1164, 1272, 2732;  1164, 1274, 2096;  1164, 1474, 2544;  1164, 1558, 2180;  1164, 1742, 2196;  1164, 2178, 2878;  1164, 2546, 2734;  1166, 1274, 2734;  1166, 1476, 2546;  1166, 1744, 2198;  1166, 2180, 2880;  1248, 1270, 2176;  1248, 1312, 2542;  1248, 1472, 1554;  1248, 1740, 2830;  1248, 2092, 2908;  1248, 2194, 2730;  1250, 1272, 2178;  1250, 1314, 2544;  1250, 1474, 1556;  1250, 1742, 2832;  1250, 2094, 2910;  1250, 2196, 2732;  1252, 1274, 2180;  1252, 1316, 2546;  1252, 1476, 1558;  1252, 1744, 2834;  1252, 2096, 2912;  1252, 2198, 2734;  1255, 1683, 2683;  1255, 1725, 2525;  1255, 1747, 1835;  1256, 1270, 2542;  1256, 1312, 2968;  1256, 1472, 2176;  1256, 2408, 2830;  1256, 2730, 2876;  1257, 1685, 2685;  1257, 1749, 1837;  1258, 1272, 2544;  1258, 1314, 2970;  1258, 1474, 2178;  1258, 2410, 2832;  1258, 2732, 2878;  1259, 1687, 2687;  1259, 1751, 1839;  1260, 1274, 2546;  1260, 1316, 2972;  1260, 1476, 2180;  1260, 2412, 2834;  1260, 2734, 2880;  1270, 1312, 2194;  1270, 1472, 2968;  1270, 1554, 2908;  1270, 2408, 2876;  1272, 1314, 2196;  1272, 1474, 2970;  1272, 1556, 2910;  1272, 2410, 2878;  1274, 1316, 2198;  1274, 1476, 2972;  1274, 1558, 2912;  1274, 2412, 2880;  1312, 1472, 1740;  1312, 2092, 2830;  1312, 2176, 2408;  1312, 2876, 2908;  1314, 1474, 1742;  1314, 2094, 2832;  1314, 2178, 2410;  1314, 2878, 2910;  1316, 1476, 1744;  1316, 2096, 2834;  1316, 2180, 2412;  1316, 2880, 2912;  1441, 1683, 1833;  1441, 1739, 2769;  1441, 1747, 2403;  1443, 1685, 1835;  1443, 1741, 2771;  1443, 1749, 2405;  1445, 1687, 1837;  1445, 1743, 2773;  1445, 1751, 2407;  1472, 2830, 2876;  1474, 2832, 2878;  1476, 2834, 2880;  1523, 1725, 1833;  1523, 2195, 2403;  1525, 1727, 1835;  1525, 2197, 2405;  1527, 1729, 1837;  1527, 2199, 2407;  1554, 1740, 2876;  1554, 2092, 2730;  1554, 2194, 2968;  1554, 2542, 2830;  1556, 1742, 2878;  1556, 2094, 2732;  1556, 2196, 2970;  1556, 2544, 2832;  1558, 1744, 2880;  1558, 2096, 2734;  1558, 2198, 2972;  1558, 2546, 2834;  1683, 1725, 1739;  1683, 2403, 2525;  1685, 1727, 1741;  1685, 2405, 2527;  1687, 1729, 1743;  1687, 2407, 2529;  1739, 1835, 2403;  1740, 2092, 2542;  1740, 2730, 2968;  1741, 1837, 2405;  1742, 2094, 2544;  1742, 2732, 2970;  1743, 1839, 2407;  1744, 2096, 2546;  1744, 2734, 2972;  1747, 2195, 2769;  1749, 2197, 2771;  1751, 2199, 2773;  1833, 1835, 2525;  1835, 1837, 2527;  1837, 1839, 2529;  2092, 2176, 2194;  2092, 2408, 2968;  2094, 2178, 2196;  2094, 2410, 2970;  2096, 2180, 2198;  2096, 2412, 2972;  2176, 2730, 2908;  2178, 2732, 2910;  2180, 2734, 2912;  2194, 2408, 2908;  2195, 2525, 2683;  2196, 2410, 2910;  2197, 2527, 2685;  2198, 2412, 2912;  2199, 2529, 2687;  2830, 2908, 2968;  2832, 2910, 2970;  2834, 2912, 2972;  0, 3, 1598;  0, 4, 1454;  0, 5, 716;  0, 6, 987;  0, 7, 2567;  0, 9, 2241;  0, 10, 2749;  0, 11, 1690;  0, 12, 1591;  0, 13, 1686;  0, 15, 830;  0, 16, 1820;  0, 17, 197;  0, 19, 2082;  0, 20, 1525;  0, 21, 2885;  0, 23, 1563;  0, 24, 1911;  0, 25, 2871;  0, 28, 2960;  0, 29, 1755;  0, 30, 765;  0, 31, 633;  0, 33, 1540;  0, 34, 1378;  0, 35, 641;  0, 36, 2164;  0, 37, 1892;  0, 38, 1609;  0, 39, 1715;  0, 40, 1885;  0, 41, 2589;  0, 43, 2618;  0, 44, 1390;  0, 45, 1350;  0, 47, 2390;  0, 48, 1753;  0, 49, 614;  0, 50, 204;  0, 51, 2698;  0, 52, 2161;  0, 53, 2310;  0, 54, 711;  0, 55, 1127;  0, 57, 1407;  0, 58, 1490;  0, 59, 2010;  0, 61, 2716;  0, 62, 852;  0, 63, 2598;  0, 65, 2601;  0, 66, 1275;  0, 67, 937;  0, 68, 1624;  0, 69, 1110;  0, 70, 627;  0, 71, 263;  0, 72, 835;  0, 73, 1863;  0, 74, 1721;  0, 75, 769;  0, 76, 1865;  0, 77, 1076;  0, 79, 434;  0, 80, 2411;  0, 81, 1279;  0, 83, 2636;  0, 85, 2371;  0, 89, 1595;  0, 90, 791;  0, 91, 563;  0, 93, 2881;  0, 95, 381;  0, 97, 1171;  0, 98, 2175;  0, 99, 136;  0, 101, 2631;  0, 103, 2612;  0, 104, 2079;  0, 105, 1895;  0, 106, 1381;  0, 107, 838;  0, 109, 1083;  0, 111, 432;  0, 112, 2801;  0, 113, 1009;  0, 114, 2095;  0, 115, 2046;  0, 116, 433;  0, 117, 2712;  0, 120, 2087;  0, 121, 1805;  0, 123, 951;  0, 124, 562;  0, 125, 1728;  0, 126, 811;  0, 127, 2530;  0, 128, 2309;  0, 129, 2134;  0, 130, 319;  0, 131, 1135;  0, 132, 367;  0, 133, 1316;  0, 135, 2655;  0, 137, 2637;  0, 139, 259;  0, 140, 2083;  0, 141, 1485;  0, 142, 1677;  0, 143, 191;  0, 144, 661;  0, 145, 250;  0, 147, 889;  0, 148, 1643;  0, 149, 440;  0, 151, 1696;  0, 153, 2089;  0, 154, 2443;  0, 155, 292;  0, 156, 2325;  0, 157, 2385;  0, 159, 821;  0, 161, 258;  0, 162, 343;  0, 163, 2319;  0, 166, 2343;  0, 167, 2837;  0, 168, 1881;  0, 169, 2117;  0, 170, 2059;  0, 171, 1102;  0, 172, 2199;  0, 173, 758;  0, 174, 2744;  0, 175, 412;  0, 176, 2273;  0, 177, 2608;  0, 179, 1020;  0, 180, 1423;  0, 182, 2547;  0, 183, 2444;  0, 184, 1449;  0, 185, 2348;  0, 187, 1812;  0, 189, 1943;  0, 190, 1119;  0, 192, 916;  0, 193, 2941;  0, 194, 531;  0, 195, 215;  0, 196, 2111;  0, 198, 625;  0, 199, 1052;  0, 200, 1659;  0, 201, 1789;  0, 203, 1750;  0, 205, 901;  0, 206, 968;  0, 207, 623;  0, 209, 2700;  0, 210, 907;  0, 211, 674;  0, 212, 767;  0, 213, 702;  0, 217, 1455;  0, 218, 1757;  0, 219, 727;  0, 220, 2548;  0, 221, 372;  0, 222, 1348;  0, 223, 1247;  0, 225, 2695;  0, 227, 1623;  0, 228, 1453;  0, 229, 2630;  0, 230, 1395;  0, 231, 1456;  0, 233, 2519;  0, 234, 1650;  0, 236, 2221;  0, 237, 2745;  0, 239, 520;  0, 240, 726;  0, 241, 1227;  0, 243, 612;  0, 245, 1193;  0, 246, 2025;  0, 247, 2699;  0, 248, 349;  0, 249, 384;  0, 251, 2163;  0, 253, 1099;  0, 254, 2676;  0, 255, 1226;  0, 256, 550;  0, 257, 2422;  0, 260, 575;  0, 261, 2395;  0, 262, 2427;  0, 266, 439;  0, 267, 2162;  0, 269, 2534;  0, 270, 941;  0, 271, 1037;  0, 272, 2469;  0, 273, 1094;  0, 274, 1629;  0, 275, 846;  0, 276, 2283;  0, 277, 474;  0, 278, 1903;  0, 279, 2369;  0, 281, 1265;  0, 283, 881;  0, 285, 1466;  0, 286, 1417;  0, 287, 548;  0, 289, 2466;  0, 290, 2668;  0, 291, 1333;  0, 293, 1986;  0, 295, 2099;  0, 296, 1839;  0, 297, 2352;  0, 299, 1502;  0, 300, 2603;  0, 301, 1691;  0, 302, 606;  0, 303, 1467;  0, 304, 888;  0, 305, 1940;  0, 307, 2151;  0, 308, 1006;  0, 309, 736;  0, 311, 2081;  0, 313, 1263;  0, 314, 2541;  0, 315, 1262;  0, 317, 2849;  0, 320, 2657;  0, 321, 1364;  0, 323, 1744;  0, 324, 2051;  0, 325, 1788;  0, 326, 1028;  0, 327, 1055;  0, 328, 546;  0, 329, 332;  0, 331, 2460;  0, 333, 1489;  0, 335, 2813;  0, 336, 2953;  0, 337, 615;  0, 339, 1431;  0, 340, 1208;  0, 341, 2705;  0, 342, 1897;  0, 345, 757;  0, 346, 2219;  0, 347, 1692;  0, 350, 893;  0, 351, 861;  0, 353, 1058;  0, 354, 956;  0, 355, 2264;  0, 356, 1429;  0, 357, 613;  0, 358, 1341;  0, 359, 2536;  0, 361, 2990;  0, 363, 477;  0, 364, 2555;  0, 365, 808;  0, 368, 2552;  0, 369, 2247;  0, 370, 1697;  0, 371, 1616;  0, 373, 2308;  0, 374, 2661;  0, 375, 582;  0, 376, 1025;  0, 377, 2775;  0, 378, 1607;  0, 379, 2728;  0, 380, 741;  0, 382, 1899;  0, 383, 1955;  0, 385, 603;  0, 386, 1433;  0, 387, 909;  0, 388, 1493;  0, 389, 979;  0, 391, 1196;  0, 393, 1551;  0, 395, 2028;  0, 397, 1311;  0, 400, 2121;  0, 401, 1383;  0, 402, 983;  0, 403, 1732;  0, 404, 851;  0, 405, 688;  0, 406, 1561;  0, 407, 1323;  0, 408, 2215;  0, 409, 1507;  0, 410, 621;  0, 411, 2110;  0, 413, 1804;  0, 414, 2507;  0, 415, 2825;  0, 416, 2477;  0, 417, 961;  0, 418, 2077;  0, 419, 1552;  0, 420, 2019;  0, 421, 616;  0, 423, 925;  0, 424, 1966;  0, 425, 2379;  0, 429, 1631;  0, 431, 1926;  0, 435, 2746;  0, 437, 2449;  0, 438, 1559;  0, 441, 2973;  0, 442, 2571;  0, 443, 2665;  0, 444, 2015;  0, 445, 1080;  0, 446, 482;  0, 447, 558;  0, 449, 943;  0, 451, 2055;  0, 455, 2034;  0, 457, 1802;  0, 458, 744;  0, 460, 1813;  0, 461, 1486;  0, 462, 2979;  0, 463, 1996;  0, 464, 2713;  0, 465, 1007;  0, 466, 1195;  0, 467, 1046;  0, 469, 539;  0, 471, 2851;  0, 472, 544;  0, 473, 1800;  0, 475, 1343;  0, 478, 2485;  0, 479, 2945;  0, 480, 2963;  0, 481, 2903;  0, 483, 2582;  0, 485, 1218;  0, 487, 2793;  0, 489, 2486;  0, 491, 823;  0, 493, 1886;  0, 494, 1259;  0, 495, 1220;  0, 497, 2721;  0, 498, 1173;  0, 499, 1337;  0, 501, 529;  0, 502, 2689;  0, 504, 1573;  0, 505, 2233;  0, 506, 1103;  0, 507, 1268;  0, 508, 2971;  0, 509, 1613;  0, 510, 2931;  0, 511, 2271;  0, 513, 2286;  0, 514, 878;  0, 515, 2621;  0, 516, 2539;  0, 518, 809;  0, 519, 1803;  0, 521, 2727;  0, 522, 1594;  0, 523, 2439;  0, 524, 1517;  0, 525, 980;  0, 526, 1073;  0, 527, 1795;  0, 530, 921;  0, 532, 2432;  0, 533, 2462;  0, 535, 652;  0, 537, 1657;  0, 540, 2445;  0, 541, 1873;  0, 542, 2357;  0, 543, 2779;  0, 545, 989;  0, 547, 2362;  0, 549, 1999;  0, 551, 2136;  0, 552, 1695;  0, 553, 1157;  0, 555, 1754;  0, 557, 1305;  0, 559, 2663;  0, 561, 1589;  0, 564, 1688;  0, 565, 2400;  0, 566, 1503;  0, 567, 729;  0, 569, 759;  0, 571, 2143;  0, 573, 2758;  0, 577, 696;  0, 579, 1957;  0, 583, 2279;  0, 584, 2132;  0, 585, 1470;  0, 589, 995;  0, 590, 1662;  0, 591, 1709;  0, 593, 2576;  0, 595, 669;  0, 596, 2313;  0, 597, 873;  0, 598, 2035;  0, 599, 2858;  0, 600, 1784;  0, 601, 1045;  0, 604, 2012;  0, 605, 2361;  0, 607, 2870;  0, 608, 870;  0, 609, 1932;  0, 610, 1297;  0, 611, 2383;  0, 617, 1851;  0, 619, 655;  0, 624, 2115;  0, 626, 2584;  0, 628, 2709;  0, 629, 1772;  0, 630, 2347;  0, 631, 2009;  0, 632, 1061;  0, 635, 2365;  0, 637, 1553;  0, 639, 1952;  0, 643, 1991;  0, 644, 1313;  0, 645, 1550;  0, 646, 2845;  0, 647, 2101;  0, 648, 1831;  0, 649, 914;  0, 650, 2470;  0, 651, 2038;  0, 653, 1379;  0, 658, 2704;  0, 659, 1118;  0, 660, 1909;  0, 662, 1906;  0, 663, 2084;  0, 665, 1556;  0, 666, 2867;  0, 667, 1891;  0, 670, 2415;  0, 673, 839;  0, 676, 1539;  0, 677, 2841;  0, 679, 2681;  0, 680, 1495;  0, 681, 2001;  0, 683, 2853;  0, 686, 2212;  0, 687, 1432;  0, 689, 1889;  0, 691, 950;  0, 693, 2866;  0, 694, 827;  0, 695, 929;  0, 698, 1012;  0, 699, 1329;  0, 703, 1530;  0, 705, 1300;  0, 706, 2501;  0, 707, 2468;  0, 709, 1492;  0, 712, 1581;  0, 713, 2729;  0, 715, 2701;  0, 717, 1363;  0, 718, 857;  0, 719, 1023;  0, 721, 1224;  0, 723, 969;  0, 725, 2950;  0, 728, 1915;  0, 730, 1385;  0, 731, 2905;  0, 733, 2397;  0, 739, 2389;  0, 740, 2614;  0, 742, 1775;  0, 743, 770;  0, 745, 1786;  0, 746, 2179;  0, 747, 1646;  0, 748, 1018;  0, 749, 1810;  0, 750, 2633;  0, 751, 1527;  0, 752, 2891;  0, 753, 1166;  0, 755, 2553;  0, 760, 1737;  0, 761, 1902;  0, 762, 1841;  0, 766, 1284;  0, 768, 2821;  0, 771, 1180;  0, 772, 1001;  0, 773, 2647;  0, 775, 2878;  0, 776, 1295;  0, 777, 2819;  0, 779, 1963;  0, 781, 2177;  0, 782, 977;  0, 783, 2393;  0, 785, 1860;  0, 787, 1146;  0, 788, 1792;  0, 789, 2299;  0, 790, 2453;  0, 793, 1993;  0, 794, 1941;  0, 795, 1201;  0, 796, 853;  0, 797, 2944;  0, 798, 2677;  0, 799, 2044;  0, 803, 841;  0, 804, 1635;  0, 805, 2512;  0, 806, 1347;  0, 807, 847;  0, 810, 2183;  0, 813, 2306;  0, 814, 1147;  0, 815, 1293;  0, 820, 2331;  0, 824, 1461;  0, 825, 2459;  0, 828, 2285;  0, 829, 2154;  0, 832, 2353;  0, 833, 868;  0, 834, 2599;  0, 837, 2753;  0, 843, 1532;  0, 845, 1123;  0, 849, 1451;  0, 855, 2100;  0, 859, 2731;  0, 860, 2645;  0, 862, 2855;  0, 863, 1190;  0, 866, 1111;  0, 867, 1346;  0, 869, 1951;  0, 871, 2771;  0, 872, 1815;  0, 874, 1928;  0, 875, 1759;  0, 877, 2236;  0, 879, 1036;  0, 883, 1402;  0, 884, 2791;  0, 885, 1603;  0, 887, 2094;  0, 890, 2499;  0, 891, 1021;  0, 892, 933;  0, 895, 971;  0, 896, 1053;  0, 897, 996;  0, 899, 2675;  0, 904, 2764;  0, 905, 1678;  0, 908, 2529;  0, 910, 2611;  0, 911, 986;  0, 913, 1200;  0, 915, 2141;  0, 917, 2899;  0, 919, 1823;  0, 923, 2475;  0, 926, 1419;  0, 927, 2102;  0, 935, 2932;  0, 942, 2925;  0, 945, 1165;  0, 947, 2978;  0, 949, 1469;  0, 952, 1801;  0, 953, 2071;  0, 954, 2985;  0, 955, 1245;  0, 957, 2879;  0, 959, 2500;  0, 963, 1970;  0, 965, 1526;  0, 967, 2229;  0, 970, 2127;  0, 972, 2339;  0, 973, 2419;  0, 974, 1142;  0, 975, 2803;  0, 976, 2259;  0, 978, 2523;  0, 982, 1370;  0, 985, 2318;  0, 991, 1743;  0, 992, 1506;  0, 993, 1422;  0, 994, 2323;  0, 997, 1869;  0, 998, 1985;  0, 999, 1756;  0, 1003, 1916;  0, 1004, 2447;  0, 1005, 2654;  0, 1010, 2129;  0, 1011, 2057;  0, 1013, 2626;  0, 1015, 1840;  0, 1017, 1250;  0, 1019, 2300;  0, 1024, 2209;  0, 1027, 2171;  0, 1029, 2003;  0, 1031, 2091;  0, 1033, 1475;  0, 1035, 1515;  0, 1038, 2119;  0, 1039, 1654;  0, 1040, 1917;  0, 1041, 2384;  0, 1042, 1314;  0, 1047, 2224;  0, 1048, 1401;  0, 1049, 1077;  0, 1050, 2288;  0, 1051, 1320;  0, 1054, 1131;  0, 1057, 1773;  0, 1059, 1849;  0, 1060, 2005;  0, 1065, 2266;  0, 1066, 2185;  0, 1067, 2935;  0, 1071, 2027;  0, 1074, 2381;  0, 1075, 1937;  0, 1078, 1339;  0, 1082, 2246;  0, 1085, 1213;  0, 1087, 1244;  0, 1088, 2715;  0, 1089, 1397;  0, 1091, 1852;  0, 1093, 2818;  0, 1095, 1989;  0, 1097, 1699;  0, 1098, 2463;  0, 1100, 2435;  0, 1101, 1733;  0, 1104, 2276;  0, 1105, 1936;  0, 1106, 1701;  0, 1107, 1231;  0, 1108, 1508;  0, 1109, 2936;  0, 1113, 1694;  0, 1114, 1125;  0, 1115, 2257;  0, 1116, 2809;  0, 1117, 2854;  0, 1120, 1611;  0, 1122, 2969;  0, 1129, 1651;  0, 1132, 1335;  0, 1133, 1546;  0, 1137, 1790;  0, 1139, 2198;  0, 1140, 2489;  0, 1141, 1170;  0, 1144, 1639;  0, 1145, 1164;  0, 1149, 2829;  0, 1150, 1443;  0, 1151, 2515;  0, 1153, 2597;  0, 1154, 1877;  0, 1155, 2265;  0, 1156, 2217;  0, 1158, 2467;  0, 1161, 2396;  0, 1163, 2332;  0, 1167, 2464;  0, 1169, 2861;  0, 1174, 2344;  0, 1175, 2125;  0, 1177, 2314;  0, 1179, 2446;  0, 1181, 1758;  0, 1185, 1931;  0, 1187, 1997;  0, 1189, 2943;  0, 1191, 2527;  0, 1192, 2755;  0, 1197, 2122;  0, 1198, 1703;  0, 1199, 1519;  0, 1202, 2609;  0, 1203, 1822;  0, 1204, 1285;  0, 1205, 1396;  0, 1207, 1330;  0, 1210, 2917;  0, 1211, 2069;  0, 1212, 2777;  0, 1215, 2659;  0, 1216, 1827;  0, 1217, 1947;  0, 1219, 2492;  0, 1222, 1439;  0, 1223, 1704;  0, 1232, 1315;  0, 1233, 1511;  0, 1234, 2409;  0, 1235, 1438;  0, 1237, 2747;  0, 1239, 1615;  0, 1240, 2000;  0, 1241, 2135;  0, 1251, 2404;  0, 1253, 2213;  0, 1257, 2926;  0, 1260, 2697;  0, 1261, 2228;  0, 1264, 2399;  0, 1267, 2239;  0, 1269, 2336;  0, 1271, 1988;  0, 1273, 2800;  0, 1277, 1483;  0, 1281, 2451;  0, 1283, 1542;  0, 1287, 2392;  0, 1289, 2937;  0, 1291, 2741;  0, 1298, 2075;  0, 1299, 2139;  0, 1303, 2711;  0, 1304, 2341;  0, 1306, 2723;  0, 1310, 2193;  0, 1317, 1468;  0, 1318, 1925;  0, 1319, 2914;  0, 1325, 1829;  0, 1327, 1548;  0, 1331, 2457;  0, 1336, 2865;  0, 1340, 2189;  0, 1345, 1627;  0, 1351, 2021;  0, 1352, 1501;  0, 1353, 2575;  0, 1355, 1462;  0, 1357, 2024;  0, 1359, 2337;  0, 1360, 1487;  0, 1362, 2767;  0, 1368, 2955;  0, 1371, 2107;  0, 1375, 1633;  0, 1377, 1536;  0, 1384, 1871;  0, 1387, 1421;  0, 1389, 2235;  0, 1393, 2578;  0, 1399, 2061;  0, 1400, 1921;  0, 1403, 1948;  0, 1406, 2789;  0, 1408, 2131;  0, 1411, 2041;  0, 1413, 2479;  0, 1425, 1679;  0, 1427, 2717;  0, 1435, 2585;  0, 1442, 2873;  0, 1444, 2613;  0, 1445, 1549;  0, 1447, 1976;  0, 1448, 2535;  0, 1450, 2807;  0, 1459, 1777;  0, 1463, 1535;  0, 1465, 2812;  0, 1471, 1658;  0, 1473, 2804;  0, 1477, 2205;  0, 1478, 2949;  0, 1479, 2387;  0, 1481, 2080;  0, 1484, 1619;  0, 1497, 1533;  0, 1504, 2543;  0, 1505, 2211;  0, 1509, 1870;  0, 1510, 2351;  0, 1513, 2301;  0, 1516, 1673;  0, 1521, 2067;  0, 1528, 2751;  0, 1529, 2918;  0, 1537, 2889;  0, 1538, 2781;  0, 1541, 2329;  0, 1543, 2495;  0, 1544, 2763;  0, 1547, 2045;  0, 1557, 1927;  0, 1562, 1929;  0, 1567, 2533;  0, 1569, 1965;  0, 1575, 1861;  0, 1577, 2516;  0, 1583, 2218;  0, 1585, 1630;  0, 1586, 2375;  0, 1588, 2521;  0, 1592, 2188;  0, 1593, 1645;  0, 1597, 2739;  0, 1601, 2248;  0, 1604, 2561;  0, 1605, 1653;  0, 1610, 2651;  0, 1617, 1761;  0, 1621, 2864;  0, 1625, 2200;  0, 1634, 2281;  0, 1637, 2833;  0, 1641, 2157;  0, 1642, 1675;  0, 1647, 2452;  0, 1648, 2085;  0, 1655, 1769;  0, 1661, 2620;  0, 1663, 1880;  0, 1664, 2913;  0, 1665, 2137;  0, 1676, 2623;  0, 1681, 2528;  0, 1684, 2522;  0, 1685, 2355;  0, 1687, 1726;  0, 1707, 2255;  0, 1711, 2225;  0, 1719, 2363;  0, 1723, 2831;  0, 1729, 2156;  0, 1730, 1933;  0, 1731, 2252;  0, 1735, 2014;  0, 1741, 2882;  0, 1745, 2174;  0, 1751, 1969;  0, 1760, 2815;  0, 1762, 2413;  0, 1763, 2877;  0, 1767, 1844;  0, 1768, 2269;  0, 1771, 2431;  0, 1774, 2734;  0, 1778, 2679;  0, 1781, 2305;  0, 1783, 2893;  0, 1785, 2002;  0, 1791, 1946;  0, 1793, 2039;  0, 1796, 2795;  0, 1797, 2090;  0, 1798, 2465;  0, 1799, 1918;  0, 1809, 2511;  0, 1811, 2324;  0, 1814, 2983;  0, 1817, 1838;  0, 1819, 2702;  0, 1821, 2245;  0, 1825, 2828;  0, 1828, 2367;  0, 1837, 1874;  0, 1843, 2455;  0, 1845, 2026;  0, 1850, 2961;  0, 1855, 2811;  0, 1857, 2180;  0, 1859, 2013;  0, 1862, 2685;  0, 1867, 2759;  0, 1868, 2869;  0, 1875, 2920;  0, 1876, 2113;  0, 1888, 2733;  0, 1893, 2924;  0, 1894, 2047;  0, 1900, 2765;  0, 1901, 2017;  0, 1907, 2433;  0, 1912, 2289;  0, 1913, 2649;  0, 1923, 2043;  0, 1934, 2545;  0, 1935, 2474;  0, 1939, 2896;  0, 1942, 2345;  0, 1945, 2666;  0, 1949, 2959;  0, 1954, 2243;  0, 1958, 2587;  0, 1959, 2919;  0, 1961, 2073;  0, 1971, 2497;  0, 1973, 2605;  0, 1975, 2029;  0, 1977, 2398;  0, 1979, 2977;  0, 1982, 2149;  0, 1984, 2293;  0, 1987, 2593;  0, 1990, 2488;  0, 1995, 2373;  0, 2008, 2287;  0, 2011, 2907;  0, 2018, 2037;  0, 2020, 2641;  0, 2033, 2577;  0, 2042, 2624;  0, 2049, 2493;  0, 2063, 2554;  0, 2065, 2991;  0, 2068, 2590;  0, 2078, 2911;  0, 2086, 2947;  0, 2097, 2315;  0, 2098, 2104;  0, 2103, 2126;  0, 2105, 2662;  0, 2116, 2401;  0, 2120, 2320;  0, 2123, 2692;  0, 2133, 2817;  0, 2145, 2155;  0, 2147, 2680;  0, 2159, 2429;  0, 2165, 2923;  0, 2170, 2921;  0, 2173, 2560;  0, 2181, 2743;  0, 2186, 2569;  0, 2187, 2974;  0, 2191, 2686;  0, 2197, 2822;  0, 2203, 2915;  0, 2206, 2669;  0, 2207, 2405;  0, 2222, 2653;  0, 2223, 2967;  0, 2227, 2487;  0, 2231, 2311;  0, 2234, 2275;  0, 2237, 2591;  0, 2242, 2719;  0, 2249, 2720;  0, 2251, 2619;  0, 2253, 2278;  0, 2254, 2984;  0, 2258, 2951;  0, 2263, 2321;  0, 2267, 2674;  0, 2270, 2595;  0, 2277, 2810;  0, 2282, 2327;  0, 2291, 2738;  0, 2295, 2594;  0, 2296, 2827;  0, 2297, 2349;  0, 2307, 2890;  0, 2317, 2993;  0, 2333, 2757;  0, 2335, 2887;  0, 2342, 2839;  0, 2350, 2581;  0, 2366, 2957;  0, 2368, 2644;  0, 2377, 2966;  0, 2378, 2638;  0, 2380, 2691;  0, 2386, 2799;  0, 2391, 2461;  0, 2407, 2783;  0, 2425, 2927;  0, 2428, 2901;  0, 2437, 2794;  0, 2441, 2981;  0, 2450, 2761;  0, 2471, 2607;  0, 2473, 2847;  0, 2481, 2805;  0, 2503, 2846;  0, 2509, 2942;  0, 2531, 2740;  0, 2537, 2852;  0, 2549, 2823;  0, 2551, 2606;  0, 2557, 2648;  0, 2559, 2956;  0, 2563, 2857;  0, 2573, 2785;  0, 2579, 2888;  0, 2583, 2588;  0, 2615, 2987;  0, 2627, 2774;  0, 2629, 2798;  0, 2639, 2737;  0, 2643, 2989;  0, 2667, 2954;  0, 2671, 2933;  0, 2687, 2909;  0, 2693, 2782;  0, 2703, 2788;  0, 2824, 2863;  0, 2929, 2939;  1, 2, 1635;  1, 4, 1874;  1, 5, 542;  1, 7, 1790;  1, 8, 1433;  1, 10, 2654;  1, 13, 2827;  1, 14, 2116;  1, 16, 728;  1, 17, 1157;  1, 20, 2027;  1, 21, 2947;  1, 22, 1693;  1, 25, 122;  1, 26, 64;  1, 28, 2101;  1, 29, 2680;  1, 31, 794;  1, 32, 2311;  1, 34, 2839;  1, 39, 308;  1, 41, 1262;  1, 44, 1815;  1, 45, 2752;  1, 50, 944;  1, 51, 896;  1, 52, 2803;  1, 53, 580;  1, 58, 1049;  1, 62, 667;  1, 63, 2882;  1, 67, 1520;  1, 68, 1943;  1, 70, 332;  1, 73, 2470;  1, 74, 2008;  1, 76, 80;  1, 81, 171;  1, 82, 2309;  1, 86, 1046;  1, 88, 632;  1, 91, 838;  1, 94, 1825;  1, 99, 2516;  1, 100, 1789;  1, 104, 327;  1, 105, 853;  1, 107, 2756;  1, 110, 1736;  1, 112, 182;  1, 113, 2084;  1, 115, 904;  1, 116, 292;  1, 117, 1832;  1, 125, 472;  1, 127, 1576;  1, 128, 1469;  1, 129, 1226;  1, 130, 1245;  1, 131, 1868;  1, 133, 2618;  1, 134, 2015;  1, 136, 1593;  1, 137, 2834;  1, 140, 419;  1, 141, 2920;  1, 142, 1563;  1, 143, 2878;  1, 145, 2723;  1, 146, 727;  1, 148, 2102;  1, 149, 2617;  1, 152, 183;  1, 154, 1414;  1, 155, 495;  1, 160, 2074;  1, 163, 2381;  1, 164, 1286;  1, 166, 2719;  1, 172, 1649;  1, 173, 1191;  1, 175, 1880;  1, 176, 2708;  1, 177, 383;  1, 178, 2529;  1, 184, 1069;  1, 185, 1093;  1, 191, 2524;  1, 194, 844;  1, 195, 371;  1, 197, 1660;  1, 199, 1664;  1, 200, 2840;  1, 201, 355;  1, 202, 2431;  1, 205, 1829;  1, 206, 2811;  1, 208, 651;  1, 211, 892;  1, 212, 2447;  1, 213, 1610;  1, 214, 601;  1, 220, 2333;  1, 221, 230;  1, 223, 2254;  1, 224, 1787;  1, 229, 1544;  1, 231, 248;  1, 232, 628;  1, 235, 1978;  1, 236, 1681;  1, 237, 2107;  1, 241, 2078;  1, 242, 885;  1, 249, 389;  1, 250, 1007;  1, 251, 2509;  1, 254, 2090;  1, 255, 386;  1, 257, 1211;  1, 262, 2249;  1, 267, 2482;  1, 268, 2152;  1, 271, 1915;  1, 272, 2497;  1, 273, 467;  1, 274, 2721;  1, 275, 686;  1, 278, 653;  1, 284, 2836;  1, 286, 2137;  1, 287, 2653;  1, 290, 1947;  1, 293, 478;  1, 296, 1183;  1, 297, 1502;  1, 298, 2853;  1, 301, 1796;  1, 302, 1757;  1, 303, 2596;  1, 304, 1209;  1, 305, 2485;  1, 309, 891;  1, 310, 2227;  1, 314, 2521;  1, 315, 1304;  1, 316, 2009;  1, 320, 2204;  1, 321, 2024;  1, 322, 783;  1, 325, 2890;  1, 326, 2265;  1, 328, 1457;  1, 329, 1449;  1, 333, 670;  1, 334, 2537;  1, 337, 1923;  1, 338, 2648;  1, 340, 2468;  1, 341, 2461;  1, 346, 2198;  1, 347, 2114;  1, 350, 529;  1, 351, 578;  1, 352, 484;  1, 357, 1384;  1, 359, 2936;  1, 362, 2455;  1, 364, 2444;  1, 365, 1034;  1, 368, 2283;  1, 370, 1879;  1, 373, 1190;  1, 374, 622;  1, 376, 1777;  1, 380, 2056;  1, 381, 1634;  1, 382, 1592;  1, 385, 1322;  1, 387, 1222;  1, 392, 507;  1, 394, 1517;  1, 397, 2168;  1, 398, 604;  1, 400, 1041;  1, 401, 1430;  1, 403, 1105;  1, 404, 994;  1, 405, 2138;  1, 406, 2025;  1, 409, 2380;  1, 410, 2289;  1, 411, 2181;  1, 412, 479;  1, 413, 1024;  1, 415, 2600;  1, 416, 1119;  1, 417, 926;  1, 418, 1814;  1, 421, 2098;  1, 422, 2270;  1, 424, 2972;  1, 425, 1112;  1, 430, 1855;  1, 433, 2595;  1, 435, 1730;  1, 436, 1982;  1, 439, 1102;  1, 440, 1851;  1, 441, 718;  1, 446, 573;  1, 447, 833;  1, 448, 2017;  1, 452, 2534;  1, 454, 1803;  1, 458, 2150;  1, 459, 2258;  1, 460, 1364;  1, 461, 1382;  1, 463, 1742;  1, 465, 2704;  1, 466, 2321;  1, 476, 695;  1, 481, 1180;  1, 482, 1165;  1, 483, 2278;  1, 488, 659;  1, 490, 544;  1, 494, 2615;  1, 500, 2167;  1, 502, 1114;  1, 505, 2302;  1, 506, 835;  1, 508, 1363;  1, 509, 1966;  1, 512, 1498;  1, 514, 1945;  1, 517, 2738;  1, 518, 2733;  1, 519, 2560;  1, 524, 2325;  1, 526, 2585;  1, 531, 1731;  1, 533, 1456;  1, 536, 1574;  1, 538, 2306;  1, 543, 2667;  1, 545, 1154;  1, 547, 2522;  1, 548, 2489;  1, 549, 2081;  1, 550, 2553;  1, 551, 971;  1, 554, 1489;  1, 556, 2888;  1, 559, 2456;  1, 560, 2175;  1, 562, 596;  1, 563, 2218;  1, 567, 1702;  1, 568, 2197;  1, 572, 2396;  1, 574, 694;  1, 583, 1307;  1, 584, 710;  1, 585, 964;  1, 586, 2044;  1, 591, 649;  1, 595, 2481;  1, 597, 932;  1, 598, 771;  1, 602, 2624;  1, 608, 1005;  1, 609, 2300;  1, 610, 1462;  1, 611, 2236;  1, 614, 1919;  1, 617, 1029;  1, 620, 2906;  1, 625, 1636;  1, 627, 2174;  1, 633, 1016;  1, 634, 1871;  1, 638, 2558;  1, 640, 2353;  1, 644, 2132;  1, 645, 1960;  1, 646, 1204;  1, 647, 2774;  1, 650, 2385;  1, 652, 2919;  1, 656, 1841;  1, 658, 2188;  1, 662, 1508;  1, 674, 2199;  1, 675, 2960;  1, 676, 1039;  1, 680, 1886;  1, 681, 1479;  1, 687, 1640;  1, 688, 1318;  1, 689, 2740;  1, 692, 1695;  1, 698, 831;  1, 699, 2488;  1, 704, 1159;  1, 707, 1089;  1, 712, 2726;  1, 719, 842;  1, 724, 1933;  1, 730, 1948;  1, 731, 982;  1, 734, 2530;  1, 736, 2896;  1, 740, 2859;  1, 741, 1984;  1, 742, 2639;  1, 743, 1340;  1, 745, 2588;  1, 747, 1522;  1, 749, 2282;  1, 751, 2921;  1, 752, 1589;  1, 754, 1021;  1, 758, 1533;  1, 760, 2924;  1, 761, 2900;  1, 763, 1863;  1, 772, 2555;  1, 776, 2120;  1, 778, 1606;  1, 782, 1214;  1, 791, 2294;  1, 795, 2422;  1, 797, 946;  1, 802, 2041;  1, 805, 2768;  1, 807, 866;  1, 809, 2234;  1, 811, 2810;  1, 812, 1354;  1, 814, 2492;  1, 820, 868;  1, 824, 2944;  1, 825, 2224;  1, 830, 1946;  1, 836, 1679;  1, 841, 2242;  1, 850, 1040;  1, 856, 2942;  1, 860, 1588;  1, 862, 1487;  1, 867, 1372;  1, 872, 940;  1, 874, 1269;  1, 878, 1543;  1, 879, 1967;  1, 880, 2019;  1, 890, 2687;  1, 898, 1894;  1, 901, 2047;  1, 908, 2126;  1, 909, 913;  1, 910, 1210;  1, 911, 1690;  1, 916, 2206;  1, 920, 1195;  1, 928, 1251;  1, 934, 1869;  1, 943, 2480;  1, 950, 2477;  1, 952, 1276;  1, 956, 2825;  1, 961, 1934;  1, 962, 1881;  1, 969, 2213;  1, 970, 2631;  1, 976, 1258;  1, 977, 1718;  1, 980, 1331;  1, 981, 2782;  1, 985, 2165;  1, 986, 2824;  1, 992, 2855;  1, 993, 1418;  1, 995, 1216;  1, 997, 2367;  1, 998, 2979;  1, 999, 1822;  1, 1000, 2066;  1, 1006, 2225;  1, 1011, 2741;  1, 1018, 2086;  1, 1019, 2549;  1, 1022, 1305;  1, 1028, 2121;  1, 1030, 1115;  1, 1036, 1850;  1, 1037, 1172;  1, 1043, 2379;  1, 1047, 2452;  1, 1048, 1861;  1, 1053, 2128;  1, 1054, 2854;  1, 1055, 1907;  1, 1058, 1460;  1, 1060, 1311;  1, 1061, 1492;  1, 1064, 2030;  1, 1066, 1459;  1, 1067, 1892;  1, 1070, 1094;  1, 1072, 1385;  1, 1076, 2804;  1, 1077, 2201;  1, 1079, 1954;  1, 1082, 1291;  1, 1083, 2343;  1, 1084, 2284;  1, 1090, 1688;  1, 1095, 1227;  1, 1096, 1759;  1, 1100, 1540;  1, 1103, 2932;  1, 1107, 1670;  1, 1108, 2705;  1, 1118, 1996;  1, 1120, 1403;  1, 1124, 1525;  1, 1126, 1455;  1, 1130, 1337;  1, 1132, 2021;  1, 1133, 2320;  1, 1136, 1646;  1, 1148, 2416;  1, 1155, 2164;  1, 1156, 2866;  1, 1166, 1970;  1, 1168, 1763;  1, 1173, 2410;  1, 1174, 2613;  1, 1175, 2636;  1, 1178, 2950;  1, 1185, 1556;  1, 1192, 1691;  1, 1193, 2399;  1, 1196, 2609;  1, 1197, 2966;  1, 1198, 1378;  1, 1199, 2860;  1, 1202, 2036;  1, 1203, 1369;  1, 1205, 2535;  1, 1208, 1438;  1, 1213, 2780;  1, 1217, 2990;  1, 1220, 2765;  1, 1221, 1930;  1, 1228, 1531;  1, 1232, 2426;  1, 1233, 2691;  1, 1234, 2858;  1, 1235, 2362;  1, 1238, 2735;  1, 1240, 1846;  1, 1241, 2038;  1, 1250, 1990;  1, 1252, 1579;  1, 1261, 2770;  1, 1268, 2746;  1, 1285, 2626;  1, 1292, 2013;  1, 1298, 1598;  1, 1299, 2690;  1, 1300, 2392;  1, 1301, 2805;  1, 1309, 2984;  1, 1310, 1347;  1, 1317, 1480;  1, 1319, 1526;  1, 1321, 2696;  1, 1341, 1552;  1, 1352, 2018;  1, 1358, 2812;  1, 1361, 2650;  1, 1365, 2577;  1, 1388, 2434;  1, 1390, 2591;  1, 1396, 2279;  1, 1400, 2572;  1, 1401, 1672;  1, 1406, 1611;  1, 1408, 2345;  1, 1412, 2799;  1, 1420, 2216;  1, 1436, 2369;  1, 1439, 2192;  1, 1442, 1766;  1, 1443, 2246;  1, 1444, 2391;  1, 1448, 1983;  1, 1463, 2528;  1, 1467, 2026;  1, 1468, 1504;  1, 1478, 2080;  1, 1484, 1918;  1, 1491, 2614;  1, 1493, 2776;  1, 1503, 1604;  1, 1505, 2698;  1, 1509, 2006;  1, 1514, 1748;  1, 1516, 2465;  1, 1527, 1959;  1, 1529, 2948;  1, 1532, 2223;  1, 1539, 2255;  1, 1547, 2368;  1, 1553, 2458;  1, 1557, 2476;  1, 1580, 1870;  1, 1582, 2699;  1, 1586, 2536;  1, 1595, 1882;  1, 1599, 1654;  1, 1600, 1820;  1, 1612, 1768;  1, 1616, 2387;  1, 1618, 1732;  1, 1622, 2974;  1, 1624, 2813;  1, 1628, 1900;  1, 1630, 2744;  1, 1643, 2012;  1, 1647, 1988;  1, 1648, 2548;  1, 1652, 2627;  1, 1655, 2722;  1, 1658, 2042;  1, 1677, 2144;  1, 1682, 2054;  1, 1684, 2428;  1, 1689, 2397;  1, 1694, 1769;  1, 1696, 2745;  1, 1700, 2552;  1, 1712, 2187;  1, 1714, 2297;  1, 1724, 2531;  1, 1733, 2889;  1, 1745, 2212;  1, 1750, 2158;  1, 1751, 2644;  1, 1754, 2506;  1, 1760, 2096;  1, 1762, 1821;  1, 1774, 2819;  1, 1775, 2156;  1, 1778, 2243;  1, 1779, 2637;  1, 1780, 2674;  1, 1785, 2127;  1, 1792, 2252;  1, 1793, 2870;  1, 1802, 2475;  1, 1804, 2248;  1, 1823, 1955;  1, 1828, 2360;  1, 1834, 2002;  1, 1852, 2378;  1, 1858, 2681;  1, 1862, 2123;  1, 1888, 2290;  1, 1889, 2801;  1, 1893, 2951;  1, 1895, 2450;  1, 1904, 2170;  1, 1906, 1989;  1, 1912, 2122;  1, 1916, 2189;  1, 1924, 2621;  1, 1928, 2429;  1, 1935, 2326;  1, 1937, 2296;  1, 1941, 2267;  1, 1942, 2288;  1, 1952, 2517;  1, 1971, 2260;  1, 1977, 2818;  1, 2001, 2660;  1, 2014, 2986;  1, 2020, 2554;  1, 2039, 2884;  1, 2043, 2438;  1, 2068, 2513;  1, 2069, 2372;  1, 2072, 2933;  1, 2079, 2669;  1, 2085, 2135;  1, 2091, 2788;  1, 2103, 2980;  1, 2117, 2918;  1, 2129, 2200;  1, 2133, 2620;  1, 2210, 2319;  1, 2228, 2606;  1, 2230, 2645;  1, 2235, 2872;  1, 2240, 2927;  1, 2271, 2655;  1, 2276, 2823;  1, 2314, 2732;  1, 2315, 2440;  1, 2318, 2978;  1, 2324, 2337;  1, 2336, 2864;  1, 2338, 2883;  1, 2342, 2926;  1, 2351, 2607;  1, 2354, 2852;  1, 2356, 2975;  1, 2386, 2800;  1, 2420, 2445;  1, 2432, 2546;  1, 2446, 2462;  1, 2463, 2702;  1, 2504, 2759;  1, 2512, 2848;  1, 2570, 2902;  1, 2578, 2714;  1, 2582, 2762;  1, 2584, 2795;  1, 2590, 2693;  1, 2638, 2945;  1, 2642, 2703;  1, 2672, 2961;  1, 2729, 2794;  1, 2747, 2992;  1, 2798, 2930;  2, 5, 1234;  2, 8, 609;  2, 9, 2531;  2, 11, 1053;  2, 14, 2807;  2, 17, 1857;  2, 22, 611;  2, 23, 993;  2, 32, 713;  2, 35, 1976;  2, 41, 1199;  2, 45, 249;  2, 46, 2781;  2, 50, 2739;  2, 51, 2241;  2, 52, 1299;  2, 53, 1233;  2, 56, 1539;  2, 57, 797;  2, 59, 2787;  2, 65, 2482;  2, 68, 765;  2, 69, 2968;  2, 71, 1142;  2, 74, 2565;  2, 75, 976;  2, 76, 657;  2, 81, 242;  2, 82, 2355;  2, 83, 2381;  2, 87, 602;  2, 92, 1359;  2, 93, 1418;  2, 95, 2164;  2, 99, 2260;  2, 100, 425;  2, 101, 1106;  2, 105, 333;  2, 106, 2051;  2, 107, 322;  2, 113, 596;  2, 118, 783;  2, 119, 1311;  2, 122, 2799;  2, 123, 2949;  2, 125, 184;  2, 130, 873;  2, 131, 539;  2, 141, 1985;  2, 142, 1353;  2, 143, 2390;  2, 146, 533;  2, 147, 1725;  2, 149, 1839;  2, 153, 2956;  2, 155, 915;  2, 158, 1271;  2, 161, 975;  2, 164, 1145;  2, 165, 1796;  2, 171, 2576;  2, 172, 2326;  2, 176, 2427;  2, 179, 2879;  2, 183, 1329;  2, 185, 1988;  2, 189, 2768;  2, 191, 2501;  2, 194, 501;  2, 195, 2639;  2, 196, 795;  2, 200, 2325;  2, 201, 2201;  2, 203, 2819;  2, 206, 2633;  2, 212, 1017;  2, 214, 2645;  2, 215, 1534;  2, 219, 412;  2, 220, 677;  2, 224, 2597;  2, 227, 2361;  2, 231, 2246;  2, 233, 785;  2, 237, 473;  2, 238, 2450;  2, 239, 479;  2, 243, 2231;  2, 245, 1688;  2, 248, 1893;  2, 251, 418;  2, 255, 2446;  2, 256, 1047;  2, 260, 2901;  2, 261, 2657;  2, 269, 2729;  2, 272, 1677;  2, 273, 1265;  2, 278, 1077;  2, 279, 2505;  2, 280, 1594;  2, 284, 1946;  2, 287, 2519;  2, 296, 1769;  2, 297, 1510;  2, 298, 1133;  2, 303, 526;  2, 304, 1869;  2, 305, 1137;  2, 310, 1059;  2, 311, 2620;  2, 315, 2800;  2, 320, 1667;  2, 321, 388;  2, 323, 2933;  2, 328, 2027;  2, 329, 2621;  2, 335, 2631;  2, 341, 2591;  2, 344, 2553;  2, 345, 2157;  2, 351, 1895;  2, 357, 2003;  2, 358, 2915;  2, 359, 2872;  2, 365, 2885;  2, 370, 2722;  2, 371, 1487;  2, 375, 2909;  2, 376, 2021;  2, 381, 1262;  2, 382, 1451;  2, 383, 2169;  2, 387, 2319;  2, 395, 1887;  2, 398, 2351;  2, 401, 950;  2, 405, 1090;  2, 406, 1509;  2, 407, 743;  2, 410, 1331;  2, 411, 465;  2, 413, 956;  2, 416, 2069;  2, 423, 2543;  2, 429, 892;  2, 435, 682;  2, 437, 647;  2, 440, 2366;  2, 441, 860;  2, 446, 1904;  2, 447, 2516;  2, 453, 1898;  2, 459, 2187;  2, 460, 2423;  2, 461, 626;  2, 466, 2967;  2, 471, 2019;  2, 476, 2913;  2, 483, 1673;  2, 484, 621;  2, 485, 2354;  2, 488, 1373;  2, 489, 1352;  2, 491, 2272;  2, 496, 1781;  2, 497, 2385;  2, 506, 1757;  2, 507, 1084;  2, 508, 2444;  2, 509, 1809;  2, 513, 2830;  2, 515, 2675;  2, 518, 1397;  2, 519, 2656;  2, 520, 2943;  2, 524, 1613;  2, 525, 2469;  2, 527, 581;  2, 531, 1097;  2, 532, 687;  2, 536, 2031;  2, 542, 1155;  2, 550, 2843;  2, 551, 2668;  2, 555, 1078;  2, 560, 1731;  2, 561, 1454;  2, 563, 759;  2, 566, 1779;  2, 568, 747;  2, 569, 1732;  2, 573, 1011;  2, 574, 2067;  2, 575, 1214;  2, 585, 1186;  2, 587, 2944;  2, 605, 1156;  2, 610, 855;  2, 614, 2837;  2, 616, 1275;  2, 617, 1120;  2, 623, 898;  2, 627, 1847;  2, 628, 2775;  2, 629, 1954;  2, 633, 1073;  2, 634, 2278;  2, 635, 2871;  2, 641, 1317;  2, 646, 947;  2, 650, 1408;  2, 651, 1619;  2, 659, 1526;  2, 663, 2295;  2, 665, 2733;  2, 669, 1041;  2, 676, 2555;  2, 681, 2290;  2, 688, 951;  2, 698, 2409;  2, 701, 2595;  2, 707, 2080;  2, 717, 2002;  2, 718, 1226;  2, 719, 1634;  2, 728, 2163;  2, 729, 1979;  2, 741, 1240;  2, 746, 2957;  2, 748, 2853;  2, 753, 998;  2, 755, 1371;  2, 761, 2054;  2, 770, 2301;  2, 771, 1322;  2, 772, 2108;  2, 778, 1845;  2, 789, 813;  2, 790, 1864;  2, 796, 2363;  2, 800, 2313;  2, 808, 2489;  2, 815, 1203;  2, 819, 1691;  2, 826, 1625;  2, 827, 1918;  2, 830, 2153;  2, 831, 1395;  2, 837, 1102;  2, 842, 1739;  2, 843, 1103;  2, 849, 986;  2, 861, 1109;  2, 862, 1755;  2, 867, 1439;  2, 872, 2907;  2, 874, 2086;  2, 875, 2277;  2, 886, 2135;  2, 891, 1888;  2, 893, 1185;  2, 902, 2391;  2, 909, 2331;  2, 911, 1377;  2, 917, 1745;  2, 928, 2219;  2, 929, 969;  2, 933, 2626;  2, 941, 2530;  2, 953, 2206;  2, 957, 1000;  2, 958, 2867;  2, 963, 2302;  2, 970, 1643;  2, 971, 2938;  2, 980, 2265;  2, 987, 2223;  2, 994, 2085;  2, 1007, 1342;  2, 1012, 2321;  2, 1019, 2801;  2, 1023, 2434;  2, 1029, 1466;  2, 1030, 2073;  2, 1031, 2891;  2, 1035, 1335;  2, 1037, 1760;  2, 1054, 1859;  2, 1061, 1785;  2, 1064, 2903;  2, 1065, 2951;  2, 1071, 2847;  2, 1079, 2338;  2, 1082, 2746;  2, 1083, 1706;  2, 1085, 2109;  2, 1107, 1379;  2, 1108, 2717;  2, 1119, 1286;  2, 1125, 2638;  2, 1131, 2692;  2, 1136, 2721;  2, 1139, 2159;  2, 1143, 1323;  2, 1151, 1468;  2, 1157, 1923;  2, 1161, 2524;  2, 1163, 2579;  2, 1168, 2139;  2, 1175, 2614;  2, 1179, 2866;  2, 1181, 1682;  2, 1184, 2945;  2, 1191, 1425;  2, 1198, 1319;  2, 1204, 1355;  2, 1205, 1569;  2, 1209, 1816;  2, 1210, 1935;  2, 1211, 2897;  2, 1217, 1383;  2, 1222, 1733;  2, 1239, 2609;  2, 1241, 1574;  2, 1247, 1252;  2, 1264, 1631;  2, 1277, 2584;  2, 1283, 1697;  2, 1289, 1762;  2, 1295, 2747;  2, 1300, 2111;  2, 1301, 2884;  2, 1306, 2777;  2, 1312, 1959;  2, 1313, 1552;  2, 1341, 1949;  2, 1366, 1431;  2, 1367, 2805;  2, 1391, 1995;  2, 1402, 2825;  2, 1403, 2368;  2, 1419, 1523;  2, 1421, 2471;  2, 1444, 1721;  2, 1455, 1690;  2, 1467, 1792;  2, 1469, 2015;  2, 1479, 2848;  2, 1481, 2741;  2, 1486, 2061;  2, 1492, 2013;  2, 1504, 1511;  2, 1505, 1947;  2, 1517, 2177;  2, 1521, 2392;  2, 1540, 2793;  2, 1557, 2055;  2, 1559, 2421;  2, 1563, 2315;  2, 1581, 2267;  2, 1583, 1636;  2, 1589, 1990;  2, 1593, 2403;  2, 1595, 2811;  2, 1601, 1901;  2, 1606, 1827;  2, 1612, 2955;  2, 1618, 1767;  2, 1637, 2303;  2, 1641, 2045;  2, 1648, 2589;  2, 1661, 2842;  2, 1679, 2134;  2, 1689, 2254;  2, 1695, 2308;  2, 1696, 1709;  2, 1707, 2518;  2, 1756, 2927;  2, 1780, 1829;  2, 1804, 1875;  2, 1805, 2207;  2, 1811, 2854;  2, 1823, 2877;  2, 1833, 1948;  2, 1846, 1863;  2, 1851, 1913;  2, 1852, 2393;  2, 1870, 2033;  2, 1889, 2020;  2, 1894, 1973;  2, 1905, 2014;  2, 1907, 2386;  2, 1911, 2224;  2, 1924, 2547;  2, 1961, 2704;  2, 1971, 1972;  2, 1978, 2765;  2, 1984, 2985;  2, 2001, 2369;  2, 2026, 2397;  2, 2039, 2931;  2, 2044, 2890;  2, 2087, 2476;  2, 2093, 2590;  2, 2117, 2619;  2, 2145, 2812;  2, 2147, 2248;  2, 2165, 2823;  2, 2176, 2745;  2, 2188, 2759;  2, 2189, 2939;  2, 2205, 2817;  2, 2230, 2961;  2, 2247, 2753;  2, 2249, 2920;  2, 2259, 2297;  2, 2261, 2523;  2, 2283, 2573;  2, 2320, 2637;  2, 2344, 2349;  2, 2345, 2783;  2, 2350, 2649;  2, 2379, 2561;  2, 2451, 2651;  2, 2465, 2715;  2, 2477, 2572;  2, 2494, 2855;  2, 2507, 2925;  2, 2511, 2723;  2, 2535, 2740;  2, 2603, 2986;  2, 2667, 2981;  2, 2699, 2711;  2, 2710, 2963;  2, 2813, 2824;  2, 2883, 2889;  3, 10, 1107;  3, 15, 915;  3, 16, 543;  3, 22, 1169;  3, 33, 1853;  3, 34, 1175;  3, 47, 2656;  3, 52, 2013;  3, 64, 322;  3, 69, 1209;  3, 71, 2098;  3, 75, 2021;  3, 76, 1923;  3, 77, 1312;  3, 82, 2651;  3, 83, 2207;  3, 94, 261;  3, 100, 1570;  3, 101, 2200;  3, 106, 298;  3, 119, 2703;  3, 124, 2939;  3, 129, 1137;  3, 130, 2679;  3, 131, 651;  3, 136, 669;  3, 142, 1403;  3, 148, 1079;  3, 159, 1330;  3, 172, 772;  3, 173, 1144;  3, 177, 1534;  3, 178, 1198;  3, 190, 1348;  3, 195, 2638;  3, 201, 2074;  3, 202, 2337;  3, 213, 1985;  3, 225, 2493;  3, 232, 2621;  3, 233, 2669;  3, 243, 1545;  3, 244, 2727;  3, 256, 2303;  3, 285, 1738;  3, 299, 893;  3, 304, 2225;  3, 305, 1619;  3, 310, 2122;  3, 323, 1192;  3, 334, 753;  3, 335, 2728;  3, 339, 2950;  3, 346, 1311;  3, 352, 2968;  3, 353, 2255;  3, 357, 1942;  3, 358, 2369;  3, 359, 531;  3, 370, 2260;  3, 376, 1084;  3, 377, 2662;  3, 383, 945;  3, 388, 1960;  3, 405, 2301;  3, 406, 1552;  3, 411, 1917;  3, 417, 1582;  3, 418, 1977;  3, 423, 1864;  3, 436, 2921;  3, 437, 958;  3, 454, 1959;  3, 460, 957;  3, 478, 1966;  3, 484, 2963;  3, 485, 2764;  3, 520, 718;  3, 521, 1102;  3, 527, 946;  3, 532, 2404;  3, 533, 2896;  3, 537, 1451;  3, 538, 1665;  3, 544, 2740;  3, 555, 1450;  3, 556, 2866;  3, 561, 2542;  3, 587, 603;  3, 592, 645;  3, 597, 1780;  3, 609, 2890;  3, 617, 658;  3, 627, 1846;  3, 628, 1187;  3, 629, 2087;  3, 634, 2716;  3, 652, 1343;  3, 653, 2548;  3, 664, 2129;  3, 676, 785;  3, 677, 706;  3, 682, 694;  3, 683, 1720;  3, 736, 1787;  3, 742, 1469;  3, 754, 1264;  3, 766, 1929;  3, 779, 2494;  3, 784, 1546;  3, 796, 1024;  3, 797, 1426;  3, 809, 886;  3, 837, 2189;  3, 855, 1708;  3, 869, 2146;  3, 881, 1210;  3, 887, 2530;  3, 891, 2512;  3, 899, 2812;  3, 910, 1678;  3, 940, 2002;  3, 951, 1924;  3, 952, 1498;  3, 953, 2110;  3, 964, 2483;  3, 969, 2632;  3, 975, 2038;  3, 982, 1097;  3, 983, 2855;  3, 988, 2224;  3, 989, 1006;  3, 999, 2290;  3, 1001, 2927;  3, 1007, 2177;  3, 1013, 1600;  3, 1018, 1265;  3, 1042, 1473;  3, 1054, 1636;  3, 1055, 2278;  3, 1060, 2285;  3, 1072, 2758;  3, 1090, 1875;  3, 1096, 1293;  3, 1114, 1629;  3, 1120, 2908;  3, 1132, 1803;  3, 1145, 1522;  3, 1234, 1607;  3, 1235, 1912;  3, 1282, 2236;  3, 1301, 1990;  3, 1306, 2465;  3, 1313, 2722;  3, 1319, 1618;  3, 1324, 1533;  3, 1354, 2362;  3, 1367, 2872;  3, 1384, 2962;  3, 1396, 1816;  3, 1409, 1931;  3, 1420, 2597;  3, 1481, 2368;  3, 1487, 2590;  3, 1511, 2777;  3, 1528, 1799;  3, 1529, 2392;  3, 1541, 2518;  3, 1553, 2591;  3, 1565, 1961;  3, 1589, 2854;  3, 1594, 1768;  3, 1606, 1810;  3, 1660, 1691;  3, 1684, 2891;  3, 1726, 2279;  3, 1756, 2308;  3, 1762, 2116;  3, 1852, 1978;  3, 1865, 2884;  3, 1888, 2627;  3, 1895, 1918;  3, 1972, 2596;  3, 2080, 2489;  3, 2086, 2974;  3, 2158, 2302;  3, 2159, 2639;  3, 2170, 2351;  3, 2182, 2387;  3, 2201, 2345;  3, 2230, 2572;  3, 2309, 2770;  3, 2344, 2813;  3, 2380, 2993;  3, 2608, 2849;  3, 2692, 2987;  3, 2836, 2873;  4, 23, 341;  4, 28, 1619;  4, 29, 1678;  4, 34, 2915;  4, 47, 874;  4, 59, 2350;  4, 65, 1306;  4, 70, 1529;  4, 77, 814;  4, 94, 2189;  4, 95, 250;  4, 101, 754;  4, 149, 490;  4, 166, 1439;  4, 173, 2434;  4, 179, 544;  4, 191, 359;  4, 197, 1948;  4, 203, 376;  4, 226, 749;  4, 227, 2375;  4, 238, 2837;  4, 239, 1301;  4, 244, 1637;  4, 322, 1643;  4, 323, 808;  4, 335, 1655;  4, 353, 2381;  4, 383, 2302;  4, 388, 2675;  4, 395, 1367;  4, 401, 1834;  4, 425, 2087;  4, 431, 2057;  4, 436, 1157;  4, 437, 2249;  4, 443, 569;  4, 455, 797;  4, 466, 2315;  4, 478, 1781;  4, 479, 2003;  4, 484, 1144;  4, 497, 730;  4, 508, 2056;  4, 520, 1882;  4, 521, 2014;  4, 532, 1727;  4, 533, 1631;  4, 539, 767;  4, 598, 1511;  4, 647, 2009;  4, 683, 1583;  4, 706, 1889;  4, 779, 1372;  4, 785, 959;  4, 857, 2855;  4, 862, 2705;  4, 863, 1067;  4, 875, 2069;  4, 905, 1985;  4, 911, 2357;  4, 923, 970;  4, 965, 1336;  4, 977, 2141;  4, 982, 2801;  4, 983, 1547;  4, 989, 1774;  4, 1007, 1960;  4, 1013, 2771;  4, 1019, 1732;  4, 1025, 2735;  4, 1061, 1091;  4, 1097, 2243;  4, 1102, 2333;  4, 1115, 1391;  4, 1271, 1433;  4, 1379, 1661;  4, 1487, 1553;  4, 1517, 2165;  4, 1943, 2741;  4, 2081, 2339;  4, 2147, 2729;  4, 2159, 2645;  4, 2183, 2303;  4, 2273, 2279;  4, 2459, 2549;  5, 29, 629;  5, 41, 1115;  5, 161, 2387;  5, 437, 2489;  5, 647, 1991;  5, 809, 1649}.
The deficiency graph is connected and has girth 4.

\adfAppGap

{\boldmath $\adfPENT(3,310,33)$}, $d = 6$:
{0, 1, 2;  0, 3, 171;  0, 4, 484;  0, 5, 364;  0, 9, 384;  0, 10, 589;  0, 11, 517;  0, 12, 294;  0, 13, 580;  0, 14, 487;  0, 21, 296;  0, 22, 397;  0, 23, 174;  0, 36, 354;  0, 37, 493;  0, 38, 427;  0, 42, 180;  0, 43, 463;  0, 44, 343;  0, 51, 147;  0, 52, 523;  0, 53, 337;  0, 54, 264;  0, 55, 583;  0, 56, 451;  0, 57, 173;  0, 58, 262;  0, 59, 60;  0, 61, 485;  0, 62, 483;  0, 66, 266;  0, 67, 349;  0, 68, 75;  0, 69, 387;  0, 70, 598;  0, 71, 586;  0, 72, 149;  0, 73, 301;  0, 74, 126;  0, 76, 584;  0, 84, 369;  0, 85, 466;  0, 86, 445;  0, 90, 258;  0, 91, 469;  0, 92, 280;  0, 93, 182;  0, 94, 190;  0, 95, 114;  0, 102, 306;  0, 103, 535;  0, 104, 373;  0, 111, 389;  0, 112, 424;  0, 113, 201;  0, 115, 464;  0, 116, 462;  0, 120, 356;  0, 121, 406;  0, 122, 222;  0, 127, 524;  0, 128, 522;  0, 129, 645;  0, 130, 172;  0, 131, 449;  0, 133, 271;  0, 134, 183;  0, 139, 148;  0, 140, 422;  0, 162, 386;  0, 163, 475;  0, 164, 198;  0, 169, 181;  0, 170, 380;  0, 175, 581;  0, 176, 579;  0, 184, 470;  0, 185, 468;  0, 187, 391;  0, 188, 276;  0, 189, 342;  0, 191, 565;  0, 193, 361;  0, 194, 234;  0, 199, 590;  0, 202, 599;  0, 203, 597;  0, 205, 259;  0, 206, 434;  0, 211, 295;  0, 212, 530;  0, 223, 494;  0, 228, 633;  0, 229, 265;  0, 230, 542;  0, 235, 536;  0, 261, 363;  0, 263, 646;  0, 272, 553;  0, 277, 467;  0, 278, 465;  0, 281, 433;  0, 283, 385;  0, 284, 569;  0, 286, 307;  0, 287, 383;  0, 302, 613;  0, 303, 651;  0, 304, 388;  0, 305, 596;  0, 313, 355;  0, 314, 482;  0, 319, 370;  0, 320, 458;  0, 338, 421;  0, 344, 379;  0, 350, 634;  0, 362, 571;  0, 365, 448;  0, 374, 382;  0, 381, 572;  0, 392, 619;  0, 398, 601;  0, 407, 502;  0, 423, 585;  0, 425, 652;  0, 428, 481;  0, 446, 457;  0, 447, 647;  0, 452, 541;  0, 476, 643;  0, 488, 529;  0, 518, 568;  0, 567, 644;  0, 587, 595;  1, 3, 365;  1, 4, 363;  1, 5, 484;  1, 9, 518;  1, 11, 589;  1, 14, 580;  1, 21, 175;  1, 22, 176;  1, 23, 530;  1, 38, 493;  1, 44, 463;  1, 51, 338;  1, 53, 523;  1, 56, 583;  1, 57, 61;  1, 58, 62;  1, 59, 449;  1, 68, 542;  1, 69, 587;  1, 70, 585;  1, 71, 598;  1, 74, 422;  1, 75, 635;  1, 86, 466;  1, 91, 279;  1, 92, 469;  1, 93, 115;  1, 94, 116;  1, 95, 380;  1, 104, 535;  1, 111, 202;  1, 112, 203;  1, 113, 596;  1, 122, 482;  1, 129, 173;  1, 130, 483;  1, 131, 262;  1, 134, 434;  1, 140, 301;  1, 147, 337;  1, 148, 524;  1, 164, 569;  1, 170, 190;  1, 171, 364;  1, 172, 485;  1, 181, 464;  1, 183, 554;  1, 188, 458;  1, 194, 383;  1, 201, 653;  1, 206, 271;  1, 211, 579;  1, 212, 397;  1, 230, 349;  1, 261, 646;  1, 263, 447;  1, 265, 584;  1, 285, 308;  1, 287, 361;  1, 303, 389;  1, 304, 597;  1, 305, 424;  1, 314, 406;  1, 362, 381;  1, 370, 467;  1, 387, 586;  1, 388, 599;  1, 405, 502;  1, 423, 652;  1, 448, 645;  1, 476, 567;  2, 3, 483;  2, 4, 129;  2, 5, 62;  2, 11, 200;  2, 14, 176;  2, 23, 602;  2, 38, 224;  2, 44, 116;  2, 53, 128;  2, 57, 172;  2, 58, 364;  2, 59, 647;  2, 68, 635;  2, 69, 597;  2, 70, 303;  2, 71, 203;  2, 75, 452;  2, 76, 266;  2, 86, 278;  2, 92, 185;  2, 104, 236;  2, 111, 388;  2, 112, 586;  2, 113, 653;  2, 122, 503;  2, 130, 171;  2, 131, 363;  2, 183, 281;  2, 184, 260;  2, 201, 587;  2, 202, 389;  2, 261, 447;  2, 262, 645;  2, 263, 484;  2, 304, 387;  2, 305, 585;  2, 406, 501;  2, 424, 651;  2, 425, 598;  2, 448, 646;  2, 449, 485;  3, 4, 5;  3, 15, 143;  3, 16, 184;  3, 17, 69;  3, 39, 267;  3, 40, 472;  3, 41, 340;  3, 45, 357;  3, 46, 586;  3, 57, 261;  3, 58, 646;  3, 59, 448;  3, 70, 461;  3, 71, 459;  3, 75, 311;  3, 76, 394;  3, 77, 117;  3, 87, 297;  3, 88, 466;  3, 93, 477;  3, 95, 490;  3, 106, 526;  3, 107, 454;  3, 123, 299;  3, 124, 352;  3, 125, 285;  3, 130, 647;  3, 135, 359;  3, 136, 520;  3, 137, 237;  3, 166, 538;  3, 167, 346;  3, 172, 262;  3, 173, 485;  3, 178, 539;  3, 185, 496;  3, 190, 400;  3, 191, 225;  3, 195, 479;  3, 197, 279;  3, 208, 310;  3, 209, 593;  3, 214, 268;  3, 215, 383;  3, 226, 527;  3, 233, 437;  3, 238, 587;  3, 275, 376;  3, 286, 467;  3, 305, 622;  3, 316, 478;  3, 317, 545;  3, 322, 358;  3, 323, 533;  3, 341, 382;  3, 347, 436;  3, 353, 556;  3, 395, 604;  3, 401, 568;  3, 431, 592;  3, 449, 484;  3, 491, 544;  4, 16, 71;  4, 17, 377;  4, 41, 472;  4, 47, 586;  4, 77, 593;  4, 89, 466;  4, 107, 526;  4, 124, 287;  4, 125, 425;  4, 137, 533;  4, 142, 461;  4, 173, 364;  4, 197, 545;  4, 203, 383;  4, 209, 394;  4, 233, 304;  4, 262, 647;  5, 17, 497;  5, 47, 239;  5, 59, 131;  5, 89, 287;  5, 95, 281;  5, 107, 227;  0, 6, 274;  0, 7, 442;  0, 8, 436;  0, 15, 255;  0, 16, 257;  0, 17, 609;  0, 18, 357;  0, 19, 489;  0, 20, 431;  0, 24, 178;  0, 25, 556;  0, 26, 316;  0, 33, 65;  0, 34, 437;  0, 35, 80;  0, 39, 606;  0, 40, 144;  0, 41, 537;  0, 45, 641;  0, 46, 97;  0, 47, 410;  0, 49, 340;  0, 50, 315;  0, 63, 640;  0, 64, 353;  0, 78, 519;  0, 79, 591;  0, 81, 511;  0, 82, 321;  0, 83, 88;  0, 89, 217;  0, 96, 526;  0, 98, 557;  0, 105, 215;  0, 106, 608;  0, 107, 402;  0, 108, 401;  0, 109, 539;  0, 110, 346;  0, 117, 135;  0, 118, 547;  0, 119, 560;  0, 123, 275;  0, 124, 375;  0, 125, 623;  0, 136, 291;  0, 137, 460;  0, 145, 310;  0, 146, 250;  0, 150, 478;  0, 151, 551;  0, 152, 549;  0, 153, 248;  0, 155, 399;  0, 156, 394;  0, 157, 515;  0, 158, 532;  0, 159, 440;  0, 160, 290;  0, 161, 179;  0, 165, 408;  0, 166, 225;  0, 167, 604;  0, 177, 237;  0, 195, 242;  0, 196, 226;  0, 197, 603;  0, 207, 649;  0, 208, 533;  0, 209, 497;  0, 213, 358;  0, 214, 621;  0, 216, 525;  0, 218, 341;  0, 219, 227;  0, 220, 368;  0, 221, 330;  0, 239, 366;  0, 240, 592;  0, 241, 538;  0, 243, 611;  0, 244, 499;  0, 245, 267;  0, 247, 561;  0, 249, 607;  0, 251, 395;  0, 253, 412;  0, 254, 459;  0, 256, 367;  0, 269, 416;  0, 273, 639;  0, 289, 508;  0, 292, 311;  0, 297, 543;  0, 298, 377;  0, 299, 512;  0, 317, 631;  0, 322, 506;  0, 323, 575;  0, 325, 562;  0, 326, 520;  0, 327, 477;  0, 329, 638;  0, 331, 509;  0, 332, 573;  0, 345, 650;  0, 347, 559;  0, 351, 415;  0, 376, 632;  0, 393, 500;  0, 400, 578;  0, 403, 490;  0, 404, 610;  0, 409, 544;  0, 413, 563;  0, 429, 548;  0, 435, 514;  0, 439, 471;  0, 443, 473;  0, 461, 637;  0, 472, 577;  0, 479, 507;  0, 491, 505;  0, 495, 521;  0, 496, 574;  0, 513, 593;  0, 531, 555;  0, 605, 622;  1, 7, 549;  1, 8, 537;  1, 15, 242;  1, 16, 526;  1, 17, 157;  1, 19, 575;  1, 20, 610;  1, 25, 353;  1, 26, 244;  1, 34, 217;  1, 35, 50;  1, 39, 46;  1, 40, 441;  1, 41, 256;  1, 45, 98;  1, 47, 409;  1, 49, 509;  1, 63, 632;  1, 64, 609;  1, 65, 393;  1, 79, 611;  1, 80, 490;  1, 81, 412;  1, 82, 519;  1, 83, 560;  1, 87, 328;  1, 89, 411;  1, 97, 435;  1, 105, 135;  1, 106, 214;  1, 107, 511;  1, 109, 351;  1, 110, 322;  1, 117, 638;  1, 118, 358;  1, 119, 531;  1, 123, 267;  1, 124, 489;  1, 125, 158;  1, 137, 241;  1, 146, 395;  1, 151, 623;  1, 152, 376;  1, 153, 309;  1, 154, 442;  1, 155, 289;  1, 159, 404;  1, 161, 506;  1, 165, 273;  1, 167, 574;  1, 177, 440;  1, 178, 508;  1, 195, 520;  1, 196, 605;  1, 197, 345;  1, 207, 227;  1, 208, 403;  1, 209, 215;  1, 218, 437;  1, 219, 416;  1, 221, 299;  1, 237, 548;  1, 239, 347;  1, 245, 274;  1, 248, 555;  1, 249, 310;  1, 250, 357;  1, 254, 593;  1, 255, 507;  1, 257, 399;  1, 268, 377;  1, 269, 556;  1, 275, 352;  1, 290, 323;  1, 291, 325;  1, 311, 327;  1, 316, 561;  1, 317, 650;  1, 321, 478;  1, 326, 640;  1, 332, 525;  1, 340, 496;  1, 346, 562;  1, 368, 622;  1, 375, 410;  1, 394, 539;  1, 413, 512;  1, 429, 608;  1, 430, 459;  1, 477, 578;  1, 491, 495;  1, 497, 603;  1, 500, 538;  1, 514, 545;  1, 563, 639;  1, 573, 592;  2, 8, 441;  2, 15, 431;  2, 16, 33;  2, 17, 632;  2, 20, 81;  2, 34, 539;  2, 39, 255;  2, 45, 299;  2, 46, 640;  2, 47, 152;  2, 50, 245;  2, 64, 479;  2, 65, 315;  2, 80, 237;  2, 82, 88;  2, 83, 297;  2, 89, 525;  2, 98, 213;  2, 105, 242;  2, 107, 610;  2, 110, 545;  2, 118, 275;  2, 119, 274;  2, 123, 440;  2, 124, 532;  2, 136, 332;  2, 137, 377;  2, 146, 353;  2, 153, 178;  2, 155, 298;  2, 158, 513;  2, 160, 165;  2, 161, 562;  2, 166, 248;  2, 167, 321;  2, 177, 507;  2, 219, 225;  2, 227, 591;  2, 239, 368;  2, 244, 496;  2, 249, 395;  2, 250, 273;  2, 254, 520;  2, 257, 473;  2, 269, 593;  2, 291, 442;  2, 310, 317;  2, 322, 471;  2, 327, 358;  2, 328, 543;  2, 329, 436;  2, 340, 401;  2, 346, 489;  2, 352, 639;  2, 359, 605;  2, 394, 544;  2, 490, 514;  2, 491, 550;  2, 495, 573;  2, 497, 521;  2, 515, 575;  2, 533, 556;  2, 538, 563;  2, 604, 622}.
The deficiency graph is connected and has girth 4.

{\boldmath $\adfPENT(3,316,33)$}, $d = 6$:
{0, 1, 2;  0, 6, 372;  0, 7, 655;  0, 8, 463;  0, 12, 336;  0, 13, 622;  0, 14, 451;  0, 15, 246;  0, 16, 535;  0, 17, 397;  0, 18, 337;  0, 19, 475;  0, 20, 186;  0, 24, 228;  0, 25, 601;  0, 26, 403;  0, 48, 285;  0, 49, 592;  0, 50, 343;  0, 54, 195;  0, 55, 616;  0, 56, 352;  0, 60, 247;  0, 61, 445;  0, 62, 72;  0, 63, 273;  0, 64, 613;  0, 65, 364;  0, 66, 282;  0, 67, 619;  0, 68, 301;  0, 69, 274;  0, 70, 412;  0, 71, 81;  0, 73, 510;  0, 74, 79;  0, 75, 288;  0, 76, 499;  0, 77, 436;  0, 78, 534;  0, 80, 446;  0, 82, 561;  0, 83, 142;  0, 90, 283;  0, 91, 469;  0, 92, 222;  0, 114, 286;  0, 115, 349;  0, 116, 138;  0, 117, 196;  0, 118, 418;  0, 119, 123;  0, 124, 465;  0, 125, 136;  0, 133, 643;  0, 134, 454;  0, 135, 615;  0, 137, 419;  0, 139, 516;  0, 140, 271;  0, 143, 413;  0, 147, 211;  0, 148, 529;  0, 149, 192;  0, 151, 481;  0, 152, 198;  0, 157, 478;  0, 158, 171;  0, 168, 432;  0, 169, 661;  0, 170, 589;  0, 173, 217;  0, 174, 373;  0, 175, 595;  0, 176, 249;  0, 188, 262;  0, 194, 205;  0, 197, 329;  0, 200, 265;  0, 206, 530;  0, 212, 380;  0, 218, 479;  0, 223, 543;  0, 224, 238;  0, 230, 236;  0, 235, 554;  0, 239, 470;  0, 243, 433;  0, 244, 604;  0, 245, 318;  0, 248, 323;  0, 250, 609;  0, 251, 325;  0, 261, 621;  0, 263, 476;  0, 266, 482;  0, 270, 591;  0, 272, 350;  0, 275, 299;  0, 284, 296;  0, 287, 332;  0, 290, 422;  0, 295, 545;  0, 297, 363;  0, 298, 563;  0, 302, 544;  0, 320, 367;  0, 322, 512;  0, 326, 596;  0, 327, 351;  0, 328, 467;  0, 331, 518;  0, 338, 386;  0, 344, 517;  0, 353, 466;  0, 365, 562;  0, 368, 605;  0, 374, 440;  0, 379, 578;  0, 385, 608;  0, 398, 511;  0, 404, 553;  0, 411, 614;  0, 417, 617;  0, 421, 494;  0, 434, 458;  0, 437, 493;  0, 439, 611;  0, 452, 607;  0, 455, 577;  0, 457, 650;  0, 464, 610;  0, 477, 602;  0, 590, 649;  0, 603, 662;  1, 7, 325;  1, 8, 326;  1, 13, 262;  1, 14, 263;  1, 15, 61;  1, 16, 79;  1, 17, 80;  1, 20, 452;  1, 25, 217;  1, 26, 218;  1, 49, 271;  1, 50, 272;  1, 55, 136;  1, 56, 137;  1, 62, 398;  1, 63, 70;  1, 64, 142;  1, 65, 143;  1, 67, 238;  1, 68, 239;  1, 69, 563;  1, 71, 365;  1, 73, 535;  1, 75, 151;  1, 76, 265;  1, 77, 266;  1, 81, 274;  1, 83, 273;  1, 92, 302;  1, 116, 344;  1, 117, 467;  1, 119, 353;  1, 123, 196;  1, 124, 616;  1, 125, 195;  1, 134, 206;  1, 135, 197;  1, 139, 592;  1, 140, 285;  1, 141, 275;  1, 147, 578;  1, 149, 455;  1, 152, 437;  1, 158, 404;  1, 169, 367;  1, 170, 368;  1, 171, 229;  1, 176, 464;  1, 187, 622;  1, 236, 479;  1, 237, 284;  1, 243, 650;  1, 245, 590;  1, 249, 373;  1, 261, 338;  1, 296, 470;  1, 297, 411;  1, 298, 561;  1, 299, 413;  1, 323, 446;  1, 327, 417;  1, 328, 465;  1, 329, 419;  1, 332, 350;  1, 351, 617;  1, 352, 418;  1, 363, 614;  1, 364, 412;  1, 380, 530;  1, 386, 476;  1, 422, 482;  1, 435, 500;  1, 440, 596;  1, 453, 644;  1, 458, 605;  1, 466, 615;  2, 17, 321;  2, 56, 327;  2, 63, 561;  2, 64, 411;  2, 65, 297;  2, 69, 412;  2, 70, 141;  2, 71, 274;  2, 81, 365;  2, 82, 563;  2, 116, 286;  2, 117, 418;  2, 118, 135;  2, 119, 196;  2, 123, 353;  2, 124, 467;  2, 125, 616;  2, 136, 329;  2, 137, 351;  2, 140, 592;  2, 142, 299;  2, 143, 363;  2, 171, 404;  2, 188, 622;  2, 195, 617;  2, 197, 466;  2, 206, 453;  2, 238, 296;  2, 243, 604;  2, 250, 611;  2, 262, 386;  2, 275, 562;  2, 285, 593;  2, 328, 615;  2, 352, 419;  2, 364, 413;  3, 4, 5;  3, 9, 237;  3, 11, 352;  3, 15, 297;  3, 16, 664;  3, 17, 496;  3, 21, 238;  3, 22, 502;  3, 23, 81;  3, 28, 538;  3, 29, 460;  3, 51, 369;  3, 52, 604;  3, 53, 388;  3, 63, 333;  3, 64, 658;  3, 70, 646;  3, 71, 442;  3, 75, 298;  3, 77, 207;  3, 83, 214;  3, 94, 466;  3, 95, 141;  3, 118, 454;  3, 119, 267;  3, 135, 423;  3, 136, 610;  3, 142, 477;  3, 143, 292;  3, 153, 327;  3, 154, 622;  3, 155, 400;  3, 159, 424;  3, 160, 592;  3, 161, 225;  3, 172, 471;  3, 173, 178;  3, 179, 407;  3, 191, 251;  3, 197, 219;  3, 202, 448;  3, 209, 232;  3, 215, 503;  3, 221, 286;  3, 227, 274;  3, 239, 305;  3, 250, 473;  3, 275, 593;  3, 293, 467;  3, 299, 347;  3, 323, 533;  3, 329, 341;  3, 335, 485;  3, 340, 557;  3, 371, 377;  3, 376, 581;  3, 383, 472;  3, 389, 580;  3, 401, 556;  3, 436, 515;  3, 443, 514;  3, 461, 478;  3, 484, 653;  4, 10, 214;  4, 11, 215;  4, 16, 232;  4, 17, 233;  4, 23, 353;  4, 28, 292;  4, 29, 293;  4, 64, 286;  4, 65, 287;  4, 77, 497;  4, 95, 461;  4, 119, 443;  4, 136, 274;  4, 137, 275;  4, 154, 322;  4, 155, 323;  4, 203, 389;  4, 251, 407;  4, 305, 503;  4, 341, 533;  4, 377, 449;  4, 437, 455;  0, 3, 354;  0, 4, 278;  0, 5, 45;  0, 9, 626;  0, 10, 568;  0, 11, 624;  0, 21, 447;  0, 22, 121;  0, 23, 279;  0, 27, 305;  0, 28, 505;  0, 29, 550;  0, 30, 551;  0, 31, 599;  0, 32, 184;  0, 33, 242;  0, 34, 570;  0, 35, 153;  0, 36, 221;  0, 37, 303;  0, 38, 519;  0, 39, 292;  0, 40, 391;  0, 41, 291;  0, 43, 639;  0, 44, 128;  0, 46, 163;  0, 47, 252;  0, 51, 159;  0, 52, 558;  0, 84, 395;  0, 85, 640;  0, 86, 399;  0, 87, 514;  0, 88, 340;  0, 89, 496;  0, 93, 426;  0, 94, 572;  0, 95, 416;  0, 97, 497;  0, 98, 627;  0, 99, 202;  0, 100, 253;  0, 101, 377;  0, 102, 215;  0, 103, 371;  0, 104, 280;  0, 109, 491;  0, 110, 131;  0, 111, 566;  0, 112, 207;  0, 120, 334;  0, 122, 209;  0, 126, 484;  0, 127, 213;  0, 129, 571;  0, 144, 513;  0, 145, 537;  0, 146, 165;  0, 154, 314;  0, 155, 579;  0, 161, 587;  0, 162, 339;  0, 164, 381;  0, 166, 625;  0, 167, 632;  0, 178, 304;  0, 179, 532;  0, 180, 574;  0, 181, 569;  0, 182, 227;  0, 183, 268;  0, 189, 448;  0, 190, 523;  0, 191, 408;  0, 201, 637;  0, 203, 347;  0, 208, 597;  0, 219, 382;  0, 220, 547;  0, 225, 487;  0, 226, 580;  0, 231, 542;  0, 232, 375;  0, 233, 506;  0, 241, 269;  0, 254, 531;  0, 259, 346;  0, 260, 357;  0, 267, 428;  0, 276, 652;  0, 277, 405;  0, 281, 555;  0, 293, 345;  0, 309, 663;  0, 310, 586;  0, 313, 515;  0, 316, 584;  0, 317, 460;  0, 335, 646;  0, 341, 355;  0, 356, 359;  0, 370, 645;  0, 383, 409;  0, 387, 489;  0, 388, 471;  0, 389, 410;  0, 392, 631;  0, 393, 651;  0, 400, 653;  0, 401, 507;  0, 406, 459;  0, 407, 560;  0, 415, 629;  0, 427, 525;  0, 472, 509;  0, 473, 557;  0, 483, 520;  0, 485, 524;  0, 488, 526;  0, 490, 638;  0, 495, 533;  0, 508, 549;  0, 527, 647;  0, 538, 641;  0, 539, 665;  0, 541, 573;  0, 548, 585;  0, 556, 583;  0, 559, 581;  0, 565, 575;  0, 567, 598;  0, 628, 664;  1, 3, 507;  1, 4, 575;  1, 5, 392;  1, 9, 95;  1, 10, 52;  1, 21, 164;  1, 22, 51;  1, 27, 146;  1, 28, 565;  1, 31, 219;  1, 32, 557;  1, 34, 292;  1, 35, 154;  1, 37, 663;  1, 38, 304;  1, 39, 580;  1, 40, 44;  1, 41, 488;  1, 43, 551;  1, 45, 524;  1, 46, 127;  1, 47, 461;  1, 53, 573;  1, 85, 310;  1, 86, 178;  1, 89, 259;  1, 93, 369;  1, 94, 241;  1, 97, 484;  1, 98, 647;  1, 100, 254;  1, 101, 278;  1, 104, 382;  1, 109, 293;  1, 110, 629;  1, 111, 651;  1, 112, 527;  1, 113, 166;  1, 121, 628;  1, 122, 153;  1, 128, 525;  1, 131, 213;  1, 145, 358;  1, 155, 191;  1, 159, 201;  1, 160, 585;  1, 161, 638;  1, 163, 645;  1, 165, 209;  1, 167, 253;  1, 177, 341;  1, 179, 260;  1, 181, 537;  1, 182, 268;  1, 183, 305;  1, 184, 626;  1, 202, 513;  1, 207, 335;  1, 220, 355;  1, 221, 381;  1, 227, 242;  1, 232, 473;  1, 233, 333;  1, 277, 652;  1, 279, 309;  1, 280, 315;  1, 281, 560;  1, 291, 345;  1, 303, 359;  1, 311, 394;  1, 314, 489;  1, 317, 574;  1, 339, 448;  1, 346, 632;  1, 347, 506;  1, 356, 664;  1, 370, 542;  1, 371, 399;  1, 375, 416;  1, 377, 471;  1, 387, 531;  1, 395, 449;  1, 400, 584;  1, 406, 526;  1, 407, 665;  1, 410, 519;  1, 447, 548;  1, 459, 491;  1, 460, 515;  1, 472, 599;  1, 485, 587;  1, 490, 521;  1, 495, 579;  1, 496, 598;  1, 533, 567;  1, 538, 639;  1, 539, 566;  1, 549, 646;  1, 555, 572;  2, 3, 101;  2, 4, 640;  2, 9, 555;  2, 10, 129;  2, 11, 580;  2, 22, 448;  2, 27, 131;  2, 28, 598;  2, 29, 304;  2, 32, 556;  2, 34, 461;  2, 35, 538;  2, 38, 575;  2, 41, 166;  2, 44, 269;  2, 45, 100;  2, 46, 167;  2, 52, 560;  2, 53, 291;  2, 87, 572;  2, 93, 652;  2, 95, 202;  2, 104, 489;  2, 112, 369;  2, 113, 155;  2, 122, 513;  2, 128, 311;  2, 130, 215;  2, 146, 490;  2, 153, 405;  2, 159, 164;  2, 161, 665;  2, 165, 208;  2, 179, 533;  2, 182, 381;  2, 184, 537;  2, 207, 303;  2, 209, 519;  2, 214, 358;  2, 220, 497;  2, 225, 314;  2, 231, 267;  2, 232, 316;  2, 233, 375;  2, 242, 573;  2, 254, 569;  2, 260, 550;  2, 278, 585;  2, 293, 645;  2, 305, 401;  2, 334, 495;  2, 335, 376;  2, 339, 599;  2, 340, 383;  2, 341, 449;  2, 345, 472;  2, 359, 460;  2, 370, 532;  2, 371, 406;  2, 377, 407;  2, 388, 639;  2, 447, 568;  2, 471, 581;  2, 473, 597;  2, 485, 646;  2, 557, 586;  2, 574, 628;  3, 148, 257}.
The deficiency graph is connected and has girth 4.

\adfAppGap

{\boldmath $\adfPENT(4,60)$}, $d = 5$:
{1, 4, 5, 180;  0, 2, 11, 176;  1, 3, 7, 182;  2, 4, 13, 178;  0, 3, 19, 174;  11, 161, 87, 129;  66, 47, 152, 145;  92, 30, 28, 138;  17, 33, 180, 93;  180, 161, 16, 34;  129, 86, 28, 149;  13, 158, 96, 16;  53, 166, 141, 3;  63, 126, 44, 143;  50, 177, 132, 14;  112, 135, 34, 84;  100, 58, 123, 29;  142, 117, 84, 14;  110, 43, 94, 146;  49, 180, 80, 174;  89, 1, 155, 42;  33, 88, 111, 95;  182, 74, 169, 69;  178, 137, 119, 7;  102, 142, 23, 181;  46, 62, 178, 146;  121, 76, 147, 5;  87, 116, 162, 177;  31, 36, 148, 19;  105, 42, 179, 183;  0, 12, 26, 146;  0, 8, 101, 151;  0, 14, 41, 171;  0, 13, 51, 111;  0, 15, 61, 136;  0, 24, 66, 156;  0, 32, 56, 86;  0, 17, 81, 159;  0, 28, 76, 99;  0, 18, 126, 164;  0, 42, 91, 161;  0, 29, 109, 131;  1, 8, 14, 16;  1, 32, 69, 92;  1, 9, 53, 129;  0, 37, 49, 96;  1, 13, 49, 94;  1, 38, 57, 142;  1, 22, 74, 109;  0, 79, 104, 141;  1, 68, 97, 114;  1, 34, 158, 179;  1, 58, 82, 153;  1, 117, 152, 173;  1, 37, 99, 168;  0, 64, 139, 147;  2, 93, 124, 134;  0, 34, 89, 150;  0, 33, 59, 133;  2, 28, 67, 94;  2, 53, 88, 184;  3, 18, 48, 154;  0, 97, 129, 167;  0, 44, 112, 132;  0, 137, 144, 168;  0, 93, 152, 182;  0, 73, 107, 157;  0, 43, 68, 117;  0, 53, 110, 177;  0, 48, 60, 163;  0, 58, 63, 95;  0, 47, 83, 130;  0, 25, 113, 140;  0, 20, 50, 77;  0, 40, 92, 105}.
The deficiency graph is connected and has girth 5.
Identifying code: $\{x: 0 \le x < 185,\; x \equiv 0 \text{ or } 2 \adfmod{5}\}$,
$\pi_0 = 0$, $\pi_1 = 74$, $\pi_2 = 111$.

{\boldmath $\adfPENT(4,65)$}, $d = 2$:
{1, 2, 11, 198;  0, 5, 190, 197;  121, 189, 135, 100;  140, 168, 128, 45;  60, 126, 0, 191;  164, 26, 44, 87;  130, 108, 103, 1;  186, 116, 96, 130;  60, 34, 89, 102;  59, 108, 181, 11;  166, 11, 155, 193;  73, 186, 114, 32;  190, 103, 145, 101;  195, 18, 116, 2;  136, 160, 51, 57;  20, 167, 71, 189;  0, 6, 25, 54;  0, 8, 52, 116;  0, 17, 32, 63;  0, 23, 122, 161;  0, 57, 124, 183;  0, 30, 88, 125;  0, 38, 137, 157;  0, 69, 85, 149;  0, 36, 145, 179;  0, 49, 96, 163;  0, 33, 106, 181;  0, 107, 135, 167;  0, 47, 83, 175;  1, 25, 91, 131;  0, 53, 129, 141;  0, 81, 139, 165;  0, 50, 100, 150;  1, 51, 101, 151}.
The last 2 blocks represent short orbits.
The deficiency graph is connected and has girth 5.
Identifying code: $\{x + 20i: x \in \{0,2,8,9,11,13,14,15\}, 0 \le i < 10\}$,
$\pi_0 = 0$, $\pi_1 = 80$, $\pi_2 = 120$.

{\boldmath $\adfPENT(4,73)$}, $d = 2$:
{5, 11, 57, 165;  60, 168, 214, 220;  97, 99, 204, 121;  12, 181, 66, 28;  104, 22, 203, 37;  25, 35, 7, 117;  16, 142, 178, 218;  57, 4, 30, 197;  140, 19, 220, 121;  74, 206, 41, 73;  178, 137, 157, 77;  0, 1, 45, 222;  0, 4, 13, 131;  0, 7, 14, 215;  0, 8, 25, 97;  0, 10, 29, 185;  0, 12, 33, 163;  0, 18, 86, 106;  0, 24, 72, 102;  0, 31, 35, 120;  0, 28, 65, 77;  0, 34, 75, 128;  0, 42, 109, 143;  0, 32, 71, 166;  0, 44, 95, 149;  0, 43, 60, 219;  0, 66, 145, 187;  0, 83, 113, 179;  0, 55, 81, 155;  0, 63, 189, 197;  0, 69, 107, 147;  0, 73, 84, 177;  0, 133, 195, 209;  0, 87, 137, 173;  0, 85, 110, 221;  0, 61, 74, 124;  0, 56, 112, 168;  1, 57, 113, 169}.
The last 2 blocks represent short orbits.
The deficiency graph is connected and has girth 6.

{\boldmath $\adfPENT(4,80)$}, $d = 5$:
{1, 4, 5, 240;  0, 2, 11, 236;  1, 3, 7, 242;  2, 4, 13, 238;  0, 3, 19, 234;  53, 146, 172, 79;  3, 180, 55, 63;  80, 186, 42, 149;  142, 191, 232, 163;  71, 87, 142, 28;  31, 226, 219, 50;  201, 14, 182, 69;  114, 87, 198, 29;  52, 229, 27, 140;  84, 133, 156, 92;  81, 217, 40, 184;  144, 46, 204, 108;  36, 61, 127, 83;  123, 243, 187, 59;  84, 96, 2, 19;  57, 12, 152, 0;  66, 161, 226, 134;  220, 139, 190, 23;  47, 183, 143, 1;  171, 213, 25, 51;  16, 133, 94, 44;  214, 194, 90, 169;  87, 152, 116, 244;  61, 157, 39, 237;  33, 111, 51, 174;  108, 113, 208, 47;  29, 115, 223, 88;  126, 21, 123, 198;  173, 66, 149, 231;  215, 14, 230, 56;  0, 7, 21, 136;  0, 13, 46, 101;  0, 14, 31, 206;  0, 16, 171, 225;  0, 17, 51, 151;  0, 18, 56, 166;  0, 22, 66, 141;  0, 24, 121, 161;  0, 23, 96, 131;  0, 25, 181, 211;  0, 27, 81, 126;  0, 32, 201, 216;  0, 37, 111, 116;  0, 28, 76, 119;  0, 39, 91, 147;  0, 43, 196, 219;  0, 64, 221, 239;  0, 72, 172, 231;  1, 8, 32, 199;  1, 9, 14, 133;  0, 61, 93, 148;  1, 13, 28, 58;  1, 18, 39, 83;  1, 22, 38, 123;  1, 34, 68, 148;  1, 52, 74, 78;  1, 49, 128, 222;  1, 94, 134, 233;  1, 113, 119, 214;  1, 82, 122, 184;  1, 132, 162, 238;  1, 144, 157, 244;  1, 53, 193, 207;  1, 107, 139, 142;  0, 33, 74, 184;  0, 34, 105, 154;  0, 40, 99, 134;  0, 42, 109, 189;  0, 47, 54, 209;  0, 52, 149, 224;  0, 58, 114, 229;  2, 17, 59, 199;  0, 75, 204, 214;  0, 38, 84, 112;  0, 113, 199, 228;  2, 33, 144, 197;  2, 14, 123, 193;  2, 148, 198, 229;  0, 179, 198, 217;  2, 39, 162, 213;  2, 54, 77, 187;  0, 35, 127, 197;  0, 83, 122, 145;  0, 77, 163, 192;  0, 97, 143, 168;  0, 55, 118, 208;  0, 67, 130, 233;  0, 80, 178, 213;  0, 73, 95, 202;  0, 62, 65, 203;  0, 70, 158, 160;  0, 53, 60, 142;  0, 50, 137, 173;  0, 102, 110, 227;  0, 45, 212, 232}.
The deficiency graph is connected and has girth 5.
Identifying code: $\{x: 0 \le x < 245,\; x \equiv 0 \text{ or } 2 \adfmod{5}\}$,
$\pi_0 = 0$, $\pi_1 = 98$, $\pi_2 = 147$.

{\boldmath $\adfPENT(4,85)$}, $d = 2$:
{1, 2, 11, 258;  0, 5, 250, 257;  185, 77, 2, 217;  223, 67, 122, 72;  219, 241, 17, 1;  71, 101, 185, 234;  124, 68, 98, 50;  112, 69, 84, 224;  255, 1, 168, 73;  116, 235, 75, 156;  212, 161, 196, 243;  178, 258, 195, 145;  103, 122, 1, 26;  138, 69, 256, 167;  24, 237, 105, 68;  56, 148, 156, 78;  37, 205, 96, 245;  153, 0, 249, 125;  198, 60, 162, 74;  21, 214, 3, 89;  66, 46, 87, 244;  23, 200, 219, 128;  0, 6, 57, 95;  0, 12, 133, 157;  0, 23, 110, 214;  0, 24, 58, 85;  0, 25, 32, 98;  0, 33, 134, 176;  0, 55, 143, 152;  0, 45, 106, 174;  0, 53, 90, 206;  0, 38, 137, 166;  0, 35, 161, 184;  0, 43, 60, 123;  0, 71, 115, 146;  0, 64, 177, 239;  0, 167, 221, 233;  0, 39, 52, 193;  0, 129, 189, 203;  1, 3, 49, 119;  1, 35, 91, 185;  0, 59, 181, 207;  0, 65, 130, 195}.
The last block represents a short orbit.
The deficiency graph is connected and has girth 5.
Identifying code: $\{x + 20i: x \in \{0,2,8,9,11,13,14,15\}, 0 \le i < 13\}$,
$\pi_0 = 0$, $\pi_1 = 104$, $\pi_2 = 156$.

{\boldmath $\adfPENT(4,89)$}, $d = 2$:
{1, 4, 243, 268;  0, 25, 30, 249;  176, 92, 128, 235;  19, 34, 220, 219;  90, 216, 1, 204;  36, 51, 142, 169;  260, 65, 160, 209;  258, 5, 165, 185;  145, 226, 246, 187;  168, 72, 219, 250;  173, 49, 179, 213;  141, 108, 237, 137;  163, 77, 262, 165;  194, 173, 67, 129;  0, 2, 9, 265;  0, 3, 11, 229;  0, 6, 23, 105;  0, 10, 31, 165;  0, 13, 215, 250;  0, 14, 53, 245;  0, 16, 111, 209;  0, 18, 61, 223;  0, 24, 65, 211;  0, 26, 63, 153;  0, 28, 103, 163;  0, 34, 101, 115;  0, 40, 137, 159;  0, 32, 149, 201;  0, 45, 161, 208;  0, 91, 123, 168;  1, 13, 51, 115;  0, 52, 125, 203;  0, 83, 141, 261;  0, 131, 157, 167;  0, 85, 113, 189;  1, 19, 75, 141;  0, 47, 180, 235;  0, 93, 110, 148;  0, 89, 102, 164;  0, 42, 121, 140;  0, 58, 118, 196;  0, 54, 128, 184;  0, 50, 120, 200;  0, 44, 90, 156;  0, 68, 136, 204;  1, 69, 137, 205}.
The last 2 blocks represent short orbits.
The deficiency graph is connected and has girth 6.

{\boldmath $\adfPENT(4,100)$}, $d = 5$:
{1, 4, 5, 300;  0, 2, 11, 296;  1, 3, 7, 302;  2, 4, 13, 298;  0, 3, 19, 294;  56, 172, 286, 104;  75, 294, 157, 297;  177, 87, 245, 239;  90, 265, 11, 147;  100, 208, 25, 164;  21, 101, 29, 238;  13, 41, 215, 1;  206, 140, 162, 30;  151, 237, 143, 62;  113, 251, 300, 12;  108, 191, 50, 30;  60, 285, 301, 114;  279, 196, 96, 12;  83, 133, 206, 12;  262, 102, 81, 236;  164, 43, 146, 204;  188, 66, 193, 29;  99, 52, 297, 12;  156, 78, 207, 41;  22, 9, 181, 288;  84, 74, 171, 62;  191, 45, 226, 178;  75, 149, 56, 282;  30, 107, 220, 173;  261, 209, 13, 117;  273, 66, 153, 121;  205, 44, 270, 20;  288, 83, 249, 52;  8, 294, 243, 222;  168, 36, 224, 280;  220, 125, 66, 162;  21, 178, 216, 190;  155, 32, 239, 134;  65, 2, 273, 219;  301, 270, 149, 105;  0, 7, 21, 126;  0, 8, 41, 101;  0, 12, 36, 276;  0, 13, 71, 156;  0, 14, 81, 96;  0, 18, 86, 211;  0, 15, 121, 251;  0, 23, 46, 186;  0, 27, 56, 191;  0, 29, 76, 221;  0, 30, 201, 291;  0, 34, 111, 231;  1, 6, 31, 44;  1, 8, 236, 252;  1, 18, 24, 261;  0, 28, 216, 266;  1, 32, 211, 244;  0, 32, 206, 282;  0, 39, 47, 281;  0, 38, 166, 257;  0, 63, 232, 271;  1, 23, 48, 212;  0, 52, 59, 116;  1, 28, 47, 287;  1, 29, 53, 102;  1, 37, 73, 198;  1, 57, 82, 207;  1, 43, 74, 152;  1, 62, 79, 272;  1, 83, 129, 257;  1, 54, 103, 177;  1, 89, 93, 107;  1, 99, 104, 188;  1, 148, 169, 303;  1, 139, 159, 288;  1, 164, 193, 247;  1, 144, 179, 243;  1, 149, 194, 294;  1, 157, 263, 304;  1, 173, 202, 279;  1, 174, 253, 289;  1, 199, 232, 284;  1, 204, 274, 299;  1, 214, 237, 264;  0, 25, 114, 264;  0, 33, 124, 214;  0, 35, 104, 244;  0, 40, 149, 229;  0, 42, 234, 289;  0, 43, 129, 194;  0, 49, 174, 255;  0, 94, 137, 269;  0, 53, 179, 254;  0, 48, 159, 279;  0, 83, 164, 274;  2, 18, 84, 254;  0, 62, 119, 262;  0, 72, 142, 259;  2, 28, 99, 159;  0, 138, 199, 227;  0, 68, 127, 169;  2, 39, 98, 208;  2, 124, 158, 243;  2, 37, 137, 204;  2, 34, 82, 197;  2, 53, 129, 223;  0, 88, 204, 267;  0, 60, 162, 212;  0, 67, 70, 287;  0, 122, 167, 197;  2, 17, 58, 188;  0, 107, 135, 292;  0, 73, 97, 200;  0, 87, 198, 215;  0, 125, 253, 297;  0, 117, 150, 298;  0, 93, 173, 252;  0, 112, 188, 278;  0, 92, 243, 283;  0, 85, 248, 303;  0, 45, 145, 268;  0, 98, 203, 238;  0, 168, 233, 263;  0, 153, 213, 228;  0, 113, 158, 273}.
The deficiency graph is connected and has girth 5.
Identifying code: $\{x: 0 \le x < 305,\; x \equiv 0 \text{ or } 2 \adfmod{5}\}$,
$\pi_0 = 0$, $\pi_1 = 122$, $\pi_2 = 183$.

{\boldmath $\adfPENT(4,105)$}, $d = 2$:
{1, 2, 11, 318;  0, 5, 310, 317;  143, 162, 262, 49;  94, 246, 31, 188;  105, 2, 182, 32;  174, 276, 160, 75;  189, 227, 240, 61;  106, 215, 291, 308;  280, 21, 35, 254;  262, 18, 52, 241;  128, 303, 250, 83;  76, 39, 245, 127;  132, 40, 12, 277;  20, 281, 51, 279;  40, 296, 62, 257;  103, 78, 311, 264;  264, 158, 138, 174;  116, 151, 274, 186;  91, 220, 285, 313;  310, 128, 305, 219;  254, 306, 283, 225;  150, 278, 66, 109;  0, 6, 18, 260;  0, 8, 56, 77;  0, 17, 274, 293;  0, 23, 72, 96;  0, 27, 136, 173;  0, 33, 112, 144;  0, 39, 216, 278;  0, 40, 85, 206;  0, 41, 50, 238;  0, 44, 99, 174;  0, 49, 164, 231;  0, 57, 83, 222;  0, 59, 74, 252;  0, 71, 172, 251;  0, 89, 196, 313;  0, 38, 133, 149;  0, 54, 167, 179;  0, 115, 135, 157;  0, 129, 165, 273;  0, 121, 183, 289;  0, 137, 203, 267;  0, 97, 171, 253;  0, 215, 263, 287;  0, 131, 187, 255;  0, 139, 249, 309;  0, 147, 193, 225;  0, 161, 277, 295;  1, 31, 85, 189;  1, 7, 41, 143;  1, 51, 121, 173;  0, 80, 160, 240;  1, 81, 161, 241}.
The last 2 blocks represent short orbits.
The deficiency graph is connected and has girth 5.
Identifying code: $\{x + 20i: x \in \{0,2,8,9,11,13,14,15\}, 0 \le i < 16\}$,
$\pi_0 = 0$, $\pi_1 = 128$, $\pi_2 = 192$.

{\boldmath $\adfPENT(4,113)$}, $d = 2$:
{12, 27, 229, 332;  67, 116, 279, 318;  136, 274, 71, 161;  63, 248, 247, 262;  204, 238, 99, 254;  324, 271, 68, 248;  71, 181, 298, 211;  221, 233, 94, 25;  143, 171, 267, 198;  43, 21, 271, 312;  153, 253, 95, 246;  169, 313, 149, 278;  268, 5, 79, 162;  269, 312, 45, 24;  270, 35, 287, 343;  226, 13, 30, 222;  171, 122, 160, 92;  189, 221, 100, 20;  187, 96, 88, 203;  151, 249, 97, 198;  78, 228, 60, 3;  0, 1, 147, 342;  0, 5, 9, 265;  0, 6, 19, 37;  0, 10, 71, 105;  0, 20, 63, 143;  0, 22, 67, 277;  0, 26, 55, 137;  0, 28, 87, 237;  0, 32, 213, 285;  0, 33, 41, 339;  0, 40, 191, 197;  0, 42, 175, 219;  0, 36, 161, 283;  0, 44, 129, 311;  0, 54, 225, 249;  0, 47, 185, 315;  0, 57, 60, 323;  0, 97, 165, 333;  0, 83, 153, 204;  0, 65, 281, 307;  0, 136, 309, 319;  0, 98, 199, 244;  0, 92, 271, 313;  0, 66, 233, 240;  1, 41, 105, 167;  1, 53, 113, 227;  1, 3, 51, 157;  0, 62, 321, 335;  0, 103, 122, 236;  0, 94, 207, 220;  0, 64, 134, 224;  0, 52, 130, 242;  0, 58, 132, 216;  0, 48, 144, 226;  0, 46, 156, 228;  0, 86, 172, 258;  1, 87, 173, 259}.
The last 2 blocks represent short orbits.
The deficiency graph is connected and has girth 6.

{\boldmath $\adfPENT(4,120)$}, $d = 5$:
{1, 4, 5, 360;  0, 2, 11, 356;  1, 3, 7, 362;  2, 4, 13, 358;  0, 3, 19, 354;  77, 334, 69, 149;  321, 215, 351, 28;  301, 34, 177, 122;  315, 323, 212, 11;  58, 209, 73, 151;  251, 66, 244, 80;  348, 261, 151, 20;  141, 258, 118, 75;  199, 301, 221, 45;  351, 138, 39, 201;  4, 36, 278, 75;  293, 162, 130, 165;  298, 327, 144, 99;  200, 108, 103, 240;  338, 280, 115, 73;  23, 154, 231, 37;  259, 99, 343, 25;  299, 333, 35, 339;  3, 29, 122, 186;  244, 46, 360, 96;  266, 355, 158, 240;  321, 112, 196, 354;  154, 276, 2, 239;  53, 328, 119, 201;  167, 68, 295, 116;  286, 349, 96, 131;  169, 58, 102, 182;  306, 106, 187, 41;  57, 199, 111, 94;  114, 96, 302, 155;  203, 24, 53, 81;  143, 362, 74, 17;  40, 221, 28, 322;  282, 358, 255, 249;  353, 68, 211, 243;  124, 107, 166, 251;  308, 233, 119, 339;  316, 227, 162, 277;  180, 107, 296, 198;  41, 319, 312, 118;  286, 82, 274, 65;  76, 264, 223, 242;  342, 147, 58, 358;  289, 16, 363, 242;  203, 34, 305, 230;  277, 217, 150, 92;  3, 314, 119, 191;  162, 138, 194, 5;  218, 99, 345, 352;  277, 10, 140, 235;  161, 215, 243, 298;  82, 90, 42, 67;  362, 113, 304, 265;  50, 32, 162, 94;  155, 233, 305, 327;  296, 215, 2, 199;  164, 11, 271, 122;  78, 54, 31, 57;  332, 199, 278, 316;  25, 200, 147, 325;  0, 12, 47, 132;  0, 13, 72, 142;  0, 14, 52, 327;  0, 15, 77, 242;  0, 16, 37, 197;  0, 20, 107, 202;  0, 21, 117, 257;  0, 23, 57, 207;  0, 24, 177, 222;  0, 29, 92, 192;  0, 31, 167, 322;  0, 33, 82, 337;  0, 53, 277, 352;  1, 6, 62, 322;  1, 8, 77, 267;  0, 34, 272, 302;  0, 38, 102, 247;  0, 36, 162, 297;  0, 48, 212, 281;  1, 9, 32, 246;  0, 54, 56, 287;  0, 49, 246, 332;  0, 73, 206, 252;  0, 63, 191, 232;  1, 13, 117, 351;  0, 59, 217, 231;  1, 23, 68, 107;  1, 14, 287, 306;  1, 37, 146, 184;  1, 18, 53, 332;  1, 29, 142, 171;  1, 44, 252, 276;  1, 49, 67, 311;  1, 94, 113, 192;  1, 38, 114, 197;  2, 33, 58, 294;  2, 28, 89, 109;  2, 38, 98, 258;  2, 104, 153, 193;  2, 29, 164, 223;  2, 124, 168, 219;  2, 14, 138, 253;  2, 43, 163, 249;  2, 229, 289, 293;  2, 178, 228, 324;  2, 103, 174, 264;  2, 189, 314, 339;  2, 48, 218, 319;  0, 25, 143, 278;  0, 30, 153, 333;  0, 68, 85, 193;  0, 55, 313, 343;  0, 88, 134, 173;  0, 113, 146, 243;  0, 148, 156, 218;  0, 198, 201, 363;  1, 46, 103, 161;  1, 108, 213, 226;  0, 98, 308, 329;  1, 69, 144, 293;  0, 46, 248, 316;  0, 96, 174, 358;  0, 93, 266, 291;  0, 188, 271, 346;  1, 93, 188, 231;  0, 64, 208, 321;  1, 139, 278, 314;  0, 86, 216, 299;  0, 90, 211, 251;  0, 70, 141, 205;  0, 50, 126, 289;  0, 91, 220, 319;  0, 41, 111, 344;  0, 69, 261, 285;  0, 100, 229, 296;  0, 120, 259, 286;  0, 89, 185, 336;  0, 105, 284, 331;  0, 45, 124, 155;  0, 184, 219, 339;  0, 114, 244, 314;  0, 104, 109, 195;  0, 159, 214, 224;  0, 60, 204, 254;  0, 84, 199, 309}.
The deficiency graph is connected and has girth 5.
Identifying code: $\{x: 0 \le x < 365,\; x \equiv 0 \text{ or } 2 \adfmod{5}\}$,
$\pi_0 = 0$, $\pi_1 = 146$, $\pi_2 = 219$.

{\boldmath $\adfPENT(4,121)$}, $d = 2$:
{16, 43, 289, 352;  37, 80, 326, 333;  341, 177, 289, 356;  337, 121, 297, 4;  33, 319, 58, 244;  130, 41, 32, 96;  195, 257, 50, 322;  205, 6, 67, 330;  44, 57, 326, 222;  344, 118, 25, 247;  287, 297, 183, 155;  249, 96, 186, 210;  146, 157, 363, 274;  203, 179, 305, 323;  292, 103, 186, 74;  15, 327, 232, 212;  204, 13, 331, 339;  113, 81, 269, 333;  131, 305, 327, 130;  85, 8, 244, 345;  202, 149, 268, 322;  71, 314, 201, 83;  209, 309, 338, 215;  328, 182, 156, 197;  258, 53, 39, 310;  0, 2, 5, 21;  0, 4, 37, 219;  0, 6, 73, 335;  0, 8, 91, 139;  0, 10, 35, 81;  0, 12, 121, 205;  0, 14, 69, 155;  0, 17, 204, 227;  0, 18, 65, 283;  0, 22, 213, 309;  0, 28, 85, 161;  0, 31, 188, 233;  0, 36, 123, 201;  0, 30, 299, 357;  0, 38, 267, 355;  0, 40, 223, 359;  0, 42, 363, 365;  0, 103, 263, 361;  0, 167, 235, 295;  0, 307, 311, 345;  1, 31, 155, 199;  0, 53, 107, 187;  0, 119, 185, 259;  0, 277, 347, 367;  0, 60, 173, 351;  0, 88, 231, 257;  0, 93, 130, 214;  0, 62, 221, 292;  0, 51, 78, 194;  0, 68, 202, 349;  0, 70, 144, 268;  0, 56, 136, 208;  0, 46, 140, 266;  0, 58, 168, 250;  0, 48, 156, 206;  0, 92, 184, 276;  1, 93, 185, 277}.
The last 2 blocks represent short orbits.
The deficiency graph is connected and has girth 6.

{\boldmath $\adfPENT(4,125)$}, $d = 2$:
{1, 2, 11, 378;  0, 5, 370, 377;  287, 281, 255, 161;  220, 142, 169, 5;  97, 17, 268, 257;  138, 104, 21, 357;  65, 1, 243, 297;  13, 225, 118, 31;  61, 353, 42, 217;  42, 148, 93, 206;  371, 333, 172, 297;  150, 285, 118, 170;  371, 268, 176, 132;  323, 358, 32, 309;  300, 14, 214, 295;  375, 377, 125, 304;  201, 312, 42, 276;  264, 297, 56, 227;  335, 250, 274, 13;  234, 240, 56, 356;  348, 274, 154, 109;  221, 335, 267, 34;  66, 163, 362, 224;  362, 202, 341, 164;  129, 207, 214, 278;  136, 329, 144, 167;  34, 351, 299, 178;  14, 128, 71, 365;  281, 337, 30, 132;  232, 295, 372, 21;  93, 44, 315, 293;  277, 165, 324, 311;  0, 12, 37, 67;  0, 14, 59, 211;  0, 16, 91, 353;  0, 17, 299, 358;  0, 18, 83, 279;  0, 21, 40, 268;  0, 26, 217, 313;  0, 28, 127, 318;  0, 41, 132, 162;  0, 35, 60, 284;  0, 42, 98, 246;  0, 46, 118, 188;  0, 47, 88, 245;  0, 111, 128, 341;  0, 48, 153, 298;  0, 43, 108, 314;  0, 69, 166, 349;  0, 87, 210, 365;  0, 68, 147, 280;  0, 101, 124, 327;  0, 50, 126, 276;  0, 93, 117, 225;  0, 123, 151, 163;  0, 207, 223, 273;  0, 77, 137, 347;  0, 89, 179, 255;  0, 113, 229, 371;  1, 21, 145, 193;  0, 53, 145, 227;  1, 43, 105, 177;  0, 95, 190, 285}.
The last block represents a short orbit.
The deficiency graph is connected and has girth 5.
Identifying code: $\{x + 20i: x \in \{0,2,8,9,11,13,14,15\}, 0 \le i < 19\}$,
$\pi_0 = 0$, $\pi_1 = 152$, $\pi_2 = 228$.

{\boldmath $\adfPENT(4,129)$}, $d = 2$:
{30, 271, 289, 362;  69, 104, 122, 325;  199, 73, 87, 338;  260, 368, 20, 181;  223, 304, 284, 351;  98, 293, 304, 182;  42, 349, 165, 383;  38, 89, 48, 177;  125, 254, 147, 26;  110, 172, 310, 46;  124, 226, 53, 363;  89, 115, 319, 143;  93, 218, 37, 100;  311, 256, 267, 306;  257, 376, 234, 350;  351, 294, 248, 104;  70, 278, 242, 347;  171, 247, 388, 156;  345, 236, 292, 10;  391, 371, 390, 147;  31, 29, 44, 201;  42, 105, 254, 207;  79, 262, 20, 237;  217, 355, 143, 352;  333, 27, 59, 237;  69, 274, 134, 206;  94, 184, 273, 263;  97, 143, 0, 354;  70, 297, 27, 147;  0, 2, 9, 235;  0, 4, 25, 49;  0, 6, 37, 107;  0, 8, 93, 155;  0, 13, 156, 168;  0, 17, 86, 100;  0, 16, 48, 70;  0, 19, 28, 244;  0, 24, 104, 218;  0, 27, 58, 124;  0, 33, 92, 352;  0, 29, 262, 296;  0, 35, 182, 256;  0, 39, 298, 350;  0, 43, 170, 304;  0, 61, 246, 359;  0, 65, 178, 293;  0, 71, 272, 347;  0, 76, 269, 310;  0, 154, 317, 355;  0, 153, 213, 230;  0, 87, 177, 204;  0, 112, 231, 297;  0, 78, 223, 369;  0, 106, 305, 389;  0, 95, 173, 363;  0, 225, 261, 325;  0, 83, 215, 371;  0, 281, 387, 391;  0, 117, 189, 197;  0, 151, 303, 343;  0, 229, 287, 337;  1, 13, 155, 207;  1, 7, 49, 141;  1, 17, 131, 161;  0, 98, 196, 294;  1, 99, 197, 295}.
The last 2 blocks represent short orbits.
The deficiency graph is connected and has girth 6.

{\boldmath $\adfPENT(4,137)$}, $d = 2$:
{166, 199, 250, 271;  146, 195, 218, 223;  232, 315, 45, 198;  381, 232, 385, 372;  18, 69, 404, 129;  225, 117, 106, 330;  131, 1, 283, 64;  281, 99, 20, 297;  222, 338, 365, 105;  63, 329, 23, 136;  37, 52, 262, 257;  299, 326, 197, 241;  28, 409, 120, 87;  215, 387, 311, 135;  236, 174, 361, 304;  251, 357, 162, 387;  69, 304, 300, 118;  362, 412, 386, 298;  315, 10, 267, 301;  273, 248, 41, 396;  365, 119, 155, 41;  97, 282, 362, 351;  115, 153, 27, 108;  395, 412, 158, 216;  137, 100, 17, 139;  318, 211, 412, 240;  62, 321, 274, 289;  215, 84, 169, 239;  157, 177, 59, 219;  193, 293, 305, 362;  133, 336, 352, 159;  239, 367, 18, 280;  328, 241, 155, 114;  0, 1, 53, 414;  0, 6, 23, 385;  0, 8, 99, 327;  0, 10, 145, 407;  0, 12, 43, 159;  0, 14, 107, 297;  0, 18, 121, 253;  0, 19, 109, 233;  0, 20, 315, 337;  0, 22, 137, 395;  0, 29, 32, 171;  0, 28, 101, 245;  0, 35, 48, 313;  0, 38, 161, 361;  0, 36, 133, 201;  0, 54, 129, 321;  0, 44, 157, 207;  1, 19, 167, 231;  0, 167, 251, 307;  0, 173, 341, 351;  0, 95, 269, 275;  0, 281, 355, 363;  0, 98, 273, 339;  0, 71, 110, 232;  0, 74, 285, 330;  0, 169, 176, 216;  0, 63, 134, 290;  0, 66, 142, 391;  0, 102, 279, 280;  0, 65, 124, 180;  0, 46, 152, 272;  0, 52, 164, 246;  0, 70, 188, 288;  0, 42, 138, 284;  0, 60, 150, 308;  0, 104, 208, 312;  1, 105, 209, 313}.
The last 2 blocks represent short orbits.
The deficiency graph is connected and has girth 6.

{\boldmath $\adfPENT(4,140)$}, $d = 5$:
{1, 4, 5, 420;  0, 2, 11, 416;  1, 3, 7, 422;  2, 4, 13, 418;  0, 3, 19, 414;  34, 285, 230, 278;  231, 351, 156, 209;  253, 324, 200, 111;  411, 273, 82, 357;  282, 307, 345, 268;  102, 260, 161, 218;  77, 160, 395, 221;  148, 313, 60, 54;  50, 187, 150, 315;  302, 203, 380, 37;  47, 226, 154, 142;  186, 284, 335, 264;  250, 318, 117, 184;  79, 286, 161, 69;  52, 170, 2, 55;  209, 78, 174, 63;  17, 35, 386, 353;  3, 316, 170, 271;  215, 84, 193, 327;  242, 324, 22, 51;  382, 120, 161, 26;  238, 409, 274, 197;  50, 223, 264, 418;  412, 388, 29, 98;  171, 229, 178, 372;  51, 122, 48, 279;  223, 176, 27, 126;  33, 424, 393, 175;  389, 79, 407, 350;  49, 134, 401, 251;  194, 252, 421, 231;  121, 323, 405, 336;  31, 388, 201, 167;  35, 277, 246, 373;  104, 150, 351, 16;  175, 337, 352, 397;  402, 132, 385, 63;  197, 370, 229, 161;  229, 99, 334, 402;  206, 314, 153, 220;  98, 201, 372, 345;  330, 401, 282, 38;  39, 295, 58, 220;  223, 391, 113, 349;  358, 123, 131, 236;  231, 285, 349, 201;  421, 38, 161, 334;  208, 254, 313, 357;  62, 249, 160, 287;  79, 83, 133, 16;  26, 408, 353, 330;  283, 100, 315, 243;  294, 367, 333, 56;  377, 32, 105, 207;  182, 68, 213, 105;  9, 58, 136, 120;  370, 314, 291, 49;  143, 245, 52, 181;  179, 324, 136, 69;  198, 224, 176, 68;  293, 415, 214, 71;  89, 299, 293, 373;  147, 81, 115, 262;  142, 360, 221, 382;  330, 3, 209, 211;  241, 56, 408, 209;  153, 327, 176, 248;  52, 117, 247, 19;  128, 211, 111, 391;  91, 345, 394, 52;  1, 187, 207, 324;  0, 7, 21, 56;  0, 8, 66, 241;  0, 12, 36, 151;  0, 13, 76, 186;  0, 14, 31, 111;  0, 18, 161, 176;  0, 15, 181, 236;  0, 20, 46, 386;  0, 24, 51, 256;  0, 28, 156, 316;  0, 25, 231, 271;  0, 29, 136, 196;  0, 34, 96, 296;  0, 38, 311, 381;  0, 43, 131, 261;  0, 42, 106, 376;  0, 72, 91, 116;  0, 40, 126, 266;  0, 52, 301, 396;  1, 9, 14, 66;  0, 83, 401, 406;  1, 13, 62, 238;  0, 44, 191, 268;  0, 73, 321, 408;  1, 17, 33, 38;  1, 27, 34, 243;  1, 19, 107, 133;  1, 28, 84, 248;  1, 29, 42, 234;  1, 52, 104, 263;  1, 47, 87, 293;  1, 77, 94, 333;  1, 114, 139, 272;  1, 39, 132, 167;  1, 112, 134, 302;  1, 53, 117, 264;  1, 102, 183, 204;  1, 122, 169, 249;  1, 93, 217, 244;  1, 127, 157, 269;  1, 142, 212, 267;  1, 147, 252, 419;  1, 197, 233, 337;  1, 124, 262, 308;  1, 198, 289, 408;  1, 173, 227, 317;  1, 213, 232, 352;  1, 129, 292, 379;  1, 273, 332, 398;  1, 242, 314, 378;  1, 254, 294, 349;  1, 237, 322, 414;  0, 30, 153, 308;  0, 33, 93, 213;  0, 35, 97, 203;  0, 45, 163, 193;  0, 54, 198, 233;  0, 58, 158, 298;  0, 63, 138, 250;  0, 67, 223, 398;  2, 39, 163, 188;  2, 53, 123, 229;  0, 99, 273, 423;  0, 113, 199, 228;  0, 128, 157, 288;  0, 69, 117, 413;  0, 152, 208, 378;  2, 63, 148, 404;  0, 122, 293, 402;  0, 139, 247, 373;  2, 244, 268, 384;  2, 219, 263, 308;  0, 144, 243, 397;  2, 78, 154, 399;  2, 73, 209, 262;  3, 124, 199, 289;  0, 227, 328, 389;  0, 92, 192, 333;  0, 60, 202, 417;  0, 80, 317, 392;  2, 59, 137, 247;  0, 79, 197, 340;  0, 182, 195, 384;  0, 87, 149, 334;  0, 219, 269, 332;  0, 105, 234, 382;  0, 167, 299, 399;  0, 65, 155, 367;  0, 135, 319, 322;  0, 47, 185, 255;  0, 134, 329, 337;  0, 279, 339, 404;  0, 140, 349, 394;  0, 84, 220, 330;  0, 74, 194, 295;  0, 114, 264, 280;  0, 50, 159, 275;  0, 119, 274, 344;  0, 59, 120, 245}.
The deficiency graph is connected and has girth 5.
Identifying code: $\{x: 0 \le x < 425,\; x \equiv 0 \text{ or } 2 \adfmod{5}\}$,
$\pi_0 = 0$, $\pi_1 = 170$, $\pi_2 = 255$.

{\boldmath $\adfPENT(4,145)$}, $d = 2$:
{1, 2, 11, 438;  0, 5, 430, 437;  38, 219, 269, 231;  398, 240, 153, 46;  421, 43, 293, 94;  306, 197, 238, 292;  32, 265, 96, 337;  87, 119, 132, 11;  59, 403, 292, 439;  106, 259, 76, 397;  118, 290, 314, 192;  131, 305, 160, 154;  236, 106, 412, 164;  72, 343, 238, 161;  332, 310, 162, 78;  94, 407, 218, 416;  193, 250, 389, 218;  128, 123, 269, 222;  284, 22, 178, 5;  437, 333, 190, 112;  361, 65, 363, 183;  168, 189, 268, 103;  75, 435, 172, 321;  76, 205, 139, 369;  148, 106, 261, 203;  406, 45, 326, 26;  210, 118, 189, 401;  419, 341, 174, 85;  299, 381, 224, 433;  305, 366, 139, 52;  411, 24, 249, 374;  200, 107, 39, 428;  155, 275, 60, 23;  0, 59, 86, 259;  0, 8, 31, 381;  0, 12, 40, 56;  0, 17, 280, 422;  0, 20, 45, 234;  0, 26, 62, 108;  0, 27, 33, 230;  0, 29, 336, 402;  0, 34, 73, 202;  0, 41, 146, 364;  0, 43, 278, 392;  0, 49, 132, 390;  0, 52, 109, 252;  0, 53, 190, 255;  0, 61, 96, 224;  0, 70, 163, 185;  0, 69, 138, 219;  0, 77, 184, 338;  0, 83, 116, 266;  0, 85, 103, 236;  0, 121, 135, 288;  0, 125, 165, 304;  0, 137, 144, 355;  0, 180, 377, 397;  0, 112, 339, 393;  0, 131, 201, 276;  0, 120, 323, 349;  0, 175, 367, 401;  0, 127, 285, 369;  0, 123, 249, 409;  0, 239, 337, 425;  0, 329, 385, 429;  0, 187, 309, 421;  0, 235, 299, 391;  0, 223, 265, 359;  0, 101, 269, 317;  1, 17, 47, 317;  1, 25, 173, 289;  1, 29, 103, 233;  0, 110, 220, 330;  1, 111, 221, 331}.
The last 2 blocks represent short orbits.
The deficiency graph is connected and has girth 5.
Identifying code: $\{x + 20i: x \in \{0,2,8,9,11,13,14,15\}, 0 \le i < 22\}$,
$\pi_0 = 0$, $\pi_1 = 176$, $\pi_2 = 264$.

{\boldmath $\adfPENT(4,153)$}, $d = 2$:
{71, 121, 186, 278;  225, 241, 344, 394;  427, 454, 14, 93;  12, 442, 11, 59;  228, 359, 24, 387;  235, 133, 18, 267;  388, 143, 424, 366;  16, 270, 357, 201;  357, 222, 148, 399;  419, 333, 125, 29;  448, 129, 247, 431;  307, 413, 208, 117;  342, 65, 296, 90;  22, 63, 123, 175;  212, 101, 408, 323;  278, 156, 158, 6;  148, 20, 9, 80;  1, 356, 215, 4;  442, 379, 169, 375;  145, 199, 47, 235;  182, 287, 437, 192;  217, 37, 408, 168;  367, 149, 284, 314;  450, 211, 408, 272;  105, 238, 92, 40;  267, 348, 344, 223;  441, 220, 158, 300;  82, 53, 247, 179;  304, 202, 431, 197;  114, 149, 135, 65;  421, 349, 344, 258;  12, 129, 202, 119;  278, 421, 329, 52;  76, 50, 32, 158;  91, 418, 239, 319;  34, 147, 130, 107;  298, 147, 437, 222;  387, 240, 95, 268;  21, 204, 43, 67;  196, 438, 98, 176;  272, 208, 444, 251;  0, 1, 88, 458;  0, 3, 40, 174;  0, 8, 19, 304;  0, 9, 250, 450;  0, 15, 56, 448;  0, 12, 37, 306;  0, 27, 302, 432;  0, 29, 48, 364;  0, 39, 276, 321;  0, 57, 166, 220;  0, 61, 70, 154;  0, 38, 144, 284;  0, 55, 75, 90;  0, 67, 69, 332;  0, 63, 280, 346;  0, 85, 234, 385;  0, 110, 235, 248;  0, 169, 176, 347;  0, 123, 237, 282;  0, 104, 297, 421;  0, 114, 275, 308;  0, 149, 208, 375;  0, 164, 337, 367;  0, 213, 239, 425;  0, 175, 231, 269;  0, 155, 259, 359;  0, 133, 243, 309;  0, 197, 261, 279;  0, 89, 287, 433;  0, 289, 411, 417;  1, 89, 197, 333;  1, 9, 167, 201;  0, 253, 265, 407;  1, 63, 139, 387;  0, 207, 351, 409;  0, 116, 232, 348;  1, 117, 233, 349}.
The last 2 blocks represent short orbits.
The deficiency graph is connected and has girth 6.

{\boldmath $\adfPENT(4,160)$}, $d = 5$:
{1, 4, 5, 480;  0, 2, 11, 476;  1, 3, 7, 482;  2, 4, 13, 478;  0, 3, 19, 474;  337, 157, 378, 322;  118, 269, 447, 484;  9, 454, 250, 12;  385, 364, 122, 275;  20, 377, 402, 432;  327, 466, 186, 250;  4, 311, 482, 382;  250, 192, 464, 39;  413, 319, 249, 57;  231, 473, 93, 455;  401, 37, 261, 335;  385, 29, 383, 180;  364, 163, 127, 184;  111, 308, 330, 15;  367, 353, 68, 308;  296, 455, 124, 6;  150, 191, 147, 263;  230, 3, 136, 65;  477, 267, 88, 185;  432, 96, 225, 119;  179, 275, 47, 302;  288, 440, 186, 141;  210, 410, 158, 91;  382, 161, 49, 207;  249, 149, 463, 166;  174, 431, 361, 133;  453, 54, 261, 280;  394, 18, 231, 78;  236, 372, 63, 167;  377, 330, 464, 133;  242, 372, 352, 94;  288, 423, 447, 281;  445, 183, 84, 71;  147, 249, 446, 237;  433, 295, 421, 118;  383, 254, 293, 121;  475, 353, 321, 158;  262, 422, 400, 13;  468, 74, 82, 209;  241, 293, 468, 63;  372, 217, 175, 440;  7, 384, 302, 25;  257, 194, 39, 470;  241, 472, 84, 280;  397, 461, 57, 279;  408, 457, 73, 320;  400, 188, 333, 260;  201, 478, 57, 415;  424, 115, 7, 352;  460, 339, 91, 144;  241, 447, 81, 345;  422, 472, 36, 234;  397, 23, 282, 325;  484, 293, 145, 112;  436, 326, 484, 418;  38, 329, 469, 3;  222, 156, 65, 273;  192, 58, 257, 297;  153, 98, 228, 278;  254, 36, 483, 212;  366, 388, 189, 46;  376, 453, 16, 229;  417, 166, 356, 451;  183, 325, 268, 168;  411, 165, 465, 284;  323, 53, 306, 255;  42, 262, 332, 479;  433, 132, 31, 199;  242, 78, 326, 147;  129, 265, 83, 437;  270, 126, 444, 360;  230, 185, 114, 218;  307, 370, 431, 281;  359, 90, 5, 165;  214, 11, 25, 433;  3, 317, 84, 458;  63, 88, 182, 371;  153, 349, 74, 235;  341, 286, 11, 480;  139, 460, 61, 3;  458, 387, 265, 374;  418, 4, 310, 451;  49, 465, 250, 270;  373, 249, 6, 3;  14, 353, 387, 66;  207, 7, 124, 175;  116, 389, 360, 88;  121, 211, 0, 386;  331, 380, 424, 480;  305, 122, 160, 46;  70, 356, 397, 476;  92, 295, 310, 270;  450, 120, 259, 114;  316, 187, 203, 216;  54, 433, 76, 37;  333, 16, 154, 442;  145, 69, 335, 202;  435, 305, 461, 246;  386, 483, 172, 422;  338, 450, 126, 18;  0, 74, 115, 437;  0, 7, 35, 52;  0, 8, 92, 435;  0, 13, 182, 455;  0, 12, 210, 224;  0, 21, 425, 449;  0, 31, 49, 235;  0, 23, 390, 477;  0, 34, 135, 240;  0, 37, 360, 415;  0, 38, 64, 175;  0, 46, 335, 383;  0, 56, 230, 297;  0, 59, 225, 305;  0, 39, 76, 120;  0, 65, 143, 217;  0, 97, 144, 342;  0, 62, 176, 468;  0, 79, 83, 153;  0, 101, 128, 337;  0, 103, 181, 263;  0, 106, 171, 323;  0, 93, 99, 353;  0, 94, 104, 163;  0, 133, 184, 228;  0, 117, 136, 243;  0, 131, 146, 242;  0, 112, 147, 462;  0, 151, 167, 313;  0, 132, 178, 398;  0, 158, 202, 234;  0, 177, 268, 361;  0, 188, 227, 378;  0, 198, 206, 312;  0, 159, 191, 387;  0, 199, 321, 478;  0, 218, 314, 407;  0, 236, 316, 358;  0, 201, 232, 448;  0, 252, 276, 443;  0, 248, 372, 453;  0, 229, 254, 368;  0, 249, 338, 458;  0, 253, 306, 392;  0, 298, 401, 438;  0, 344, 394, 469;  0, 256, 318, 439;  0, 329, 367, 393;  0, 351, 402, 424;  0, 397, 451, 472;  0, 362, 454, 456;  0, 301, 376, 416;  0, 278, 432, 459;  0, 399, 419, 461;  1, 154, 398, 463;  1, 59, 88, 388;  2, 23, 402, 468;  1, 34, 208, 418;  1, 128, 283, 344;  1, 133, 138, 282;  2, 64, 198, 433;  1, 48, 171, 327;  1, 114, 303, 413;  1, 64, 117, 328;  1, 57, 223, 349;  2, 127, 263, 429;  1, 182, 268, 324;  2, 234, 258, 469;  1, 58, 294, 382;  1, 152, 212, 443;  2, 314, 409, 458;  2, 253, 284, 414;  1, 104, 132, 253;  2, 79, 283, 319;  2, 169, 203, 344;  2, 82, 359, 464;  1, 257, 439, 472;  1, 146, 292, 412;  1, 86, 247, 397;  1, 29, 192, 377;  1, 36, 127, 352;  1, 217, 269, 379;  2, 159, 204, 389;  1, 26, 227, 381;  1, 149, 287, 474;  1, 51, 236, 334;  1, 31, 256, 384;  1, 136, 394, 399;  1, 6, 239, 409;  1, 174, 181, 241;  1, 189, 224, 414;  1, 69, 124, 389;  1, 9, 89, 209;  1, 254, 369, 459;  1, 44, 109, 194}.
The deficiency graph is connected and has girth 5.
Identifying code: $\{x: 0 \le x < 485,\; x \equiv 0 \text{ or } 2 \adfmod{5}\}$,
$\pi_0 = 0$, $\pi_1 = 194$, $\pi_2 = 291$.

{\boldmath $\adfPENT(4,161)$}, $d = 2$:
{62, 315, 426, 451;  35, 38, 174, 455;  386, 288, 107, 136;  327, 266, 399, 124;  147, 234, 359, 164;  356, 477, 19, 484;  374, 270, 261, 31;  120, 299, 92, 291;  194, 145, 435, 80;  211, 302, 154, 199;  5, 46, 52, 105;  203, 104, 117, 280;  340, 376, 430, 467;  279, 265, 448, 468;  426, 47, 392, 122;  383, 50, 66, 407;  217, 171, 215, 210;  413, 283, 45, 349;  148, 228, 451, 317;  236, 147, 413, 192;  345, 318, 159, 368;  361, 287, 479, 202;  239, 58, 168, 264;  224, 413, 373, 118;  70, 333, 344, 441;  309, 446, 113, 168;  230, 286, 288, 145;  246, 457, 260, 284;  99, 301, 425, 224;  308, 7, 473, 152;  465, 455, 326, 245;  465, 241, 181, 478;  458, 332, 52, 232;  298, 438, 341, 374;  456, 270, 7, 363;  465, 135, 218, 28;  180, 293, 274, 10;  478, 85, 152, 468;  77, 203, 409, 192;  95, 1, 44, 448;  80, 59, 247, 199;  96, 475, 315, 331;  474, 34, 216, 217;  186, 428, 212, 445;  165, 255, 481, 107;  0, 3, 4, 470;  0, 8, 40, 350;  0, 9, 416, 428;  0, 15, 260, 289;  0, 30, 118, 358;  0, 31, 232, 436;  0, 33, 86, 442;  0, 35, 164, 213;  0, 47, 66, 268;  0, 63, 254, 420;  0, 55, 212, 380;  0, 67, 73, 290;  0, 42, 117, 376;  0, 92, 193, 288;  0, 87, 144, 222;  0, 115, 174, 294;  0, 102, 236, 339;  0, 161, 192, 483;  0, 111, 261, 330;  0, 123, 227, 300;  0, 150, 335, 473;  0, 74, 367, 387;  0, 116, 361, 453;  0, 205, 369, 423;  0, 157, 175, 273;  0, 147, 257, 427;  0, 137, 327, 365;  0, 239, 301, 381;  0, 153, 425, 461;  0, 215, 383, 409;  1, 53, 129, 235;  1, 5, 33, 89;  1, 43, 113, 289;  1, 79, 233, 315;  1, 51, 147, 387;  0, 122, 244, 366;  1, 123, 245, 367}.
The last 2 blocks represent short orbits.
The deficiency graph is connected and has girth 6.

{\boldmath $\adfPENT(4,165)$}, $d = 2$:
{1, 2, 11, 498;  0, 5, 490, 497;  356, 397, 242, 121;  219, 311, 376, 57;  37, 379, 441, 422;  116, 281, 7, 40;  437, 97, 248, 374;  478, 15, 315, 181;  72, 470, 185, 404;  434, 216, 278, 17;  438, 28, 432, 71;  382, 302, 185, 2;  338, 74, 225, 370;  246, 215, 69, 463;  344, 337, 207, 225;  449, 495, 311, 92;  95, 100, 332, 152;  175, 207, 403, 34;  455, 341, 48, 476;  141, 55, 107, 324;  43, 167, 416, 235;  15, 13, 363, 387;  374, 45, 192, 324;  199, 182, 328, 100;  255, 167, 98, 190;  18, 293, 277, 448;  127, 56, 147, 390;  24, 131, 38, 271;  361, 328, 141, 62;  101, 137, 419, 388;  99, 231, 422, 110;  193, 111, 426, 210;  497, 112, 130, 471;  327, 125, 414, 72;  141, 425, 219, 114;  436, 250, 48, 409;  158, 187, 482, 348;  418, 87, 439, 194;  317, 51, 354, 422;  45, 104, 392, 134;  159, 137, 78, 253;  407, 79, 349, 2;  0, 8, 55, 95;  0, 12, 35, 451;  0, 16, 61, 135;  0, 20, 227, 307;  0, 22, 419, 449;  0, 24, 167, 423;  0, 25, 67, 465;  0, 26, 187, 295;  0, 28, 131, 453;  0, 34, 229, 305;  0, 36, 111, 321;  0, 38, 471, 485;  0, 39, 137, 335;  0, 42, 277, 431;  0, 49, 373, 429;  0, 51, 89, 101;  0, 57, 179, 289;  0, 44, 117, 297;  0, 64, 325, 481;  1, 67, 137, 241;  0, 149, 291, 459;  1, 49, 213, 313;  0, 199, 243, 437;  0, 163, 191, 377;  0, 109, 279, 351;  0, 193, 206, 415;  1, 55, 119, 411;  0, 128, 455, 461;  0, 48, 145, 170;  0, 123, 178, 350;  0, 85, 164, 270;  0, 78, 225, 316;  0, 130, 331, 340;  0, 74, 162, 382;  0, 84, 208, 362;  0, 86, 226, 326;  0, 56, 116, 252;  0, 46, 194, 302;  0, 40, 94, 442;  0, 104, 214, 358;  0, 125, 250, 375}.
The last block represents a short orbit.
The deficiency graph is connected and has girth 5.
Identifying code: $\{x + 20i: x \in \{0,2,8,9,11,13,14,15\}, 0 \le i < 25\}$,
$\pi_0 = 0$, $\pi_1 = 200$, $\pi_2 = 300$.

{\boldmath $\adfPENT(4,169)$}, $d = 2$:
{188, 199, 324, 445;  68, 175, 314, 339;  274, 151, 497, 21;  322, 99, 180, 176;  160, 427, 489, 324;  99, 293, 372, 155;  225, 329, 149, 154;  130, 468, 465, 61;  191, 219, 469, 21;  252, 446, 168, 313;  365, 204, 306, 145;  373, 436, 97, 392;  56, 174, 240, 387;  476, 143, 47, 345;  47, 352, 161, 147;  267, 107, 351, 42;  265, 410, 394, 318;  74, 159, 355, 324;  203, 205, 59, 416;  470, 438, 397, 194;  6, 453, 340, 253;  347, 418, 155, 481;  47, 403, 198, 168;  17, 242, 133, 224;  258, 401, 372, 459;  182, 86, 446, 78;  215, 222, 505, 95;  237, 182, 125, 272;  1, 243, 10, 177;  363, 120, 239, 81;  389, 206, 372, 112;  355, 36, 229, 437;  253, 299, 75, 164;  406, 217, 476, 179;  250, 277, 278, 347;  222, 331, 296, 191;  411, 239, 171, 259;  340, 257, 9, 224;  86, 216, 347, 328;  310, 211, 89, 346;  469, 6, 464, 0;  257, 376, 209, 239;  77, 213, 88, 127;  499, 357, 507, 458;  357, 454, 200, 431;  0, 1, 108, 111;  0, 2, 23, 75;  0, 7, 13, 197;  0, 9, 38, 53;  0, 10, 151, 495;  0, 12, 79, 409;  0, 14, 51, 325;  0, 20, 77, 137;  0, 24, 115, 327;  0, 22, 191, 263;  0, 26, 237, 491;  0, 42, 195, 229;  0, 46, 363, 395;  0, 45, 99, 462;  0, 47, 427, 467;  0, 43, 275, 437;  0, 106, 245, 461;  0, 86, 313, 337;  0, 81, 123, 333;  0, 129, 302, 487;  1, 91, 189, 295;  1, 5, 17, 27;  0, 127, 221, 374;  1, 79, 227, 307;  0, 58, 159, 356;  0, 64, 343, 380;  0, 98, 288, 451;  0, 60, 208, 417;  0, 177, 192, 232;  0, 110, 308, 479;  0, 78, 238, 371;  0, 68, 341, 354;  0, 93, 206, 386;  0, 80, 269, 362;  0, 56, 218, 340;  0, 72, 212, 312;  0, 52, 134, 222;  0, 34, 202, 326;  0, 62, 182, 358;  0, 128, 256, 384;  1, 129, 257, 385}.
The last 2 blocks represent short orbits.
The deficiency graph is connected and has girth 6.

{\boldmath $\adfPENT(4,177)$}, $d = 2$:
{232, 304, 423, 493;  44, 114, 165, 373;  290, 49, 180, 427;  231, 100, 322, 475;  358, 65, 129, 332;  447, 32, 412, 255;  338, 458, 13, 411;  115, 442, 120, 104;  463, 458, 358, 366;  335, 98, 373, 137;  139, 229, 150, 14;  198, 247, 75, 256;  203, 8, 279, 513;  492, 19, 94, 27;  69, 140, 446, 85;  367, 260, 61, 191;  23, 112, 150, 240;  356, 108, 106, 467;  236, 311, 494, 218;  289, 33, 237, 355;  214, 499, 175, 397;  95, 297, 51, 315;  205, 88, 289, 15;  406, 521, 111, 209;  533, 38, 132, 283;  98, 526, 382, 263;  377, 355, 483, 433;  198, 299, 485, 6;  103, 60, 190, 470;  40, 331, 154, 153;  264, 338, 159, 274;  31, 289, 186, 285;  35, 159, 98, 241;  141, 277, 257, 510;  216, 186, 349, 253;  218, 137, 406, 196;  183, 177, 494, 103;  326, 270, 506, 16;  203, 450, 131, 156;  104, 206, 503, 161;  471, 336, 309, 21;  160, 456, 18, 77;  381, 188, 242, 15;  521, 303, 46, 270;  98, 63, 360, 207;  369, 247, 502, 221;  312, 74, 160, 215;  128, 153, 372, 151;  0, 41, 187, 397;  0, 1, 468, 483;  0, 3, 301, 532;  0, 6, 27, 523;  0, 9, 12, 515;  0, 14, 325, 335;  0, 13, 127, 513;  0, 17, 45, 149;  0, 19, 53, 323;  0, 20, 205, 365;  0, 24, 89, 451;  0, 28, 377, 435;  0, 31, 173, 197;  0, 29, 36, 313;  0, 32, 155, 265;  0, 42, 249, 433;  0, 44, 343, 411;  0, 40, 387, 499;  0, 78, 429, 471;  1, 63, 157, 417;  0, 87, 179, 395;  1, 15, 61, 169;  1, 13, 153, 349;  0, 147, 235, 352;  0, 158, 385, 485;  0, 50, 231, 279;  0, 77, 172, 371;  0, 129, 208, 369;  0, 146, 487, 519;  0, 85, 196, 358;  0, 76, 171, 200;  0, 83, 176, 366;  0, 81, 96, 404;  0, 48, 234, 437;  0, 80, 164, 246;  0, 62, 168, 386;  0, 104, 216, 420;  0, 34, 174, 376;  0, 52, 118, 382;  0, 60, 148, 414;  0, 134, 268, 402;  1, 135, 269, 403}.
The last 2 blocks represent short orbits.
The deficiency graph is connected and has girth 6.

{\boldmath $\adfPENT(4,180)$}, $d = 5$:
{1, 4, 5, 540;  0, 2, 11, 536;  1, 3, 7, 542;  2, 4, 13, 538;  0, 3, 19, 534;  25, 377, 285, 409;  146, 165, 379, 188;  345, 383, 255, 464;  291, 504, 26, 155;  382, 35, 136, 201;  157, 386, 407, 94;  416, 70, 245, 482;  415, 445, 16, 479;  371, 174, 418, 234;  264, 456, 443, 449;  511, 257, 57, 124;  452, 106, 353, 494;  302, 434, 27, 162;  504, 77, 268, 127;  164, 404, 267, 22;  377, 52, 16, 103;  339, 508, 227, 310;  421, 97, 146, 173;  501, 354, 151, 463;  470, 97, 513, 258;  248, 169, 36, 458;  509, 321, 173, 213;  376, 361, 439, 189;  155, 25, 381, 482;  124, 178, 42, 39;  398, 383, 465, 4;  129, 413, 544, 121;  119, 296, 84, 489;  536, 490, 59, 176;  529, 418, 479, 334;  382, 407, 127, 473;  293, 31, 14, 260;  163, 406, 231, 16;  393, 495, 304, 357;  400, 241, 485, 151;  403, 359, 417, 15;  209, 467, 135, 358;  44, 199, 331, 301;  27, 258, 542, 176;  539, 246, 460, 361;  65, 416, 253, 336;  108, 496, 493, 290;  540, 484, 316, 227;  464, 288, 185, 266;  281, 298, 45, 442;  519, 465, 436, 235;  495, 387, 75, 20;  431, 335, 63, 426;  240, 268, 25, 252;  111, 242, 371, 448;  72, 462, 203, 92;  120, 219, 280, 227;  6, 66, 231, 160;  459, 125, 10, 122;  222, 381, 534, 203;  304, 361, 533, 124;  32, 388, 158, 461;  459, 181, 195, 486;  390, 346, 508, 6;  480, 400, 387, 472;  185, 230, 38, 162;  523, 393, 9, 121;  135, 104, 425, 184;  301, 371, 464, 511;  323, 518, 106, 500;  160, 533, 62, 365;  322, 137, 531, 255;  327, 58, 535, 287;  447, 237, 336, 302;  392, 126, 13, 100;  359, 30, 293, 458;  523, 164, 72, 368;  520, 219, 268, 316;  543, 328, 298, 203;  327, 130, 501, 104;  263, 432, 149, 73;  277, 524, 459, 478;  341, 501, 333, 48;  204, 290, 114, 459;  487, 526, 504, 323;  254, 331, 523, 404;  398, 434, 416, 160;  502, 537, 108, 197;  394, 404, 145, 278;  357, 155, 281, 31;  115, 272, 246, 291;  524, 364, 350, 452;  531, 136, 187, 220;  314, 539, 350, 75;  377, 497, 504, 350;  230, 281, 324, 203;  327, 248, 74, 304;  498, 524, 273, 328;  189, 72, 98, 218;  71, 512, 111, 69;  219, 311, 543, 436;  492, 230, 103, 414;  113, 313, 520, 252;  157, 446, 264, 507;  5, 125, 12, 518;  30, 237, 215, 431;  287, 412, 254, 499;  518, 374, 335, 36;  420, 141, 462, 524;  449, 17, 220, 402;  106, 412, 233, 164;  474, 519, 304, 492;  345, 339, 229, 532;  73, 2, 253, 17;  461, 226, 278, 257;  146, 192, 251, 292;  430, 288, 333, 77;  519, 200, 109, 539;  294, 88, 225, 272;  394, 87, 428, 371;  0, 212, 318, 456;  0, 8, 234, 424;  0, 13, 109, 394;  0, 16, 134, 474;  0, 15, 59, 159;  0, 17, 269, 389;  0, 20, 314, 419;  0, 21, 149, 529;  0, 24, 36, 359;  0, 25, 219, 439;  0, 32, 199, 204;  0, 31, 139, 374;  0, 41, 89, 164;  0, 37, 404, 499;  0, 35, 444, 469;  0, 53, 254, 324;  0, 61, 309, 364;  1, 13, 184, 224;  1, 8, 29, 299;  0, 52, 379, 504;  0, 58, 82, 479;  1, 17, 114, 379;  0, 57, 179, 527;  0, 173, 417, 494;  0, 87, 289, 317;  0, 76, 129, 467;  1, 33, 82, 274;  1, 36, 244, 318;  1, 14, 57, 461;  1, 26, 264, 496;  1, 38, 272, 284;  1, 34, 72, 96;  1, 56, 168, 494;  1, 39, 446, 513;  0, 77, 472, 524;  1, 63, 133, 219;  1, 73, 174, 246;  1, 87, 188, 419;  1, 89, 153, 212;  1, 111, 228, 514;  1, 107, 269, 328;  1, 144, 238, 412;  1, 104, 172, 232;  1, 177, 353, 394;  2, 43, 192, 349;  1, 123, 192, 229;  1, 103, 167, 319;  2, 88, 312, 499;  2, 59, 72, 167;  2, 83, 229, 432;  2, 82, 474, 513;  2, 269, 428, 463;  2, 284, 308, 388;  2, 268, 319, 518;  2, 149, 338, 413;  2, 47, 224, 367;  2, 239, 243, 393;  0, 40, 177, 307;  0, 60, 222, 251;  0, 62, 117, 390;  0, 50, 422, 491;  0, 78, 132, 426;  0, 88, 302, 366;  0, 66, 442, 470;  0, 97, 111, 486;  0, 105, 247, 511;  0, 122, 196, 382;  0, 106, 383, 487;  1, 127, 297, 418;  1, 136, 332, 422;  0, 167, 306, 506;  0, 135, 387, 497;  0, 152, 261, 481;  1, 97, 181, 323;  0, 287, 381, 502;  1, 137, 242, 488;  0, 127, 241, 361;  0, 108, 200, 482;  0, 348, 392, 436;  0, 86, 322, 376;  1, 131, 357, 388;  1, 138, 191, 467;  0, 377, 473, 496;  2, 48, 183, 328;  0, 100, 357, 395;  0, 156, 378, 471;  0, 121, 323, 521;  0, 140, 288, 476;  0, 186, 283, 335;  0, 178, 281, 488;  0, 145, 298, 326;  0, 123, 221, 355;  0, 56, 343, 483;  0, 48, 286, 365;  0, 218, 278, 431;  0, 113, 235, 300;  0, 83, 193, 325;  0, 110, 268, 305;  0, 93, 143, 380;  0, 73, 503, 528;  0, 95, 328, 375;  0, 208, 320, 533;  0, 98, 163, 468;  0, 438, 538, 543;  0, 63, 248, 523}.
The deficiency graph is connected and has girth 5.
Identifying code: $\{x: 0 \le x < 545,\; x \equiv 0 \text{ or } 2 \adfmod{5}\}$,
$\pi_0 = 0$, $\pi_1 = 218$, $\pi_2 = 327$.

{\boldmath $\adfPENT(4,181)$}, $d = 2$:
{17, 206, 342, 415;  59, 134, 491, 532;  513, 213, 246, 462;  325, 82, 271, 527;  298, 275, 479, 441;  133, 478, 52, 245;  434, 154, 348, 11;  229, 448, 545, 169;  442, 354, 434, 275;  315, 225, 246, 432;  340, 188, 301, 463;  212, 122, 516, 121;  214, 121, 267, 320;  95, 33, 92, 122;  474, 188, 533, 181;  152, 114, 458, 174;  20, 195, 325, 184;  188, 69, 522, 79;  186, 334, 103, 512;  529, 82, 308, 335;  398, 306, 444, 159;  458, 248, 292, 235;  517, 473, 543, 110;  179, 14, 389, 88;  334, 465, 344, 302;  146, 142, 148, 212;  486, 67, 257, 73;  362, 191, 112, 539;  528, 17, 4, 82;  306, 488, 436, 116;  308, 93, 52, 532;  194, 2, 222, 395;  499, 108, 484, 411;  102, 160, 521, 329;  299, 342, 65, 425;  461, 453, 12, 247;  259, 520, 439, 290;  507, 157, 110, 15;  494, 130, 93, 207;  149, 448, 464, 379;  59, 55, 499, 64;  171, 84, 182, 341;  227, 338, 282, 489;  478, 100, 49, 129;  213, 85, 395, 110;  0, 1, 124, 530;  0, 5, 270, 277;  0, 9, 96, 534;  0, 12, 43, 404;  0, 20, 114, 350;  0, 21, 408, 512;  0, 25, 266, 300;  0, 26, 134, 188;  0, 33, 202, 402;  0, 35, 290, 498;  0, 39, 100, 145;  0, 49, 63, 352;  0, 48, 174, 239;  0, 55, 118, 420;  0, 61, 288, 355;  0, 40, 111, 336;  0, 72, 171, 380;  0, 76, 160, 327;  0, 89, 238, 331;  0, 101, 276, 383;  0, 112, 229, 232;  0, 133, 155, 390;  0, 62, 219, 247;  0, 115, 215, 368;  0, 205, 234, 451;  0, 127, 237, 254;  0, 132, 413, 443;  0, 82, 279, 347;  0, 116, 481, 501;  0, 213, 225, 533;  0, 387, 423, 471;  0, 343, 477, 529;  0, 367, 417, 503;  0, 241, 337, 409;  0, 231, 297, 371;  0, 195, 369, 403;  0, 335, 351, 499;  0, 289, 353, 513;  1, 121, 253, 405;  1, 47, 223, 305;  1, 19, 43, 269;  1, 41, 119, 451;  1, 77, 179, 337;  1, 3, 95, 331;  1, 33, 155, 425;  0, 137, 274, 411;  1, 138, 275, 412}.
The last 2 blocks represent short orbits.
The deficiency graph is connected and has girth 6.

{\boldmath $\adfPENT(4,185)$}, $d = 2$:
{1, 2, 11, 558;  0, 5, 550, 557;  212, 327, 166, 384;  32, 310, 476, 516;  162, 104, 179, 75;  124, 158, 223, 256;  34, 88, 16, 525;  232, 526, 52, 15;  202, 179, 106, 406;  431, 360, 381, 160;  129, 91, 298, 379;  189, 14, 94, 313;  516, 475, 61, 511;  68, 218, 536, 473;  26, 235, 180, 34;  420, 317, 345, 133;  22, 455, 311, 381;  296, 422, 383, 341;  159, 112, 306, 97;  542, 56, 363, 464;  165, 46, 308, 277;  97, 118, 429, 3;  210, 537, 190, 25;  48, 371, 388, 501;  435, 188, 250, 335;  405, 131, 423, 249;  492, 270, 331, 549;  342, 151, 390, 549;  376, 153, 246, 73;  75, 133, 411, 377;  253, 523, 375, 61;  539, 4, 227, 279;  404, 193, 299, 345;  253, 435, 17, 200;  107, 56, 531, 373;  97, 502, 25, 144;  293, 39, 295, 226;  253, 313, 491, 269;  403, 53, 397, 230;  201, 33, 52, 115;  547, 71, 318, 154;  559, 39, 360, 483;  359, 381, 448, 132;  526, 350, 22, 99;  531, 367, 103, 301;  434, 425, 85, 361;  325, 478, 42, 237;  0, 39, 281, 470;  0, 6, 41, 97;  0, 12, 43, 511;  0, 14, 107, 203;  0, 16, 143, 345;  0, 19, 168, 191;  0, 22, 125, 215;  0, 24, 163, 241;  0, 25, 230, 259;  0, 26, 59, 263;  0, 28, 205, 325;  0, 27, 38, 461;  0, 30, 243, 451;  0, 32, 335, 389;  0, 36, 269, 301;  0, 37, 44, 197;  0, 42, 293, 481;  0, 52, 277, 415;  0, 50, 447, 515;  0, 64, 425, 533;  0, 79, 181, 456;  0, 88, 403, 517;  0, 117, 245, 310;  0, 187, 419, 449;  1, 27, 207, 361;  0, 157, 169, 409;  0, 111, 261, 332;  1, 25, 195, 371;  0, 141, 238, 505;  0, 121, 135, 390;  0, 145, 165, 331;  0, 66, 431, 466;  0, 106, 264, 401;  0, 101, 178, 320;  0, 82, 171, 352;  0, 179, 192, 312;  0, 131, 184, 306;  0, 136, 287, 404;  0, 109, 234, 416;  0, 108, 226, 412;  0, 68, 196, 370;  0, 100, 210, 372;  0, 84, 198, 336;  0, 60, 274, 344;  0, 102, 236, 348;  0, 140, 280, 420;  1, 141, 281, 421}.
The last 2 blocks represent short orbits.
The deficiency graph is connected and has girth 5.
Identifying code: $\{x + 20i: x \in \{0,2,8,9,11,13,14,15\}, 0 \le i < 28\}$,
$\pi_0 = 0$, $\pi_1 = 224$, $\pi_2 = 336$.

{\boldmath $\adfPENT(4,189)$}, $d = 2$:
{94, 233, 267, 478;  189, 306, 340, 385;  111, 35, 64, 302;  448, 129, 98, 337;  401, 553, 194, 230;  71, 215, 325, 94;  121, 185, 83, 236;  441, 361, 272, 461;  474, 287, 472, 350;  235, 357, 464, 45;  377, 63, 320, 105;  212, 262, 550, 169;  276, 199, 96, 427;  171, 527, 333, 87;  436, 161, 213, 147;  82, 214, 495, 196;  512, 491, 85, 4;  41, 180, 483, 15;  342, 195, 170, 94;  392, 224, 538, 359;  108, 147, 260, 214;  159, 345, 451, 317;  77, 479, 42, 439;  478, 3, 277, 360;  463, 184, 447, 270;  385, 148, 315, 271;  4, 14, 214, 505;  5, 396, 301, 97;  354, 50, 387, 233;  547, 202, 245, 452;  161, 63, 120, 329;  302, 94, 88, 299;  194, 141, 180, 201;  391, 393, 380, 148;  235, 110, 173, 334;  75, 535, 418, 486;  183, 239, 231, 364;  352, 15, 437, 116;  386, 31, 530, 401;  463, 90, 159, 571;  481, 204, 370, 390;  219, 350, 506, 422;  537, 549, 38, 357;  157, 266, 212, 154;  304, 505, 392, 12;  330, 30, 362, 166;  99, 472, 94, 407;  184, 0, 165, 360;  221, 410, 358, 463;  0, 1, 241, 564;  0, 4, 87, 409;  0, 12, 149, 367;  0, 16, 115, 287;  0, 17, 22, 559;  0, 19, 29, 255;  0, 23, 141, 257;  0, 24, 187, 399;  0, 26, 133, 415;  0, 27, 77, 203;  0, 28, 65, 213;  0, 30, 431, 571;  0, 38, 157, 379;  0, 40, 309, 403;  0, 42, 109, 353;  0, 51, 55, 445;  0, 44, 175, 265;  0, 59, 485, 557;  0, 56, 525, 531;  0, 61, 259, 317;  0, 71, 159, 295;  0, 90, 219, 513;  0, 75, 468, 561;  0, 223, 547, 565;  1, 97, 221, 359;  0, 339, 523, 545;  1, 47, 129, 285;  0, 275, 417, 449;  0, 247, 301, 325;  0, 147, 347, 467;  0, 127, 195, 535;  0, 225, 261, 270;  0, 110, 307, 452;  0, 60, 395, 470;  0, 62, 351, 368;  0, 78, 329, 398;  0, 150, 328, 483;  0, 66, 293, 412;  0, 96, 254, 370;  0, 80, 206, 434;  0, 98, 246, 374;  0, 108, 262, 402;  0, 70, 260, 352;  0, 74, 216, 316;  0, 48, 130, 438;  0, 143, 286, 429}.
The last block represents a short orbit.
The deficiency graph is connected and has girth 5.

{\boldmath $\adfPENT(4,193)$}, $d = 2$:
{178, 375, 406, 513;  72, 137, 210, 449;  171, 219, 490, 371;  341, 235, 7, 323;  118, 64, 51, 392;  313, 403, 406, 97;  227, 558, 359, 123;  384, 419, 385, 380;  273, 236, 14, 155;  196, 301, 118, 318;  131, 270, 455, 133;  313, 26, 110, 515;  177, 51, 505, 422;  322, 362, 168, 23;  279, 412, 129, 151;  387, 456, 141, 328;  285, 256, 390, 188;  194, 371, 53, 227;  44, 359, 403, 323;  90, 283, 162, 357;  504, 568, 127, 134;  417, 471, 429, 233;  233, 392, 507, 163;  333, 86, 398, 356;  350, 46, 498, 483;  358, 166, 579, 450;  171, 374, 511, 347;  2, 552, 329, 429;  567, 3, 173, 304;  132, 320, 195, 296;  13, 321, 317, 255;  107, 234, 304, 577;  304, 465, 493, 446;  191, 360, 565, 556;  169, 350, 237, 226;  22, 252, 208, 173;  57, 250, 580, 44;  173, 457, 166, 326;  61, 316, 106, 201;  581, 231, 408, 272;  165, 41, 328, 348;  35, 82, 266, 552;  178, 135, 297, 495;  315, 334, 274, 516;  69, 498, 165, 203;  441, 215, 478, 61;  252, 577, 77, 424;  324, 427, 129, 215;  215, 535, 4, 567;  305, 397, 426, 148;  563, 518, 469, 306;  0, 2, 23, 487;  0, 3, 8, 101;  0, 6, 31, 363;  0, 10, 175, 191;  0, 12, 81, 91;  0, 14, 87, 345;  0, 15, 223, 546;  0, 16, 367, 523;  0, 17, 77, 235;  0, 18, 75, 187;  0, 22, 201, 231;  0, 26, 225, 281;  0, 27, 85, 501;  0, 28, 347, 399;  0, 30, 525, 575;  0, 49, 89, 113;  0, 90, 229, 583;  0, 36, 159, 341;  0, 43, 373, 528;  0, 71, 322, 467;  0, 55, 303, 349;  0, 62, 369, 441;  0, 51, 217, 365;  0, 109, 258, 551;  1, 7, 179, 371;  0, 227, 433, 509;  1, 15, 117, 405;  1, 27, 109, 187;  0, 116, 411, 533;  0, 117, 125, 180;  0, 111, 240, 473;  0, 127, 234, 316;  0, 108, 248, 419;  0, 80, 215, 376;  0, 76, 302, 573;  0, 50, 162, 499;  0, 88, 334, 481;  0, 58, 299, 324;  0, 52, 126, 478;  0, 100, 252, 408;  0, 94, 204, 440;  0, 32, 118, 486;  0, 96, 198, 416;  0, 46, 166, 410;  0, 66, 170, 360;  0, 146, 292, 438;  1, 147, 293, 439}.
The last 2 blocks represent short orbits.
The deficiency graph is connected and has girth 6.

{\boldmath $\adfPENT(4,197)$}, $d = 2$:
{20, 545, 576, 587;  10, 52, 265, 333;  390, 203, 301, 265;  105, 139, 285, 368;  299, 345, 543, 273;  515, 4, 347, 151;  2, 66, 181, 318;  494, 247, 268, 254;  194, 433, 112, 285;  12, 253, 410, 586;  442, 77, 483, 191;  81, 3, 352, 440;  505, 255, 397, 131;  178, 445, 278, 520;  537, 247, 368, 440;  479, 477, 91, 502;  163, 28, 61, 347;  216, 114, 121, 106;  217, 144, 255, 272;  291, 510, 383, 482;  101, 170, 477, 509;  61, 374, 505, 506;  210, 317, 299, 488;  52, 144, 471, 327;  98, 393, 397, 176;  282, 497, 344, 564;  497, 352, 491, 432;  264, 569, 307, 470;  591, 55, 92, 159;  110, 512, 106, 49;  361, 381, 317, 113;  62, 50, 431, 143;  160, 367, 87, 359;  362, 202, 564, 19;  355, 576, 445, 46;  101, 332, 436, 319;  64, 122, 367, 83;  454, 47, 192, 287;  445, 7, 84, 432;  506, 538, 166, 295;  103, 76, 327, 100;  230, 56, 236, 391;  536, 271, 553, 501;  286, 482, 87, 342;  45, 258, 344, 35;  323, 403, 294, 151;  83, 573, 1, 82;  299, 446, 485, 96;  154, 15, 365, 351;  487, 91, 550, 44;  247, 452, 564, 323;  481, 569, 453, 111;  203, 436, 376, 497;  0, 2, 85, 259;  0, 5, 21, 483;  0, 9, 570, 593;  0, 10, 63, 175;  0, 16, 103, 285;  0, 18, 49, 505;  0, 25, 119, 552;  0, 30, 347, 401;  0, 34, 421, 495;  0, 37, 77, 271;  0, 38, 181, 395;  0, 36, 351, 481;  0, 45, 94, 531;  0, 46, 101, 233;  0, 57, 493, 526;  0, 71, 193, 249;  0, 54, 485, 543;  0, 76, 311, 553;  0, 75, 195, 291;  0, 68, 209, 509;  0, 79, 386, 591;  0, 125, 555, 585;  1, 127, 277, 405;  1, 51, 207, 435;  0, 157, 184, 577;  0, 243, 265, 313;  1, 67, 237, 323;  0, 117, 162, 455;  0, 123, 277, 292;  0, 188, 417, 517;  0, 124, 301, 503;  0, 191, 324, 549;  0, 106, 219, 380;  0, 114, 250, 473;  0, 122, 264, 537;  0, 74, 306, 411;  0, 50, 208, 360;  0, 130, 268, 424;  0, 52, 168, 426;  0, 118, 244, 448;  0, 98, 276, 410;  0, 108, 228, 378;  0, 48, 214, 414;  0, 84, 238, 384;  0, 96, 260, 404;  0, 149, 298, 447}.
The last block represents a short orbit.
The deficiency graph is connected and has girth 6.

\adfAppGap

{\boldmath $\adfPENT(4,57,28)$}, $d = 40$:
{0, 1, 2, 3;  0, 4, 92, 192;  0, 5, 36, 135;  0, 6, 134, 193;  0, 7, 93, 133;  0, 28, 39, 95;  0, 29, 38, 132;  0, 30, 37, 195;  0, 31, 94, 194;  0, 32, 128, 164;  0, 33, 108, 159;  0, 34, 158, 165;  0, 35, 129, 157;  0, 88, 111, 131;  0, 89, 110, 156;  0, 90, 109, 167;  0, 91, 130, 166;  1, 4, 30, 36;  1, 5, 39, 94;  1, 6, 38, 192;  1, 7, 95, 134;  1, 28, 92, 133;  1, 29, 135, 195;  1, 31, 37, 193;  1, 32, 90, 108;  1, 33, 111, 130;  1, 34, 110, 164;  1, 35, 131, 158;  1, 88, 128, 157;  1, 89, 159, 167;  1, 91, 109, 165;  1, 93, 132, 194;  1, 129, 156, 166;  2, 4, 28, 37;  2, 5, 29, 133;  2, 6, 31, 95;  2, 7, 135, 193;  2, 30, 39, 134;  2, 32, 88, 109;  2, 33, 89, 157;  2, 34, 91, 131;  2, 35, 159, 165;  2, 36, 92, 194;  2, 38, 93, 195;  2, 90, 111, 158;  2, 94, 132, 192;  2, 108, 128, 166;  2, 110, 129, 167;  2, 130, 156, 164;  3, 4, 38, 133;  3, 5, 31, 195;  3, 6, 36, 94;  3, 7, 28, 192;  3, 29, 92, 193;  3, 30, 95, 132;  3, 32, 110, 157;  3, 33, 91, 167;  3, 34, 108, 130;  3, 35, 88, 164;  3, 37, 93, 134;  3, 39, 135, 194;  3, 89, 128, 165;  3, 90, 131, 156;  3, 109, 129, 158;  3, 111, 159, 166;  4, 5, 6, 7;  4, 29, 95, 194;  4, 31, 93, 135;  4, 32, 120, 188;  4, 33, 116, 143;  4, 34, 142, 189;  4, 35, 121, 141;  4, 39, 132, 193;  4, 68, 144, 156;  4, 69, 104, 155;  4, 70, 154, 157;  4, 71, 145, 153;  4, 76, 112, 160;  4, 77, 108, 139;  4, 78, 138, 161;  4, 79, 113, 137;  4, 80, 107, 147;  4, 81, 106, 152;  4, 82, 105, 159;  4, 83, 146, 158;  4, 94, 134, 195;  4, 96, 111, 115;  4, 97, 110, 136;  4, 98, 109, 163;  4, 99, 114, 162;  4, 100, 119, 123;  4, 101, 118, 140;  4, 102, 117, 191;  4, 103, 122, 190;  5, 28, 134, 194;  5, 30, 93, 192;  5, 32, 102, 116;  5, 33, 119, 122;  5, 34, 118, 188;  5, 35, 123, 142;  5, 37, 92, 132;  5, 38, 95, 193;  5, 68, 82, 104;  5, 69, 107, 146;  5, 70, 106, 156;  5, 71, 147, 154;  5, 76, 98, 108;  5, 77, 111, 114;  5, 78, 110, 160;  5, 79, 115, 138;  5, 80, 144, 153;  5, 81, 155, 159;  5, 83, 105, 157;  5, 96, 112, 137;  5, 97, 139, 163;  5, 99, 109, 161;  5, 100, 120, 141;  5, 101, 143, 191;  5, 103, 117, 189;  5, 113, 136, 162;  5, 121, 140, 190;  5, 145, 152, 158;  6, 28, 132, 195;  6, 29, 39, 93;  6, 30, 92, 135;  6, 32, 100, 117;  6, 33, 101, 141;  6, 34, 103, 123;  6, 35, 143, 189;  6, 37, 133, 194;  6, 68, 80, 105;  6, 69, 81, 153;  6, 70, 83, 147;  6, 71, 155, 157;  6, 76, 96, 109;  6, 77, 97, 137;  6, 78, 99, 115;  6, 79, 139, 161;  6, 82, 107, 154;  6, 98, 111, 138;  6, 102, 119, 142;  6, 104, 144, 158;  6, 106, 145, 159;  6, 108, 112, 162;  6, 110, 113, 163;  6, 114, 136, 160;  6, 116, 120, 190;  6, 118, 121, 191;  6, 122, 140, 188;  6, 146, 152, 156;  7, 29, 37, 94;  7, 30, 38, 194;  7, 31, 36, 132;  7, 32, 118, 141;  7, 33, 103, 191;  7, 34, 116, 122;  7, 35, 100, 188;  7, 39, 92, 195;  7, 68, 106, 153;  7, 69, 83, 159;  7, 70, 104, 146;  7, 71, 80, 156;  7, 76, 110, 137;  7, 77, 99, 163;  7, 78, 108, 114;  7, 79, 96, 160;  7, 81, 144, 157;  7, 82, 147, 152;  7, 97, 112, 161;  7, 98, 115, 136;  7, 101, 120, 189;  7, 102, 123, 140;  7, 105, 145, 154;  7, 107, 155, 158;  7, 109, 113, 138;  7, 111, 139, 162;  7, 117, 121, 142;  7, 119, 143, 190;  8, 9, 10, 11;  8, 12, 100, 160;  8, 13, 56, 147;  8, 14, 146, 161;  8, 15, 101, 145;  8, 20, 104, 192;  8, 21, 72, 139;  8, 22, 138, 193;  8, 23, 105, 137;  8, 40, 59, 103;  8, 41, 58, 144;  8, 42, 57, 163;  8, 43, 102, 162;  8, 52, 75, 107;  8, 53, 74, 136;  8, 54, 73, 195;  8, 55, 106, 194;  9, 12, 42, 56;  9, 13, 59, 102;  9, 14, 58, 160;  9, 15, 103, 146;  9, 20, 54, 72;  9, 21, 75, 106;  9, 22, 74, 192;  9, 23, 107, 138;  9, 40, 100, 145;  9, 41, 147, 163;  9, 43, 57, 161;  9, 52, 104, 137;  9, 53, 139, 195;  9, 55, 73, 193;  9, 101, 144, 162;  9, 105, 136, 194;  10, 12, 40, 57;  10, 13, 41, 145;  10, 14, 43, 103;  10, 15, 147, 161;  10, 20, 52, 73;  10, 21, 53, 137;  10, 22, 55, 107;  10, 23, 139, 193;  10, 42, 59, 146;  10, 54, 75, 138;  10, 56, 100, 162;  10, 58, 101, 163;  10, 72, 104, 194;  10, 74, 105, 195;  10, 102, 144, 160;  10, 106, 136, 192;  11, 12, 58, 145;  11, 13, 43, 163;  11, 14, 56, 102;  11, 15, 40, 160;  11, 20, 74, 137;  11, 21, 55, 195;  11, 22, 72, 106;  11, 23, 52, 192;  11, 41, 100, 161;  11, 42, 103, 144;  11, 53, 104, 193;  11, 54, 107, 136;  11, 57, 101, 146;  11, 59, 147, 162;  11, 73, 105, 138;  11, 75, 139, 194;  12, 13, 14, 15;  12, 16, 44, 164;  12, 17, 24, 51;  12, 18, 50, 165;  12, 19, 45, 49;  12, 20, 27, 47;  12, 21, 26, 48;  12, 22, 25, 167;  12, 23, 46, 166;  12, 41, 103, 162;  12, 43, 101, 147;  12, 59, 144, 161;  12, 60, 136, 188;  12, 61, 128, 187;  12, 62, 186, 189;  12, 63, 137, 185;  12, 68, 131, 139;  12, 69, 130, 184;  12, 70, 129, 191;  12, 71, 138, 190;  12, 76, 140, 196;  12, 77, 104, 179;  12, 78, 178, 197;  12, 79, 141, 177;  12, 88, 107, 143;  12, 89, 106, 176;  12, 90, 105, 199;  12, 91, 142, 198;  12, 102, 146, 163;  13, 16, 22, 24;  13, 17, 27, 46;  13, 18, 26, 164;  13, 19, 47, 50;  13, 20, 44, 49;  13, 21, 51, 167;  13, 23, 25, 165;  13, 40, 146, 162;  13, 42, 101, 160;  13, 45, 48, 166;  13, 57, 100, 144;  13, 58, 103, 161;  13, 60, 70, 128;  13, 61, 131, 138;  13, 62, 130, 188;  13, 63, 139, 186;  13, 68, 136, 185;  13, 69, 187, 191;  13, 71, 129, 189;  13, 76, 90, 104;  13, 77, 107, 142;  13, 78, 106, 196;  13, 79, 143, 178;  13, 88, 140, 177;  13, 89, 179, 199;  13, 91, 105, 197;  13, 137, 184, 190;  13, 141, 176, 198;  14, 16, 20, 25;  14, 17, 21, 49;  14, 18, 23, 47;  14, 19, 51, 165;  14, 22, 27, 50;  14, 24, 44, 166;  14, 26, 45, 167;  14, 40, 144, 163;  14, 41, 59, 101;  14, 42, 100, 147;  14, 46, 48, 164;  14, 57, 145, 162;  14, 60, 68, 129;  14, 61, 69, 185;  14, 62, 71, 139;  14, 63, 187, 189;  14, 70, 131, 186;  14, 76, 88, 105;  14, 77, 89, 177;  14, 78, 91, 143;  14, 79, 179, 197;  14, 90, 107, 178;  14, 104, 140, 198;  14, 106, 141, 199;  14, 128, 136, 190;  14, 130, 137, 191;  14, 138, 184, 188;  14, 142, 176, 196;  15, 16, 26, 49;  15, 17, 23, 167;  15, 18, 24, 46;  15, 19, 20, 164;  15, 21, 44, 165;  15, 22, 47, 48;  15, 25, 45, 50;  15, 27, 51, 166;  15, 41, 57, 102;  15, 42, 58, 162;  15, 43, 56, 144;  15, 59, 100, 163;  15, 60, 130, 185;  15, 61, 71, 191;  15, 62, 128, 138;  15, 63, 68, 188;  15, 69, 136, 189;  15, 70, 139, 184;  15, 76, 106, 177;  15, 77, 91, 199;  15, 78, 104, 142;  15, 79, 88, 196;  15, 89, 140, 197;  15, 90, 143, 176;  15, 105, 141, 178;  15, 107, 179, 198;  15, 129, 137, 186;  15, 131, 187, 190;  16, 17, 18, 19;  16, 21, 47, 166;  16, 23, 45, 51;  16, 27, 48, 165;  16, 46, 50, 167;  17, 20, 50, 166;  17, 22, 45, 164;  17, 25, 44, 48;  17, 26, 47, 165;  18, 20, 48, 167;  18, 21, 27, 45;  18, 22, 44, 51;  18, 25, 49, 166;  19, 21, 25, 46;  19, 22, 26, 166;  19, 23, 24, 48;  19, 27, 44, 167;  20, 21, 22, 23;  20, 24, 45, 165;  20, 26, 46, 51;  20, 53, 107, 194;  20, 55, 105, 139;  20, 75, 136, 193;  20, 106, 138, 195;  21, 24, 50, 164;  21, 52, 138, 194;  21, 54, 105, 192;  21, 73, 104, 136;  21, 74, 107, 193;  22, 46, 49, 165;  22, 52, 136, 195;  22, 53, 75, 105;  22, 54, 104, 139;  22, 73, 137, 194;  23, 26, 44, 50;  23, 27, 49, 164;  23, 53, 73, 106;  23, 54, 74, 194;  23, 55, 72, 136;  23, 75, 104, 195;  24, 25, 26, 27;  24, 47, 49, 167;  25, 47, 51, 164;  28, 29, 30, 31;  28, 36, 93, 193;  28, 38, 94, 135;  29, 36, 134, 192;  30, 94, 133, 193;  31, 38, 92, 134;  31, 39, 133, 192;  32, 33, 34, 35;  32, 89, 131, 166;  32, 91, 129, 159;  32, 101, 123, 190;  32, 103, 121, 143;  32, 111, 156, 165;  32, 119, 140, 189;  32, 122, 142, 191;  32, 130, 158, 167;  33, 88, 158, 166;  33, 90, 129, 164;  33, 100, 142, 190;  33, 102, 121, 188;  33, 109, 128, 156;  33, 110, 131, 165;  33, 117, 120, 140;  33, 118, 123, 189;  34, 88, 156, 167;  34, 89, 111, 129;  34, 90, 128, 159;  34, 100, 140, 191;  34, 101, 119, 121;  34, 102, 120, 143;  34, 109, 157, 166;  34, 117, 141, 190;  35, 89, 109, 130;  35, 90, 110, 166;  35, 91, 108, 156;  35, 101, 117, 122;  35, 102, 118, 190;  35, 103, 116, 140;  35, 111, 128, 167;  35, 119, 120, 191;  36, 37, 38, 39;  36, 95, 133, 195;  37, 95, 135, 192;  40, 56, 101, 161;  40, 58, 102, 147;  41, 56, 146, 160;  42, 102, 145, 161;  43, 58, 100, 146;  43, 59, 145, 160;  52, 72, 105, 193;  52, 74, 106, 139;  53, 72, 138, 192;  54, 106, 137, 193;  55, 74, 104, 138;  55, 75, 137, 192;  56, 103, 145, 163;  57, 103, 147, 160;  60, 69, 139, 190;  60, 71, 137, 187;  60, 131, 184, 189;  60, 138, 186, 191;  61, 68, 186, 190;  61, 70, 137, 188;  61, 129, 136, 184;  61, 130, 139, 189;  62, 68, 184, 191;  62, 69, 131, 137;  62, 70, 136, 187;  62, 129, 185, 190;  63, 69, 129, 138;  63, 70, 130, 190;  63, 71, 128, 184;  63, 131, 136, 191;  68, 81, 147, 158;  68, 83, 145, 155;  68, 107, 152, 157;  68, 128, 137, 189;  68, 130, 138, 187;  68, 146, 154, 159;  69, 80, 154, 158;  69, 82, 145, 156;  69, 105, 144, 152;  69, 106, 147, 157;  69, 128, 186, 188;  70, 80, 152, 159;  70, 81, 107, 145;  70, 82, 144, 155;  70, 105, 153, 158;  70, 138, 185, 189;  71, 81, 105, 146;  71, 82, 106, 158;  71, 83, 104, 152;  71, 107, 144, 159;  71, 130, 136, 186;  71, 131, 185, 188;  72, 107, 137, 195;  73, 107, 139, 192;  76, 89, 143, 198;  76, 91, 141, 179;  76, 97, 115, 162;  76, 99, 113, 139;  76, 107, 176, 197;  76, 111, 136, 161;  76, 114, 138, 163;  76, 142, 178, 199;  77, 88, 178, 198;  77, 90, 141, 196;  77, 96, 138, 162;  77, 98, 113, 160;  77, 105, 140, 176;  77, 106, 143, 197;  77, 109, 112, 136;  77, 110, 115, 161;  78, 88, 176, 199;  78, 89, 107, 141;  78, 90, 140, 179;  78, 96, 136, 163;  78, 97, 111, 113;  78, 98, 112, 139;  78, 105, 177, 198;  78, 109, 137, 162;  79, 89, 105, 142;  79, 90, 106, 198;  79, 91, 104, 176;  79, 97, 109, 114;  79, 98, 110, 162;  79, 99, 108, 136;  79, 107, 140, 199;  79, 111, 112, 163;  80, 104, 145, 157;  80, 106, 146, 155;  81, 104, 154, 156;  82, 146, 153, 157;  83, 106, 144, 154;  83, 107, 153, 156;  88, 104, 141, 197;  88, 106, 142, 179;  88, 108, 129, 165;  88, 110, 130, 159;  89, 104, 178, 196;  89, 108, 158, 164;  90, 130, 157, 165;  90, 142, 177, 197;  91, 106, 140, 178;  91, 107, 177, 196;  91, 110, 128, 158;  91, 111, 157, 164;  96, 108, 113, 161;  96, 110, 114, 139;  97, 108, 138, 160;  98, 114, 137, 161;  99, 110, 112, 138;  99, 111, 137, 160;  100, 116, 121, 189;  100, 118, 122, 143;  101, 116, 142, 188;  102, 122, 141, 189;  103, 118, 120, 142;  103, 119, 141, 188;  104, 143, 177, 199;  104, 147, 153, 159;  105, 143, 179, 196;  105, 147, 155, 156;  108, 115, 137, 163;  108, 131, 157, 167;  109, 115, 139, 160;  109, 131, 159, 164;  116, 123, 141, 191;  117, 123, 143, 188;  128, 139, 185, 191;  129, 139, 187, 188}.
The deficiency graph is connected and has girth 4.

\adfAppGap

{\boldmath $\adfPENT(7,50,49)$}, $d = 70$:
{0, 1, 2, 3, 4, 5, 6;  0, 7, 50, 66, 167, 234, 337;  0, 8, 52, 69, 164, 232, 336;  0, 9, 54, 65, 161, 237, 342;  0, 10, 49, 68, 165, 235, 341;  0, 11, 51, 64, 162, 233, 340;  0, 12, 53, 67, 166, 231, 339;  0, 13, 55, 63, 163, 236, 338;  0, 56, 155, 192, 230, 276, 288;  0, 57, 157, 195, 227, 274, 287;  0, 58, 159, 191, 224, 279, 293;  0, 59, 154, 194, 228, 277, 292;  0, 60, 156, 190, 225, 275, 291;  0, 61, 158, 193, 229, 273, 290;  0, 62, 160, 189, 226, 278, 289;  1, 7, 49, 64, 164, 237, 339;  1, 8, 51, 67, 161, 235, 338;  1, 9, 53, 63, 165, 233, 337;  1, 10, 55, 66, 162, 231, 336;  1, 11, 50, 69, 166, 236, 342;  1, 12, 52, 65, 163, 234, 341;  1, 13, 54, 68, 167, 232, 340;  1, 56, 154, 190, 227, 279, 290;  1, 57, 156, 193, 224, 277, 289;  1, 58, 158, 189, 228, 275, 288;  1, 59, 160, 192, 225, 273, 287;  1, 60, 155, 195, 229, 278, 293;  1, 61, 157, 191, 226, 276, 292;  1, 62, 159, 194, 230, 274, 291;  2, 7, 55, 69, 161, 233, 341;  2, 8, 50, 65, 165, 231, 340;  2, 9, 52, 68, 162, 236, 339;  2, 10, 54, 64, 166, 234, 338;  2, 11, 49, 67, 163, 232, 337;  2, 12, 51, 63, 167, 237, 336;  2, 13, 53, 66, 164, 235, 342;  2, 56, 160, 195, 224, 275, 292;  2, 57, 155, 191, 228, 273, 291;  2, 58, 157, 194, 225, 278, 290;  2, 59, 159, 190, 229, 276, 289;  2, 60, 154, 193, 226, 274, 288;  2, 61, 156, 189, 230, 279, 287;  2, 62, 158, 192, 227, 277, 293;  3, 7, 54, 67, 165, 236, 336;  3, 8, 49, 63, 162, 234, 342;  3, 9, 51, 66, 166, 232, 341;  3, 10, 53, 69, 163, 237, 340;  3, 11, 55, 65, 167, 235, 339;  3, 12, 50, 68, 164, 233, 338;  3, 13, 52, 64, 161, 231, 337;  3, 56, 159, 193, 228, 278, 287;  3, 57, 154, 189, 225, 276, 293;  3, 58, 156, 192, 229, 274, 292;  3, 59, 158, 195, 226, 279, 291;  3, 60, 160, 191, 230, 277, 290;  3, 61, 155, 194, 227, 275, 289;  3, 62, 157, 190, 224, 273, 288;  4, 7, 53, 65, 162, 232, 338;  4, 8, 55, 68, 166, 237, 337;  4, 9, 50, 64, 163, 235, 336;  4, 10, 52, 67, 167, 233, 342;  4, 11, 54, 63, 164, 231, 341;  4, 12, 49, 66, 161, 236, 340;  4, 13, 51, 69, 165, 234, 339;  4, 56, 158, 191, 225, 274, 289;  4, 57, 160, 194, 229, 279, 288;  4, 58, 155, 190, 226, 277, 287;  4, 59, 157, 193, 230, 275, 293;  4, 60, 159, 189, 227, 273, 292;  4, 61, 154, 192, 224, 278, 291;  4, 62, 156, 195, 228, 276, 290;  5, 7, 52, 63, 166, 235, 340;  5, 8, 54, 66, 163, 233, 339;  5, 9, 49, 69, 167, 231, 338;  5, 10, 51, 65, 164, 236, 337;  5, 11, 53, 68, 161, 234, 336;  5, 12, 55, 64, 165, 232, 342;  5, 13, 50, 67, 162, 237, 341;  5, 56, 157, 189, 229, 277, 291;  5, 57, 159, 192, 226, 275, 290;  5, 58, 154, 195, 230, 273, 289;  5, 59, 156, 191, 227, 278, 288;  5, 60, 158, 194, 224, 276, 287;  5, 61, 160, 190, 228, 274, 293;  5, 62, 155, 193, 225, 279, 292;  6, 7, 51, 68, 163, 231, 342;  6, 8, 53, 64, 167, 236, 341;  6, 9, 55, 67, 164, 234, 340;  6, 10, 50, 63, 161, 232, 339;  6, 11, 52, 66, 165, 237, 338;  6, 12, 54, 69, 162, 235, 337;  6, 13, 49, 65, 166, 233, 336;  6, 56, 156, 194, 226, 273, 293;  6, 57, 158, 190, 230, 278, 292;  6, 58, 160, 193, 227, 276, 291;  6, 59, 155, 189, 224, 274, 290;  6, 60, 157, 192, 228, 279, 289;  6, 61, 159, 195, 225, 277, 288;  6, 62, 154, 191, 229, 275, 287;  7, 8, 9, 10, 11, 12, 13;  7, 56, 176, 206, 216, 248, 330;  7, 57, 178, 209, 213, 246, 329;  7, 58, 180, 205, 210, 251, 335;  7, 59, 175, 208, 214, 249, 334;  7, 60, 177, 204, 211, 247, 333;  7, 61, 179, 207, 215, 245, 332;  7, 62, 181, 203, 212, 250, 331;  7, 119, 141, 185, 258, 269, 274;  7, 120, 143, 188, 255, 267, 273;  7, 121, 145, 184, 252, 272, 279;  7, 122, 140, 187, 256, 270, 278;  7, 123, 142, 183, 253, 268, 277;  7, 124, 144, 186, 257, 266, 276;  7, 125, 146, 182, 254, 271, 275;  7, 133, 169, 192, 202, 241, 281;  7, 134, 171, 195, 199, 239, 280;  7, 135, 173, 191, 196, 244, 286;  7, 136, 168, 194, 200, 242, 285;  7, 137, 170, 190, 197, 240, 284;  7, 138, 172, 193, 201, 238, 283;  7, 139, 174, 189, 198, 243, 282;  8, 56, 175, 204, 213, 251, 332;  8, 57, 177, 207, 210, 249, 331;  8, 58, 179, 203, 214, 247, 330;  8, 59, 181, 206, 211, 245, 329;  8, 60, 176, 209, 215, 250, 335;  8, 61, 178, 205, 212, 248, 334;  8, 62, 180, 208, 216, 246, 333;  8, 119, 140, 183, 255, 272, 276;  8, 120, 142, 186, 252, 270, 275;  8, 121, 144, 182, 256, 268, 274;  8, 122, 146, 185, 253, 266, 273;  8, 123, 141, 188, 257, 271, 279;  8, 124, 143, 184, 254, 269, 278;  8, 125, 145, 187, 258, 267, 277;  8, 133, 168, 190, 199, 244, 283;  8, 134, 170, 193, 196, 242, 282;  8, 135, 172, 189, 200, 240, 281;  8, 136, 174, 192, 197, 238, 280;  8, 137, 169, 195, 201, 243, 286;  8, 138, 171, 191, 198, 241, 285;  8, 139, 173, 194, 202, 239, 284;  9, 56, 181, 209, 210, 247, 334;  9, 57, 176, 205, 214, 245, 333;  9, 58, 178, 208, 211, 250, 332;  9, 59, 180, 204, 215, 248, 331;  9, 60, 175, 207, 212, 246, 330;  9, 61, 177, 203, 216, 251, 329;  9, 62, 179, 206, 213, 249, 335;  9, 119, 146, 188, 252, 268, 278;  9, 120, 141, 184, 256, 266, 277;  9, 121, 143, 187, 253, 271, 276;  9, 122, 145, 183, 257, 269, 275;  9, 123, 140, 186, 254, 267, 274;  9, 124, 142, 182, 258, 272, 273;  9, 125, 144, 185, 255, 270, 279;  9, 133, 174, 195, 196, 240, 285;  9, 134, 169, 191, 200, 238, 284;  9, 135, 171, 194, 197, 243, 283;  9, 136, 173, 190, 201, 241, 282;  9, 137, 168, 193, 198, 239, 281;  9, 138, 170, 189, 202, 244, 280;  9, 139, 172, 192, 199, 242, 286;  10, 56, 180, 207, 214, 250, 329;  10, 57, 175, 203, 211, 248, 335;  10, 58, 177, 206, 215, 246, 334;  10, 59, 179, 209, 212, 251, 333;  10, 60, 181, 205, 216, 249, 332;  10, 61, 176, 208, 213, 247, 331;  10, 62, 178, 204, 210, 245, 330;  10, 119, 145, 186, 256, 271, 273;  10, 120, 140, 182, 253, 269, 279;  10, 121, 142, 185, 257, 267, 278;  10, 122, 144, 188, 254, 272, 277;  10, 123, 146, 184, 258, 270, 276;  10, 124, 141, 187, 255, 268, 275;  10, 125, 143, 183, 252, 266, 274;  10, 133, 173, 193, 200, 243, 280;  10, 134, 168, 189, 197, 241, 286;  10, 135, 170, 192, 201, 239, 285;  10, 136, 172, 195, 198, 244, 284;  10, 137, 174, 191, 202, 242, 283;  10, 138, 169, 194, 199, 240, 282;  10, 139, 171, 190, 196, 238, 281;  11, 56, 179, 205, 211, 246, 331;  11, 57, 181, 208, 215, 251, 330;  11, 58, 176, 204, 212, 249, 329;  11, 59, 178, 207, 216, 247, 335;  11, 60, 180, 203, 213, 245, 334;  11, 61, 175, 206, 210, 250, 333;  11, 62, 177, 209, 214, 248, 332;  11, 119, 144, 184, 253, 267, 275;  11, 120, 146, 187, 257, 272, 274;  11, 121, 141, 183, 254, 270, 273;  11, 122, 143, 186, 258, 268, 279;  11, 123, 145, 182, 255, 266, 278;  11, 124, 140, 185, 252, 271, 277;  11, 125, 142, 188, 256, 269, 276;  11, 133, 172, 191, 197, 239, 282;  11, 134, 174, 194, 201, 244, 281;  11, 135, 169, 190, 198, 242, 280;  11, 136, 171, 193, 202, 240, 286;  11, 137, 173, 189, 199, 238, 285;  11, 138, 168, 192, 196, 243, 284;  11, 139, 170, 195, 200, 241, 283;  12, 56, 178, 203, 215, 249, 333;  12, 57, 180, 206, 212, 247, 332;  12, 58, 175, 209, 216, 245, 331;  12, 59, 177, 205, 213, 250, 330;  12, 60, 179, 208, 210, 248, 329;  12, 61, 181, 204, 214, 246, 335;  12, 62, 176, 207, 211, 251, 334;  12, 119, 143, 182, 257, 270, 277;  12, 120, 145, 185, 254, 268, 276;  12, 121, 140, 188, 258, 266, 275;  12, 122, 142, 184, 255, 271, 274;  12, 123, 144, 187, 252, 269, 273;  12, 124, 146, 183, 256, 267, 279;  12, 125, 141, 186, 253, 272, 278;  12, 133, 171, 189, 201, 242, 284;  12, 134, 173, 192, 198, 240, 283;  12, 135, 168, 195, 202, 238, 282;  12, 136, 170, 191, 199, 243, 281;  12, 137, 172, 194, 196, 241, 280;  12, 138, 174, 190, 200, 239, 286;  12, 139, 169, 193, 197, 244, 285;  13, 56, 177, 208, 212, 245, 335;  13, 57, 179, 204, 216, 250, 334;  13, 58, 181, 207, 213, 248, 333;  13, 59, 176, 203, 210, 246, 332;  13, 60, 178, 206, 214, 251, 331;  13, 61, 180, 209, 211, 249, 330;  13, 62, 175, 205, 215, 247, 329;  13, 119, 142, 187, 254, 266, 279;  13, 120, 144, 183, 258, 271, 278;  13, 121, 146, 186, 255, 269, 277;  13, 122, 141, 182, 252, 267, 276;  13, 123, 143, 185, 256, 272, 275;  13, 124, 145, 188, 253, 270, 274;  13, 125, 140, 184, 257, 268, 273;  13, 133, 170, 194, 198, 238, 286;  13, 134, 172, 190, 202, 243, 285;  13, 135, 174, 193, 199, 241, 284;  13, 136, 169, 189, 196, 239, 283;  13, 137, 171, 192, 200, 244, 282;  13, 138, 173, 195, 197, 242, 281;  13, 139, 168, 191, 201, 240, 280;  14, 15, 16, 17, 18, 19, 20;  14, 21, 71, 101, 181, 255, 281;  14, 22, 73, 104, 178, 253, 280;  14, 23, 75, 100, 175, 258, 286;  14, 24, 70, 103, 179, 256, 285;  14, 25, 72, 99, 176, 254, 284;  14, 26, 74, 102, 180, 252, 283;  14, 27, 76, 98, 177, 257, 282;  14, 35, 92, 129, 188, 241, 337;  14, 36, 94, 132, 185, 239, 336;  14, 37, 96, 128, 182, 244, 342;  14, 38, 91, 131, 186, 242, 341;  14, 39, 93, 127, 183, 240, 340;  14, 40, 95, 130, 187, 238, 339;  14, 41, 97, 126, 184, 243, 338;  15, 21, 70, 99, 178, 258, 283;  15, 22, 72, 102, 175, 256, 282;  15, 23, 74, 98, 179, 254, 281;  15, 24, 76, 101, 176, 252, 280;  15, 25, 71, 104, 180, 257, 286;  15, 26, 73, 100, 177, 255, 285;  15, 27, 75, 103, 181, 253, 284;  15, 35, 91, 127, 185, 244, 339;  15, 36, 93, 130, 182, 242, 338;  15, 37, 95, 126, 186, 240, 337;  15, 38, 97, 129, 183, 238, 336;  15, 39, 92, 132, 187, 243, 342;  15, 40, 94, 128, 184, 241, 341;  15, 41, 96, 131, 188, 239, 340;  16, 21, 76, 104, 175, 254, 285;  16, 22, 71, 100, 179, 252, 284;  16, 23, 73, 103, 176, 257, 283;  16, 24, 75, 99, 180, 255, 282;  16, 25, 70, 102, 177, 253, 281;  16, 26, 72, 98, 181, 258, 280;  16, 27, 74, 101, 178, 256, 286;  16, 35, 97, 132, 182, 240, 341;  16, 36, 92, 128, 186, 238, 340;  16, 37, 94, 131, 183, 243, 339;  16, 38, 96, 127, 187, 241, 338;  16, 39, 91, 130, 184, 239, 337;  16, 40, 93, 126, 188, 244, 336;  16, 41, 95, 129, 185, 242, 342;  17, 21, 75, 102, 179, 257, 280;  17, 22, 70, 98, 176, 255, 286;  17, 23, 72, 101, 180, 253, 285;  17, 24, 74, 104, 177, 258, 284;  17, 25, 76, 100, 181, 256, 283;  17, 26, 71, 103, 178, 254, 282;  17, 27, 73, 99, 175, 252, 281;  17, 35, 96, 130, 186, 243, 336;  17, 36, 91, 126, 183, 241, 342;  17, 37, 93, 129, 187, 239, 341;  17, 38, 95, 132, 184, 244, 340;  17, 39, 97, 128, 188, 242, 339;  17, 40, 92, 131, 185, 240, 338;  17, 41, 94, 127, 182, 238, 337;  18, 21, 74, 100, 176, 253, 282;  18, 22, 76, 103, 180, 258, 281;  18, 23, 71, 99, 177, 256, 280;  18, 24, 73, 102, 181, 254, 286;  18, 25, 75, 98, 178, 252, 285;  18, 26, 70, 101, 175, 257, 284;  18, 27, 72, 104, 179, 255, 283;  18, 35, 95, 128, 183, 239, 338;  18, 36, 97, 131, 187, 244, 337;  18, 37, 92, 127, 184, 242, 336;  18, 38, 94, 130, 188, 240, 342;  18, 39, 96, 126, 185, 238, 341;  18, 40, 91, 129, 182, 243, 340;  18, 41, 93, 132, 186, 241, 339;  19, 21, 73, 98, 180, 256, 284;  19, 22, 75, 101, 177, 254, 283;  19, 23, 70, 104, 181, 252, 282;  19, 24, 72, 100, 178, 257, 281;  19, 25, 74, 103, 175, 255, 280;  19, 26, 76, 99, 179, 253, 286;  19, 27, 71, 102, 176, 258, 285;  19, 35, 94, 126, 187, 242, 340;  19, 36, 96, 129, 184, 240, 339;  19, 37, 91, 132, 188, 238, 338;  19, 38, 93, 128, 185, 243, 337;  19, 39, 95, 131, 182, 241, 336;  19, 40, 97, 127, 186, 239, 342;  19, 41, 92, 130, 183, 244, 341;  20, 21, 72, 103, 177, 252, 286;  20, 22, 74, 99, 181, 257, 285;  20, 23, 76, 102, 178, 255, 284;  20, 24, 71, 98, 175, 253, 283;  20, 25, 73, 101, 179, 258, 282;  20, 26, 75, 104, 176, 256, 281;  20, 27, 70, 100, 180, 254, 280;  20, 35, 93, 131, 184, 238, 342;  20, 36, 95, 127, 188, 243, 341;  20, 37, 97, 130, 185, 241, 340;  20, 38, 92, 126, 182, 239, 339;  20, 39, 94, 129, 186, 244, 338;  20, 40, 96, 132, 183, 242, 337;  20, 41, 91, 128, 187, 240, 336;  21, 22, 23, 24, 25, 26, 27;  21, 28, 36, 45, 83, 87, 288;  21, 29, 38, 48, 80, 85, 287;  21, 30, 40, 44, 77, 90, 293;  21, 31, 35, 47, 81, 88, 292;  21, 32, 37, 43, 78, 86, 291;  21, 33, 39, 46, 82, 84, 290;  21, 34, 41, 42, 79, 89, 289;  21, 105, 120, 227, 244, 325, 330;  21, 106, 122, 230, 241, 323, 329;  21, 107, 124, 226, 238, 328, 335;  21, 108, 119, 229, 242, 326, 334;  21, 109, 121, 225, 239, 324, 333;  21, 110, 123, 228, 243, 322, 332;  21, 111, 125, 224, 240, 327, 331;  21, 133, 155, 185, 251, 311, 344;  21, 134, 157, 188, 248, 309, 343;  21, 135, 159, 184, 245, 314, 349;  21, 136, 154, 187, 249, 312, 348;  21, 137, 156, 183, 246, 310, 347;  21, 138, 158, 186, 250, 308, 346;  21, 139, 160, 182, 247, 313, 345;  22, 28, 35, 43, 80, 90, 290;  22, 29, 37, 46, 77, 88, 289;  22, 30, 39, 42, 81, 86, 288;  22, 31, 41, 45, 78, 84, 287;  22, 32, 36, 48, 82, 89, 293;  22, 33, 38, 44, 79, 87, 292;  22, 34, 40, 47, 83, 85, 291;  22, 105, 119, 225, 241, 328, 332;  22, 106, 121, 228, 238, 326, 331;  22, 107, 123, 224, 242, 324, 330;  22, 108, 125, 227, 239, 322, 329;  22, 109, 120, 230, 243, 327, 335;  22, 110, 122, 226, 240, 325, 334;  22, 111, 124, 229, 244, 323, 333;  22, 133, 154, 183, 248, 314, 346;  22, 134, 156, 186, 245, 312, 345;  22, 135, 158, 182, 249, 310, 344;  22, 136, 160, 185, 246, 308, 343;  22, 137, 155, 188, 250, 313, 349;  22, 138, 157, 184, 247, 311, 348;  22, 139, 159, 187, 251, 309, 347;  23, 28, 41, 48, 77, 86, 292;  23, 29, 36, 44, 81, 84, 291;  23, 30, 38, 47, 78, 89, 290;  23, 31, 40, 43, 82, 87, 289;  23, 32, 35, 46, 79, 85, 288;  23, 33, 37, 42, 83, 90, 287;  23, 34, 39, 45, 80, 88, 293;  23, 105, 125, 230, 238, 324, 334;  23, 106, 120, 226, 242, 322, 333;  23, 107, 122, 229, 239, 327, 332;  23, 108, 124, 225, 243, 325, 331;  23, 109, 119, 228, 240, 323, 330;  23, 110, 121, 224, 244, 328, 329;  23, 111, 123, 227, 241, 326, 335;  23, 133, 160, 188, 245, 310, 348;  23, 134, 155, 184, 249, 308, 347;  23, 135, 157, 187, 246, 313, 346;  23, 136, 159, 183, 250, 311, 345;  23, 137, 154, 186, 247, 309, 344;  23, 138, 156, 182, 251, 314, 343;  23, 139, 158, 185, 248, 312, 349;  24, 28, 40, 46, 81, 89, 287;  24, 29, 35, 42, 78, 87, 293;  24, 30, 37, 45, 82, 85, 292;  24, 31, 39, 48, 79, 90, 291;  24, 32, 41, 44, 83, 88, 290;  24, 33, 36, 47, 80, 86, 289;  24, 34, 38, 43, 77, 84, 288;  24, 105, 124, 228, 242, 327, 329;  24, 106, 119, 224, 239, 325, 335;  24, 107, 121, 227, 243, 323, 334;  24, 108, 123, 230, 240, 328, 333;  24, 109, 125, 226, 244, 326, 332;  24, 110, 120, 229, 241, 324, 331;  24, 111, 122, 225, 238, 322, 330;  24, 133, 159, 186, 249, 313, 343;  24, 134, 154, 182, 246, 311, 349;  24, 135, 156, 185, 250, 309, 348;  24, 136, 158, 188, 247, 314, 347;  24, 137, 160, 184, 251, 312, 346;  24, 138, 155, 187, 248, 310, 345;  24, 139, 157, 183, 245, 308, 344;  25, 28, 39, 44, 78, 85, 289;  25, 29, 41, 47, 82, 90, 288;  25, 30, 36, 43, 79, 88, 287;  25, 31, 38, 46, 83, 86, 293;  25, 32, 40, 42, 80, 84, 292;  25, 33, 35, 45, 77, 89, 291;  25, 34, 37, 48, 81, 87, 290;  25, 105, 123, 226, 239, 323, 331;  25, 106, 125, 229, 243, 328, 330;  25, 107, 120, 225, 240, 326, 329;  25, 108, 122, 228, 244, 324, 335;  25, 109, 124, 224, 241, 322, 334;  25, 110, 119, 227, 238, 327, 333;  25, 111, 121, 230, 242, 325, 332;  25, 133, 158, 184, 246, 309, 345;  25, 134, 160, 187, 250, 314, 344;  25, 135, 155, 183, 247, 312, 343;  25, 136, 157, 186, 251, 310, 349;  25, 137, 159, 182, 248, 308, 348;  25, 138, 154, 185, 245, 313, 347;  25, 139, 156, 188, 249, 311, 346;  26, 28, 38, 42, 82, 88, 291;  26, 29, 40, 45, 79, 86, 290;  26, 30, 35, 48, 83, 84, 289;  26, 31, 37, 44, 80, 89, 288;  26, 32, 39, 47, 77, 87, 287;  26, 33, 41, 43, 81, 85, 293;  26, 34, 36, 46, 78, 90, 292;  26, 105, 122, 224, 243, 326, 333;  26, 106, 124, 227, 240, 324, 332;  26, 107, 119, 230, 244, 322, 331;  26, 108, 121, 226, 241, 327, 330;  26, 109, 123, 229, 238, 325, 329;  26, 110, 125, 225, 242, 323, 335;  26, 111, 120, 228, 239, 328, 334;  26, 133, 157, 182, 250, 312, 347;  26, 134, 159, 185, 247, 310, 346;  26, 135, 154, 188, 251, 308, 345;  26, 136, 156, 184, 248, 313, 344;  26, 137, 158, 187, 245, 311, 343;  26, 138, 160, 183, 249, 309, 349;  26, 139, 155, 186, 246, 314, 348;  27, 28, 37, 47, 79, 84, 293;  27, 29, 39, 43, 83, 89, 292;  27, 30, 41, 46, 80, 87, 291;  27, 31, 36, 42, 77, 85, 290;  27, 32, 38, 45, 81, 90, 289;  27, 33, 40, 48, 78, 88, 288;  27, 34, 35, 44, 82, 86, 287;  27, 105, 121, 229, 240, 322, 335;  27, 106, 123, 225, 244, 327, 334;  27, 107, 125, 228, 241, 325, 333;  27, 108, 120, 224, 238, 323, 332;  27, 109, 122, 227, 242, 328, 331;  27, 110, 124, 230, 239, 326, 330;  27, 111, 119, 226, 243, 324, 329;  27, 133, 156, 187, 247, 308, 349;  27, 134, 158, 183, 251, 313, 348;  27, 135, 160, 186, 248, 311, 347;  27, 136, 155, 182, 245, 309, 346;  27, 137, 157, 185, 249, 314, 345;  27, 138, 159, 188, 246, 312, 344;  27, 139, 154, 184, 250, 310, 343;  28, 29, 30, 31, 32, 33, 34;  35, 36, 37, 38, 39, 40, 41;  42, 43, 44, 45, 46, 47, 48;  49, 50, 51, 52, 53, 54, 55;  56, 57, 58, 59, 60, 61, 62;  63, 64, 65, 66, 67, 68, 69}.
The deficiency graph is connected and has girth 4.

}

%

\end{document}